\documentclass[12pt,twoside]{amsart}
\usepackage{geometry}
\usepackage{amssymb,amsmath,amsthm, amscd, enumerate, mathrsfs}
\usepackage{graphicx, hhline}
\usepackage[all]{xy}
\usepackage[dvipdfmx]{hyperref}
\usepackage{color}

\title[MMP for projective morphisms]
{Minimal model program for projective morphisms 
between complex analytic spaces}
\author{Osamu Fujino}
\date{2022/1/27, version 0.01}
\address{Department of 
Mathematics, Graduate School of Science, 
Kyoto University, Kyoto 606-8502, Japan}
\email{fujino@math.kyoto-u.ac.jp}

\DeclareMathOperator{\mult}{mult}
\DeclareMathOperator{\rank}{rank}
\DeclareMathOperator{\Nklt}{Nklt}
\DeclareMathOperator{\Supp}{Supp}
\DeclareMathOperator{\Mov}{\overline{Mov}}
\DeclareMathOperator{\NE}{\overline{NE}}
\DeclareMathOperator{\Amp}{Amp}
\DeclareMathOperator{\an}{an}
\DeclareMathOperator{\codim}{codim}
\DeclareMathOperator{\WDiv}{WDiv}
\DeclareMathOperator{\Exc}{Exc}
\DeclareMathOperator{\Fix}{Fix}
\DeclareMathOperator{\Pic}{Pic}
\DeclareMathOperator{\Hom}{Hom}
\DeclareMathOperator{\Sym}{Sym}
\DeclareMathOperator{\Projan}{Projan}
\DeclareMathOperator{\Coker}{Coker}
\DeclareMathOperator{\ddiv}{div}
\DeclareMathOperator{\NLC}{NLC}
\DeclareMathOperator{\Ker}{Ker}
\DeclareMathOperator{\NS}{NS}
\newtheorem{thm}{Theorem}[section]
\newtheorem{lem}[thm]{Lemma}
\newtheorem{cor}[thm]{Corollary}

\newtheorem{conj}[thm]{Conjecture}

\newtheorem{theorema}{Theorem}

\theoremstyle{definition}
\newtheorem{step}{Step}
\newtheorem{defn}[thm]{Definition}
\newtheorem{rem}[thm]{Remark}
\newtheorem{ex}[thm]{Example}
\newtheorem*{ack}{Acknowledgments}  
\newtheorem{say}[thm]{}

\makeatletter
    
    \@addtoreset{equation}{section}
\makeatother

\begin{document}

\maketitle 

\begin{abstract} 
We discuss the minimal model program for 
projective morphisms of complex analytic spaces. 
Roughly speaking, we show 
that the results obtained by 
Birkar--Cascini--Hacon--M\textsuperscript{c}Kernan 
hold true for projective 
morphisms between complex analytic spaces. 
We also treat some related topics. 
\end{abstract}

\tableofcontents

\section{Introduction}\label{a-sec1}
This paper is the first step of the minimal model program 
for projective morphisms of complex analytic spaces. 

In \cite{bchm} and \cite{hacon-mckernan}, 
Birkar, Cascini, Hacon, and M\textsuperscript{c}Kernan 
established many important results on the minimal model 
program for quasi-projective kawamata log terminal pairs 
defined over the complex number field 
(see \cite[Theorems A, B, C, D, E, and F]{bchm}). 
Thus we can run the minimal model program 
with scaling. 

\begin{thm}[{Minimal model program with scaling, see 
\cite[Corollary 1.4.2]{bchm}}]\label{a-thm1.1}
Let $\pi\colon X\to Y$ be a projective morphism 
of normal quasi-projective varieties. 
Let $(X, \Delta)$ be a $\mathbb Q$-factorial 
kawamata log terminal pair, where $K_X+\Delta$ 
is $\mathbb R$-Cartier and $\Delta$ is $\pi$-big. 
Let $C$ be an effective $\mathbb R$-divisor on $X$. 
If $K_X+\Delta+C$ is kawamata log terminal 
and $\pi$-nef, then we can run the $(K_X+\Delta)$-minimal 
model program over $Y$ with scaling of $C$. 
The output of this minimal model program is a log 
terminal model {\em{(}}resp.~a Mori fiber space{\em{)}} 
over $Y$ when $K_X+\Delta$ is 
$\pi$-pseudo-effective {\em{(}}resp.~not 
$\pi$-pseudo-effective{\em{)}}. 
\end{thm}

Hence we can easily check: 

\begin{thm}[{\cite[Theorem 1.2]{bchm}}]\label{a-thm1.2} 
Let $(X, \Delta)$ be a kawamata log terminal 
pair, where $K_X+\Delta$ is $\mathbb R$-Cartier. 
Let $\pi\colon X\to Y$ be a projective morphism 
of quasi-projective varieties. 
If either $\Delta$ is $\pi$-big and 
$K_X+\Delta$ is $\pi$-pseudo-effective or 
$K_X+\Delta$ is $\pi$-big, 
then 
\begin{itemize}
\item[(1)] $K_X+\Delta$ has a log terminal model over $Y$, 
\item[(2)] if $K_X+\Delta$ is $\pi$-big then 
$K_X+\Delta$ has a log canonical model over $Y$, and 
\item[(3)] if $K_X+\Delta$ is $\mathbb Q$-Cartier, 
then 
\begin{equation*}
R(X/Y, K_X+\Delta):=\bigoplus _{m\in \mathbb N} 
\pi_*\mathcal O\left(\lfloor m(K_X+\Delta)\rfloor\right)
\end{equation*}
is finitely generated as an $\mathcal O_Y$-algebra. 
\end{itemize}
\end{thm}

The main purpose of this paper is to generalize 
the results obtained in \cite{bchm} and 
\cite{hacon-mckernan} 
for projective morphisms of complex analytic spaces 
under some suitable assumptions. 
One of the main difficulties to translate \cite{bchm} 
and \cite{hacon-mckernan} 
into the complex analytic setting is to 
find a reasonable formulation. 
For the general understanding of the theory of minimal models 
in the complex analytic setting, see \cite[Example 2.17]{kollar-mori}. 
To the best knowledge of the author, the minimal model 
program for projective morphisms 
between complex analytic spaces is not discussed in 
standard literature. 

\begin{rem}[Minimal model program 
for compact K\"ahler threefolds]\label{a-rem1.3}
In a series of 
papers (see \cite{horing-peternell1}, 
\cite{horing-peternell2}, and \cite{chp}), 
the theory of minimal models was generalized 
for compact K\"ahler threefolds (see 
also \cite{horing-peternell3}). It is different from our direction 
and is another 
complex analytic generalization of the 
minimal model program. 
The minimal model theory for log surfaces in 
Fujiki's class $\mathcal C$ was described in 
\cite{fujino-classC}. 
Based on the idea that the essence of the theory of 
minimal models is projectivity, we think that 
our formulation in this paper is more natural than 
the minimal model program for compact K\"ahler varieties. 
\end{rem}

Let us see an easy example. 

\begin{ex}\label{a-ex1.4} 
Let $\{P_k\}_{k\in \mathbb N}$ be a set of 
mutually distinct discrete points of $Y:=\mathbb C^2$ and 
let $\pi\colon X\to Y$ be the blow-up whose 
center is $\{P_k\}_{k\in \mathbb N}$. 
We put $E_k:=\pi^{-1}(P_k)$ for every $k$. 
Since the line bundle 
$\mathcal O_X\left(-\sum _{k\in \mathbb N}E_k\right)$ is $\pi$-ample, 
$\pi$ is a projective bimeromorphic 
morphism of smooth complex surfaces. 
In this case, there are infinitely many mutually 
disjoint $\pi$-exceptional 
curves on $X$. 
Hence, there exists no naive generalization 
of the minimal model program working for this 
projective bimeromorphic morphism $\pi\colon X\to Y$. 
\end{ex}

By Example \ref{a-ex1.4}, 
it seems to be reasonable and indispensable to fix a compact 
subset $W$ of $Y$ with some good properties 
and only treat $\pi\colon X\to Y$ over some open 
neighborhood of $W$. We need some finiteness condition 
in order to formulate the minimal model program for 
projective morphisms of complex analytic spaces. 
In this paper, we will mainly 
consider a projective morphism 
$\pi\colon X\to Y$ of complex analytic spaces and 
a compact subset $W$ of $Y$ with the 
following properties: 
\begin{itemize}
\item[(P1)] $X$ is a normal complex variety, 
\item[(P2)] $Y$ is a Stein space, 
\item[(P3)] $W$ is a Stein compact subset of $Y$, and 
\item[(P4)] $W\cap Z$ has only finitely many connected 
components for any analytic subset 
$Z$ which is defined over an open neighborhood 
of $W$. 
\end{itemize}

Since we are trying to discuss the minimal 
model program, (P1) is indispensable. 
So we almost always assume that $X$ is a normal complex 
variety. 
We note that the Steinness of $Y$, that is, (P2),  is a 
substitute of the quasi-projectivity of $Y$ in \cite{bchm} 
and that the quasi-projectivity of $Y$ is indispensable 
in \cite{bchm}. 
Let $\mathcal F$ be a coherent sheaf on 
$X$ such that $\pi_*\mathcal F\ne 0$. 
Then we see that $\Gamma (X, \mathcal F)
=\Gamma (Y, \pi_*\mathcal F)\ne 0$ holds by Cartan's 
Theorem A. Here we need the Steinness of $Y$. 
We also note that the analytic space naturally associated to 
an affine scheme is Stein. 
Let us explain (P3) and (P4). 
A compact subset on an analytic space is said to be 
{\em{Stein compact}} if it admits a fundamental system of Stein open 
neighborhoods. Note that a holomorphically convex hull 
$\widehat K$ of a compact subset $K$ on a 
Stein space is a Stein compact subset. 
Therefore, we can find many Stein compact subsets on a 
given Stein space. If $W$ satisfies (P3) and we are only 
interested in 
objects defined over some open neighborhood of $W$, 
then we can freely replace $Y$ with a small Stein 
open neighborhood of $W$. 
The condition in (P4) is not so easy to understand and 
looks somewhat artificial. However, it is a very natural condition. 
It is known that a Stein compact subset $W$ 
satisfies (P4) if and only if $\Gamma (W, \mathcal O_Y)$ is 
noetherian by Siu's theorem 
(see \cite[Theorem 1]{siu}). 
If $W$ 
is a Stein compact semianalytic subset, 
then it satisfies (P3) and (P4). 
Thus we see that 
there are many Stein compact subsets satisfying (P4) 
on a given Stein space $Y$. 
We consider the free abelian group $Z_1(X/Y; W)$ generated 
by the projective integral curves $C$ on $X$ such that 
$\pi(C)$ is a point of $W$. 
We take $C_1, C_2\in Z_1(X/Y; W)\otimes _{\mathbb Z}\mathbb R$. 
If $C_1\cdot \mathcal L=C_2\cdot \mathcal L$ holds 
for every $\mathcal L\in \Pic\left(\pi^{-1}(U)\right)$ and 
every open neighborhood $U$ of $W$, 
then we write $C_1\equiv _W C_2$. 
We set $N_1(X/Y; W):=Z_1(X/Y; W)\otimes _{\mathbb Z} \mathbb R
/\equiv_{W}$. 
Then, by (P4), we can check that 
$N_1(X/Y; W)$ is a finite-dimensional 
$\mathbb R$-vector space 
(see \cite[Chapter II.~5.19.~Lemma]{nakayama3}). 
When $N_1(X/Y; W)$ is finite-dimensional, 
we can define 
the {\em{Kleiman--Mori cone}} 
$\NE(X/Y; W)$ in $N_1(X/Y; W)$ for $\pi\colon X\to Y$ and 
$W$, 
that is, 
$\NE(X/Y; W)$ is the closure of the convex cone 
in $N_1(X/Y; W)$ generated by projective integral curves $C$ 
on $X$ such that $\pi(C)$ is a point in $W$. 
Without any difficulties, we can establish Kleiman's 
ampleness criterion in the complex analytic setting. 
We further assume that there exists an $\mathbb R$-divisor 
$\Delta$ on $X$ such that 
$(X, \Delta)$ is kawamata log terminal. 
Then we can formulate the cone theorem as usual  
\begin{equation*}
\NE(X/Y; W)=\NE(X/Y; W)_{K_X+\Delta\geq 0}
+\sum _j R_j
\end{equation*} 
and prove the contraction theorem for 
each $(K_X+\Delta)$-negative extremal ray $R_j$ 
over some open 
neighborhood of $W$. 
This is essentially due to Nakayama (see \cite{nakayama2}). 
By the above observations, we recognize that 
(P1), (P2), (P3), and (P4) are reasonable. 
From now on, we usually simply say that 
$\pi\colon X\to Y$ and $W$ satisfies (P) if it 
satisfies (P1), (P2), (P3), and (P4). 

\begin{rem}\label{a-rem1.5}
Let $D$ be an $\mathbb R$-Cartier $\mathbb R$-divisor on $X$. 
Then $D\geq 0$ on $\NE(X/Y; W)$ means 
that $D\cdot C\geq 0$ for every projective integral curve $C$ on 
$X$ such that $\pi(C)$ is a point in $W$. 
Unfortunately, however, $D$ is not necessarily nef over some open 
neighborhood of $W$ even when 
$D\geq 0$ on $\NE(X/Y; W)$. 
This fact often causes troublesome problems 
when we discuss the minimal model program 
for projective morphisms between 
complex analytic spaces. 
\end{rem}

In this paper, we will prove: 

\begin{thm}[Main theorem, see 
Theorem \ref{a-thm1.13} below]\label{a-thm1.6}
Let $\pi\colon X\to Y$ be a projective morphism 
of complex analytic spaces and let $W$ be 
a compact subset of $Y$ such that 
$\pi\colon X\to Y$ and $W$ satisfies {\em{(P)}}. 
Then Theorems A, B, C, D, E, and F in \cite{bchm} 
hold true with some suitable modifications. 
\end{thm}

For the precise statement, 
see Theorems \ref{thm-a}, \ref{thm-b}, 
\ref{thm-c}, \ref{thm-d}, \ref{thm-e}, and 
\ref{thm-f} in Subsection \ref{a-subsec1.1} and 
Theorem \ref{a-thm1.13} below. 
To the best knowledge of the author, 
our results are new even in dimension three. 
For projective morphisms of complex analytic spaces, 
the minimal model program with scaling (see 
Theorem \ref{a-thm1.1}) becomes as follows. 

\begin{thm}[{Minimal model program 
with scaling for projective morphisms of complex 
analytic spaces, see \cite[Corollary 1.4.2]{bchm}}]\label{a-thm1.7}
Let $\pi\colon X\to Y$ be a projective surjective 
morphism of complex analytic spaces and let $W$ be a compact 
subset of $Y$. 
Assume that $Y$ is Stein and that $W$ is a Stein 
compact subset of $Y$ such that $\Gamma(W, \mathcal O_Y)$ is 
noetherian. 
Let $(X, \Delta)$ be a kawamata log terminal 
pair such that $\Delta$ is $\pi$-big and 
that 
$X$ is $\mathbb Q$-factorial over $W$. 
If $C$ is an effective $\mathbb R$-divisor 
on $X$ such that 
$K_X+\Delta+C$ is kawamata log terminal and it is nef 
over $W$. 
Then we can run the $(K_X+\Delta)$-minimal 
model program with scaling of $C$ over $Y$ 
around $W$. 
More precisely, we have a finite sequence of flips and divisorial 
contractions over $Y$ 
starting from $(X, \Delta)$: 
\begin{equation*}
(X, \Delta)=:(X_0, \Delta_0)\overset{\varphi_0}{\dashrightarrow} 
(X_1, \Delta_1)\overset{\varphi_1}{\dashrightarrow} 
\cdots \overset{\varphi_{m-1}}{\dashrightarrow} 
(X_m ,\Delta_m),  
\end{equation*}
where $\Delta_{i+1}:=(\varphi_i)_*\Delta_i$ for every $i\geq 0$, 
such that $(X_m, \Delta_m)$ is a log terminal model 
{\em{(}}resp.~has a Mori fiber space structure{\em{)}} over some 
open neighborhood of $W$ when $K_X+\Delta$ is 
$\pi$-pseudo-effective {\em{(}}resp.~not 
$\pi$-pseudo-effective{\em{)}}. 
We note that each step $\varphi_i$ exists only after shrinking 
$Y$ around $W$ suitably. 
Hence we have to replace $Y$ with a small Stein open 
neighborhood 
of $W$ repeatedly in the above process. 
\end{thm}

By Theorem \ref{a-thm1.7}, we can prove:  

\begin{thm}[{\cite[Theorem 1.2]{bchm}}]\label{a-thm1.8}
Let $(X, \Delta)$ be a kawamata log terminal pair, where 
$K_X+\Delta$ is $\mathbb R$-Cartier. 
Let $\pi\colon X\to Y$ be a projective morphism 
of complex analytic spaces and let $W$ be 
a compact subset of $Y$ such that 
$\pi\colon X\to Y$ and $W$ satisfies {\em{(P)}}. 
If either $\Delta$ is $\pi$-big and 
$K_X+\Delta$ is $\pi$-pseudo-effective or $K_X+\Delta$ 
is $\pi$-big, then 
\begin{itemize}
\item[(1)] $K_X+\Delta$ has a log terminal model 
over some open neighborhood of $W$, 
\item[(2)] if $K_X+\Delta$ is $\pi$-big then 
$K_X+\Delta$ has a log canonical model 
over some open neighborhood of $W$, and 
\item[(3)] if $K_X+\Delta$ is $\mathbb Q$-Cartier, 
then 
\begin{equation*}
R(X/Y, K_X+\Delta)=\bigoplus _{m\in \mathbb N}
\pi_*\mathcal O_X(\lfloor m(K_X+\Delta)\rfloor)
\end{equation*}
is a locally finitely generated graded $\mathcal O_Y$-algebra. 
\end{itemize}
\end{thm}

Of course, Theorem \ref{a-thm1.8} is an analytic version of 
Theorem \ref{a-thm1.2}. 
We note that this paper is not self-contained. 
If the proof is essentially the same as the original one 
for quasi-projective varieties, then 
we will only explain how to modify 
arguments in \cite{bchm} and \cite{hacon-mckernan} 
in order to make them work 
for projective morphisms between analytic spaces 
satisfying (P). 
In this paper, we always assume that 
complex analytic spaces are {\em{Hausdorff}} and 
{\em{second-countable}}. 

\begin{say}[Motivation]\label{a-say1.9}
Let us explain the motivation of this paper. 
We sometimes have to consider the following 
setting when we study degenerations of 
algebraic varieties. 
Let $\pi\colon X\to \Delta$ be a projective morphism 
from a complex manifold $X$ onto a disc $\Delta=\{z
\in \mathbb C\, |\, |z|<1\}$ with connected 
fibers. Suppose that $\pi$ is smooth over $\Delta\setminus 
\{0\}$ and $K_{X_z}\sim _{\mathbb Q}0$ for every $z\in 
\Delta\setminus \{0\}$, where $X_z=\pi^{-1}(z)$, and 
that $\pi^*0$ is a reduced simple normal crossing divisor on $X$. 
Roughly speaking, $\pi\colon X\setminus \pi^{-1}(0)\to \Delta
\setminus \{0\}$ is a smooth projective family of 
Calabi--Yau manifolds and $\pi\colon X\to \Delta$ 
is a semistable degeneration. 
Since $\Delta$ is not a quasi-projective 
algebraic variety, we can not directly use \cite{bchm} for 
the study of $\pi\colon X\to \Delta$. By using 
the results established in this paper, 
after slightly shrinking $\Delta$ around $0$ repeatedly, 
we can construct a finite sequence of 
flips and divisorial contractions starting from $X$: 
\begin{equation*}
X=:X_0\dashrightarrow X_1\dashrightarrow X_2\dashrightarrow 
\cdots \dashrightarrow X_m 
\end{equation*}
over $\Delta$ such that $X_m$ is a minimal 
model of $X$ over $\Delta$. 
More precisely, $X_m$ has only terminal singularities 
and $K_{X_m}$ is $\mathbb Q$-linearly trivial. 
The above result is a typical application of our result in this paper, 
which is not covered by \cite{bchm}. 
It is a complex analytic generalization of 
the semistable minimal model program established 
in \cite{fujino-semistable}. 
\end{say}

\begin{say}[{Background, see \cite[3.5, 3.6]{fujino-fundamental}}]
\label{a-say1.10}
In the traditional framework of the minimal model program, 
the most important and natural object is a 
quasi-projective kawamata log terminal pair 
(see \cite{kmm}, \cite{kollar-mori}, \cite{matsuki}, \cite{bchm}, 
\cite{hacon-mckernan}, and so on). 
From the Hodge theoretic viewpoint, 
we think that there exists the following correspondence. 
\begin{equation*}
 \fbox{
\begin{tabular}{c}
\begin{minipage}{5.2cm}
Kawamata log terminal pairs
\end{minipage}
\end{tabular}
}
\Longleftrightarrow 
\fbox{
 \begin{tabular}{c}
\begin{minipage}{4.0cm}
Pure Hodge structures
\end{minipage}
\end{tabular}
}
\end{equation*}
We have already used mixed 
Hodge structures on cohomology with compact 
support systematically for the study of minimal models 
(see \cite{fujino-fundamental}, \cite{fujino-foundations}, and so on). 
Then we succeeded in greatly expanding the framework 
of the minimal model program. Roughly speaking, 
we established the following correspondence. 
\begin{equation*}
 \fbox{
\begin{tabular}{c}
\begin{minipage}{3.3cm}
Quasi-log schemes
\end{minipage}
\end{tabular}
}
\Longleftrightarrow 
\fbox{
 \begin{tabular}{c}
\begin{minipage}{4.3cm}
Mixed Hodge structures
\end{minipage}
\end{tabular}
}
\end{equation*}
For the details of this direction, see also 
\cite{fujino-slc-trivial}, \cite{fujino-fujisawa-liu}, 
\cite{fujino-cone}, and so on. 

On the other hand, from the complex analytic 
viewpoint, we know the following correspondence. 
\begin{equation*}
 \fbox{
\begin{tabular}{c}
\begin{minipage}{5.2cm}
Kawamata log terminal pairs
\end{minipage}
\end{tabular}
}
\Longleftrightarrow 
\fbox{
 \begin{tabular}{c}
\begin{minipage}{2.3cm}
$L^2$ condition
\end{minipage}
\end{tabular}
}
\end{equation*}
Hence, it is natural to think that we can generalize 
the minimal model program for quasi-projective kawamata log 
terminal pairs established in \cite{bchm} 
and \cite{hacon-mckernan} to the one in 
the complex analytic setting. 
We note that the projectivity plays a crucial role in the theory 
of minimal models. 
Therefore, it is reasonable to discuss the minimal model program 
for projective morphisms between complex analytic spaces. 
This naive idea is now realized in this paper. 

Unfortunately, we have not established 
complex analytic methods to treat varieties 
whose singularities are worse than kawamata log terminal singularities yet. 
Thus, it is a challenging problem 
to consider some analytic generalization 
of the theory of quasi-log schemes (see \cite[Chapter 6]{fujino-foundations} 
and \cite{fujino-on-quasi}). 
\end{say}

\begin{say}[How to use (P)]\label{a-say1.11}
Before we see the main results in 
Subsection \ref{a-subsec1.1}, let us explain how to use 
(P) for the reader's convenience. 
Let $\pi\colon X\to Y$ be a projective morphism 
of complex analytic spaces and let $W$ be a 
compact subset of $Y$  such that $\pi\colon X\to Y$ and $W$ 
satisfies (P). 
We usually consider an $\mathbb R$-divisor $\Delta$ on $X$ such 
that $K_X+\Delta$ is $\mathbb R$-Cartier. 
It sometimes happens that some properties hold true only 
over an open neighborhood $U$ of $W$. 
Since $W$ is a Stein compact subset of $Y$, we can always 
take a relatively compact Stein open neighborhood $Y'$ 
of $W$ in $Y$ satisfying 
\begin{equation*}
W \subset Y' \Subset U \subset Y. 
\end{equation*} 
Of course, $\pi'\colon X'\to Y'$ and $W$ satisfies (P), 
where $X':=\pi^{-1}(Y')$ and 
$\pi':=\pi|_{X'}$. 
We frequently replace $\pi\colon X\to Y$ with 
$\pi'\colon X'\to Y'$ without mentioning it explicitly. 
We note that the support of $\Delta$ 
is only locally finite by definition. 
In general, the support of $\Delta$ may have 
infinitely many irreducible components. 
By construction, $X'$ is a relatively compact 
open subset of $X$. 
Hence the support of $\Delta|_{X'}$ is finite. 
On the other hand, let $V$ be a relatively compact 
open neighborhood of $W$ in $Y$. 
Since $Y$ is Stein, we can take an Oka--Weil domain $V'$ 
satisfying 
\begin{equation*}
W\subset V\subset \overline {V} \subset 
V'\subset \overline{V'}\subset Y. 
\end{equation*} 
By construction, $V'$ can be seen as a 
complex analytic subspace 
of a polydisc. Hence we can take a 
semianalytic Stein compact subset $W'$ such that 
\begin{equation*}
W\subset V\subset \overline{V}\subset 
W' \subset V' \subset 
\overline {V'}\subset Y. 
\end{equation*} 
Note that $W'$ satisfies (P4) since 
it is semianalytic. 
Therefore, $\pi\colon X\to Y$ and $W'$ satisfies 
(P) and $W'$ contains a given 
relatively compact open neighborhood 
$V$ of $W$. 
This argument is useful and indispensable 
when we try to check that some properties 
hold true over an open neighborhood 
of $W$. 
For example, when we prove that $K_X+\Delta$ 
is nef over some open neighborhood 
of $W$, we sometimes have to consider $\pi\colon X\to Y$ and $W'$. 
\end{say}

\subsection{Main results}\label{a-subsec1.1}
Here, we state the main results of this paper. 
The following theorems look very similar to 
those in \cite{bchm} although they treat complex analytic spaces. 

\begin{theorema}[{Existence of pl-flips, \cite[Theorem A]{bchm}}]
\label{thm-a}
Let $\varphi\colon X\to Z$ be an analytic pl-flipping 
contraction for a 
purely log terminal pair $(X, \Delta)$. 
Then the flip $\varphi^+\colon X^+\to Z$ of $\varphi$ always 
exists. 
\end{theorema}

\begin{theorema}[{Special finiteness, \cite[Theorem B]{bchm}}]
\label{thm-b}
Let $\pi\colon X\to Y$ be a projective morphism 
of complex analytic spaces and let $W$ be a compact 
subset of $Y$ such that $\pi\colon X\to Y$ and $W$ satisfies 
{\em{(P)}}. Suppose that 
$X$ is $\mathbb Q$-factorial over $W$. 
Let $V$ be a finite-dimensional affine subspace of 
$\WDiv_{\mathbb R}(X)$, which is defined over the rationals, 
let $S$ be the sum of finitely many prime divisors and let $A$ be 
a general $\pi$-ample $\mathbb Q$-divisor 
on $X$ such that the number of the irreducible components 
of $\Supp A$ is finite. 
Let $(X, \Delta_0)$ be a divisorial log terminal 
pair such that $S\leq \Delta_0$. 
We fix a finite set $\mathfrak C$ of prime divisors 
on $X$. 
Then, after shrinking $Y$ around $W$ suitably, 
there are finitely many bimeromorphic maps 
$\phi_i\colon X\dashrightarrow Z_i$ over $Y$ for 
$1\leq i\leq k$ with 
the following property. 
If $U$ is an open neighborhood of $W$ and 
$\phi\colon \pi^{-1}(U)\dashrightarrow Z$ is any 
weak log canonical model over $W$ 
of $(K_X+\Delta)|_{\pi^{-1}(U)}$ 
such that $Z$ is $\mathbb Q$-factorial over $W$, 
where $\Delta\in \mathcal L_{S+A}(V; \pi^{-1}(W))$, 
which only contracts elements of $\mathfrak C$ and 
which does not contract every component of $S$, 
then there exists an index $1\leq i\leq k$ such that, 
after shrinking $Y$ and $U$ around $W$ suitably, 
the induced bimeromorphic map 
$\xi\colon Z_i\dashrightarrow 
Z$ is an isomorphism 
in a neighborhood of the strict transform of $S$. 
\end{theorema}

\begin{theorema}[{Existence of log terminal 
models, \cite[Theorem C]{bchm}}]\label{thm-c}
Let $\pi\colon X\to Y$ be a projective 
surjective morphism of complex analytic spaces 
and let $W$ be a compact subset of $Y$ 
such that $\pi\colon X\to Y$ and $W$ satisfies 
{\em{(P)}}. 
Suppose that $(X, \Delta)$ is kawamata log terminal 
and that $\Delta$ is big over $Y$. 
If there exists an $\mathbb R$-divisor 
$D$ on $X$ such that $K_X+\Delta\sim _{\mathbb R} D\geq 0$, 
then $(X, \Delta)$ has a log terminal model over 
some open neighborhood of $W$. 
\end{theorema}

\begin{theorema}[{Nonvanishing theorem, \cite[Theorem D]{bchm}}]
\label{thm-d}
Let $(X, \Delta)$ be a kawamata log terminal pair and let 
$\pi\colon X\to Y$ be a projective morphism 
of complex analytic spaces such that $Y$ is Stein. 
Assume that $\Delta$ is big over $Y$ and that 
$K_X+\Delta$ 
is pseudo-effective over $Y$. 
Let $U$ be any relatively compact Stein open subset of $Y$. 
Then there exists a globally $\mathbb R$-Cartier 
$\mathbb R$-divisor $D$ on 
$\pi^{-1}(U)$ such that $(K_X+\Delta)|_{\pi^{-1}(U)}
\sim _{\mathbb R}D\geq 0$. 
\end{theorema}

\begin{theorema}[{Finiteness of models, \cite[Theorem E]{bchm}}]
\label{thm-e}
Let $\pi\colon X\to Y$ be a projective 
morphism of complex analytic spaces 
and let $W$ be a compact 
subset of $Y$ such that 
$\pi\colon X\to Y$ and $W$ satisfies {\em{(P)}}. 
We fix a general $\pi$-ample $\mathbb Q$-divisor $A\geq 0$ on $X$ 
such that the number of the irreducible components 
of $\Supp A$ is finite. 
Let $V$ be a finite-dimensional 
affine subspace of $\WDiv_{\mathbb R} (X)$ which 
is defined over the rationals. 
Suppose that there is a kawamata log terminal 
pair $(X, \Delta_0)$. 
Then, after shrinking $Y$ around 
$W$ suitably, there are finitely many bimeromorphic 
maps $\psi_j\colon X\dashrightarrow Z_j$ over $Y$ for $1\leq j\leq 
l$ with the following property. 
If $U$ is an open neighborhood of $W$ and $\psi\colon 
\pi^{-1}(U)\dashrightarrow Z$ is 
a weak log canonical model of $(K_X+\Delta)|_{\pi^{-1}(U)}$ 
over $W$ for some 
$\Delta\in \mathcal L_A(V; \pi^{-1}(W))$, then 
there exists an index $1\leq j\leq l$ and an isomorphism 
$\xi\colon Z_j\to Z$ such that 
$\psi=\xi\circ \psi_j$ after shrinking $Y$ and $U$ around 
$W$ suitably.  
\end{theorema}

\begin{theorema}[{Finite generation, \cite[Theorem F]{bchm}}]
\label{thm-f}
Let $\pi\colon X\to Y$ be a projective morphism of complex 
analytic spaces and let $W$ be a compact subset of $Y$ 
such that $\pi\colon X\to Y$ and $W$ satisfies 
{\em{(P)}}. 
Let $(X, \Delta =A+B)$ be a kawamata 
log terminal pair, 
where $A\geq 0$ is 
a $\pi$-ample $\mathbb Q$-divisor and $B\geq 0$. 
We assume that the number of the irreducible components of 
$\Supp \Delta$ is finite. 
If $K_X+\Delta$ is pseudo-effective over $Y$, then 
\begin{itemize}
\item[(1)] After shrinking $Y$ around $W$ suitably, 
the pair $(X, \Delta)$ has a log terminal model 
$\mu\colon X\dashrightarrow Z$ over $Y$. 
In particular if $K_X+\Delta$ is $\mathbb Q$-Cartier, then the 
log canonical ring 
\begin{equation*}
R(X/Y, K_X+\Delta)=\bigoplus_{m\in \mathbb N} 
\pi_*\mathcal O_X(\lfloor m(K_X+\Delta)\rfloor)
\end{equation*} 
is a locally finitely generated graded $\mathcal O_Y$-algebra. 
\item[(2)] Let $V\subset \WDiv_{\mathbb R} (X)$ be the 
vector space spanned by the components of $\Delta$. 
Then, after shrinking $Y$ around $W$ suitably, 
there is a constant $\delta>0$ such that 
if $G$ is a prime divisor contained in 
the stable base locus of $K_X+\Delta$ 
over $Y$ and $\Xi\in 
\mathcal L_A(V; \pi^{-1}(W))$ such that 
$|\!| \Xi-\Delta|\!|<\delta$, 
then $G$ is contained in the stable base locus of 
$K_X+\Xi$ over $Y$. 
\item[(3)] 
Let $V'\subset V$ be the smallest affine subspace of 
$\WDiv_{\mathbb R}(X)$ 
containing $\Delta$, which is defined over the 
rationals. 
Then, after shrinking $Y$ around $W$ suitably, 
there exists a constant $\eta>0$ and a 
positive integer $r>0$ such that 
if $\Xi\in V'$ is any divisor 
and $k$ is any positive integer such that 
$|\!| \Xi-\Delta|\!|<\eta$ and $k(K_X+\Xi)/r$ is 
Cartier, then every component of $\Fix(k(K_X+\Xi))$ 
is a component of the stable base locus 
of $K_X+\Delta$ over $Y$. 
\end{itemize}
\end{theorema}

\subsection{How to prove the main results}\label{a-subsec1.2} 
The formulation of Theorem \ref{thm-c} 
is not appropriate for our inductive treatment of 
the main results in this paper. 
Therefore, we prepare a somewhat artificial statement, 
which is a slight generalization 
of Theorem \ref{thm-c}. 
We will use it instead of Theorem \ref{thm-c} in the inductive proof of 
the main results.

\begin{theorema}[Existence of good log terminal models]\label{thm-g} 
Let $\pi\colon X\to Y$ be a projective 
surjective morphism of complex analytic spaces 
and let $W$ be a compact subset of $Y$ 
such that $\pi\colon X\to Y$ and $W$ satisfies 
{\em{(P)}}. Assume that $\pi\colon X\to Y^\flat \to Y$ such that 
$Y^\flat$ is projective over $Y$. 
Suppose that $(X, \Delta)$ is kawamata log terminal 
and that $\Delta$ is big over $Y$. 
If there exists an $\mathbb R$-divisor 
$D$ on $X$ such that $K_X+\Delta\sim _{\mathbb R} D\geq 0$, 
then, after shrinking $Y$ around $W$ suitably,  
there exists a bimeromorphic contraction 
$\phi\colon X\dashrightarrow Z$ over $Y^\flat$ 
such that $\phi$ is $(K_X+\Delta)$-negative, $Z$ is $\mathbb Q$-factorial 
over $W$, $K_Z+\Gamma$ is semiample 
over $Y^\flat$, where $\Gamma =\phi_*\Delta$. 
This means that $(Z, \Gamma)$ is a good log terminal model 
of $(X, \Delta)$ over $Y^\flat$. 
\end{theorema}

As an obvious remark, we have: 

\begin{rem}\label{a-rem1.12}
If we put $Y^\flat=Y$ in Theorem \ref{thm-g}, then we 
obtain Theorem \ref{thm-c} as a special case of Theorem \ref{thm-g}. 
Therefore, it is sufficient to prove Theorem \ref{thm-g}. 
\end{rem}

Note that 
Theorem \ref{thm-a}$_n$ refers to Theorem \ref{thm-a} in the case when the 
dimension of $X$ is $n$. 
In \cite{bchm} and \cite{hacon-mckernan}, 
Theorem A, Theorem B, Theorem C, Theorem D, 
Theorem E, and Theorem F were proved by induction on $n$ as follows. 

\begin{itemize}
\item Theorem F$_{n-1}$ implies Theorem A$_n$. 
\item Theorem E$_{n-1}$ implies Theorem B$_n$. 
\item Theorem A$_n$ and Theorem B$_n$ imply Theorem C$_n$. 
\item Theorem D$_{n-1}$, Theorem B$_n$ and Theorem C$_n$  imply Theorem D$_n$. 
\item Theorem C$_n$ and Theorem D$_n$ imply Theorem E$_n$. 
\item Theorem C$_n$ and Theorem D$_n$ imply Theorem F$_n$. 
\end{itemize}
Our strategy in this paper is essentially the same as that of \cite{bchm}. 
However, it is slightly simpler. 
We first note that we can easily check: 
\begin{itemize}
\item Theorem \ref{thm-d}$_n$ holds true for arbitrary $n$. 
\item Theorem \ref{thm-g}$_n$ implies 
Theorem \ref{thm-c}$_n$ for arbitrary $n$. 
\end{itemize}
Hence it is sufficient to prove Theorem \ref{thm-a}, 
Theorem \ref{thm-b}, Theorem \ref{thm-e}, Theorem 
\ref{thm-f}, and Theorem \ref{thm-g} by induction on $n$ 
as follows. 
\begin{itemize}
\item 
Theorem \ref{thm-f}$_{n-1}$ implies Theorem 
\ref{thm-a}$_n$. 
\item 
Theorem \ref{thm-e}$_{n-1}$ implies Theorem \ref{thm-b}$_{n}$. 
\item 
Theorem \ref{thm-a}$_n$ and 
Theorem \ref{thm-b}$_n$ imply Theorem \ref{thm-g}$_{n}$. 
\item 
Theorem \ref{thm-g}$_n$ 
implies Theorem \ref{thm-e}$_n$.
\item Theorem \ref{thm-g}$_n$ 
implies Theorem \ref{thm-f}$_n$.
\end{itemize}
Although there are some new difficulties, the proof of 
each step is similar to the original one in \cite{bchm} 
and \cite{hacon-mckernan}. Precisely speaking, 
we make great efforts to find a suitable formulation in 
order to make the original proof work 
with only some minor modifications.  

The correct statement of 
Theorem \ref{a-thm1.6} should be: 

\begin{thm}[Main theorem]\label{a-thm1.13}
Theorems \ref{thm-a}, \ref{thm-b}, \ref{thm-c}, 
\ref{thm-d}, \ref{thm-e}, \ref{thm-f}, and \ref{thm-g} 
hold true in any dimension. 
\end{thm}

\subsection{Some other results}\label{a-subsec1.3}

We have already known that many results follow from 
\cite{bchm}. 
We can prove that some of them hold true even 
in the complex analytic setting if we take some 
suitable modifications. 

Once we have Theorem \ref{thm-g}, it is 
not difficult to prove the existence of kawamata 
log terminal flips in the complex analytic setting. 

\begin{thm}[Existence of kawamata 
log terminal flips]\label{a-thm1.14}
Let $(X, \Delta)$ be a kawamata log terminal pair. 
Let $\varphi\colon X\to Z$ be a small projective surjective 
morphism of normal complex varieties. 
Then the flip $\varphi^+\colon X^+\to Z$ 
of $\varphi$ always exists. This means that 
there exists the following commutative diagram:  
$$
\xymatrix{
X \ar[dr]_-\varphi\ar@{-->}[rr]^-\phi& & \ar[dl]^-{\varphi^+}X^+\\ 
& Z &
}
$$
where 
\begin{itemize}
\item[(1)] $\varphi^+\colon X^+\to Z$ is a small 
projective morphism of normal complex varieties, and 
\item[(2)] $K_{X^+}+\Delta^+$ is $\varphi^+$-ample, 
where $\Delta^+:=\phi_*\Delta$. 
\end{itemize}
Note that $(X^+, \Delta^+)$ is automatically 
a kawamata log terminal pair. 
\end{thm}

\begin{rem}\label{a-rem1.15}
Theorem \ref{a-thm1.14} generalizes 
Mori's flip theorem (see \cite[(0.4.1) Flip Theorem]{mori-flip}). 
Roughly speaking, Mori coarsely classified three-dimensional 
flipping contractions analytically and checked the existence of three-dimensional terminal flips. 
\end{rem}

The next one is a result on partial resolutions of 
singularities for complex varieties. 

\begin{thm}[Existence of canonicalizations]\label{a-thm1.16}
Let $X$ be a complex variety. 
Then there exists a projective 
bimeromorphic morphism $f\colon Z\to X$, which is 
the identity map over a nonempty Zariski open 
subset where $X$ has only canonical singularities, 
from a normal complex variety $Z$ with 
only canonical singularities such that 
$K_Z$ is $f$-ample. 
\end{thm}

If $K_X+\Delta$ is not pseudo-effective over $Y$, 
then we can run the minimal model program with scaling to get a Mori 
fiber space. 

\begin{thm}[{\cite[Corollary 1.3.3]{bchm}}]\label{a-thm1.17} 
Let $\pi\colon X\to Y$ be a projective morphism 
of complex analytic spaces and let $W$ be a compact 
subset of $Y$ such that 
$\pi \colon X\to Y$ and $W$ satisfies {\em{(P)}}. 
Let $(X, \Delta)$ be a divisorial log terminal pair such that 
$X$ is $\mathbb Q$-factorial over $W$. 
Suppose that $K_X+\Delta$ is not pseudo-effective over 
$Y$. Then we can run a $(K_X+\Delta)$-minimal model 
program and finally obtain a Mori fiber space over some 
open neighborhood of $W$. 
\end{thm}

We can prove the finite generation theorem 
for kawamata log 
terminal pairs in full generality in the complex analytic 
setting. 

\begin{thm}[{\cite[Corollary 1.1.2]{bchm}}]\label{a-thm1.18}
Let $(X, \Delta)$ be a kawamata log terminal pair, where 
$K_X+\Delta$ is $\mathbb Q$-Cartier, and 
let $\pi\colon X\to Y$ be a projective morphism 
of complex analytic spaces. 
Then 
\begin{equation*}
R(X/Y, K_X+\Delta)=\bigoplus _{m\in 
\mathbb N}\pi_*\mathcal O_X(\lfloor 
m(K_X+\Delta)\rfloor)
\end{equation*}
is a locally finitely generated graded $\mathcal O_Y$-algebra. 
\end{thm}

The original algebraic version of Theorem \ref{a-thm1.19} below 
was first obtained in \cite{fujino-semistable} as an application of 
\cite{bchm}. 

\begin{thm}[{\cite[Theorem 1.3]{fujino-semistable}}]
\label{a-thm1.19}
Let $(X, \Delta)$ be a divisorial 
log terminal pair and let $\pi\colon X\to Y$ be a projective 
morphism onto a disc $Y=\{z\in \mathbb C\, |\, |z|<1\}$ with 
connected fibers. Assume that 
$(K_X+\Delta)|_F\sim _{\mathbb R}0$ 
holds for an analytically 
sufficiently general fiber $F$ of $\pi$. 
We further assume that $W$ is a Stein compact subset of $Y$ such that 
$\Gamma (W, \mathcal O_Y)$ is noetherian and 
that $X$ is $\mathbb Q$-factorial over $W$. 
Then we can run the $(K_X+\Delta)$-minimal 
model program over $Y$ in a neighborhood 
of $W$ with ample scaling. 
More precisely, we have a finite sequence of flips and divisorial 
contractions over $Y$ 
starting from $(X, \Delta)$: 
\begin{equation*}
(X, \Delta)=:(X_0, \Delta_0)\overset{\varphi_0}{\dashrightarrow} 
(X_1, \Delta_1)\overset{\varphi_1}{\dashrightarrow} 
\cdots \overset{\varphi_{m-1}}{\dashrightarrow} 
(X_m ,\Delta_m),  
\end{equation*}
where $\Delta_{i+1}:=(\varphi_i)_*\Delta_i$ for every $i\geq 0$, 
such that $(X_m, \Delta_m)$ is a log terminal model over some 
open neighborhood of $W$. 
Moreover, 
$K_{X_m}+\Delta_m\sim_{\mathbb R} (\pi_m)^*D$ 
for some 
$\mathbb R$-Cartier $\mathbb R$-divisor $D$ on 
$Y$ 
after shrinking $Y$ around $W$ suitably, where $\pi_m\colon X_m\to Y$. 
We note that each step $\varphi_i$ exists only after shrinking 
$Y$ around $W$ suitably. 
\end{thm}

\begin{rem}\label{a-rem1.20}
In Theorem \ref{a-thm1.19}, 
$W=\{z\in \mathbb C\, |\, |z|\leq r\}$ for $0\leq r<1$ is a 
Stein compact subset of $Y$ such that 
$\Gamma (W, \mathcal O_Y)$ is noetherian. 
\end{rem}

The following theorem is an analytic version of 
{\em{dlt blow-ups}}. In the recent developments of the minimal model 
theory for higher-dimensional algebraic varieties, dlt blow-ups are 
very useful and important. 

\begin{thm}[Dlt blow-ups, I]\label{a-thm1.21}
Let $X$ be a normal complex variety and let $\Delta$ be 
an effective $\mathbb R$-divisor 
on $X$ such that 
$K_X+\Delta$ is $\mathbb R$-Cartier. 
Let $U$ be any relatively compact Stein open subset of $X$ 
and let $V$ be any relatively compact open subset of $U$. 
Then we can take a Stein compact subset $W$ of $U$ 
such that $\Gamma (W, \mathcal O_X)$ is noetherian, 
$V\subset W$, and, after shrinking $X$ around $W$ suitably, 
we can construct a projective bimeromorphic 
morphism 
$f\colon Z\to X$ from a normal complex variety 
$Z$ with the following properties: 
\begin{itemize}
\item[(1)] $Z$ is $\mathbb Q$-factorial over $W$, 
\item[(2)] $a(E, X, \Delta)\leq -1$ for every 
$f$-exceptional divisor $E$ on $Z$, and 
\item[(3)] $\left(Z, \Delta^{<1}_Z+\Supp \Delta^{\geq 1}_Z\right)$ is 
divisorial log terminal, where $K_Z+\Delta_Z=f^*(K_X+\Delta)$. 
\end{itemize}
Note that if $(X, \Delta)$ is log canonical 
then $\Delta_Z=\Delta^{<1}_Z+\Supp \Delta^{\geq 1}_Z$ holds.  
\end{thm}

We can use Theorem \ref{a-thm1.21} for the study of 
log canonical singularities, which are not necessarily algebraic. 
The following result is a generalization of 
Theorem \ref{a-thm1.18} for K\"ahler manifolds. 
When $Y$ is a point, Theorem \ref{a-thm1.22} is 
\cite[Theorem 1.8]{fujino-some}.  

\begin{thm}[{Finite generation for K\"ahler manifolds, see 
\cite{fujino-some}}]\label{a-thm1.22}
Let $\pi\colon X\to Y$ be a proper morphism from a K\"ahler 
manifold $X$ to a complex analytic space $Y$. 
Let $\Delta$ be an effective $\mathbb Q$-divisor 
on $X$ such that $\lfloor \Delta\rfloor =0$ and that 
$\Supp \Delta$ is a simple normal crossing divisor on $X$. 
Then 
\begin{equation*}
R(X/Y, K_X+\Delta)=\bigoplus _{m\in \mathbb N} 
\pi_*\mathcal O_X(\lfloor m(K_X+\Delta)\rfloor)
\end{equation*} 
is a locally finitely generated graded $\mathcal O_Y$-algebra. 
\end{thm}

The assumption that $\lfloor \Delta\rfloor=0$ holds in 
Theorem \ref{a-thm1.22} is very important. 
The following conjecture is widely open even when 
$\pi\colon X\to Y$ is a projective morphism 
between quasi-projective varieties (see \cite{fujino-gongyo2}). 

\begin{conj}\label{a-conj1.23}
Let $\pi\colon X\to Y$ be a proper morphism from a K\"ahler 
manifold $X$ to a complex analytic space $Y$. 
Let $\Delta$ be a boundary $\mathbb Q$-divisor 
on $X$ such that 
$\Supp \Delta$ is a simple normal crossing divisor on $X$. 
Then 
\begin{equation*}
R(X/Y, K_X+\Delta)=\bigoplus _{m\in \mathbb N} 
\pi_*\mathcal O_X(\lfloor m(K_X+\Delta)\rfloor)
\end{equation*} 
is a locally finitely generated graded $\mathcal O_Y$-algebra. 
\end{conj}

Let $X$ be a normal complex variety and let $L\subset K$ be 
compact subsets of $X$. 
Assume that $X$ is $\mathbb Q$-factorial at $K$. 
Unfortunately, however, $X$ is not necessarily $\mathbb Q$-factorial 
at $L$. Therefore, the following theorem seems to be much more useful than 
we expected and is indispensable. 

\begin{thm}[Small $\mathbb Q$-factorializations]\label{a-thm1.24}
Let $\pi\colon X\to Y$ be a projective morphism of complex 
analytic spaces and let $W$ be a compact subset 
of $Y$ such that $\pi\colon X\to Y$ and $W$ satisfies {\em{(P)}}. 
Suppose that there exists $\Delta$ such that $(X, \Delta)$ 
is kawamata log terminal. 
Then, after shrinking $Y$ around $W$ suitably, 
there exists a small projective bimeromorphic contraction 
morphism $f\colon Z\to X$ from a normal complex variety $Z$ 
such that $Z$ is projective over $Y$ and is $\mathbb Q$-factorial 
over $W$. 
\end{thm}

As an obvious corollary of Theorem \ref{a-thm1.24}, we have: 

\begin{cor}\label{a-cor1.25}
Let $\pi\colon X\to Y$ be a projective morphism of 
complex analytic spaces and let $W$ be a compact subset of $Y$ 
such that $\pi\colon X\to Y$ and $W$ satisfies {\em{(P)}}. 
Let $\phi\colon X\dashrightarrow Z$ be a log terminal model 
over $W$. 
Let $W'$ be a Stein compact subset of $Y$ such that $W'\subset W$ and 
that $\Gamma (W', \mathcal O_Y)$ is noetherian. 
Then there exists a log terminal model $\phi'\colon X\dashrightarrow Z'$ 
over $W'$ after shrinking $Y$ around $W'$ suitably. 

We further assume that there is an open 
neighborhood $U$ of $W'$ such that 
$U\Subset W$. Then $\phi'\colon X\dashrightarrow Z'$ 
is a log terminal model 
over some open neighborhood of $W'$. 
\end{cor}

We can also prove: 

\begin{thm}\label{a-thm1.26} 
Let $(X, \Delta)$ be a kawamata log terminal pair 
and let $D$ be an integral Weil divisor on $X$. 
Then $\bigoplus _{m\in \mathbb N}\mathcal O_X(mD)$ 
is a locally finitely generated graded $\mathcal O_X$-algebra. 
\end{thm}

Theorem \ref{a-thm1.26} is a complete 
generalization of \cite[Theorem 6.1]{kawamata-crepant} 
(see also \cite[Theorem 7.2]{fujino-some}). 
The next result will be indispensable for further 
studies of the minimal model program 
for projective morphisms of complex analytic spaces. 

\begin{thm}[Dlt blow-ups, II]\label{a-thm1.27} 
Let $\pi\colon X\to Y$ be a projective morphism of complex 
analytic spaces and let $W$ be a compact subset 
of $Y$ such that $\pi\colon X\to Y$ and $W$ satisfies {\em{(P)}}. 
Let $\Delta$ be an effective $\mathbb R$-divisor on $X$ such that 
$K_X+\Delta$ is $\mathbb R$-Cartier. 
Then, after shrinking $Y$ around $W$ suitably, 
we can construct a projective bimeromorphic 
morphism 
$f\colon Z\to X$ from a normal complex variety $Z$ with the following 
properties: 
\begin{itemize}
\item[(1)] $Z$ is projective over $Y$ and 
is $\mathbb Q$-factorial over $W$, 
\item[(2)] $a(E, X, \Delta)\leq -1$ for every 
$f$-exceptional divisor $E$ on $Z$, and 
\item[(3)] $\left(Z, \Delta^{<1}_Z+\Supp \Delta^{\geq 1}_Z\right)$ is 
divisorial log terminal, where $K_Z+\Delta_Z=f^*(K_X+\Delta)$. 
\end{itemize}
We note that $\Delta_Z=\Delta^{<1}_Z+\Supp \Delta^{\geq 1}_Z$ 
holds when $(X, \Delta)$ is log canonical. 
\end{thm}

We can generalize \cite[Theorem 1]{kawamata-flop} as follows. 

\begin{thm}[{see \cite[Corollary 3.3]{birkar2}}]\label{a-thm1.28}
Let $\pi_1\colon X_1\to Y$ and $\pi_2\colon X_2\to Y$ be 
projective morphisms such that $Y$ is Stein and 
let $W$ be a Stein compact subset of $Y$ such that 
$\Gamma (W, \mathcal O_Y)$ is noetherian. 
Let $(X_1, \Delta_1)$ and $(X_2, \Delta_2)$ be two kawamata 
log terminal pairs such that 
$K_{X_1}+\Delta_1$ and $K_{X_2}+\Delta_2$ are nef over 
$W$, $X_1$ and $X_2$ are $\mathbb Q$-factorial over $W$, 
$X_1$ and $X_2$ are isomorphic in codimension one, and $\Delta_2$ 
is the strict transform of $\Delta_1$. 
Then, after shrinking $Y$ around $W$ suitably, 
$X_1$ and $X_2$ are connected by a sequence of 
flops with respect to $K_{X_1}+\Delta_1$. 
\end{thm}

\begin{rem}\label{a-rem1.29}
Precisely speaking, 
the proof of Theorem \ref{a-thm1.28} shows that 
there exists an effective $\mathbb Q$-Cartier 
$\mathbb Q$-divisor 
$D_1$ on $X_1$ such that 
$(X_1, \Delta_1+D_1)$ is kawamata log terminal, 
$X_1$ and $X_2$ are connected by a finite 
sequence of 
flips with respect to $K_{X_1}+\Delta_1+D_1$, and 
$K_{X_1}+\Delta_1$ is numerically trivial over $W$ in each 
flip. 
\end{rem}

On the abundance conjecture, 
we have: 

\begin{thm}[Abundance theorem, see 
Theorem \ref{a-thm23.2}]\label{a-thm1.30}
Assume that the abundance conjecture holds for projective 
kawamata log terminal pairs in dimension $n$. 

Let $\pi\colon X\to Y$ be a projective surjective morphism 
of complex analytic spaces and let $W$ be a compact 
subset of $Y$ such that 
$\pi\colon X\to Y$ and $W$ satisfies {\em{(P)}}. 
Assume that $K_X+\Delta$ is nef over $Y$ and $\dim X-\dim Y=n$. 
Then $K_X+\Delta$ is $\pi$-semiample over some open 
neighborhood of $W$. 
\end{thm}

\begin{rem}\label{a-rem1.31}The abundance conjecture 
for projective kawamata log terminal pairs was 
completely solved affirmatively in dimension $\leq 3$. 
Therefore, in Theorem \ref{a-thm1.30}, 
$K_X+\Delta$ is $\pi$-semiample 
over some open neighborhood of $W$ when 
$\dim X-\dim Y\leq 3$. 
This result seems to be new even when 
$\dim X=3$. 
\end{rem}

We make some remarks on the proof of our results in this paper. 

\begin{rem}\label{a-rem1.32}
We will use the fact that every extremal ray is spanned by a 
rational curve of low degree (see \cite[Theorem 3.8.1]{bchm} and 
Theorem \ref{a-thm9.2}) for the proof of 
Theorems \ref{thm-a}, \ref{thm-b}, \ref{thm-c}, \ref{thm-e}, 
\ref{thm-f}, and \ref{thm-g}. 
Note that in \cite{bchm} it was only 
used for the proof of the finiteness of negative extremal 
rays (see \cite[Corollary 3.8.2]{bchm}). 
However, it is well known as a part of the cone theorem 
(see also Theorems \ref{a-thm7.2} and \ref{a-thm7.3}). 
Therefore, \cite{bchm} is independent of Mori's bend and break technique, 
which relies on methods in 
positive characteristic. 
\end{rem}

\begin{rem}\label{a-rem1.33}
We can easily reduce Theorem \ref{thm-d} to the case where $Y$ 
is a point. In this case, $X$ is projective and then Theorem \ref{thm-d} 
becomes a special case of \cite[Theorem D]{bchm}. 

In \cite[Section 3]{birkar-paun}, P\u aun proved a slightly 
weaker version of the nonvanishing theorem for projective varieties 
(see \cite[Theorem 1.5]{birkar-paun}). The proof is complex analytic and is 
independent of the theory of minimal models. 
By combining it with \cite[Theorem 0.1 
and Corollary 3.3]{ckp}, we can recover the nonvanishing theorem 
for projective varieties in full generality (see \cite[Theorem D]{bchm}). 
By adopting this approach, Theorem \ref{thm-d} in this paper becomes 
completely independent of the framework of the minimal model program. 

Anyway, we can prove Theorem \ref{thm-d} without 
any difficulties by using some known results.  
\end{rem}

\begin{rem}\label{a-rem1.34}
In \cite{cascini-lazic}, Cascini and Lazi\'c directly 
proved the finite generation of adjoint 
rings by using Hironaka's resolution and the Kawamata--Viehweg 
vanishing theorem. 
Their proof does not use the minimal model program. Then, 
in \cite{corti-lazic}, Corti and Lazi\'c recovered many results on the minimal 
model program from \cite[Theorem A]{cascini-lazic}. The author does not 
know whether this approach works or not in our complex analytic setting. 
\end{rem}

\begin{rem}\label{a-rem1.35}
In dimension three, some results were formulated 
and established in the complex analytic 
setting (see, for example, \cite{kawamata-crepant}, \cite{mori-flip}, 
and \cite[Chapter 6]{kollar-mori}). It is not surprising 
because the theory of minimal models for $3$-folds 
originally owes to the study of various 
singularities (see, for example, \cite{mori-terminal}, 
\cite{kawamata-crepant}, and \cite{mori-flip}). 
Although the classification of surface singularities 
is indispensable for adjunction, we do not need 
any classifications of higher-dimensional singularities 
in \cite{bchm}. Hence, we think that the minimal model 
program in dimension $\geq 4$ has been formulated 
and studied only for algebraic varieties. 
\end{rem}

We look at the organization of this paper. 
Section \ref{a-sec2} is a long preliminary 
section, where we will collect and explain 
some basic definitions and results on 
complex analytic spaces. We think that the 
reader can understand that (P4) is reasonable. 
To the best knowledge of the author, 
some of the results in this section seem to be new. 
In Section \ref{a-sec3}, we will explain singularities of 
pairs in the complex analytic setting. 
The definitions of singularities 
become slightly complicated in the complex analytic setting. 
In Section \ref{a-sec4}, we will define 
Kleiman--Mori cones and establish 
Kleiman's ampleness criterion in the complex 
analytic setting. Here, the property (P4) plays a crucial role. 
Section \ref{a-sec5} is a very short section, where 
we explain only two vanishing theorems. 
From Section \ref{a-sec6} to Section \ref{a-sec8}, we will 
establish several basepoint-free theorems, the cone and contraction 
theorem, and so on, in the complex analytic setting. This part is 
essentially due to Nakayama (see \cite{nakayama2}). 
We note that we have to treat $\mathbb R$-divisors. 
Therefore, some parts are harder than the classical setting 
discussed in \cite{nakayama2}. 
In Section \ref{a-sec9}, 
we will prove that every negative extremal ray is spanned by 
a rational curve of low degree. 
Note that we need 
some result obtained by Mori's bend and break technique, 
which relies on methods in 
positive characteristic. The result in this section will 
play an important role in the subsequent sections. 
In Sections \ref{a-sec11} and \ref{a-sec12}, 
we will prepare various basic definitions to 
establish the main results of this paper. 
We closely follow the treatment of \cite{bchm}. 
However, we have to reformulate some of them 
in order to make them suitable for our complex 
analytic setting. Hence we strongly recommend 
the reader to read these sections carefully. 
In Section \ref{a-sec13}, we will explain the minimal model 
program with scaling in the complex analytic setting in detail. 
It is very useful for various geometric applications. 
From Section \ref{a-sec14} to Section \ref{a-sec19}, 
we will prove Theorems \ref{thm-a}, 
\ref{thm-b}, \ref{thm-c}, 
\ref{thm-d}, \ref{thm-e}, \ref{thm-f}, and \ref{thm-g}. 
Although there are many technical differences between 
the original proof for quasi-projective varieties (see 
\cite{bchm} and 
\cite{hacon-mckernan}) and the one given here in the 
complex analytic setting, the strategy of the whole proof is the same. 
In some parts, we will only explain how to modify the original 
proof in order to make it work in our complex analytic setting. 
In Section \ref{a-sec20}, we will prove almost all 
the theorems given in Section \ref{a-sec1}. 
We think that the reader who is familiar with the minimal 
model program for quasi-projective varieties 
can understand this section without any difficulties. 
In the last three sections, we will treat some advanced 
topics. 
In Section \ref{a-sec21}, we will briefly 
discuss a canonical bundle formula in the complex analytic 
setting and prove Theorems \ref{a-thm1.18} and 
\ref{a-thm1.22}. This section needs some deep results 
on the theory of variations of Hodge structure. Hence 
the topic in Section \ref{a-sec21} is slightly different from 
the other sections. In Section \ref{a-sec22}, 
we will discuss the minimal model program with scaling again. 
Then we will prove Theorem \ref{a-thm1.28} as an easy 
application. In Section \ref{a-sec23}, we will explain how to 
reduce the abundance conjecture for projective morphisms 
between complex analytic spaces to the original 
abundance conjecture 
for projective varieties (see Theorem \ref{a-thm1.30}). 
Note that we will only treat kawamata log terminal pairs 
in Section \ref{a-sec23}. 
The abundance conjecture for log canonical pairs 
looks much harder than the one for kawamata log 
terminal pairs. 

As is well known, the recent developments 
of the theory of minimal models heavily owe 
to many ideas and results obtained by Shokurov. 
They are scattered in his 
papers (see, for example, \cite{shokurov1}, 
\cite{shokurov2}, and \cite{shokurov3}). 
In this paper, we will freely use them without referring to 
Shokurov's original papers.  

\begin{ack}\label{a-ack}
The author was partially 
supported by JSPS KAKENHI Grant Numbers 
JP19H01787, JP20H00111, JP21H00974. 
He would like to thank Hiromichi Takagi for reading the 
first version of \cite{bchm} with him at Nagoya in 2006. 
He thanks Yoshinori Gongyo, Kenta Hashizume, Yuji Odaka, 
Keisuke Miyamoto, Shigeharu Takayama, and 
Yuga Tsubouchi very much. 
\end{ack}

The set of integers (resp.~rational numbers, real numbers, 
complex numbers) is denoted by $\mathbb Z$ 
(resp.~$\mathbb Q$, $\mathbb R$, $\mathbb C$). 
The set of nonnegative integers (resp.~positive integers, 
positive rational numbers, positive real numbers, 
nonnegative real numbers) 
is denoted by $\mathbb N$ (resp.~$\mathbb Z_{>0}$, 
$\mathbb Q_{>0}$, 
$\mathbb R_{>0}$, $\mathbb R_{\geq 0}$). 

\section{Preliminaries}\label{a-sec2} 
In this section, we will collect some basic definitions and 
explain various standard results. 
We note that 
every complex analytic space in this paper 
is assumed to be {\em{Hausdorff}} and 
{\em{second-countable}}. 
The books \cite{banica}, 
\cite{fischer}, and \cite{gunning-rossi} 
are standard references of complex analytic geometry 
for algebraic geometers. 
A relatively new book by Noguchi (see \cite{noguchi}) is 
a very accessible textbook on several complex variables and 
complex analytic spaces. 
Demailly's book (see \cite{demailly}) is also helpful 
and contains a proof of Grauert's theorem 
on direct images of coherent sheaves. 
We will freely use Serre's GAGA principle 
(see, for example, \cite[Chapter 13]{taylor} and 
\cite[Expos\'e XII]{sga1}) throughout this 
paper. 

Let us start with the definition of {\em{holomorphically 
convex hulls}}. 

\begin{defn}[Holomorphically convex hulls]\label{a-def2.1}
Let $X$ be an analytic space and let $K$ be a compact 
subset of $X$. 
The {\em{holomorphically convex hull}} $\widehat K$ of $K$ 
in $X$ is the set 
\begin{equation*}
\widehat K=\left\{x\in X\, 
\left|\, \text{$|f(x)|\leq \underset{z\in K}{\sup}
|f(z)|$ for every $f\in \Gamma(X, \mathcal O_X)$}\right.\right\}.  
\end{equation*}
We note that $K\subset \widehat K$ always holds by definition. 
A compact subset $K$ of $X$ is said to 
be {\em{holomorphically convex}} in $X$ if 
$\widehat K=K$ holds. 
\end{defn}

Let us recall the definition of {\em{Stein spaces}} for the 
reader's convenience. We note that 
the analytic space naturally associated to 
an affine scheme is Stein

\begin{defn}[Stein spaces]\label{a-def2.2} 
A complex analytic space $X$ is said to 
be {\em{Stein}} if 
\begin{itemize}
\item[(i)] the global sections of $\mathcal O_X$ separate 
points in $X$, 
\item[(ii)] for each $x\in X$, the maximal ideal 
of $\mathcal O_{X, x}$ is generated by a set 
of global sections of $\mathcal O_X$, and 
\item[(iii)] $X$ is holomorphically convex, that is, 
$\widehat K$ is compact for every compact subset $K$ of $X$, 
where $\widehat K$ is the holomorphically convex hull of 
$K$ in $X$. 
\end{itemize}
\end{defn}

The notion of {\em{Stein compact subsets}} 
plays a crucial role in this paper. 

\begin{defn}[Stein compact subsets]\label{a-def2.3}
A compact subset $K$ of a complex 
analytic space is called {\em{Stein compact}} if 
it admits a fundamental system of Stein open neighborhoods. 
\end{defn}

The notion of {\em{Oka--Weil domains}} is very useful when 
we construct desired Stein open neighborhoods. 

\begin{defn}[{Oka--Weil domains, see 
\cite[Chapter VII, Section A, 
$2'$.~Definition]{gunning-rossi}}]\label{a-def2.4}
Let $X$ be a complex analytic space. 
An {\em{Oka--Weil domain}} on $X$ is a relatively 
compact open subset $W$ such that 
there exists a holomorphic map $\varphi$ defined in a neighborhood 
of $\overline W$, and with values in $\mathbb C^n$, such that 
$\varphi|_W$ is a biholomorphic mapping onto a 
closed complex analytic subspace of a polydisc in $\mathbb C^n$. 
We note that $W$ itself is Stein. 
\end{defn}

By the following lemma, we know that we can find many 
Stein compact subsets on a given Stein space. 

\begin{lem}\label{a-lem2.5} 
Let $K$ be a compact subset of a Stein space 
$X$. Let $\widehat K$ be a 
holomorphically convex hull of $K$ in $X$
Then $\widehat K$ is a Stein compact subset of $X$. 
\end{lem}

\begin{proof}
Since $X$ is holomorphically convex, 
$\widehat K$ is a compact subset of $X$. 
Let $U$ be any open subset of $X$ with 
$\widehat K\subset U$. 
Then we can take an Oka--Weil domain $W$ of $X$ such that 
$\widehat K\subset W\subset \overline W\subset U$ 
(see \cite[Chapter VII, Section A, 3.~Proposition]
{gunning-rossi}). 
This means that $\widehat K$ 
admits a fundamental system of Stein open 
neighborhoods since $W$ is a Stein space. 
Hence $\widehat K$ is a Stein compact subset of $X$. 
\end{proof}

Throughout this paper, we freely use 
Cartan's Theorems A and B without mentioning it explicitly. 
We include them here for the reader's convenience. 

\begin{thm}[Cartan's Theorems]\label{a-thm2.6}
Let $X$ be a Stein space and let 
$\mathcal F$ be a coherent sheaf on $X$. 
Then 
\begin{itemize}
\item[(1)] {\em{(Cartan's Theorem A).}} $\Gamma (X, \mathcal F)$ 
generates $\mathcal F_x$ at every point $x\in X$, and 
\item[(2)] {\em{(Cartan's Theorem B).}} $H^i(X, \mathcal F)=0$ 
holds for every $i>0$. 
\end{itemize}
\end{thm}

The following cohomological 
characterization of Stein spaces is useful 
and may help algebraic geometers 
understand the definition of Stein spaces. 

\begin{thm}\label{a-thm2.7} 
Let $X$ be a complex analytic space. 
Then $X$ is Stein if and only if $H^1(X, \mathcal F)=0$ for 
every coherent sheaf $\mathcal F$ on $X$. 
\end{thm}

\begin{proof}
If $X$ is Stein, then $H^1(X, \mathcal F)=0$ for every 
coherent sheaf $\mathcal F$ on $X$ by 
Cartan's Theorem B. 
On the other hand, if $H^1(X, \mathcal I)=0$ for every coherent ideal sheaf 
$\mathcal I$ on $X$, then it is an easy exercise to check that 
(i) and (ii) 
in Definition \ref{a-def2.2} hold true. 
Let $K$ be a compact subset of $X$. Suppose 
that the holomorphically convex hull $\widehat K$ of $K$ 
is not compact. 
Then we can take a discrete sequence $\{x_k\}_{k\in \mathbb N}
\subset \widehat K$. 
Note that $V:=\{x_k \, |\, 0\leq k< \infty\}$ is a closed 
analytic subspace of $X$. 
Hence the defining ideal sheaf $\mathcal I_V$ of $V$ is a coherent 
sheaf on $X$. 
Thus $H^1(X, \mathcal I_V)=0$ holds 
by assumption. 
This implies that the natural map $H^0(X, \mathcal O_X)\to 
H^0(X, \mathcal O_X/\mathcal I_V)$ is surjective. 
Therefore, we can take $f\in H^0(X, \mathcal O_X)$ such that 
$f(x_n)=n$. 
On the other hand, 
\begin{equation*}
n=|f(x_n)|\leq \sup _{x\in K} |f(x)|<\infty
\end{equation*} 
for 
every $n$ since $x_n\in \widehat K$. 
This is a contradiction. 
Thus, $\widehat K$ is always compact, that is, 
(iii) holds true. We finish the proof. 
\end{proof}

As an easy consequence of Theorem \ref{a-thm2.7}, we have: 

\begin{thm}\label{a-thm2.8} 
Let $f\colon Z\to X$ be a finite morphism 
between complex analytic spaces. 
If $X$ is Stein, then so is $Z$. 
\end{thm}

\begin{proof}
Let $\mathcal F$ be any coherent sheaf on $Z$. 
Since $f\colon Z\to X$ is finite, 
$f_*\mathcal F$ is coherent and $H^1(Z, \mathcal F)
=H^1(X, f_*\mathcal F)=0$ holds by 
the Steinness of $X$. 
Hence, by Theorem \ref{a-thm2.7}, $Z$ is Stein. 
\end{proof}

We have already explained that every Stein space 
$X$ has many good properties. 
Unfortunately, however, 
$\Gamma(X, \mathcal O_X)$ is not noetherian if 
$X$ does not consist of only finitely many points. 

\begin{ex}\label{a-ex2.9}
Let $X$ be a Stein space and let $\{P_k\}_{k\in \mathbb 
N}$ be a set of 
mutually distinct discrete points of $X$. 
Then $Z_n:=\{P_n, P_{n+1}, \ldots\}$ can be seen 
as a closed analytic subspace of $X$ for every $n\in \mathbb N$. 
Let $\mathcal I_n$ be the defining ideal sheaf of $Z_n$ on $X$. 
It is well known that $\mathcal I_n$ is a coherent 
sheaf on $X$. 
We put $\mathfrak a_n:=\Gamma (X, \mathcal I_n)$ for 
every $n$. 
Then 
\begin{equation*}
\mathfrak a_0\subsetneq \mathfrak a_1\subsetneq \cdots \subsetneq 
\mathfrak a_n \subsetneq \cdots
\end{equation*}
is a strictly increasing sequence of ideals of 
$\Gamma(X, \mathcal O_X)$. This means that 
$\Gamma(X, \mathcal O_X)$ is not noetherian. 
\end{ex}

Siu's theorem clarifies the meaning of the condition in (P4). 

\begin{thm}[{\cite[Theorem 1]{siu}}]\label{a-thm2.10} 
Let $K$ be a Stein compact subset of a complex analytic 
space $X$. Then 
$K\cap Z$ has only finitely many connected 
components for any analytic subset 
$Z$ which is defined over an open neighborhood 
of $K$ if and only if 
\begin{equation*}
\mathcal O_X(K)=
\Gamma (K, \mathcal O_Y)=
\underset{K\subset U}\varinjlim
\Gamma (U, \mathcal O_X), 
\end{equation*} 
where $U$ runs through all the open neighborhoods of 
$K$, is noetherian.
\end{thm}

\begin{proof}For the details, see, for example, \cite[Chapter V, \S 3]{banica}. 
\end{proof}

One point is a Stein compact subset satisfying (P4). 

\begin{ex}\label{a-ex2.11}
Let $X$ be a complex analytic space and let $P$ be any point 
of $X$. 
Then $P$ is a Stein compact subset of $X$ 
and $\mathcal O_{X, P}=\Gamma(P, \mathcal O_X)$ 
is noetherian. 
\end{ex}

The Cantor set is a Stein compact subset which does not 
satisfy (P4). 

\begin{ex}\label{a-ex2.12}
We note that $\mathbb C$ is Stein and that 
any open subset of $\mathbb C$ is also Stein since 
it is holomorphically convex. 
We put $X=\{z\in \mathbb C\, |\, |z|<2\}$ and 
consider the Cantor set $\mathcal C$. 
It is easy to see that $\mathcal C$ ($\subset 
[0, 1]\subset X$) is a Stein compact subset of $X$ and 
that $X$ is Stein. 
We can easily check that for any given $x_1, x_2\in \mathcal C$ 
there exists $x_3\not \in \mathcal C$ with 
$x_3\in [x_1, x_2]$. 
Hence $\mathcal C$ does not satisfy (P4). 
Thus, $\Gamma (\mathcal C, \mathcal O_X)$ is not noetherian 
by Theorem \ref{a-thm2.10} . 
More explicitly, 
we put 
\begin{equation*}
\mathfrak a_n:=\left\{f\in \Gamma (\mathcal C, \mathcal O_X) \, 
\left|
\, {\text{$f(z)=0$ for any $z\in \mathcal C\cap 
\left[0, \frac{1}{3^n}\right]$}}\right.\right\}
\end{equation*}
for every $n\in \mathbb N$. 
Then we can check that $\mathfrak a_n\subsetneq 
\mathfrak a_{n+1}$ holds for every $n\in \mathbb N$. 
Therefore, we get a strictly increasing sequence of ideals of 
$\Gamma (\mathcal C, \mathcal O_X)$: 
\begin{equation*}
\mathfrak a_0\subsetneq \mathfrak a_1\subsetneq 
\cdots \subsetneq \mathfrak a_n\subsetneq \cdots. 
\end{equation*} 
This implies that $\Gamma (\mathcal C, \mathcal O_X)$ is not 
noetherian. 
\end{ex}

We supplement Theorem \ref{a-thm2.8}. 

\begin{thm}\label{a-thm2.13}
Let $f\colon Z\to X$ be a finite morphism 
of complex analytic spaces such that 
$X$ is Stein. 
Let $K$ be a Stein compact subset of $X$ such that 
$\Gamma(K, \mathcal O_X)$ is noetherian. 
Then $f^{-1}(K)$ is a Stein compact 
subset of $Z$ such that $\Gamma (f^{-1}(K), \mathcal O_Z)$ 
is noetherian. 
\end{thm}

\begin{proof}
Since $f$ is finite, $f^{-1}(K)$ is a compact subset of $Z$. 
Let $\{U_\lambda\}_{\lambda\in \Lambda}$ be 
a fundamental system of Stein open neighborhoods 
of $K$. 
By Theorem \ref{a-thm2.8}, 
$Z$ is Stein and $f^{-1}(U_\lambda)$ is also 
Stein for every $\lambda\in \Lambda$. 
Hence $\{f^{-1}(U_\lambda)\}_{\lambda\in \Lambda}$ 
is a fundamental system of Stein open neighborhoods 
of $f^{-1}(K)$. 
On the other hand, since $f$ is finite, 
$f_*\mathcal O_Z$ is a coherent sheaf on $X$. 
By the Stein compactness of $K$, 
there exist a Stein open neighborhood 
$U$ of $K$ and 
a surjection 
\begin{equation*}
\mathcal O_U^{\oplus N} \to f_*\mathcal O_Z|_U\to 0
\end{equation*} 
for some positive integer $N$. 
This implies the surjection 
\begin{equation*}
\Gamma (K, \mathcal O_X)^{\oplus N}\to 
\Gamma (K, f_*\mathcal O_Z)\to 0.  
\end{equation*}
Hnece $\Gamma (K, f_*\mathcal O_Z)$ is a finitely generated 
$\Gamma (K, \mathcal O_X)$-module. 
We note that $\Gamma (K, f_*\mathcal O_Z)=\Gamma 
(f^{-1}(K), \mathcal O_Z)$ and that 
$\Gamma (K, \mathcal O_X)$ is noetherian. 
Thus, $\Gamma (f^{-1}(K), \mathcal O_Z)$ is noetherian. 
This means that $f^{-1}(K)$ is a Stein 
compact subset of $Z$ such that 
$\Gamma (f^{-1}(K), \mathcal O_Z)$ is noetherian. 
\end{proof}

\begin{rem}\label{a-rem2.14}
Let $\pi\colon X\to Y$ be a projective morphism between 
complex analytic spaces and let $W$ be a compact subset 
of $Y$ such that 
$\pi\colon X\to Y$ and $W$ satisfies (P). 
As an easy consequence of Theorems \ref{a-thm2.8} 
and \ref{a-thm2.13}, 
we usually may assume that 
$\pi$ is surjective and that $Y$ is a Stein variety 
by replacing $Y$ with $\pi(X)$. 
For some purposes, we sometimes 
replace $Y$ with its normalization and 
further assume that $Y$ is a normal 
Stein variety. 
By taking the Stein factorization of $\pi\colon X\to Y$, 
we sometimes further assume that 
$\pi$ has connected fibers and that 
$Y$ is a normal Stein variety, 
that is, $\pi\colon X\to Y$ is a contraction 
of normal complex varieties.    
\end{rem}

We note: 

\begin{defn}\label{a-def2.15}
A proper morphism $\pi\colon X\to Y$ of 
normal complex varieties is called 
a {\em{contraction}} if $\pi_*\mathcal O_X\simeq 
\mathcal O_Y$ holds. 
\end{defn}

When we enlarge a given Stein compact subset satisfying 
(P4) slightly, we need the following lemma. 

\begin{lem}\label{a-lem2.16}
Let $X$ be a Stein space and let 
$K$ be a holomorphically convex compact subset of $X$. 
If $U$ is any open neighborhood 
of $K$, then there exists an Oka--Weil domain 
$V$, defined by global holomorphic functions on $X$, such that 
$K\subset V \subset \overline {V} \subset U$. 
Note that $V$ can be seen as a closed complex analytic subspace 
of a polydisc $\Delta(0, r)=\{(z_1, \ldots, z_n) \, |\, 
{\text{$|z_i|<r$ for $1\leq i\leq n$}}\}$ for some $r>0$ and 
$n\in \mathbb Z_{>0}$. 
We put $L:=V\cap \overline {\Delta} (0, r-\varepsilon)$ 
with $0<\varepsilon <r$. 
Then $L$ is compact, semianalytic, and holomorphically 
convex in $V$. In particular, $L$ is a Stein compact 
subset such that $\Gamma (L, \mathcal O_X)$ is 
noetherian. 
Furthermore, if 
$U'$ is a relatively compact open neighborhood of 
$K$ in $X$, then we can choose $U$, $V$, and $L$ 
such that 
\begin{equation*}
K\subset U'\subset L\subset V\subset \overline {V}\subset 
U\subset X
\end{equation*} 
holds. 
\end{lem}

\begin{proof}
For the existence of a desired Oka--Weil domain $V$, 
see, for example, \cite[Chapter VII, Section A, 3.~Proposition]
{gunning-rossi}. 
By definition, $L$ is obviously compact and semianalytic. 
Since $L$ is defined by 
$|z_1|\leq r-\varepsilon, \ldots, 
|z_n|\leq r-\varepsilon$ such that $z_i \in \Gamma (V, \mathcal O_V)$ 
for every $i$, it is easy to see that $L$ is holomorphically 
convex in $V$. 
By Lemma \ref{a-lem2.5}, $L$ is Stein compact. 
Since $L$ is compact and semianalytic, 
it is well known that $L$ satisfies (P4) 
(see, for example, \cite[Corollary 2.7 (2)]{bierstone-milman1}). 
Thus, $\Gamma (L, \mathcal O_X)$ is noetherian by 
Theorem \ref{a-thm2.10}. 
Or, by applying \cite[Th\'eor\`eme (I,9)]{frisch} to $L$, 
we obtain that $\Gamma (L, \mathcal O_X)$ is noetherian.  
By the above construction, the last statement is obvious. 
\end{proof}

We frequently use the following property of coherent 
sheaves on complex analytic spaces. 

\begin{lem}\label{a-lem2.17}
Let $\mathcal F$ be a coherent sheaf on a complex 
analytic space $X$ and 
let 
\begin{equation*}
\mathcal F_0\subset \mathcal F_1\subset \mathcal F_2
\subset \cdots \subset \mathcal F
\end{equation*} 
be an increasing chain of coherent subsheaves. 
Then this chain is stationary over any relatively 
compact subset of $X$. 
\end{lem}
\begin{proof}
See, for example, \cite[0.40.~Proposition and 
Corollary]{fischer}. 
\end{proof}

Note that the arguments in \cite[\S 2 Basic theorems 
on coherent $\mathcal O$-modules]{kaup} work for 
Stein compact subsets $K$ satisfying (P4) with 
obvious modifications. 

\begin{defn}[{$\mathcal O_X$-exhaustions, see 
\cite[2.9 Definition]{kaup}}]\label{a-def2.18}
Let $\mathcal M$ be an $\mathcal O_X$-module on a complex analytic 
space $X$. 
An {\em{$\mathcal O_X$-exhaustion}} of 
$\mathcal M$ is an increasing 
sequence 
\begin{equation*}
\mathcal M_0\subset \mathcal M_1
\subset \cdots \subset \mathcal M_k\subset 
\cdots \subset \mathcal M
\end{equation*} 
of coherent sub-$\mathcal O_X$-modules such that 
$\mathcal M=\bigcup _k \mathcal M_k$. 
\end{defn}

\begin{lem}[{see \cite[2.10 Proposition]{kaup}}]\label{a-lem2.19}
Let $K$ be a Stein compact subset of a complex analytic 
space $X$ such that $\Gamma (K, \mathcal O_X)$ is noetherian. 
Let $\mathcal M$ and $\mathcal M'$ be $\mathcal O_X$-modules 
on $X$ which admit $\mathcal O_X$-exhaustions. 
If $\phi\colon \mathcal M\to \mathcal M'$ is a surjective 
$\mathcal O_X$-homomorphism, 
then the induced map $\Gamma (K, \mathcal M)\to 
\Gamma (K, \mathcal M')$ is surjective. 
\end{lem}
\begin{proof}
For the details, see the proof of \cite[2.10 Proposition]{kaup}. 
Although $K$ is a polydisc in \cite{kaup}, 
the proof of \cite[2.10 Proposition]{kaup} works in our setting.  
\end{proof}

Since we are working on complex analytic spaces, 
we note: 

\begin{rem}\label{a-rem2.20}
Let $\mathcal F_m$ be a coherent sheaf on a complex 
analytic space $X$ for every $m\in \mathbb N$. 
Then the natural map 
\begin{equation*}
\bigoplus _{m\in \mathbb N} \Gamma (X, \mathcal F_m)\to 
\Gamma \left(X, \bigoplus _{m \in \mathbb N} \mathcal F_m\right)
\end{equation*} 
is not necessarily an isomorphism. 
Fortunately, 
\begin{equation*}
\bigoplus _{m\in \mathbb N} \Gamma (K, \mathcal F_m)\simeq 
\Gamma \left(K, \bigoplus _{m \in \mathbb N} \mathcal F_m\right)
\end{equation*} 
holds for any compact subset $K$ of $X$. 
\end{rem}

\begin{ex}\label{a-ex2.21} 
Let $X$ be a noncompact complex analytic 
space and let $\{x_m\}_{m\in \mathbb N}$ be 
a discrete sequence of $X$. 
Let $\mathbb C(x_m):=(i_{x_m})_*\mathbb C$ be 
a skyscraper sheaf, where $i_{x_m}\colon x_m\hookrightarrow 
X$ is the inclusion map. 
Then $\bigoplus _{m \in \mathbb N} \mathbb C(x_m)$ 
is a coherent sheaf on $X$. 
In this case, 
\begin{equation*}
\Gamma\left(X, \bigoplus _{m\in \mathbb N}\mathbb C(x_m)\right)=
\prod_{m\in \mathbb N} \mathbb C=\prod_{m\in \mathbb N} 
\Gamma(X, \mathbb C(x_m)). 
\end{equation*} 
Hence, the natural map 
\begin{equation*}
\bigoplus _{m\in \mathbb N} \Gamma(X, \mathbb C(x_m))
\to \Gamma\left(X, \bigoplus _{m\in \mathbb N}\mathbb C(x_m)\right)
\end{equation*} 
is not an isomorphism. 
\end{ex}

In this paper, we will have to treat graded $\mathcal O_X$-algebras 
on a complex analytic spaces. 
So we prepare some definitions and basic properties. 

\begin{defn}[{see \cite[Chapter II.~\S 1.~b.~Spec and Proj]{nakayama3}}]
\label{a-def2.22}
Let $X$ be a complex analytic space and let $\mathbb C[x]
=\mathbb C[x_1, \cdots, x_l]$ be the polynomial 
ring of $l$-variables $x=(x_1, \cdots, x_l)$. 
An $\mathcal O_X$-algebra $\mathcal A$ is 
called {\em{of finite presentation}} 
if there exists a surjective $\mathcal O_X$-algebra 
homomorphism 
\begin{equation*}
\mathcal O_X[x]=\mathcal O_X[x_1, \cdots, x_l]=
\mathcal O_X\otimes _{\mathbb C} \mathbb C[x]\twoheadrightarrow 
\mathcal A
\end{equation*}
for some $l$ whose kernel is generated by a finite 
number of polynomials belonging to $H^0(X, \mathcal O_X)[x]$. 
If $\mathcal A|_{U_\lambda}$ is of finite 
presentation for an open covering $X=\bigcup _{\lambda\in 
\Lambda}U_\lambda$, 
then $\mathcal A$ is called {\em{locally of finite 
presentation}}. 
\end{defn}

The notion of locally finitely generated 
graded $\mathcal O_X$-algebras 
is indispensable. 

\begin{defn}[Locally finitely generated 
graded $\mathcal O_X$-algebras]
\label{a-def2.23}
Let $X$ be a complex analytic space. 
An $\mathcal O_X$-algebra $\mathcal A=\bigoplus _{m\in 
\mathbb N} \mathcal A_m$ is called a {\em{finitely 
generated graded $\mathcal O_X$-algebra}} 
if there exists a surjective $\mathcal O_X$-algebra homomorphism 
\begin{equation*}
\mathcal O_X[x]=\mathcal O_X[x_1, \cdots, x_l]=
\mathcal O_X\otimes _{\mathbb C} \mathbb C[x]\twoheadrightarrow 
\mathcal A
\end{equation*}
for some $l$ such that $x_i$ is mapped to a 
homogeneous element of $H^0(X, \mathcal A)$ for every $i$. 
If $\mathcal A|_{U_\lambda}$ is a finitely generated graded 
$\mathcal O_{U_\lambda}$-algebra for some open covering 
$X=\bigcup _{\lambda\in \Lambda} U_\lambda$, then 
$\mathcal A$ is called a {\em{locally finitely generated 
graded $\mathcal O_X$-algebra}}. 
\end{defn}

We note the following basic property of locally finitely generated 
graded $\mathcal O_X$-algebras. 

\begin{lem}[{see \cite[Chapter II.~1.6.~Corollary]{nakayama3}}]\label{a-lem2.24} 
Let $X$ be a complex analytic space and 
let $\mathcal A=\bigoplus_{m\in \mathbb N}\mathcal A_m$ 
be a locally finitely generated graded $\mathcal O_X$-algebra 
such that $\mathcal A_m$ are all coherent $\mathcal O_X$-modules. 
Then $\mathcal A$ is locally of finite presentation. 
\end{lem}

Before we prove Lemma \ref{a-lem2.24}, we note: 

\begin{rem}\label{a-rem2.25}
Any locally finitely generated graded $\mathcal O_X$-algebra 
$\mathcal A=\bigoplus _{m\in \mathbb N}\mathcal A_m$ in this 
paper satisfies the condition that 
$\mathcal A_m$ is a coherent $\mathcal O_X$-module for 
every $m\in \mathbb N$. 
We do not treat the case where $\mathcal A_m$ is not 
a coherent $\mathcal O_X$-module.  
\end{rem}

Let us prove Lemma \ref{a-lem2.24}. 

\begin{proof}[Proof of Lemma \ref{a-lem2.24}]
We take an arbitrary point $P\in X$ and replace $X$ with a small 
Stein open neighborhood of $P$. 
Then we have an exact sequence 
\begin{equation*}
\phi\colon \mathcal O_X[x]\to \mathcal A\to 0
\end{equation*}
for $x=(x_1, \cdots, x_l)$ such that 
$x_i$ is mapped to a homogeneous 
element of $H^0(X, \mathcal A)$ for every $i$. 
We take a relatively compact Stein open neighborhood 
$U$ of $P$ and a Stein compact subset $K$ of $Y$ such that 
$U\subset K$ and $\Gamma (K, \mathcal O_Y)$ is noetherian. 
Then 
\begin{equation*}
\phi(K)\colon \mathcal O_X(K)[x]\to \mathcal A(K)\to 0
\end{equation*} 
is exact by Lemma \ref{a-lem2.19}. 
Since $\mathcal O_X(K)$ is noetherian, the kernel of 
$\phi(K)$ is generated by weighted homogeneous 
polynomials $f_1,\ldots, f_N$ in $\mathcal O_X(K)[x]$. 
Hence we obtain a Stein open neighborhood 
$U'$ of $K$ and a homomorphism 
\begin{equation*}
\psi:=(f_1, \cdots, f_N)\colon 
\bigoplus _{i=1}^N \mathcal O_{U'}[x](-\deg f_i)\to \mathcal O_{U'}[x]
\end{equation*}
of graded $\mathcal O_{U'}[x]$-modules such 
that the image of $\psi(K)$ is $\left(\Ker \phi\right)(K)$. 
By construction, there exists the following 
natural surjection 
\begin{equation*}
\alpha_m\colon \left (\Coker \psi\right) _m \twoheadrightarrow 
\mathcal A_m|_{U'}
\end{equation*} 
for every $m\in \mathbb N$, where 
$\left (\Coker \psi\right) _m$ is the degree $m$ part of 
$\Coker \psi$. 
By construction again, we can check that 
$\left (\Coker \psi\right) _m(K)\simeq \mathcal A_m(K)$ holds for 
every $m\in \mathbb N$. 
This implies that the kernel of $\alpha_m$ is zero on $U$. 
Hence $\mathcal A|_U$ is of finite presentation. 
Therefore, $\mathcal A$ is locally of finite presentation.  
\end{proof}

The lemma below is important and 
will be used repeatedly without mentioning it 
explicitly. The proof is much harder than 
that of the corresponding statement for 
algebraic varieties (see \cite[Corollary 1.1.2.6]{cox-rings}). 

\begin{lem}\label{a-lem2.26} 
Let $X$ be a complex analytic space and let 
$\mathcal A=\bigoplus _{m\in \mathbb N}
\mathcal A_m$ be a graded 
$\mathcal O_X$-algebra such that 
$\mathcal A_m$ is a coherent $\mathcal O_X$-module 
for every $m$ 
and $\mathcal A(U)$ is an integral domain 
for every nonempty connected open subset $U$ of $X$. 
We put $\mathcal A^{(d)}:=\bigoplus _{m\in \mathbb N}
\mathcal A_{dm}$. 
Then $\mathcal A$ is a locally finitely generated 
graded $\mathcal O_X$-algebra if and only if 
so is $\mathcal A^{(d)}$. 
\end{lem}

\begin{proof}
Since $\mathcal O_X(U)$ is not 
necessarily noetherian, the proof is not so obvious. 
\setcounter{step}{0}
\begin{step}\label{a-lem2.26-step1}
We assume that $\mathcal A^{(d)}$ is a locally finitely generated 
graded $\mathcal O_X$-algebra. 
We take an arbitrary point $P\in X$. 
By shrinking $X$ around $P$, 
there exists a surjective $\mathcal O_X$-algebra homomorphism 
\begin{equation}\label{a-lem2.26-eq1}
\mathcal O_X[x]=\mathcal O_X[x_1, \cdots, x_l]
\twoheadrightarrow 
\mathcal A^{(d)}
\end{equation}
for some $l$ such that $x_i$ is mapped to a 
homogeneous element of $H^0(X, \mathcal A^{(d)})$ for every $i$. 
We take an open neighborhood $U$ of $P\in X$ and a Stein 
compact subset $K$ of $X$ such that $P\in U \subset 
K$ and 
that $K$ satisfies (P4). 
Without loss of generality, we may assume 
that $K$ is 
connected. 
Since $\mathcal O_X[x_1, \cdots, x_l]$ and 
$\mathcal A^{(d)}$ admit 
$\mathcal O_X$-exhaustions, 
we obtain the surjection 
\begin{equation}\label{a-lem2.26-eq2} 
\mathcal O_X(K)[x_1, \cdots, x_l]\to \mathcal 
A^{(d)}(K)\to 0
\end{equation}  
induced by \eqref{a-lem2.26-eq1}. 
This means that $\mathcal A^{(d)}(K)$ is a finitely 
generated graded $\mathcal O_X(K)$-algebra. 
Note that $\mathcal O_X(K)$ is noetherian 
and that $\mathcal A(K)$ is an integral domain. 
Therefore, we see that $\mathcal A(K)$ is 
a finitely generated graded 
$\mathcal O(K)$-algebra (see, for example, 
\cite[Lemma 2.25 (ii)]{cascini-lazic} and 
\cite[Proposition 1.1.2.5]{cox-rings}). 
Since $K$ is a Stein compact subset, for any nonnegative 
integer 
$m$, $\mathcal A_m(K)$ generates $\mathcal A_{m, x}$ for 
every $x\in K$. 
Hence 
we can find a surjective $\mathcal O_U$-algebra homomorphism 
\begin{equation*}
\mathcal O_U[y]=\mathcal O_U[y_1, \cdots, y_k]\twoheadrightarrow 
\mathcal A|_U
\end{equation*} 
for some $k$ such that $y_j$ is mapped to 
a homogeneous element of $H^0(U, \mathcal A)$ for 
every $j$. 
This means that $\mathcal A$ is a locally finitely generated 
graded $\mathcal O_X$-algebra. 
\end{step}
\begin{step}\label{a-lem2.26-step2} 
As in Step \ref{a-lem2.26-step1}, we take an arbitrary point $P\in X$ 
and shrink $X$ around $P$. 
Then there exists a surjective $\mathcal O_X$-algebra homomorphism 
\begin{equation}\label{a-lem2.26-eq3}
\mathcal O_X[x]=\mathcal O_X[x_1, \cdots, x_l]
\twoheadrightarrow 
\mathcal A
\end{equation}
for some $l$ such that $x_i$ is mapped to a 
homogeneous element of $H^0(X, \mathcal A)$ for every $i$. 
In this case, it is easy to see that 
there exists a surjective $\mathcal O_X$-algebra homomorphism 
\begin{equation*}
\mathcal O_X[y]=\mathcal O_X[y_1, \cdots, y_k]\twoheadrightarrow 
\mathcal A^{(d)}
\end{equation*} 
for some $k$ such that $y_j$ is mapped to 
a homogeneous element of $H^0(X, \mathcal A^{(d)})$ for 
every $j$ (see, for example, 
\cite[Lemma 2.25 (i)]{cascini-lazic} and \cite[Proposition 
1.1.2.4]{cox-rings}). 
Hence $\mathcal A^{(d)}$ is a locally finitely generated 
graded $\mathcal O_X$-algebra. 
\end{step}
We finish the proof. 
\end{proof}

In order to construct flips and log canonical models in the category 
of complex analytic spaces, 
we need the notion of $\Projan$. 
For the details of $\Projan$, 
see \cite[\S 1.b.~Spec and Proj]{nakayama3}. 

\begin{rem}[$\Projan$]\label{a-rem2.27} 
Let $X$ be a complex analytic space and 
let $\mathcal A=\bigoplus _{m\in \mathbb N}\mathcal A_m$ 
be a locally finitely generated graded $\mathcal O_X$-algebra 
such that $\mathcal A_m$ is a coherent $\mathcal O_X$-modules 
for every $m$. 
Then we can define an analytic space 
$\Projan_X\mathcal A$ which is proper over $X$. 
More generally, we can define $\Projan_X\mathcal A$ under 
a weaker assumption that 
$\mathcal A$ is locally of finite presentation. 
Let $\mathcal F$ be a coherent $\mathcal O_X$-module. 
We note that $\Projan_X \Sym \mathcal F$ 
is usually denoted by $\mathbb P_X(\mathcal F)$, 
where $\Sym \mathcal F=
\bigoplus _{m\in \mathbb N}\Sym ^m \mathcal F$. 
\end{rem}

Now we can define projective morphisms of complex analytic spaces. 

\begin{defn}\label{a-def2.28}
Let $\pi\colon X\to Y$ be a proper 
morphism of complex analytic spaces and 
let $\mathcal L$ be a line bundle on $X$. 
Then $\mathcal L$ is said to be {\em{$\pi$-very 
ample}} or {\em{relatively very ample over 
$Y$}} if $\mathcal L$ is {\em{$\pi$-free}}, 
that is, $\pi^*\pi_*\mathcal L\to \mathcal L$ is surjective, 
and the induced morphism 
$X\to \mathbb P_Y(f_*\mathcal L)$ over 
$Y$ is a closed embedding. 
A line bundle $\mathcal L$ on $X$ is called 
{\em{$\pi$-ample}} or {\em{ample over 
$Y$}} if for any point $y\in Y$ there are an 
open neighborhood $U$ of $y$ and a positive 
integer $m$ such that $\mathcal L^{\otimes m}|_{\pi^{-1}(U)}$ 
is relatively very ample over $U$. 
Let $D$ be a Cartier divisor on $X$. Then 
we say that $D$ is {\em{$\pi$-very ample}}, {\em{$\pi$-free}}, 
and {\em{$\pi$-ample}} if the line bundle $\mathcal O_X(D)$ is 
so, respectively. 
We note that $\pi$ is said to be {\em{projective}} 
when there exists a $\pi$-ample line bundle 
on $X$. 
\end{defn}

For the basic properties of $\pi$-ample line bundles, 
see \cite[Chapter IV]{banica} and \cite[Chapter II.~\S1.~c.~Ample 
line bundles]{nakayama3}. 

\begin{defn}[Semiampleness]\label{a-def2.29}
Let $\pi\colon X\to Y$ be a proper morphism 
of complex analytic spaces and let $\mathcal L$ be a 
line bundle on $X$. 
If there exist an open covering $Y=\bigcup _{\lambda\in \Lambda} 
U_\lambda$ and positive integers $m_\lambda$ such that 
$\mathcal L^{\otimes m_\lambda}|_{\pi^{-1}(U_\lambda)}$ is 
$\pi|_{\pi^{-1}(U_\lambda)}$-free for every $\lambda\in \Lambda$, then 
$\mathcal L$ is called {\em{$\pi$-semiample}} 
or {\em{relatively semiample over $Y$}}. 
Let $D$ be a Cartier divisor on $X$. 
If $\mathcal O_X(D)$ is $\pi$-semiample, 
then $D$ is called {\em{$\pi$-semiample}} 
or {\em{relatively semiample over $Y$}}. 
\end{defn}

Here, we recall the precise definition of {\em{bimeromorphic 
maps}} for the sake of completeness. 

\begin{defn}[Meromorphic 
maps]\label{a-def2.30}
A {\em{meromorphic map}} $f\colon X\dashrightarrow Y$ 
of complex analytic varieties is defined by 
the graph $\Gamma_f\subset X\times Y$ such that 
$\Gamma_f$ is a subvariety of $X\times Y$ and that 
the first projection is an isomorphism 
over a Zariski open dense subset of $X$. 
Note that a {\em{Zariski open subset}} is 
the complement of an analytic subset. 
If further the second projection $\Gamma _f\to Y$ is 
proper and is an isomorphism over a Zariski open dense 
subset of $Y$, then $f\colon X\dashrightarrow 
Y$ is called a {\em{bimeromorphic map}}. 
We say that a bimeromorphic map $f\colon X\dashrightarrow Y$ 
of normal complex varieties is a 
{\em{bimeromorphic contraction}} if $f^{-1}$ does not 
contract any divisors. 
If in addition $f^{-1}$ is also a bimeromorphic 
contraction, then we say that $f$ is a {\em{small}} 
bimeromorphic map. 
Let $f\colon X\to Y$ be a bimeromorphic morphism 
of complex normal varieties, equivalently, 
$f\colon X\dashrightarrow Y$ is a bimeromorphic map 
of normal complex varieties and the first projection $\Gamma_f\to X$ 
is an isomorphism. 
Then, we put $\Exc(f):=\{x\in X\, |\, {\text{$f$ is not 
an isomorphism at $x$}}\}$ and call it the 
{\em{exceptional locus}} of $f$. 
\end{defn}

\begin{rem}\label{a-rem2.31}
Let $X$ be a complex analytic space and let 
$U$ be a Zariski open subset 
of $X$. 
Let $V$ be a Zariski open subset of $U$. 
Unfortunately, $V$ is not necessarily a Zariski open subset of 
$X$. 
This is because the analytic subset $\Gamma :=U\setminus 
V$ of $U$ can not always be extended to 
an analytic subset of $X$. 
\end{rem}

In this paper, we discuss the minimal model program. 
Therefore, we need {\em{$\mathbb Q$-divisors}} and 
{\em{$\mathbb R$-divisors}}. 

\begin{defn}[Divisors, $\mathbb Q$-divisors, 
and $\mathbb R$-divisors]\label{a-def2.32}
Let $X$ be a normal complex variety. 
A {\em{prime divisor}} on $X$ is an irreducible 
and reduced closed subvariety of codimension one. 
An {\em{$\mathbb R$-divisor}} $D$ on $X$ is a formal 
sum 
\begin{equation*}
D=\sum _i a_i D_i, 
\end{equation*} 
where $D_i$ is a prime divisor on $X$ with 
$D_i\ne D_j$ for $i\ne j$, 
$a_i\in \mathbb R$ for every $i$, and the {\em{support}} 
\begin{equation*}
\Supp D:=\bigcup _{a_i\ne 0}D_i
\end{equation*} 
is a closed analytic subset of $X$. 
In other words, the formal sum $\sum _i a_i 
D_i$ is locally finite. 
If $a_i\in \mathbb Z$ (resp.~$a_i\in 
\mathbb Q$) for 
every $i$, then $D$ is called 
a {\em{divisor}} (resp.~{\em{$\mathbb Q$-divisor}}) on $X$. 
Note that a divisor is sometimes called 
an {\em{integral Weil divisor}} in 
order to emphasize the condition that $a_i\in \mathbb Z$ for every $i$. 
If $0\leq a_i\leq 1$ (resp.~$a_i\leq 1$) 
holds for every $i$, then 
an $\mathbb R$-divisor $D$ is called a {\em{boundary}} 
(resp.~{\em{subboundary}}) $\mathbb R$-divisor. 

Let $D=\sum _i a_i D_i$ be an $\mathbb R$-divisor 
on $X$ such that $D_i$ is a prime divisor 
for every $i$ with $D_i\ne D_j$ for $i\ne j$. 
The {\em{round-down}} $\lfloor D\rfloor$ of $D$ is 
defined to be the divisor 
\begin{equation*}
\lfloor D\rfloor =\sum _i \lfloor a_i\rfloor D_i. 
\end{equation*} 
The {\em{round-up}} and the 
{\em{fractional part}} of $D$ are defined to be 
\begin{equation*}
\lceil D \rceil :=-\lfloor -D\rfloor, \quad 
\text{and} \quad \{D\}:=D-\lfloor D\rfloor, 
\end{equation*} 
respectively. We put 
\begin{equation*}
D^{=1}=\sum _{a_i=1}D_i, \quad 
D^{<1}:=\sum _{a_i<1} a_i D_i, \quad \text{and} \quad 
D^{\geq 1}:=\sum _{a_i\geq 1}a_i D_i. 
\end{equation*}
We sometimes use 
\begin{equation*}
D_+:=\sum _{a_i>0} a_i D_i, \quad 
\text{and} \quad D_-:=-\sum _{a_i<0} a_i D_i \geq 0. 
\end{equation*}
By definition, $D=D_+-D_-$ holds. 

Let $D$ be an $\mathbb R$-divisor on $X$ 
and let $x$ be a point of $X$. 
If $D$ is written as a finite $\mathbb R$-linear 
(resp.~$\mathbb Q$-linear) combination of Cartier 
divisors on some open 
neighborhood of $x$, 
then $D$ is said to be {\em{$\mathbb R$-Cartier at $x$}} 
(resp.~{\em{$\mathbb Q$-Cartier at $x$}}). 
If $D$ is $\mathbb R$-Cartier 
(resp.~$\mathbb Q$-Cartier) at $x$ for every $x\in X$, 
then $D$ is said to be {\em{$\mathbb R$-Cartier}} 
(resp.~{\em{$\mathbb Q$-Cartier}}). 
Note that a $\mathbb Q$-Cartier $\mathbb R$-divisor 
$D$ is automatically a $\mathbb Q$-Cartier 
$\mathbb Q$-divisor by definition. 
If $D$ is a finite $\mathbb R$-linear (resp.~$\mathbb Q$-linear) 
combination of Cartier divisors on $X$, 
then we sometimes say that $D$ 
is a {\em{globally $\mathbb R$-Cartier $\mathbb R$-divisor}} 
(resp.~{\em{globally $\mathbb Q$-Cartier $\mathbb Q$-divisor}}).  
\end{defn}

Example \ref{a-ex2.33} below shows a big difference between 
divisors on algebraic varieties and those on complex analytic 
spaces. 

\begin{ex}[Weierstrass]\label{a-ex2.33}
Let $D$ be a divisor on $\mathbb C$. We note that 
$\Supp D$ may be any discrete subset of $\mathbb C$. 
By the classical Weierstrass theorem, 
we can construct a meromorphic function $f$ 
on $\mathbb C$ such that $\ddiv (f)=D$. 
\end{ex}

\begin{defn}\label{a-def2.34}
Let $X$ be a normal variety. 
A real vector space 
spanned by the prime divisors on $X$ is 
denoted by $\WDiv_{\mathbb R}(X)$. 
It has a canonical basis given 
by the prime divisors. 
Let $D$ be an element of $\WDiv_{\mathbb R}(X)$. 
Then the sup norm of $D$ with respect to 
this basis is denoted by $|\!|D|\!|$. 
Note that an $\mathbb R$-divisor $D$ is an 
element of $\WDiv_{\mathbb R}(X)$ if and only if 
$\Supp D$ has only finitely many irreducible components. 
\end{defn}

We need the notion of semiample $\mathbb Q$-divisors. 

\begin{defn}[Relatively semiample $\mathbb Q$-divisors]
\label{a-def2.35} 
Let $\pi\colon X\to Y$ be a projective 
morphism of complex analytic spaces such that $X$ is a normal 
variety. A $\mathbb Q$-Cartier 
$\mathbb Q$-divisor $D$ on $X$ is 
called a {\em{$\pi$-semiample $\mathbb Q$-divisor}} 
on $X$ 
if it is a finite $\mathbb Q_{>0}$-linear combination of 
$\pi$-semiample Cartier divisors on $X$. 
\end{defn}

The following lemma is very important. 

\begin{lem}\label{a-lem2.36} 
Let $\pi\colon X\to Y$ be a projective morphism of 
complex analytic spaces such that $X$ is a normal complex 
variety and let $D$ be a $\pi$-semiample 
$\mathbb Q$-divisor on $X$. 
Then $\bigoplus _{m\in \mathbb N}
\pi_*\mathcal O_X(\lfloor mD\rfloor)$ is a locally 
finitely generated graded $\mathcal O_Y$-algebra. 
In particular, if $\mathcal L$ is a $\pi$-ample 
line bundle on $X$, then 
$\bigoplus _{m\in \mathbb N} 
\pi_*\mathcal L^{\otimes m}$ is a locally 
finitely generated graded $\mathcal O_Y$-algebra. 
\end{lem}

\begin{proof}
Throughout this proof, we fix a point $y\in Y$ and repeatedly 
shrink $Y$ around 
$y$. 
In Step \ref{a-lem2.36-step1}, 
we will reduce the problem to the case 
where $\mathcal O_X(D)$ is $\pi$-ample. 
In Step \ref{a-lem2.36-step2}, 
we will prove the desired finite generation. 
\setcounter{step}{0}
\begin{step}\label{a-lem2.36-step1}
By shrinking $Y$ around $y$, we may assume that 
there exists a positive integer $d$ such that 
$\mathcal O_X(dD)$ is $\pi$-free. 
By Lemma \ref{a-lem2.26}, 
we may further assume that 
$\mathcal O_X(D)$ is $\pi$-free 
by replacing $D$ with $dD$. 
We consider a contraction morphism 
over $Y$ associated to $\mathcal O_X(D)$ and 
take the Stein factorization.  
Then there exist a contraction morphism 
$\varphi\colon X\to Z$ over $Y$ with $\varphi_*\mathcal 
O_X\simeq \mathcal O_Z$ and 
some $\pi_Z$-ample line bundle $\mathcal L$ on $Z$, where 
$\pi_Z\colon Z\to Y$ is the structure morphism, such that 
$\mathcal O_X(D)\simeq \varphi^*\mathcal L$. 
Since $\pi_*\mathcal O_X(mD)\simeq (\pi_Z)_*\mathcal L^{\otimes m}$ 
holds for every $m$, 
we may further assume that 
$\mathcal O_X(D)$ is $\pi$-ample by 
replacing $X$ and $\mathcal O_X(D)$ with $Z$ and $\mathcal L$, 
respectively. 
\end{step}
\begin{step}\label{a-lem2.36-step2} 
From now on, we assume that 
$\mathcal O_X(D)$ is $\pi$-ample. 
By Lemma \ref{a-lem2.26}, we may further assume that 
$\mathcal O_X(D)$ is $\pi$-very ample. 
Therefore, after shrinking $Y$ around $y$ suitably, 
there exists the following commutative diagram 
\begin{equation*}
\xymatrix{
X \ar[dr]_-\pi\ar@{^{(}->}[r]^-\iota&
\ar[d]^-{p_1}Y\times \mathbb P^N \ar[r]_-{p_2}& \mathbb P^N\\ 
& Y& 
}
\end{equation*}
such that $\mathcal O_X(D)\simeq \iota^*p^*_2\mathcal O_{\mathbb P^N}(1)$. 
We put $\mathcal N:=p^*_2\mathcal O_{\mathbb P^N}(1)$. 
Then there exists a positive integer $m_0$ 
such that $(p_1)_*\mathcal N^{\otimes m} 
\to \pi_*\mathcal O_X(mD)$ is surjective for every $m\geq m_0$. 
We note that 
\begin{equation*}
(p_1)_*\mathcal N^{\otimes m} 
\simeq (p_1)_*p^*_2\mathcal O_{\mathbb P^N}(1)\simeq 
\mathcal O_Y[X_0, \cdots, X_N]_m, 
\end{equation*}
where $\mathcal O_Y[X_0, \cdots, X_N]_m$ is the degree 
$m$ part of $\mathcal O_Y[X_0, \cdots, X_N]$. 
Since $\pi_*\mathcal O_X(mD)$ is a coherent $\mathcal O_Y$-module 
for every $0\leq m<m_0$, 
after replacing $Y$ with a small Stein open neighborhood 
of $y$ if necessary, 
we see that there is a surjective $\mathcal O_Y$-algebra homomorphism 
\begin{equation*}
\mathcal O_Y[X_0, \cdots, X_N, X_{N+1}, \cdots, 
X_{N+M}]\twoheadrightarrow \bigoplus _{m\in 
\mathbb N} \pi_*\mathcal O_X(mD)
\end{equation*} 
for some $M$ such that 
each $X_i$ is mapped to an element of 
$H^0(Y, \pi_*\mathcal O_X(m_iD))$ for some $m_i \in 
\mathbb N$. 
This means that 
$\bigoplus_{m\in \mathbb N}\pi_*\mathcal O_X(mD)$ is a locally 
finitely generated graded $\mathcal O_Y$-algebra. 
\end{step}
By Step \ref{a-lem2.36-step1} and Step \ref{a-lem2.36-step2}, 
$\bigoplus _{m\in \mathbb N}\pi_*\mathcal O_X(\lfloor 
mD\rfloor)$ and $\bigoplus _{m\in \mathbb N} 
\pi_*\mathcal L^{\otimes m}$ are both locally finitely generated 
graded $\mathcal O_Y$-algebras.  
\end{proof}

For almost all applications, we may assume that 
any $\mathbb R$-divisor has finitely many components 
by the following lemma. 

\begin{lem}\label{a-lem2.37} 
Let $D$ be an $\mathbb R$-Cartier $\mathbb R$-divisor 
{\em{(}}resp.~$\mathbb Q$-divisor{\em{)}} 
on a normal complex variety $X$. 
Let $U$ be any relatively compact open subset of $X$. 
Then $D$ is a finite $\mathbb R$-linear 
{\em{(}}$\mathbb Q$-linear{\em{)}} combination 
of Cartier divisors in a neighborhood of $\overline U$, that is, 
$D$ is a globally $\mathbb R$-Cartier $\mathbb R$-divisor 
{\em{(}}globally $\mathbb Q$-Cartier $\mathbb Q$-divisor{\em{)}} 
in a neighborhood of $\overline {U}$. 
\end{lem}

\begin{proof}
Without loss of generality, we may assume that 
$D$ is a finite $\mathbb R$-linear combination 
of prime divisors by shrinking $X$. 
We write $D=\sum _{i=1}^ka_i D_i$, where $a_i\in \mathbb R$ and 
$D_i$ is a prime divisor for every $i$ with $D_i\ne 
D_j$ for $i\ne j$. 
We consider the following $\mathbb R$-vector space 
\begin{equation*}
V=\{ x_1D_1+\cdots +x_k D_k \, |\, 
{\text{$x_i \in \mathbb R$ for every $i$}}\}\simeq 
\mathbb R^k. 
\end{equation*} 
We take an arbitrary point $x\in X$. 
Then there exists an open neighborhood 
$U_x$ of $x$ such that $D|_{U_x}$ is a finite $\mathbb R$-linear 
combination of Cartier divisors. Hence 
we can find an affine subspace $V^x$ of $V$ defined 
over the rationals such that $D\in V^x$ and 
that any member of $V^x$ is $\mathbb R$-Cartier 
in a neighborhood of $x$. 
Since $U$ is relatively compact, 
there exists an affine subspace $\Sigma$ of $V$ defined 
over the rationals such that 
$D\in \Sigma$ and that every element of $\Sigma$ is 
$\mathbb R$-Cartier in a neighborhood of $\overline U$. 
Hence, by the standard argument, we can write 
$D$ as a finite $\mathbb R$-linear combination of 
Cartier divisors in a neighborhood of $\overline U$. 
When $D$ is a $\mathbb Q$-divisor, it can be written as 
a finite $\mathbb Q$-linear combination of Cartier divisors 
in a neighborhood of $\overline U$. 
Thus, we get the desired statement. 
\end{proof}

The definition of {\em{$\mathbb Q$-factoriality}} 
is very subtle. 

\begin{defn}[{$\mathbb Q$-factoriality, 
see \cite[Definition 4.13]{nakayama2}}]\label{a-def2.38}
Let $X$ be a normal complex variety 
and let $K$ be a compact subset of $X$. 
Then $X$ is said to be {\em{$\mathbb Q$-factorial at $K$}} 
if every prime divisor defined on an open neighborhood 
$U$ of $K$ is $\mathbb Q$-Cartier at any point $x\in K$. 

Let $\pi\colon X\to Y$ be a projective 
morphism and let $W$ be a compact subset of $Y$. 
If $X$ is $\mathbb Q$-factorial at $\pi^{-1}(W)$, 
then we usually say that 
$X$ is {\em{$\mathbb Q$-factorial over $W$}}. 
\end{defn}

\begin{rem}\label{a-rem2.39} 
Let $\pi\colon X\to Y$ be a projective morphism 
and let $W$ be a compact subset of $Y$. 
We take a compact subset $W'$ of $Y$ with $W'\subset W$. 
It is very important to note that 
$X$ is not necessarily $\mathbb Q$-factorial 
over $W'$ even if $X$ is $\mathbb Q$-factorial over $W$. 
This is because there may exist a divisor defined over 
an open neighborhood of $\pi^{-1}(W')$ which can not 
be extended to a divisor defined over 
an open neighborhood of $\pi^{-1}(W)$. 
\end{rem}

We adopt the following definition of 
{\em{linear, $\mathbb Q$-linear, and 
$\mathbb R$-linear equivalences}} in this paper. 
Although it may be somewhat artificial, it is sufficient for 
our minimal model program for projective morphisms 
of complex analytic spaces. 

\begin{defn}[Linear, $\mathbb Q$-linear, and 
$\mathbb R$-linear equivalences]\label{a-def2.40} 
Two $\mathbb R$-divisors $D_1$ and $D_2$ are said to 
be {\em{linearly equivalent}} if 
$D_1-D_2$ is a principal Cartier divisor. 
The linear equivalence is denoted by $D_1\sim D_2$. 
Two $\mathbb R$-divisors $D_1$ and $D_2$ are 
said to be {\em{$\mathbb R$-linearly equivalent}} 
(resp.~{\em{$\mathbb Q$-linearly equivalent}}) 
if $D_1-D_2$ is a {\em{finite}} $\mathbb R$-linear 
(resp.~$\mathbb Q$-linear) combination 
of principal Cartier divisors. 
When $D_1$ is $\mathbb R$-linearly (resp.~$\mathbb Q$-linearly) 
equivalent to 
$D_2$, we write $D_1\sim _{\mathbb R}D_2$ 
(resp.~$D_1\sim _{\mathbb Q}D_2$). 
\end{defn}

\begin{ex}\label{a-ex2.41}
Let $X$ be a noncompact complex manifold with $\dim X=1$. 
Then it is known that $X$ is always 
Stein. 
We assume that $H^2(X, \mathbb Z)=0$ holds. 
Let $D$ be an $\mathbb R$-divisor 
on $X$ such that $\Supp D$ is finite. 
Then $D\sim _{\mathbb R}0$, that is, 
$D$ is $\mathbb R$-linearly trivial. 
On the other hand, if $\Supp D$ is not finite, 
then $D$ is not necessarily 
$\mathbb R$-linearly trivial in the sense of 
Definition \ref{a-def2.40}. 
As in Example \ref{a-ex2.33}, if 
$D$ is an integral Weil divisor on $X$, 
then $D\sim 0$ always holds. 
\end{ex}

We will use the following lemma in the proof of 
Theorem \ref{thm-f} (3). 

\begin{lem}[{see \cite[Lemma 1.12]{kawamata-crepant} 
and \cite[Chapter II.~2.12.~Lemma]{nakayama3}}]\label{a-lem2.42} 
Let $X$ be a normal complex variety with 
only rational singularities and let $K$ be 
a compact subset of $X$. 
Let $D_i$ be an integral Weil divisor on $X$ such that 
$D_i$ is $\mathbb Q$-Cartier at $K$ for $1\leq i\leq k$. 
Then there exists a positive integer $m$ such that 
$mD_i$ is Cartier on some open neighborhood of $K$ for 
every $1\leq i\leq k$. 
\end{lem}
\begin{proof}
We take an arbitrary point $x\in K$. 
By \cite[Chapter II.~2.12.~Lemma]{nakayama3} 
(see \cite[Lemma 1.12]{kawamata-crepant}), 
there exists a positive integer $m_x$ such that 
$m_xD_i$ is Cartier at $x$ for $1\leq i\leq k$. 
This means that there exists an open neighborhood 
$U_x$ of $x$ such that 
$m_xD_i$ is Cartier on $U_x$ for every $1\leq i\leq k$. 
Since $K$ is compact, 
we can take a positive integer $m$ and an open 
neighborhood 
$U$ of $K$ such that 
$mD_i$ is Cartier for every $1\leq i\leq k$. 
This is what we wanted. 
\end{proof}

In this paper, we usually consider the case where 
the base space $Y$ is Stein and the morphism $\pi\colon X\to Y$ is 
projective. In this setting, 
we have many good properties. 

\begin{rem}\label{a-rem2.43}
Let $\pi\colon X\to Y$ be a projective morphism 
from a normal complex variety $X$ to a 
Stein space $Y$. 
We take a $\pi$-ample 
line bundle $\mathcal A$ on $X$. 
Let $\omega_X$ be the canonical sheaf of $X$ 
(see Definition \ref{a-def3.1} below). 
Since there exists a sufficiently large positive 
integer $m$ such that 
\begin{equation*} 
H^0(X, \omega_X\otimes \mathcal A^{\otimes m})
\simeq H^0(Y, \pi_*(\omega_X\otimes \mathcal A^{\otimes m}))
\ne 0
\end{equation*} 
and 
\begin{equation*}
H^0(X, \mathcal A^{\otimes m})\simeq H^0(Y, \pi_*\mathcal 
A^{\otimes m})\ne 0, 
\end{equation*} 
we can always take a Weil divisor $K_X$ on $X$ satisfying 
$\omega_X\simeq \mathcal O_X(K_X)$. 
As usual, we call it the {\em{canonical divisor}} of $X$. 
More generally, let $\mathcal L$ be a line bundle 
(resp.~reflexive sheaf of rank one) on $X$. 
By the same argument as above, 
we can take a Cartier (resp.~Weil) divisor 
$D$ on $X$ such that $\mathcal L\simeq 
\mathcal O_X(D)$. 
\end{rem}

\begin{say}[Ample, semiample, big, pseudo-effective, and nef 
$\mathbb R$-divisors]\label{a-say2.44} 
In our framework of the minimal model 
program for projective morphisms of complex analytic spaces, 
we have to use $\mathbb R$-divisors. Hence 
we need the following definitions:~Definitions 
\ref{a-def2.45}, \ref{a-def2.46}, \ref{a-def2.47}, 
and \ref{a-def2.48}. 
We state them explicitly here for the sake of completeness. 

\begin{defn}[Ample and semiample 
$\mathbb Q$-divisors and 
$\mathbb R$-divisors, see Definition \ref{a-def2.35}]\label{a-def2.45} 
Let $\pi\colon X\to Y$ be a projective morphism 
of complex analytic spaces. 
A finite $\mathbb R_{>0}$-linear (resp.~$\mathbb Q_{>0}$-linear) 
combination of $\pi$-ample 
Cartier divisors is called a {\em{$\pi$-ample 
$\mathbb R$-divisor}} (resp.~{\em{$\pi$-ample 
$\mathbb Q$-divisor}}). 
A finite $\mathbb R_{>0}$-linear 
(resp.~$\mathbb Q_{>0}$-linear) combination of $\pi$-semiample 
Cartier divisors is called a {\em{$\pi$-semiample 
$\mathbb R$-divisor}} 
(resp.~{\em{$\pi$-semiample 
$\mathbb Q$-divisor}}). 
\end{defn}

In this paper, we adopt the following 
definition of big $\mathbb R$-divisors. 

\begin{defn}[Bigness]\label{a-def2.46}
Let $\pi\colon X\to Y$ be a projective morphism 
of complex analytic spaces such that 
$X$ is a normal complex variety. 
When $Y$ is Stein, an $\mathbb R$-divisor $D$ is 
said to be {\em{big 
over $Y$}} or {\em{$\pi$-big}} if $D\sim _{\mathbb R} A+B$, 
where $A$ is a $\pi$-ample 
$\mathbb R$-divisor and $B$ is an effective 
$\mathbb R$-divisor. 
In general, if $D|_{\pi^{-1}(U)}$ is big over $U$ for 
any Stein open subset of $Y$, then $D$ is said to be {\em{big over $Y$}} 
or {\em{$\pi$-big}}. 
We note that $D$ is not necessarily $\mathbb R$-Cartier. 
\end{defn}

\begin{defn}[Pseudo-effective $\mathbb R$-divisors]\label{a-def2.47}
Let $\pi\colon X\to Y$ be a projective morphism 
of complex analytic spaces such that $X$ is a normal complex variety. 
An $\mathbb R$-Cartier $\mathbb R$-divisor $D$ on $X$ is said 
to be {\em{pseudo-effective over $Y$}} or 
{\em{$\pi$-pseudo-effective}} 
if $D+A$ is big over $Y$ for every 
$\pi$-ample 
$\mathbb R$-divisor $A$ on $X$. 
\end{defn}

\begin{defn}[Nefness]\label{a-def2.48}
Let $\pi\colon X\to Y$ be a projective 
morphism of complex analytic spaces such that 
$X$ is a normal complex 
variety and let $W$ be a compact subset of $Y$. 
Let $D$ be an $\mathbb R$-Cartier 
$\mathbb R$-divisor on $X$. 
If $D\cdot C\geq 0$ for every projective 
integral curve $C$ on $X$ such that 
$\pi(C)$ is a point, 
then $D$ is said to be {\em{$\pi$-nef}} 
or {\em{nef over $Y$}}. 
If $D\cdot C\geq 0$ for every projective integral 
curve $C$ on $X$ such that 
$\pi(C)$ is a point of $W$, 
then $D$ is said to be {\em{nef over $W$}} 
or {\em{$\pi$-nef over $W$}}. 
\end{defn}
\end{say}

Let us recall the definition of {\em{analytically meagre subsets}}. 
As we saw in Remark \ref{a-rem2.31}, the notion of 
Zariski open subsets does not work well in the category of 
complex analytic spaces. So we frequently have to use 
analytically meagre subsets. 

\begin{defn}[Analytically meagre subsets]\label{a-def2.49} 
A subset $\mathcal S$ of a complex analytic space $X$ 
is said to be {\em{analytically
meagre}} if
\begin{equation*}
\mathcal S\subset \bigcup _{n\in \mathbb N} Y_n, 
\end{equation*} 
where each $Y_n$ is a locally closed analytic 
subset of $X$ of codimension $\geq 1$.
\end{defn}

\begin{defn}[Analytically sufficiently general points and 
fibers]\label{a-def2.50}
Let $X$ be a complex analytic space. We say that 
a property $P$ holds for an analytically sufficiently 
general point $x\in X$ when $P$ holds 
for every point $x$ contained in $X\setminus \mathcal S$ for some 
analytically meagre subset $\mathcal S$ of $X$. 

Let $f\colon X\to Y$ be a morphism of analytic spaces. 
Similarly, we say that a property $P$ holds for an 
{\em{analytically sufficiently general fiber}} of $f\colon X\to Y$ when 
$P$ holds for $f^{-1}(y)$ for 
every $y\in Y\setminus \mathcal S$, 
where $\mathcal S$ is some analytically meagre subset of $Y$. 
\end{defn}

We sometimes use the notion of general $\pi$-ample $\mathbb Q$-divisors. 

\begin{defn}\label{a-def2.51}
Let $\pi\colon X\to Y$ be a 
projective morphism from a normal 
variety $X$ to a Stein space $Y$. 
Let $A$ be a $\pi$-ample $\mathbb Q$-divisor on $X$. 
We say that $A$ is a {\em{general}} $\pi$-ample 
$\mathbb Q$-divisor on $X$ if there exist 
\begin{itemize}
\item[(i)] a large and divisible positive integer $k$ such that 
$kA$ is $\pi$-very ample, 
\item[(ii)] a finite-dimensional linear subspace 
$V$ of $H^0(X, \mathcal O_X(kA))$ 
which generates $\mathcal O_X(kA)$, and 
\item[(iii)] some analytically meagre subset $\mathcal S$ 
of $\Lambda:=
(V\setminus \{0\})/\mathbb C^\times$, 
\end{itemize} 
such that $A=\frac{1}{k} A'$ for some $A'\in \Lambda\setminus 
\mathcal S$. 
Note that $\Lambda\simeq \mathbb P^N$ for some $N\in \mathbb N$. 
\end{defn}

\begin{rem}\label{a-rem2.52}
Let $\pi\colon X\to Y$ be a projective 
morphism from a normal complex variety $X$ to a Stein space 
$Y$. Let $\sigma\colon Z\to X$ be a projective 
bimeromorphic morphism from a smooth variety $Z$ and let 
$\Sigma$ be a simple normal crossing divisor on $Z$. 
Let $\{p_i\}_{i\in \mathbb N}$ be a set of 
points of $X$. 
Let $A$ be a general $\pi$-ample $\mathbb Q$-divisor 
in the sense of Definition \ref{a-def2.51}. 
Then, by Bertini's theorem (see \cite[(II.5) Theorem]{manaresi} 
and \cite[Theorem 3.2]{fujino-bertini}), we may assume that  
$\sigma^{-1}_*A=\sigma^*A$ holds, $\Sigma+A'$ is 
a simple normal crossing divisor on $Z$, where 
$A=\frac{1}{k}A'$ as in Definition \ref{a-def2.51}, 
and $p_i\not\in \Supp A$ for every $i$, and so on. 
\end{rem}

The final result in this section is a very useful lemma. 
We will repeatedly use this lemma in the 
subsequent sections. 

\begin{lem}[{\cite[Lemma 3.2.1]{bchm} and \cite
[Lemma 2.10]{hashizume-hu}}]\label{a-lem2.53}
Let $\pi\colon X\to Y$ be a projective 
morphism of complex varieties such that 
$X$ is normal and that $Y$ is Stein. 
Let $D$ be a globally $\mathbb R$-Cartier 
$\mathbb R$-divisor, that is, a finite 
$\mathbb R$-linear combination of Cartier divisors on $X$. 
Let $F$ be an analytically sufficiently general fiber of $\pi\colon 
X\to Y$. 
Assume that $D':=D|_F\sim _{\mathbb R}B'\geq 0$ holds 
for some $\mathbb R$-divisor $B'$ on $F$. 
Then there exists a globally $\mathbb R$-Cartier 
$\mathbb R$-divisor $B$ on $X$  such that 
$D\sim _{\mathbb R}B\geq 0$. 
\end{lem}

We closely follow the proof of \cite[Lemma 2.10]{hashizume-hu}. 
In the proof of Lemma \ref{a-lem2.53}, 
we will freely use the semicontinuity theorem and 
the base change theorem described in 
\cite[Chapter III]{banica}, whose proof 
is much harder than the proof of the corresponding statements 
for algebraic varieties 
(see \cite[Chapter III, Section 12]{hartshorne}). 

\begin{defn}[Iitaka--Kodaira dimensions]\label{a-def2.54}
Let $X$ be a normal projective variety and let $D$ be a $\mathbb Q$-Cartier 
$\mathbb Q$-divisor on $X$. 
Then $\kappa (X, D)$ denotes the {\em{Iitaka--Kodaira dimension}} 
of $D$. 
\end{defn}

\begin{proof}[Proof of Lemma \ref{a-lem2.53}]
We can take a nonempty Zariski open subset $U$ of $Y$ such that 
$\pi\colon X\to Y$ is flat over $U$. 
\setcounter{step}{0}
\begin{step}\label{a-lem2.53-step1}
We fix a representation $D=\sum_{i=1}^{n}r_{i}D_{i}$ of $D$ 
as a finite $\mathbb R$-linear combination of Cartier divisors. 
Since $F$ is an analytically sufficiently general fiber of 
$\pi\colon X\to Y$, 
we may assume that $D_{i}|_F$ are well defined as 
integral Weil divisors for all $i$. 
For any closed point $y\in U$, 
the fiber of $\pi$ over $y$ is denoted by $X_y$. 
For any $\boldsymbol{p}=(p_1,
\ldots, p_n)\in \mathbb Q^{n}$, 
we set $D_{\boldsymbol{p}}=\sum_{i=1}^{n}p_{i}D_{i}$. 
We fix a positive integer 
$k_{\boldsymbol{p}}$ such 
that $k_{\boldsymbol{p}}D_{\boldsymbol{p}}$ is Cartier. 
For every $\boldsymbol{p}\in \mathbb Q^n$ and every 
$m\in \mathbb{Z}_{>0}$, we put 
\begin{equation*}
S_{\boldsymbol{p},m}
=\left\{ z\in U\,\left|\, \dim H^{0}(X_y,\mathcal{O}_{X_y}
(mk_{\boldsymbol{p}}D_{\boldsymbol{p}}|_{X_y}))=0 \right.\right\}. 
\end{equation*} 
Then $S_{\boldsymbol{p},m}=\emptyset$ holds or 
$U\setminus S_{\boldsymbol{p},m}=\emptyset$ is 
analytically meagre by the upper semicontinuity theorem. 
We set 
\begin{equation*} 
J:=\left\{ (\boldsymbol{p},m)\, 
|\, \boldsymbol{p}\in\mathbb{Q}^{n},m\in 
\mathbb{Z}_{>0}, S_{\boldsymbol{p},m}\neq \emptyset \right\},  
\end{equation*} 
and put 
\begin{equation*}
W=\bigcap_{(\boldsymbol{p},m)\in J}S_{\boldsymbol{p},m}. 
\end{equation*}
Then $U\setminus W$ is analytically meagre. 
Hence $Y\setminus W$ is also analytically meagre. 
We may assume that $F=X_{y_0}$ for some $y_{0}\in W$ 
since $F$ is an analytically sufficiently general fiber of 
$\pi\colon X\to Y$. 
Then, for any $\mathbb Q$-Cartier $\mathbb Q$-divisor 
$D_{\boldsymbol{p}'}$ associated to 
$\boldsymbol{p}'=(p'_{1},\ldots,p'_{n})\in \mathbb Q^{n}$, 
an inequality $\kappa(F ,D_{\boldsymbol{p}'}|_F)\geq 0$ holds 
if and only if 
$D_{\boldsymbol{p}'}
\sim_{\mathbb Q}E_{\mathbb{\boldsymbol{p}'}}$ for 
some $E_{\mathbb{\boldsymbol{p}'}}\geq 0$. 
Indeed, $\kappa(F ,D_{\boldsymbol{p}'}|_{F})\geq 0$ if 
and only if $y_{0}\not\in S_{\boldsymbol{p}',m}$ for some $m$. 
By 
the above definitions of $W$ and $J$, the 
condition $y_{0}\not\in S_{\boldsymbol{p}',m}$ 
is equivalent to $S_{\boldsymbol{p}',m}=\emptyset$ 
since $y_{0}\in W$. 
Since $k_{\boldsymbol{p}'}
D_{\boldsymbol{p}'}$ is a Cartier divisor and 
$Y$ is Stein, by the construction 
of $S_{\boldsymbol{p}',m}$, 
it is easy to check that $S_{\boldsymbol{p}',m}=\emptyset$ 
for some $m$ if and only if 
$D_{\boldsymbol{p}'}
\sim_{\mathbb Q}E_{\mathbb{\boldsymbol{p}'}}$ 
for some $E_{\mathbb{\boldsymbol{p}'}}\geq 0$. 
\end{step}
\begin{step}\label{a-lem2.53-step2}
From our assumption that $D'\sim _{\mathbb R} B'\geq 0$, 
there are positive real numbers 
$a_{1},\ldots, a_{s}$, effective integral Weil divisors 
$E_{1},\ldots, E_{s}$ on $F$, real numbers 
$b_{1},\ldots, b_{t}$, and meromorphic functions 
$\phi_{1},\ldots, \phi_{t}$ on $F$ 
such that 
\begin{equation*} 
D|_{F}=
\sum_{i=1}^{n}r_{i}D_{i}|_{F}
=\sum_{j=1}^{s}a_{j}E_{j}+
\sum_{k=1}^{t}b_{k}\cdot \ddiv(\phi_{k})
\end{equation*} 
holds as $\mathbb R$-divisors on $F$. 
We consider the following set 
\begin{equation*}
\Biggl\{\boldsymbol{v}'
=\bigl((r'_{i})_{i},(a'_{j})_{j},(b'_{k})_{k}\bigr)
\in \mathbb{R}^{n}\times (\mathbb{R}_{\geq0})^{s}
\times \mathbb{R}^{t}\Bigg{|}\sum_{i=1}^{n}r'_{i}D_{i}|_{F}
=\sum_{j=1}^{s}a'_{j}E_{j}+\sum_{k=1}^{t}b'_{k}
\cdot \ddiv(\phi_{k})\Biggr\},
\end{equation*}
which contains 
$\boldsymbol{v}:=\bigl((r_{i})_{i},(a_{j})_{j},(b_{k})_{k}\bigr)$. 
Since all $D_{i}|_{F}$ are well defined as 
integral Weil divisors, we can find positive real numbers 
$\alpha_{1},\ldots, \alpha_{l_{0}}$ and 
rational points 
$\boldsymbol{v}_{1},\ldots, \boldsymbol{v}_{l_{0}}$ 
in the above set such that 
$\sum_{l=1}^{l_{0}}\alpha_{l}
=1$ and $\sum_{l=1}^{l_{0}}
\alpha_{l}\boldsymbol{v}_{l}=\boldsymbol{v}$. 
This shows that 
there are $\mathbb Q$-Cartier 
$\mathbb Q$-divisors 
$D^{(1)},\ldots, D^{(l_{0})}$ on $X$ such that $\sum_{l=1}^{l_{0}}
\alpha_{l}D^{(l)}=D$ and $\kappa(F ,D^{(l)}|_{F})
\geq0$ for every $1\leq l\leq l_{0}$. 
We note that each $D^{(l)}$ is a finite 
$\mathbb Q$-linear combination of Cartier divisors 
by construction. 
By the argument in 
Step \ref{a-lem2.53-step1}, for every $1\leq l\leq l_{0}$, 
there exists a $\mathbb{Q}$-divisor $E^{(l)}\geq0$ on $X$ 
with  $D^{(l)}\sim_{\mathbb Q}E^{(l)}$. 
We put $B=\sum_{l=1}^{l_{0}}\alpha_{l}E^{(l)}$. 
Then we have $D\sim_{\mathbb R}B \geq0$. 
\end{step}
We finish the proof. 
\end{proof}

We close this section with an important remark on 
\cite[Lemma 3.2.1]{bchm}. 

\begin{rem}\label{a-rem2.55}
In \cite{bchm}, 
\cite[Lemma 3.2.1]{bchm} plays a crucial role. 
We think that the quasi-projectivity is 
indispensable in the framework of \cite{bchm} since we have 
to assume that $U$ is quasi-projective in \cite[Lemma 3.2.1]{bchm}. 
Let $\pi\colon X\to U$ be a projective morphism of normal complete 
algebraic varieties with connected fibers such that 
$X$ is a smooth projective variety and $\Pic(U)=\{0\}$. 
Let $D$ be a Cartier divisor on $X$ such that 
$-D$ is effective, $D\ne 0$, and $\pi(D)\subsetneq U$. 
Then there exists no effective $\mathbb R$-divisor 
$B$ on $X$ satisfying $D\sim _{\mathbb R, U} B\geq 0$. 
This means that \cite[Lemma 3.2.1]{bchm} does not 
always hold true without assuming the quasi-projectivity of $U$. 
\end{rem}

\section{Singularities of pairs}\label{a-sec3}
In this section, we will define singularities of pairs 
in the complex analytic setting. The definition 
is essentially the same as the one for algebraic varieties.

\begin{defn}[Singularities of pairs]\label{a-def3.1}
Let $X$ be a normal complex variety. The {\em{canonical 
sheaf}} $\omega_X$ of $X$ is the unique reflexive sheaf 
whose restriction to $X_{\mathrm{sm}}$ is isomorphic 
to the sheaf $\Omega^n_{X_{\mathrm{sm}}}$, 
where $X_{\mathrm{sm}}$ is the smooth locus 
of $X$ and $n=\dim X$. 
Let $\Delta$ be an $\mathbb R$-divisor 
on $X$. 
We say that $K_X+\Delta$ is $\mathbb R$-Cartier 
at $x\in X$ if there exist an open neighborhood 
$U_x$ of $x$ and 
a Weil divisor $K_{U_x}$ on $U_x$ with $\mathcal O_{U_x}
(K_{U_x})\simeq \omega_X|_{U_x}$ such that 
$K_{U_x}+\Delta|_{U_x}$ is $\mathbb R$-Cartier at 
$x$. We simply say that $K_X+\Delta$ is $\mathbb R$-Cartier 
when $K_X+\Delta$ is $\mathbb R$-Cartier at any point 
$x\in X$. Unfortunately, we can not define 
$K_X$ globally with $\mathcal O_X(K_X)\simeq 
\omega_X$. 
It only exists locally on $X$. However, we use the 
symbol $K_X$ as a formal divisor class 
with an isomorphism 
$\mathcal O_X(K_X)\simeq \omega_X$ and call it 
the {\em{canonical divisor}} of $X$ if there is no danger of 
confusion. 

Let $f\colon Y\to X$ be a proper bimeromorphic morphism 
between normal complex varieties. 
Suppose that $K_X+\Delta$ is $\mathbb R$-Cartier 
in the above sense. 
We take a small Stein open subset $U$ of $X$ where 
$K_U+\Delta|_U$ is a well-defined 
$\mathbb R$-Cartier $\mathbb R$-divisor on $U$. 
In this situation, we can define $K_{f^{-1}(U)}$ and $K_U$ such that 
$f_*K_{f^{-1}(U)}=K_U$. 
Then we can write 
\begin{equation*}
K_{f^{-1}(U)}=f^*(K_U+ 
\Delta|_U)+E_U
\end{equation*}
as usual. Note that $E_U$ is a well-defined 
$\mathbb R$-divisor on $f^{-1}(U)$ such that 
$f_*E_U=\Delta|_U$. 
Then we have the following formula 
\begin{equation*}
K_Y=f^*(K_X+\Delta)+
\sum _Ea(E, X, \Delta)E
\end{equation*} 
as in the algebraic case. 
We note that $\sum _Ea(E, X, \Delta)E$ is a globally well-defined 
$\mathbb R$-divisor on $Y$ such that 
$\left(\sum_E a(E, X, \Delta)E\right)|_{f^{-1}(U)}=E_U$ 
although $K_X$ and $K_Y$ are well defined 
only locally. 

If $\Delta$ is a boundary $\mathbb R$-divisor and 
$a(E, X, \Delta)\geq -1$ holds for any 
$f\colon Y\to X$ and every $f$-exceptional divisor $E$, 
then $(X, \Delta)$ is called a {\em{log canonical}} pair. 
If $(X, \Delta)$ is log canonical and 
$a(E, X, \Delta)>-1$ for any $f\colon Y\to X$ and 
every $f$-exceptional divisor $E$, 
then $(X, \Delta)$ is called a {\em{purely log terminal}} pair. 
If $(X, \Delta)$ is purely log terminal and $\lfloor 
\Delta\rfloor =0$, then $(X, \Delta)$ is called a 
{\em{kawamata log terminal}} pair. 
When $\Delta=0$ and $a(E, X, 0)\geq 0$ (resp.~$>0$) 
for 
any $f\colon Y\to X$ and 
every $f$-exceptional divisor $E$, 
we simply say that $X$ has only {\em{canonical 
singularities}} (resp.~{\em{terminal singularities}}). 

Let $X$ be a normal variety and let $\Delta$ be an 
effective $\mathbb R$-divisor on $X$ such that 
$K_X+\Delta$ is $\mathbb R$-Cartier. 
The image of $E$ with $a(E, X, \Delta)\leq -1$ for some $f\colon Y\to X$ 
is called 
a {\em{non-kawamata log terminal center}} of $(X, \Delta)$. 
The image of $E$ with $a(E, X, \Delta)=-1$ for some 
$f\colon Y\to X$ such that 
$(X, \Delta)$ is log canonical around general points of $f(E)$ is 
called a {\em{log canonical center}} of $(X, \Delta)$. 
When $(X, \Delta)$ is log canonical, a closed subset 
of $X$ is a log canonical 
center of $(X, \Delta)$ if and only if it is a non-kawamata log 
terminal center of $(X, \Delta)$ by definition. 
In the above setting, $(X, \Delta)$ is kawamata log terminal 
if and only if there are no non-kawamata log terminal 
centers of $(X, \Delta)$. 
\end{defn}

\begin{rem}\label{a-rem3.2}
If we only assume that $\Delta$ is a subboundary 
$\mathbb R$-divisor on $X$ in the above 
definition of 
log canonical pairs and kawamata log terminal pairs, 
then $(X, \Delta)$ is said to be a {\em{sub log canonical 
pair}} and {\em{sub kawamata log terminal pair}}, 
respectively. 
We will use sub log canonical pairs and sub kawamata 
log terminal pairs in Section \ref{a-sec21}. 
\end{rem}

\begin{rem}\label{a-rem3.3}
Let $X$ be a normal algebraic variety and let $\Delta$ be 
an $\mathbb R$-divisor on $X$ such that 
$K_X+\Delta$ is $\mathbb R$-Cartier. 
Let $X^{\an}$ be the complex analytic space 
naturally associated to $X$ and let $\Delta^{\an}$ 
be the $\mathbb R$-divisor on $X$ associated to $\Delta$. 
Then $(X^{\an}, \Delta^{\an})$ 
is terminal, canonical, kawamata log terminal, 
purely log terminal, and log canonical in 
the sense of Definition \ref{a-def3.1} if and only if 
$(X, \Delta)$ is terminal, canonical, kawamata log terminal, 
purely log terminal, and log canonical in 
the usual sense, respectively. 
For the details, see, for example, \cite[Proposition 4-4-4]{matsuki}. 
\end{rem}

In this paper, we need the following local definition of {\em{log canonical 
singularities}} and {\em{kawamata log terminal singularities}}. 

\begin{defn}\label{a-def3.4}
Let $X$ be a normal complex variety and let $\Delta$ be an 
effective $\mathbb R$-divisor 
on $X$. 
We say that $(X, \Delta)$ is {\em{log canonical}} 
(resp.~{\em{kawamata log terminal}}) {\em{at $x\in X$}} 
if there exits an open neighborhood $U_x$ of $x$ such that 
$(U_x, \Delta|_{U_x})$ is a log canonical 
pair (resp.~kawamata log terminal pair). 
We note that $(X, \Delta)$ is log 
canonical (resp.~kawamata log terminal) in the 
sense of Definition \ref{a-def3.1} if and only if 
$(X, \Delta)$ is log canonical 
(resp.~kawamata log terminal) at any point $x$ of $X$. 

Let $K$ be a compact subset of $X$. Then we say 
that $(X, \Delta)$ is {\em{log canonical}} 
(resp.~{\em{kawamata log terminal}}) {\em{at $K$}} if 
$(X, \Delta)$ is log canonical 
(resp.~kawamata log terminal) at any point $x$ of $K$. We note that 
$(X, \Delta)$ is log canonical (resp.~kawamata log 
terminal) at $K$ if and only if there exits an open neighborhood $U$ 
of $K$ such that $(U, \Delta|_U)$ is log canonical 
(resp.~kawamata log terminal). 
\end{defn}

The following lemma is very fundamental. 

\begin{lem}[{\cite[Lemma 3.7.2]{bchm}}]\label{a-lem3.5}
Let $X$ be a normal complex variety and let $V$ be 
a finite-dimensional affine subspace of 
$\WDiv_{\mathbb R} (X)$, which is defined over the 
rationals. 
Let $K$ be a compact subset of $X$. 
Then 
\begin{equation*}
\mathcal L(V; K):=\{\Delta\in V\, |\, 
\text{$K_X+\Delta$ is log 
canonical at $K$}\}
\end{equation*} 
is a rational polytope. 
Moreover, there 
exists an open neighborhood $U$ of $K$ such that 
$(U, \Delta|_U)$ is log canonical 
for every $\Delta\in \mathcal L(V; K)$. 
\end{lem}
\begin{proof}
We note that the set of divisors $\Delta$ such 
that $K_X+\Delta$ is $\mathbb R$-Cartier 
at $K$ forms an affine subspace $V'$ of $V$. 
Since $V$ is defined over the rationals, 
we can easily see that $V'$ is also defined 
over the rationals (see the proof of Lemma \ref{a-lem2.37}). 
Hence, by replacing $V$ with $V'$, 
we may assume that $K_X+\Delta$ is $\mathbb R$-Cartier 
at $K$ for every $\Delta\in V$. 
Since $V$ is finite-dimensional, 
there is an open neighborhood $U'$ of $K$ such that 
$K_X+\Delta$ is $\mathbb R$-Cartier on $U'$ for every 
$\Delta\in V$. 
By replacing $X$ with $U'$, we may assume 
that $K_X+\Delta$ is $\mathbb R$-Cartier for 
every $\Delta\in V$. 
Let $\Theta$ be the union of the support of any element of $V$. 
By shrinking $X$ around $K$, we can take a 
projective birational morphism 
$f\colon Y\to X$ from a smooth 
complex variety $Y$ such that 
$\Exc(f)$ and $\Exc(f)\cup \Supp f^{-1}_*\Theta$ are simple 
normal crossing divisors on $Y$. 
Thus, we can easily check that 
$\mathcal L(V; K)$ is a rational polytope. 
Let $\Delta_1, \ldots, \Delta_k$ be the vertices of $\mathcal L(V; K)$. 
Then $(X, \Delta_i)$ is log canonical on some open 
neighborhood $U_i$ of $K$ for every $i$. 
We put $U:=\bigcap_{i=1}^k U_i$. 
Then $(X, \Delta)$ is log canonical on $U$ for every $\Delta
\in \mathcal L(V; K)$. 
Hence $U$ is a desired open neighborhood of $K$. 
\end{proof}

We note the following elementary property. 
We explicitly state it for the sake of completeness. 

\begin{lem}\label{a-lem3.6}
Let $(X, \Delta)$ be a log canonical pair and let $C$ be 
an effective $\mathbb R$-Cartier $\mathbb R$-divisor 
on $X$ such that $(X, \Delta+C)$ is log canonical. 
Let $\varepsilon$ be any positive real number such that 
$0<\varepsilon \leq 1$. 
Then $V$ is a log canonical center of $(X, \Delta)$ if and only 
if $V$ is a log canonical center of 
$(X, \Delta+(1-\varepsilon)C)$.  
\end{lem}
\begin{proof}
This is obvious by definition. 
\end{proof}

The definition of {\em{divisorial log terminal pairs}} is 
very subtle. We adopt the following definition, 
which is suitable for our purposes. 

\begin{defn}[Divisorial log terminal pairs]\label{a-def3.7} 
Let $X$ be a normal complex variety and let $\Delta$ be 
a boundary $\mathbb R$-divisor 
on $X$ such that 
$K_X+\Delta$ is $\mathbb R$-Cartier. 
If there exists a proper bimeromorphic morphism 
$f\colon Y\to X$ from a smooth complex variety $Y$ such that 
$\Exc(f)$ and $\Exc(f)\cup \Supp f^{-1}_*\Delta$ are 
simple normal crossing divisors on $Y$ and that the 
discrepancy coefficient $a(E, X, \Delta)>-1$ holds 
for every $f$-exceptional divisor $E$, 
then $(X, \Delta)$ is called a 
{\em{divisorial log terminal pair}}. 
\end{defn}

\begin{rem}\label{a-rem3.8}
By definition, we can easily check that 
a divisorial log terminal pair is 
a log canonical pair. 
Let $(X, \Delta)$ be a kawamata log terminal pair and 
let $U$ be any relatively compact open subset of $X$. 
Then we can easily check that 
$(U, \Delta|_U)$ is a divisorial log terminal pair. 
\end{rem}

The morphism $f$ in Definition \ref{a-def3.7} 
can be taken as a composite of blow-ups over 
any relatively compact open subset.  

\begin{lem}\label{a-lem3.9}
Let $(X, \Delta)$ be a divisorial log terminal 
pair. 
Then there exists a morphism $\sigma\colon Z\to X$ such that, 
for any relatively compact open subset $X'$ of $X$,  
\begin{equation*}
g:=\sigma|_{\sigma^{-1}(X')}\colon 
Z':=\sigma^{-1}(X')\to X' 
\end{equation*} 
is a composite of a finite sequence of blow-ups, 
$\Exc(g)$ and $\Exc(g)\cup \Supp g^{-1}_*\Delta$ are simple 
normal crossing divisors on $Z'$, 
$a(E, X', \Delta|_{X'})>-1$ holds for every $g$-exceptional 
divisor $E$. 
In particular, we can take an effective divisor $F$ on $Z'$ such 
that $\Exc(g)=F$ and that $-F$ is $g$-very ample. 
\end{lem}

\begin{proof}
It is sufficient to apply the resolution of singularities 
explained in \cite[Sections 12 and 13]{bierstone-milman}. 
For the details, see 
\cite[Theorems 13.3 and 12.4]{bierstone-milman}. 
See also \cite[3.44 (Analytic spaces)]{kollar-resolution}, 
\cite{wlo}, and \cite[Theorem 10.45 and Proposition 
10.49]{kollar-singularities}. 
\end{proof}

Our definition of divisorial log terminal pairs is compatible with 
the usual definition of divisorial log terminal pairs for 
algebraic varieties. 

\begin{lem}\label{a-lem3.10}
Let $X$ be a normal algebraic variety and let $\Delta$ be 
an $\mathbb R$-divisor 
on $X$ such that 
$K_X+\Delta$ is $\mathbb R$-Cartier. 
Let $X^{\an}$ be the complex analytic space naturally 
associated to $X$ and let $\Delta^{\an}$ 
be the $\mathbb R$-divisor on $X$ associated to $\Delta$. 
Then $(X, \Delta)$ is divisorial log terminal in the usual sense 
if and only if $(X^{\an}, \Delta^{\an})$ is 
divisorial log terminal in the sense of Definition \ref{a-def3.7}. 
\end{lem}
\begin{proof}[Sketch of Proof of Lemma \ref{a-lem3.10}]
If $(X, \Delta)$ is divisorial log terminal in the usual sense, 
then it is obvious that 
$(X^{\an}, \Delta^{\an})$ is 
divisorial log terminal in the sense of Definition \ref{a-def3.7}. 
From now on, we assume that 
$(X^{\an}, \Delta^{\an})$ is 
divisorial log terminal in the sense of Definition \ref{a-def3.7}. 
Let $f\colon Y\to X^{\an}$ be a projective morphism 
from a smooth complex variety as in Definition \ref{a-def3.7}. 
We put $Z':=f(\Exc(f))$, which is a closed analytic subset of $X^{\an}$. 
Then $X^{\an}\setminus Z'$ is smooth and the support of $\Delta^{\an} 
|_{X^{\an}\setminus Z'}$ is a simple normal crossing divisor 
on $X^{\an} \setminus Z'$. 
We can check that 
if $g\colon V\to X$ is a projective bimeromorphic 
morphism from a smooth complex variety $V$ and $E$ is a prime divisor 
on $V$ such that $g(E)\subset Z'$ then $a(E, X, \Delta)>-1$ holds 
by the proof of \cite[Proposition 2.40]{kollar-mori}. 
Let $Z$ be the smallest closed algebraic subset of $X$ such that 
$X\setminus Z$ is smooth and the support of $\Delta|_{X\setminus Z}$ 
is a simple normal crossing divisor on $X\setminus Z$. Then 
$Z\subset Z'$ holds by definition. 
Hence, as in the proof of \cite[Proposition 4-4-4]{matsuki}, 
we see that $(X, \Delta)$ is divisorial log terminal in the usual 
sense. 
\end{proof}

Of course, Definition \ref{a-def3.7} is not analytically local. 

\begin{ex}\label{a-ex3.11}
Let $X$ be a smooth algebraic surface and let $C$ be an irreducible 
curve on $X$ with only one singular point $P$. 
Assume that $P$ is a node. 
It is obvious that $(X\setminus P, C|_{X\setminus P})$ is 
divisorial log terminal, but 
$(X, C)$ is not divisorial log terminal. 
However, there exists a small open neighborhood 
$U$ of $P$ such that $(U, C|_U)$ is divisorial 
log terminal in the sense of 
Definition \ref{a-def3.7}. Note that $C|_U$ is a simple normal crossing 
divisor on $U$ if $U$ is a small open neighborhood 
of $P$ in $X$. 
\end{ex}

The final theorem in this section is more or less 
well known to the experts. 
We will use it in the proof of Theorem \ref{thm-f} (3). 

\begin{thm}[{see \cite[Theorem 1-3-6]{kmm} and 
\cite[Chapter VII.~\S 1]{nakayama3}}]\label{a-thm3.12}
If $(X, \Delta)$ is divisorial log terminal, 
then $X$ has only rational singularities. 
\end{thm}
\begin{proof}
The arguments in \cite[3.14 Elkik--Fujita vanishing 
theorem]{fujino-foundations} work with some minor modifications 
if we use Grothendieck duality for proper 
morphisms of complex analytic spaces (see \cite{rrv}). 
We note that we have necessary vanishing theorems 
in the complex analytic setting 
(see Section \ref{a-sec5} below). 
\end{proof}

\section{Cones}\label{a-sec4}

In this section, we will define various cones and explain 
Kleiman's ampleness criterion for projective morphisms 
between complex analytic spaces. 

Throughout this section, 
let $\pi\colon X\to Y$ be a projective 
morphism of complex analytic spaces and let 
$W$ be a compact subset of $Y$. 
Let $Z_1(X/Y; W)$ be the free abelian group 
generated by the projective integral curves $C$ on $X$ such that 
$\pi(C)$ is a point of $W$. 
Let $U$ be any open neighborhood of $W$. 
Then we can consider the following intersection pairing 
\begin{equation*} 
\cdot :
\Pic\!\left(\pi^{-1}(U)\right)\times Z_1(X/Y; W)\to \mathbb Z
\end{equation*}  
given by $\mathcal L\cdot C\in \mathbb Z$ for 
$\mathcal L\in \Pic\!\left(\pi^{-1}(U)\right)$ and 
$C\in Z_1(X/Y; W)$. 
We say that $\mathcal L$ is {\em{$\pi$-numerically 
trivial over $W$}} when $\mathcal L\cdot C=0$ for 
every $C\in Z_1(X/Y; W)$. 
We take $\mathcal L_1, \mathcal L_2\in 
\Pic\!\left(\pi^{-1}(U)\right)$. 
If $\mathcal L_1\otimes \mathcal L_2^{-1}$ 
is $\pi$-numerically trivial over $W$, 
then we write $\mathcal L_1\equiv_W\mathcal L_2$ 
and say that $\mathcal L_1$ is numerically equivalent to 
$\mathcal L_2$ over $W$. 
We put 
\begin{equation*}
\widetilde A(U, W):=\Pic\!\left(\pi^{-1}(U)\right)/{\equiv_W}
\end{equation*}  
and define 
\begin{equation*}
A^1(X/Y; W):=\underset{W\subset U}\varinjlim
\widetilde A(U, W), 
\end{equation*}  
where $U$ runs through all the open neighborhoods of 
$W$. The following lemma due to Nakayama is a key result 
of the minimal model program for projective 
morphisms between complex analytic spaces.  

\begin{lem}
[{Nakayama's finiteness, see \cite[Chapter II.~5.19.~Lemma]{nakayama3}}]
\label{a-lem4.1}
Assume that $W\cap Z$ has only finitely 
many connected components 
for every analytic subset $Z$ defined over an 
open neighborhood of $W$. 
Then $A^1(X/Y; W)$ is a finitely generated 
abelian group. 
\end{lem}
\begin{proof}
For the details, 
see the proof of \cite[Chapter II.~5.19.~Lemma]{nakayama3} 
and Theorem \ref{a-thm4.7} in Subsection \ref{a-subsec4.1} below. 
\end{proof}

\begin{rem}\label{a-rem4.2}
Note that \cite[Chapter II.~5.19.~Lemma]{nakayama3}, 
that is, Lemma \ref{a-lem4.1} above, 
is a correction of \cite[Proposition 4.3 and 
Lemma 4.4]{nakayama2}. 
In Lemma \ref{a-lem4.1}, $W$ is 
not assumed to be Stein compact. 
Here, we only assume that 
$W$ is a compact subset of $Y$ satisfying (P4). 
We will discuss Lemma \ref{a-lem4.1} in detail in Subsection \ref{a-subsec4.1} 
below. 
\end{rem}

Under the assumption of Lemma \ref{a-lem4.1}, 
we can define the {\em{relative Picard number}} 
$\rho(X/Y; W)$ to be the rank of 
$A^1(X/Y; W)$. 
We put 
\begin{equation*}
N^1(X/Y; W):=A^1(X/Y; W)\otimes _{\mathbb Z} \mathbb R. 
\end{equation*} 
Let $A_1(X/Y; W)$ be the image of 
\begin{equation*} 
Z_1(X/Y; W)\to \Hom_{\mathbb Z} \left(A^1(X/Y; W), 
\mathbb Z\right)
\end{equation*} 
given by the above intersection pairing. 
Then we set 
\begin{equation*} 
N_1(X/Y; W):=A_1(X/Y; W)\otimes _{\mathbb Z}\mathbb R. 
\end{equation*}  
As usual, we can define the {\em{Kleiman--Mori cone}} 
\begin{equation*} 
\NE(X/Y; W)
\end{equation*} 
of $\pi\colon X\to Y$ over $W$, that is, 
$\NE(X/Y; W)$ is the closure of the convex cone in 
$N_1(X/Y; W)$ spanned by the projective 
integral curves $C$ such that 
$\pi(C)$ is a point of $W$. 
We also define $\Amp(X/Y; W)$ to be 
the cone in $N^1(X/Y; W)$ generated by 
line bundles $L$ such that 
$L|_{\pi^{-1}(U)}$ is $\pi$-ample 
for some open neighborhood $U$ of $W$. 
An element $\zeta\in N^1(X/Y; W)$ is called 
{\em{$\pi$-nef over $W$}} or {\em{nef over $W$}} 
if $\zeta\geq 0$ on $\NE(X/Y; W)$. 
Even when $\zeta$ is nef over $W$, 
it is not clear whether $\zeta$ is nef over some open neighborhood 
of $W$ or not. 

\begin{rem}[{see \cite[Theorem 1.2]{le}}]\label{a-rem4.3}
There exist a projective surjective morphism 
of algebraic varieties $\pi\colon X\to Y$ and an 
$\mathbb R$-Cartier $\mathbb R$-divisor $D$ 
on $X$ such that 
$\{y\in Y\, |\, {\text{$D|_{X_y}$ is nef}}\}$ is not Zariski 
open. This means that the nefness is not an open condition. 
For a criterion of openness of a family of nef line bundles, 
see \cite{moriwaki}.  
\end{rem}

On the other hand, for $\zeta\in N^1(X/Y; W)$, 
$\zeta|_{X_w}$ is ample for every $w\in W$ 
if and only if $\zeta$ is ample over 
some open neighborhood of $W$. 
This is because the ampleness is an open 
condition (see, for example, \cite[Proposition 1.41]{kollar-mori} and 
\cite[Proposition 1.4]{nakayama2}). 
Note that Kleiman's ampleness criterion holds 
true in our complex analytic setting. 

\begin{thm}[{Kleiman's 
criterion, see \cite[Proposition 4.7]{nakayama2}}]
\label{a-thm4.4}
Let $\pi\colon X\to Y$ be a projective morphism 
between complex analytic spaces and 
let $W$ be a compact subset of $Y$ 
such that $W$ satisfies {\em{(P4)}}. 
Then we have 
\begin{equation*}
\Amp(X/Y; W)=\left\{\zeta\in N^1(X/Y; W)\, |\, 
\text{$\zeta >0$ on $\NE(X/Y; W)\setminus \{0\}$}
\right\}. 
\end{equation*}
\end{thm}
\begin{proof}[Sketch of Proof of Theorem \ref{a-thm4.4}]
We note that the ample cone is an open convex cone in $N^1(X/Y; W)$. 
Hence we can easily check that it is contained in the right hand side. 
Therefore, it is sufficient to prove the opposite inclusion. 
We take a $\pi$-ample Cartier divisor $A$ on $X$. 
Let $\zeta$ be an element of $N^1(X/Y; W)$ such that 
$\zeta>0$ on $\NE(X/Y; W)\setminus \{0\}$. 
Then $\zeta-\varepsilon A>0$ on $\NE(X/Y; W)\setminus 
\{0\}$ for some small positive rational number $\varepsilon$. 
This implies that $(\zeta-\varepsilon A)|_{\pi^{-1}(w)}$ is 
nef for every $w\in W$. 
Since $A|_{\pi^{-1}(w)}$ is ample, $\zeta|_{\pi^{-1}(w)}
=(\zeta-\varepsilon A)|_{\pi^{-1}(w)}+A|_{\pi^{-1}(w)}$ 
is ample for every $w\in W$. 
We note that we can use Kleiman's ampleness criterion on 
$\pi^{-1}(w)$ since $\pi^{-1}(w)$ is projective. 
Hence, by the standard argument (see, 
for example, \cite[Section 6]{fujino-miyamoto}), we can write $\zeta$ as a 
finite $\mathbb R_{>0}$-linear combination of $\pi$-ample 
Cartier divisors over some open neighborhood of $W$. 
This is what we wanted. 
\end{proof}

We can define {\em{movable cones}} 
in our complex analytic setting (see \cite[Section 2]{kawamata-crepant}). 

\begin{defn}[{see \cite[Definition 2.1]{fujino-semistable}}]\label{a-def4.5}
Let $\pi\colon X\to Y$ be a projective 
morphism of complex analytic spaces and let $W$ be a compact 
subset of $Y$ such that $X$ is a normal complex variety. 
A Cartier divisor $D$ on $\pi^{-1}(U)$, where 
$U$ is some open neighborhood 
of $W$, 
is called {\em{$\pi$-movable over $W$}} 
if $\pi_*\mathcal O_{\pi^{-1}(U)}(D)\ne 0$ and 
if the cokernel of the natural homomorphism 
$\pi^*\pi_*\mathcal O_{\pi^{-1}(U)}(D)\to 
\mathcal O_{\pi^{-1}(U)}(D)$ has a support of 
codimension $\geq 2$. 
We define 
$\Mov(X/Y; W)$ as 
the closure of the convex cone in $N^1(X/Y; W)$ 
generated by 
the classes of $\pi$-movable Cartier divisors over $W$. 
Note that $\Mov(X/Y; W)$ is usually 
called the {\em{movable cone}} of $\pi\colon X\to Y$ and $W$. 
\end{defn}

We can easily see that a kind of negativity lemma holds. 

\begin{lem}[{Negativity lemma, see \cite[Lemma 3.8]{fujino-cone}}]
\label{a-lem4.6}
Let $\pi\colon X\to Y$ be a projective bimeromorphic 
contraction morphism of normal complex varieties and let $W$ be 
a compact subset of $Y$. 
Let $E$ be an $\mathbb R$-Cartier 
$\mathbb R$-divisor on $X$ such that $E\in \Mov(X/Y; W)$. 
Let $U$ be any open subset of $Y$ with 
$U \subset W$. 
If $-\pi_*E|_U$ is effective, then $-E|_{\pi^{-1}(U)}$ is effective. 
\end{lem}

\begin{proof}
For the details, see the proof of \cite[Lemma 3.8]{fujino-cone}. 
\end{proof}

We will use Lemma \ref{a-lem4.6} in order to terminate minimal model 
programs with scaling. 

\subsection{Nakayama's finiteness}\label{a-subsec4.1} 
In this subsection, we give a detailed proof of Nakayama's finiteness 
(see Lemma \ref{a-lem4.1}), which is the starting point of the 
minimal model program for projective morphisms 
between complex analytic spaces, for the sake of completeness. 
We will closely follow Nakayama's original proof in \cite{nakayama3}. 
The reader who is not interested in the proof can skip this subsection. 

Let us recall the statement of Nakayama's finiteness 
for the reader's convenience.  

\begin{thm}[{Nakayama's finiteness, 
see Lemma \ref{a-lem4.1} and \cite[Chapter II.~5.19.~Lemma]{nakayama3}}]
\label{a-thm4.7}
Let $\pi\colon X\to Y$ be a projective morphism of complex analytic 
spaces and let $W$ be a compact subset of $Y$. Assume that 
$W\cap Z$ has only finitely many connected components for 
every analytic subset $Z$ defined over 
an open neighborhood of $W$. 
Then $A^1(X/Y; W)$ is a finitely generated abelian group. 
\end{thm}

In this subsection, 
an {\em{$\mathbb R$-line bundle}} on a complex analytic 
space $X$ means an element of 
$\Pic(X)\otimes _{\mathbb Z} \mathbb R$. 
For simplicity of notation, we write the group law of 
$\Pic(X)\otimes _{\mathbb Z} \mathbb R$ additively. 

\begin{defn}\label{a-def4.8}Let $\pi\colon X\to Y$ be a projective 
surjective morphism of complex analytic spaces. 
An $\mathbb R$-line bundle $\mathcal L$ on $X$ is 
called {\em{$\pi$-ample}} 
if it is a finite $\mathbb R_{>0}$-linear combination 
of $\pi$-ample line bundles on $X$. 
Let $Z$ be any subset of $Y$. 
An $\mathbb R$-line bundle $\mathcal L$ on $X$ is called 
{\em{$\pi$-nef over $Z$}} and {\em{$\pi$-numerically trivial 
over $Z$}} if $\mathcal L|_{X_y}$ is nef and numerically 
trivial for every $y\in Z$, respectively, where 
$X_y:=\pi^{-1}(y)$. 
\end{defn}

Let us see numerically trivial $\mathbb R$-line bundles on 
smooth projective varieties. 

\begin{say}[Characterization of numerically trivial $\mathbb R$-line bundles 
on smooth projective varieties]\label{a-say4.9}
Let $X$ be a smooth projective variety. 
We consider the following long exact sequence: 
\begin{equation*}
\cdots \longrightarrow H^1(X, \mathcal O_X)\longrightarrow 
H^1(X, \mathcal O^*_X)\overset{c_1}{\longrightarrow}H^2(X, \mathbb Z)
\longrightarrow H^2(X, \mathcal O_X)\longrightarrow \cdots  
\end{equation*}
given by $0\to \mathbb Z \to \mathcal O_X\to \mathcal O^*_X\to 0$. 
The image of 
\begin{equation*}
c_1\colon \Pic(X)\simeq H^1(X, \mathcal O^*_X)\to 
H^2(X, \mathbb Z) 
\end{equation*} 
is denoted by $\NS(X)$. 
It is usually called the {\em{Neron--Severi group}} of 
$X$. We note that 
a line bundle $\mathcal L$ on $X$ is numerically trivial if and only 
if $c_1(\mathcal L)$ is a torsion element in $\NS(X)$ 
(see, for example, \cite[Corollary 1.4.38]{lazarsfeld1}). 
Hence $c_1$ induces the following isomorphism: 
\begin{equation*}
\Pic(X)/\!\!\equiv \overset{\sim}{\longrightarrow} \NS(X)/(\mathrm{torsion}), 
\end{equation*}
where $\equiv$ denotes the {\em{numerical equivalence}}. 
As usual, 
we put 
\begin{equation*}
N^1(X)=\left\{\Pic(X)/\!\!\equiv\right\}\otimes _{\mathbb Z} \mathbb R. 
\end{equation*} 
Let us consider an $\mathbb R$-line bundle $\mathcal L$ on $X$. 
We can define the first Chern class $c_1(\mathcal L)$ in 
$H^2(X, \mathbb R)$ since $\mathcal L$ is a finite 
$\mathbb R$-linear combination of line bundles. 
If $\mathcal L$ is numerically trivial, then 
it is easy to see that $\mathcal L$ is a finite 
$\mathbb R$-linear combination of numerically trivial line bundles on $X$. 
Therefore, $\mathcal L$ is numerically equivalent to zero, 
that is, $\mathcal L\equiv 0$, if and only 
if $c_1(\mathcal L)=0$ in $H^2(X, \mathbb R)$. 
\end{say}

Let us start with the following basic properties of ample and 
nef $\mathbb R$-line bundles. We can check them without 
any difficulties. 

\begin{lem}[{see \cite[Chapter II.~5.14.~Lemma]{nakayama3}}]\label{a-lem4.10}
Let $\pi\colon X\to Y$ be a projective surjective 
morphism of complex analytic spaces such that $Y$ is irreducible and let 
$\mathcal L$ be an $\mathbb R$-line bundle on $X$. 
\begin{itemize}
\item[(1)] Assume that $\mathcal L|_{X_{y_0}}$ is 
ample for some $y_0\in Y$. 
Then there exists an open neighborhood $U$ of $y_0$ such that 
$\mathcal L|_{\pi^{-1}(U)}$ is ample 
over $U$ and that $Y\setminus U$ is an analytically meagre subset of $Y$. 
\item[(2)] Assume that $\mathcal L|_{X_{y_0}}$ is 
nef for some $y_0\in Y$. 
Then there exists an analytically meagre subset $\mathcal S$ 
such that $\mathcal L$ is $\pi$-nef over $Y\setminus \mathcal S$ with $y_0\not 
\in \mathcal S$. 
\item[(3)] Assume that $\mathcal L|_{X_{y_0}}$ is 
numerically trivial for some $y_0\in Y$. 
Then there exists an analytically meagre subset $\mathcal S$ 
such that $\mathcal L$ is $\pi$-numerically trivial 
over $Y\setminus \mathcal S$ with $y_0\not 
\in \mathcal S$. 
\end{itemize}
\end{lem}
\begin{proof} 
(1) We can write $\mathcal L=\sum _{i=1}^k a_i \mathcal L_i$ such that 
$a_i\in \mathbb R_{>0}$ and $\mathcal L_i|_{X_{y_0}}$ is 
an ample line bundle for every $i$. 
If $U_i$ is a desired open neighborhood of $y_0$ for $\mathcal L_i$, 
then $U:=\bigcap _{i=1}^k U_i$ has the desired 
property for $\mathcal L$. 
Hence it is sufficient to prove this statement under the 
assumption that $\mathcal L$ is a line bundle. 
In this case, it is well known that there exists an open neighborhood 
$V$ of $y_0$ such that $\mathcal L|_{\pi^{-1}(V)}$ is $\pi$-ample 
over $V$ (see, for example, \cite[Proposition 1.41]{kollar-mori} and 
\cite[Proposition 1.4]{nakayama2}). 
Therefore, if $m$ is a sufficiently large positive integer, then 
$\pi^*\pi_*\mathcal L^{\otimes m}\to \mathcal L^{\otimes m}$ is surjective 
on a neighborhood of $X_{y_0}$. 
We put $Y^\dag:=Y\setminus \pi\left(\Supp \left(\Coker\pi^*\pi_*\mathcal 
L^{\otimes m} \to \mathcal L^{\otimes m}\right)\right)$. 
Then $Y^\dag$ is a Zariski open subset of $Y$ and 
$\pi^*\pi_*\mathcal L^{\otimes m}\to \mathcal L^{\otimes m}$ is surjective 
over $Y^\dag$. 
We consider the induced map 
$\varphi\colon 
X^\dag:=\pi^{-1}(Y^\dag)\to \mathbb P_{Y^\dag}\left((\pi_*\mathcal L^{\otimes m})|_{Y^\dag}\right)$ over 
$Y^\dag$. 
Since $m$ is sufficiently large, we may assume that 
$\varphi$ is a closed embedding over some open neighborhood 
of $y_0$. 
In particular, $\varphi$ is flat in a neighborhood 
of $X_{y_0}$. 
Hence, there exists a Zariski open subset $U$ of $Y^\dag$ such that 
$\varphi$ is flat and finite over $U$ (see \cite[Chapter V.~Theorem 4.5]{banica}). 
This implies that $\mathcal L|_{\pi^{-1}(U)}$ is ample 
over $U$. 
We can check that $U$ is an open neighborhood of 
$y_0$ with the desired properties. 

(2) We take an ample line bundle $\mathcal A$. 
Then $\left(m\mathcal L+\mathcal A\right)|_{X_{y_0}}$ is 
ample for every positive integer $m$. 
By (1), we can take an open neighborhood $U_m$ of $y_0$ such that 
$\left(m\mathcal L+\mathcal A\right)|_{\pi^{-1}(U_m)}$ is ample 
over $U_m$ and that 
$Y\setminus U_m$ is an analytically meagre subset for 
every $m$. 
We put $\mathcal S:=\bigcup _{m\in \mathbb Z_{>0}}\left(Y\setminus U_m\right)$. 
Then $\mathcal S$ is an analytically meagre subset and 
$\mathcal L$ is $\pi$-nef over $Y\setminus \mathcal S$. 
By construction, $y_0\not \in \mathcal S$.  
This is what we wanted. 

(3) By assumption, $\mathcal L|_{X_{y_0}}$ and 
$-\mathcal L|_{X_{y_0}}$ are both nef. By (2), there exists analytically 
meagre subsets $\mathcal S_1$ and $\mathcal S_2$ such that 
$\mathcal L$ is $\pi$-nef over $Y\setminus \mathcal S_1$ and 
$-\mathcal L$ is $\pi$-nef over $Y\setminus \mathcal S_2$. 
We put $\mathcal S:=\mathcal S_1\cup \mathcal S_2$. 
Then $\mathcal S$ is also an analytically meagre subset and 
$\mathcal L$ is $\pi$-numerically trivial over $Y\setminus \mathcal S$ 
with $y_0\not\in \mathcal S$. 
\end{proof}

\begin{rem}\label{a-rem4.11}
In Lemma \ref{a-lem4.10} (2) and (3), 
we can write $Y\setminus \mathcal S=\bigcap _{i\in \mathbb Z_{>0}} 
U_i$ such that $U_i$ is an open neighborhood 
of $y_0$ and $Y\setminus U_i$ is an analytically meagre 
subset for every $i\in \mathbb Z_{>0}$. 
\end{rem}

The following lemma is a key lemma for the proof of 
Nakayama's finiteness. 

\begin{lem}[{see \cite[Chapter II.~5.14.~Lemma]{nakayama3}}]
\label{a-lem4.12}
Let $\pi\colon X\to Y$ be a projective surjective morphism 
between complex manifolds such that 
$Y$ is connected and let $\mathcal L$ be an 
$\mathbb R$-line bundle on $X$. 
Assume that $\mathcal L|_{X_{y_0}}$ is numerically 
trivial for some $y_0\in Y$. 
Then $\mathcal L$ is $\pi$-numerically trivial over $Y$. 
\end{lem}

\begin{proof} 
By Lemma \ref{a-lem4.10} (3), 
there exists an analytically meagre subset $\mathcal S$ 
such that $\mathcal L$ is $\pi$-numerically 
trivial over $Y\setminus \mathcal S$. 
Therefore, it is sufficient to prove 
the statement under the extra assumption that $Y$ is a 
polydisc and $y_0$ is its origin. 
We can define the first Chern class 
$c_1(\mathcal L)$ in $H^2(X, \mathbb R)$ since 
$\mathcal L$ is a finite $\mathbb R$-linear 
combination of line bundles. 
It is sufficient to prove that $c_1(\mathcal L)=0$ 
in $H^2(X, \mathbb R)$. 
Since $\pi$ is smooth and proper, 
$X$ is diffeomorphic to $Y\times F$, where $F:=X_{y_0}$, 
and $\pi\colon X\to Y$ is diffeomorphic to 
the first projection $p_1\colon Y\times F\to Y$ 
(see, for example, \cite[Theorem 2.5]{kodaira}). 
On the other hand, 
\begin{equation}\label{a-eq4.1}
H^2(Y\times F, \mathbb R)\simeq H^2(F, \mathbb R) 
\end{equation}
holds as a very special case of K\"unneth formula 
(see, for example, \cite[(5.9)]{bott-tu}), 
which can be checked by Poincar\'e's lemma. 
By the isomorphism 
$H^2(X, \mathbb R)\simeq 
H^2(Y\times F, \mathbb R)$ and 
\eqref{a-eq4.1}, 
$c_1(\mathcal L)=p^*_2c_1(\mathcal L|_{X_{y_0}})$, 
where $p_2\colon Y\times F\to F$ is the second projection. Since 
$c_1(\mathcal L|_{X_{y_0}})=0$, 
we obtain $c_1(\mathcal L)=0$. 
This implies that $\mathcal L$ is $\pi$-numerically 
trivial over $Y$. 
This is what we wanted. 
\end{proof}

By Lemmas \ref{a-lem4.10} and \ref{a-lem4.12}, 
we can prove: 

\begin{thm}[{see \cite[Chapter II.~5.14.~Lemma]{nakayama3}}]\label{a-thm4.13}
Let $\pi\colon X\to Y$ be a projective surjective morphism 
between normal complex varieties and let $W$ be a compact 
subset of $W$. 
For a point $y_0\in W$, after shrinking $Y$ around $W$ suitably, 
there exists a Zariski open subset $U$ of $Y$ 
containing $y_0$ having the following 
property: If an $\mathbb R$-line bundle on $X$ is 
$\pi$-numerically trivial over the 
point $y_0\in Y$, 
then it is $\pi$-numerically trivial over $U$. 
\end{thm}

\begin{proof}Throughout this proof, 
we will repeatedly shrink $Y$ around $W$ suitably without 
mentioning it explicitly. 
By taking a resolution of singularities, we have a projective 
surjective morphism $X_0\to X$ from a smooth complex variety $X_0$. 
Let $\pi_0$ be the composition of $X_0\to X$ and $\pi\colon X\to Y$. 
Let $Y_1\subset Y$ be an analytic subset such that 
$\dim Y_1<\dim Y$, $Y\setminus Y_1$ is smooth, 
and $\pi_0$ is smooth over $Y\setminus Y_1$. 
Let $X_1\to \pi^{-1}_0(Y_1)$ be a projective bimeromorphic 
morphism from a smooth complex analytic space $X_1$ obtained by 
taking resolutions of singularities of irreducible components 
of $\pi^{-1}_0(Y_1)$. 
Then we obtain a sequence of 
analytic subsets 
\begin{equation*}
Y=:Y_0\supset Y_1\supset \cdots \supset Y_l\supset Y_{l+1}
\end{equation*} 
and projective surjective morphisms $\pi_i\colon X_i\to Y_i$, and 
projective surjective morphisms $X_i\to \pi^{-1}_{i-1}(Y_i)$ for 
$1\leq i\leq l$ such that 
\begin{itemize}
\item[(i)] $\dim _y Y_i<\dim _y Y_{i+1}$ holds at any point $y\in Y_i$, 
\item[(ii)] $Y_i\setminus Y_{i+1}$ is smooth, 
\item[(iii)] $\pi_i$ is smooth over $Y_i\setminus Y_{i+1}$, 
\item[(iv)] $\pi_i$ is nothing but the composition $X_i\to \pi^{-1}_{i-1}(Y_i)\to Y_i$, 
\item[(v)] $\pi_i$ is projective, and 
\item[(vi)] $y_0\in Y_l\setminus Y_{l+1}$. 
\end{itemize}
Let $S$ be any connected component of $Y_i\setminus Y_{i+1}$ for 
some $i\leq l$ such that 
$y_0\not\in \overline S$, where $\overline S$ is the topological 
closure of $S$ in $Y$. We note that $\overline S$ 
is an analytic 
subset of $Y$ by Remmert's extension theorem (see, 
for example, \cite[Chapter 9, \S 4, 2.~Extension Theorem for 
Analytic Sets]{grauert-remmert} and  
\cite[Theorem 6.8.1]{noguchi}) since 
$\dim _y Y_{i+1}<\dim _y Y_i$ holds at any point $y\in Y_{i+1}$. 
Let $U\subset Y$ be the Zariski open subset whose 
complement is the union of all such $\overline S$ for all $i$ and 
of $Y_{l+1}$. By Lemma \ref{a-lem4.10} (3), 
there exists an analytically meagre subset $\mathcal S$ of $Y$ such that 
$\mathcal L$ is $\pi$-numerically trivial over $Y\setminus \mathcal S$ with 
$y_0\not \in \mathcal S$. 
Let $T$ be any connected component of $Y_i\setminus Y_{i+1}$ for 
some $i\leq l$ with $y_0\in \overline T$. 
Then $T\cap (Y\setminus \mathcal S)\ne \emptyset$ by 
Remark \ref{a-rem4.11}. 
Thus, by Lemma \ref{a-lem4.12}, 
$\mathcal L$ is $\pi$-numerically trivial over $U$.  
This is what we wanted. 
\end{proof}

Let us prove Theorem \ref{a-thm4.7}. 
In the proof of Theorem \ref{a-thm4.7}, we will 
use the argument in the proof Theorem \ref{a-thm4.13}.  

\begin{proof}[Proof of Theorem \ref{a-thm4.7}]
As in the proof of Theorem \ref{a-thm4.13}, 
we construct a finite sequence of analytic subsets 
\begin{equation*}
Y=:Y_0\supset Y_1\supset \cdots \supset Y_k
\end{equation*} 
and projective surjective morphisms 
$\pi_i\colon X_i\to Y_i$ satisfying (i)--(v) in 
the proof of Theorem \ref{a-thm4.13}. 
Let $W_{i, j}$ be the connected components of $W\cap Y_i$ 
for $1\leq j\leq k_i$. 
We take a point $w_{i, j}\in W_{i, j} \setminus Y_{i+1}$ for 
any $(i, j)$ with $W_{i, j}\not\subset Y_{i+1}$. 
Then it is sufficient to show that 
\begin{equation}\label{a-eq4.2}
A^1(X/Y; W)\to \bigoplus \NS \left(\pi^{-1}_i (w_{i, j})\right)
/(\mathrm{torsion})
\end{equation}
is injective, where $\NS\left(\pi^{-1}_i (w_{i, j})\right)$ is the 
Neron--Severi group of $\pi^{-1}_i(w_{i, j})$. 
Let $\mathcal L$ be a line bundle on $\pi^{-1}(U)$, 
where $U$ is an open neighborhood 
of $W$. 
We will prove that $\mathcal L$ is $\pi$-numerically 
trivial over $W$ under the assumption that $\mathcal L$ is $\pi$-numerically 
trivial over all $w_{i,j}$. 
Since $\mathcal L$ is $\pi$-numerically 
trivial over $w_{i,j}$, 
$\mathcal L$ is $\pi$-numerically trivial over $U_{i, j}\setminus Y_{i+1}$ 
by Lemma \ref{a-lem4.12}, 
where $U_{i, j}$ is the connected component of $Y_i\cap U$ containing 
$w_{i, j}$. We note that $W\cap Y_i\subset \bigcup _j U_{i, j}$. 
Hence, $\mathcal L$ is $\pi$-numerically trivial over $W=\bigcup _i W\cap 
Y_i$. 
This is what we wanted, that is, \eqref{a-eq4.2} is injective. Therefore, 
we obtain that $A^1(X/Y; W)$ is a finitely generated abelian group since 
it is a subgroup of $\bigoplus \NS \left(\pi^{-1}_i (w_{i, j})\right)
/(\mathrm{torsion})$. 
\end{proof}

\section{Vanishing theorems}\label{a-sec5}

In this section, we will treat some vanishing theorems. 
Fortunately, the necessary vanishing theorems 
have already been established. 
We explain only two vanishing theorems 
here for the reader's convenience. 
The first one is the Kawamata--Viehweg vanishing theorem 
for projective morphisms between complex varieties. 

\begin{thm}[Kawamata--Viehweg vanishing theorem 
for projective morphisms of complex varieties]\label{a-thm5.1}
Let $X$ be a smooth complex variety and 
let $\pi\colon X\to Y$ be 
a projective morphism of complex varieties. 
Assume that $D$ is an 
$\mathbb R$-divisor on $X$ such that 
$D$ is $\pi$-nef and $\pi$-big and 
that $\Supp \{D\}$ is a simple normal 
crossing divisor on $X$. 
Then $R^i\pi_*\mathcal O_X(K_X+\lceil D\rceil)=0$ for every 
$i>0$. 
\end{thm}

\begin{proof}[Sketch of Proof of Theorem \ref{a-thm5.1}]
When $D$ is a $\mathbb Q$-divisor, this statement 
is well known. It follows from \cite[Theorem 3.4]{nakayama2}, 
\cite[Corollary 1.4]{fujino-transcendental}, and so on. 
Let $y\in Y$ be any point. 
It is sufficient to prove that $R^i\pi_*\mathcal O_X(K_X+\lceil D
\rceil)=0$ holds 
for $i>0$ on some open neighborhood $U_y$ of $y$. 
Therefore, we will freely shrink $Y$ around $y$ without mentioning 
it explicitly. 
We take a projective bimermorphic morphism 
$f\colon Z\to X$ and can reduce the problem to 
the case where $D$ is a $\mathbb Q$-divisor 
which is ample over $Y$. 
This reduction step is well known 
(see, for example, Step 2 in the proof of 
\cite[Theorem 3.4]{nakayama2}, 
the proof of \cite[Theorem 1-2-3]{kmm}, and so on). 
Hence we obtain the desired vanishing theorem. 
\end{proof}

The second one is essentially the same as the first one. 
However, we think that this formulation is 
useful for some applications. 

\begin{thm}\label{a-thm5.2}
Let $(X, \Delta)$ be a divisorial log terminal 
pair and let $\pi\colon X\to Y$ be a projective morphism 
of complex varieties. 
Let $D$ be a $\mathbb Q$-Cartier integral Weil divisor 
on $X$ such that 
$D-(K_X+\Delta)$ is $\pi$-ample. 
Then $R^i\pi_*\mathcal O_X(D)=0$ holds for every $i>0$. 
\end{thm}

\begin{proof}[Sketch of Proof of Theorem \ref{a-thm5.2}]
By Lemma \ref{a-lem3.9} and 
Theorem \ref{a-thm5.1}, 
the proof of \cite[Theorem 1-2-5]{kmm} 
works. For the details, see \cite{kmm}. 
\end{proof}

The reader can find various useful vanishing theorems 
for projective morphisms between complex varieties 
in \cite{fujino-transcendental}, \cite{fujino-matsumura}, 
\cite{matsumura}, \cite{fujino-bertini}, and so on. 

\section{Basepoint-free theorems, I}\label{a-sec6}

In this section, we will collect some necessary basepoint-free 
theorems for Cartier divisors. 

Let us start with Shokurov's nonvanishing theorem 
for smooth projective varieties, 
which can be proved by using Hironaka's 
resolution of singularities and the Kawamata--Viehweg vanishing 
theorem. 

\begin{thm}[Shokurov's nonvanishing theorem]
\label{a-thm6.1} 
Let $X$ be a smooth projective variety and let 
$D$ be a nef Cartier divisor on $X$. 
Let $A$ be an $\mathbb R$-divisor on $X$ such that 
$pD+A-K_X$ is ample for some positive integer $p$, 
$\lceil A\rceil \geq 0$, 
and $\Supp \{A\}$ is a simple normal crossing divisor. 
Then there exists some positive integer $m_0$ 
such that $H^0(X, \mathcal O_X(mD+\lceil A\rceil))\ne 0$ 
holds for every integer $m\geq m_0$. 
\end{thm}
\begin{proof}[Sketch of Proof of Theorem \ref{a-thm6.1}]
By perturbing the coefficients of $A$ slightly, 
we may assume that 
$A$ is a $\mathbb Q$-divisor. 
Then this statement is a special case 
of \cite[Theorem 2-1-1]{kmm} 
because $pD+A-K_X$ is automatically nef and big. 
For the details, see the proof of \cite[Theorem 2-1-1]{kmm}. 
\end{proof}

The following formulation of the 
basepoint-free theorem 
is suitable for our purposes in this paper. 
It is the well-known Kawamata--Shokurov basepoint-free 
theorem when $\pi\colon X\to Y$ is algebraic. 

\begin{thm}[Basepoint-free theorem]\label{a-thm6.2}
Let $\pi\colon X\to Y$ be a projective 
morphism from a normal complex variety to 
a complex analytic space $Y$ and let $\Delta$ be 
an $\mathbb R$-divisor 
on $X$ such that $(X, \Delta)$ is divisorial 
log terminal. 
Let $D$ be a $\pi$-nef Cartier divisor 
on $X$ such that $aD-(K_X+\Delta)$ is 
$\pi$-ample 
for some positive integer $a$. 
Then, for any relatively compact open subset $U$ of $Y$, 
there exists a positive integer $m_0$, which 
depends on $U$, such that 
\begin{equation*}
\pi^*\pi_*\mathcal O_X(mD)\to \mathcal O_X(mD)
\end{equation*} 
is surjective over $U$ for every integer $m\geq m_0$. 
\end{thm}

\begin{proof}[Sketch of Proof of Theorem \ref{a-thm6.2}]
By Lemma \ref{a-lem3.9} and 
Theorem \ref{a-thm6.1}, 
the proof of \cite[Theorem 3-1-1]{kmm} 
works over $U$. 
Note that we can not consider the 
generic fiber of $\pi^{-1}(U)\to U$ since 
it is a projective morphism of complex analytic spaces. 
Therefore, we apply Theorem \ref{a-thm6.1} 
to an analytically sufficiently general fiber of $\pi^{-1}(U)
\to U$ when we prove $\pi_*\mathcal O_X(mD)\ne 0$ for 
every large positive integer $m$. 
For the details, see the proof of \cite[Theorem 3-1-1]{kmm}. 
\end{proof}

\begin{rem}\label{a-rem6.3}
If $(X, \Delta)$ is kawamata log terminal in Theorem \ref{a-thm6.2}, 
then the same statement holds under a slightly weaker 
assumption that $aD-(K_X+\Delta)$ is $\pi$-nef and 
$\pi$-big. 
This is almost obvious by the proof of Theorem \ref{a-thm6.2} 
(see also the proof of \cite[Theorem 3-1-1]{kmm}). 
\end{rem}

In this paper, the following 
variant of the basepoint-free theorem 
is indispensable. Theorem \ref{a-thm6.4} 
is Koll\'ar's effective basepoint-freeness for 
projective morphisms of complex analytic spaces. 

\begin{thm}[{Effective basepoint-free theorem, 
see \cite[Theorem 3.9.1]{bchm}}]\label{a-thm6.4}
Fix a positive integer $n$. Then there exists a positive integer 
$m$ with the following properties. 

Let $\pi\colon X\to Y$ be a projective 
morphism from a normal complex variety $X$ to a complex 
analytic space $Y$ and let $D$ be a $\pi$-nef 
Cartier divisor on $X$ such that 
$D-(K_X+\Delta)$ is $\pi$-nef and 
$\pi$-big, where $(X, \Delta)$ is kawamata log terminal 
with $\dim X=n$. 

Then, $mD$ is $\pi$-free, that is, 
$$
\pi^*\pi_*\mathcal O_X(mD)\to \mathcal O_X(mD)
$$ 
is surjective. 
\end{thm}

\begin{proof}[Sketch of Proof of Theorem \ref{a-thm6.4}]
It is sufficient to prove the statement 
over any small relatively compact Stein open 
subset $U$ of $Y$. 
We will freely replace $Y$ with a small 
open subset without mentioning it explicitly. 
By the standard argument (see, for example, 
the proof of \cite[Theorem 3.9.1]{bchm}), 
we may further assume that 
$D-(K_X+\Delta)$ is $\pi$-ample and that 
$K_X+\Delta$ is $\mathbb Q$-Cartier.   
The modified basepoint-freeness method 
explained in \cite[2.1]{kollar-effective} 
works with some minor modifications. 
Note that we treat a projective morphism 
$\pi \colon X\to Y$. 
Therefore, when we prove that some sheaf is not zero, 
we restrict it to an analytically sufficiently general fiber. 
We can not use the generic fiber since we consider 
projective morphisms between complex analytic spaces. 
For the details, see \cite[Section 2]{kollar-effective}. 
\end{proof}

Theorem \ref{a-thm6.5} below is essentially due to Nakayama 
(see \cite[Theorem 4.10]{nakayama2}), which will play a crucial role 
in this paper. 

\begin{thm}[{see \cite[Theorem 4.10]{nakayama2}}]\label{a-thm6.5} 
Let $\pi\colon X\to Y$ be a projective surjective morphism 
between complex analytic spaces and let $W$ be 
a Stein compact subset of $Y$ and 
let $(X, \Delta)$ be a kawamata log terminal pair. 
Let $D$ be a Cartier divisor on $X$. 
Assume that $D$ is nef over $W$, that is, 
$D\cdot C\geq 0$ for every projective 
curve $C$ such that 
$\pi(C)$ is a point of $W$. 
We further assume that $aD-(K_X+\Delta)$ is 
$\pi$-ample for some positive real number $a$. 
Then there exist an open neighborhood 
$U$ of $W$ and 
a positive integer $m_0$ such that 
\begin{equation*}
\pi^*\pi_*\mathcal O_X(mD)\to \mathcal O_X(mD)
\end{equation*} 
is surjective over $U$ for every integer $m\geq m_0$. 
\end{thm}

Before we prove Lemma \ref{a-thm6.5}, we prepare 
an easy lemma. We describe it for the sake of completeness. 

\begin{lem}\label{a-lem6.6}
Let $\pi\colon X\to Y$ be a projective surjective morphism 
between complex varieties and let 
$\mathcal L$ be a line bundle on $X$. Assume that 
$\mathcal L|_{\pi^{-1}(y_0)}$ is nef for some 
$y_0\in Y$. 
Then there exists an analytically meagre subset 
$\mathcal S$ such that 
$\mathcal L|_{\pi^{-1}(y)}$ is nef for 
every $y\in Y\setminus \mathcal S$. 
\end{lem}

\begin{proof}[Proof of Lemma \ref{a-lem6.6}]
We take a $\pi$-ample line bundle $\mathcal H$ on $X$. 
Then $\mathcal L^{\otimes n}\otimes \mathcal H|_{\pi^{-1}(y_0)}$ 
is ample for every $n\in \mathbb Z_{>0}$. 
For each $n$, there exist a positive integer $m_n$ and an open 
neighborhood $U_n$ of $y_0$ 
such that 
\begin{equation*}
\pi^*\pi_*(\mathcal L^{\otimes n m_n}\otimes 
\mathcal H^{\otimes m_n})\to 
\mathcal L^{\otimes n m_n}\otimes 
\mathcal H^{\otimes m_n}
\end{equation*} 
is surjective on $\pi^{-1}(U_n)$ since 
$\mathcal L^{\otimes n}\otimes \mathcal H$ is ample 
over some open neighborhood of $y_0$ 
(see, for example, \cite[Proposition 1.4]{nakayama2}). 
We put 
\begin{equation*}
\mathcal F_n:=\Coker \left( 
\pi^*\pi_*(\mathcal L^{\otimes n m_n}\otimes 
\mathcal H^{\otimes m_n})\to 
\mathcal L^{\otimes n m_n}\otimes 
\mathcal H^{\otimes m_n}\right). 
\end{equation*} 
Then $\pi_*\mathcal F_n$ is a coherent sheaf on $Y$ and 
$\mathcal S_n:=\Supp \mathcal F_n$ is a closed 
analytic subset of $Y$ with $y_0\not\in \mathcal S_n$. 
By construction, $\mathcal L^{\otimes nm_n}\otimes 
\mathcal H^{\otimes m_n}$ is $\pi$-free over $Y\setminus 
\mathcal S_n$. In particular, 
$\mathcal L^{\otimes n}\otimes \mathcal H$ 
is $\pi$-nef over $Y\setminus \mathcal S_n$. 
We put $\mathcal S:=\bigcup _{n\in \mathbb Z_{>0}}\mathcal S_n$. 
If $\left(\mathcal L^{\otimes n} \otimes 
\mathcal H\right) |_{\pi^{-1}(y)}$ is nef for every 
$n\in \mathbb Z_{>0}$, then 
$\mathcal L|_{\pi^{-1}(y)}$ is nef. 
Therefore, $\mathcal S$ is the desired analytically meagre 
subset of $Y$. 
\end{proof}

Let us see the proof of Theorem \ref{a-thm6.5}. 

\begin{proof}[Sketch of Proof of Theorem \ref{a-thm6.5}]
Let $B$ be a Cartier divisor on $X$ such that 
$B$ is nef over $W$ and let $A$ be any $\pi$-ample 
$\mathbb R$-divisor on $X$. 
Then $A+B$ is $\pi$-ample 
over some open neighborhood of $W$. 
By this easy observations, 
we see that the usual proof of the basepoint-free 
theorem (see Sketch of Proof of Theorem \ref{a-thm6.2}) 
works over some open neighborhood 
$U$ of $W$. 
We note that we can use the nonvanishing theorem 
(see Theorem \ref{a-thm6.1}) on an analytically sufficiently general 
fiber by Lemma \ref{a-lem6.6}. 
Of course, we may have to replace $U$ with 
a smaller open neighborhood of $W$ finitely many times 
throughout the proof. 
\end{proof}

We will treat some basepoint-free theorems for $\mathbb R$-Cartier 
divisors in Section \ref{a-sec8}. 

\section{Cone and contraction theorem}\label{a-sec7}

This section will be devoted to the cone and 
contraction theorem for projective morphisms 
of complex analytic spaces. 

Let us start with the rationality theorem. 
We need the following formulation. 
The proof is essentially the same as that for 
algebraic varieties. 

\begin{thm}[{Rationality theorem, 
see \cite[Theorem 4.11]{nakayama2}}]\label{a-thm7.1} 
Let $\pi\colon X\to Y$ be a projective morphism 
of complex analytic spaces and 
let $W$ be a compact subset of $Y$ 
such that $\pi\colon X\to Y$ and $W$ satisfies {\em{(P)}}. 
Let $\Delta$ be an effective $\mathbb Q$-divisor 
on $X$ such that $(X, \Delta)$ is divisorial log 
terminal and that $a(K_X+\Delta)$ is Cartier 
in a neighborhood of $\pi^{-1}(W)$ for some 
positive integer $a$. 
Let $H$ be a $\pi$-ample 
Cartier divisor on $X$. 
Assume that $K_X+\Delta$ is not $\pi$-nef 
over $W$. 
Then 
\begin{equation*}
r:=\max \{ t\in \mathbb R\, |\, 
\text{$H+t(K_X+\Delta)$ is $\pi$-nef over $W$}\}
\end{equation*}
is a positive rational number. 
Furthermore, expressing $r/a=u/v$ with 
$u, v\in \mathbb Z_{>0}$ and $(u, v)=1$, 
we have $v\leq a(d+1)$, 
where $d=\max _{w\in W} \dim \pi^{-1}(w)$. 
\end{thm}

\begin{proof}[Sketch of Proof of Theorem \ref{a-thm7.1}]
The proof of \cite[Theorem 4-1-1]{kmm} 
works with some minor modifications. 
As usual, we will freely replace $Y$ with a relatively 
compact Stein open neighborhood of $W$ 
throughout this proof. 
By an easy reduction argument, we may further 
assume that $H$ is $\pi$-very ample. 
We put 
\begin{equation*}
M(x, y):=xH+ya(K_X+\Delta)
\end{equation*}
and 
\begin{equation*}
\Lambda (x, y):=\Supp \left (\Coker 
\pi^*\pi_*\mathcal O_X(M(x, y))\to 
\mathcal O_X(M(x, y))\right). 
\end{equation*}
It is not difficult to see that $\Lambda (x, y)$ 
is the same subset of $X$ for $(x, y)$ 
sufficiently large and $0<ya-xr<1$. 
We call it $\Lambda _0$. 
Moreover, let $I\subset \mathbb Z^2$ be the set of 
$(x, y)$ for which $0<ya-xr<1$ and $\Lambda(x, y)=\Lambda_0$. 
Then $I$ contains all sufficiently large $(x, y)$ with 
$0<ya-xr<1$ 
(for the details, see Claim 1 in the proof of 
\cite[Theorem 15.1]{fujino-fundamental}). 
If $r\not\in \mathbb Q$ or $v>r(d+1)$, then 
we can find $(x', y')$ sufficiently large and $0<y'a-x'r<1$ with 
$\Lambda (x', y')\subsetneq \Lambda_0$ 
(for the details, see the proof of \cite[Theorem 4-1-1]{kmm} 
or Step 7--Step 11 in \cite[Section 3.4]{kollar-mori}). 
This is a contradiction. Hence, we get the desired properties 
of $r$. 
We note that we can not consider generic fibers. 
Therefore, when we check that some sheaf is not a zero 
sheaf in the above argument, we restrict it to 
an analytically sufficiently general fiber. 
\end{proof}

It is very well known that 
the cone and contraction theorem is 
a consequence of the rationality theorem (see 
Theorem \ref{a-thm7.1}) and 
the basepoint-free theorem (see Theorem \ref{a-thm6.5}). 

\begin{thm}[{Cone and contraction theorem, 
see \cite[Theorem 4.12]{nakayama2}}]\label{a-thm7.2}
Let $(X, \Delta)$ be a divisorial log terminal pair. 
Let $\pi\colon X\to Y$ be a projective morphism of 
complex analytic spaces and let $W$ be a compact subset 
of $Y$ such that $\pi\colon X\to Y$ and 
$W$ satisfies {\em{(P)}}. 
Then we have 
\begin{equation*} 
\NE(X/Y; W)=\NE(X/Y; W)_{K_X+\Delta\geq 0} 
+\sum _j R_j
\end{equation*} 
with the following properties. 
\begin{itemize}
\item[(1)] Let $A$ be a $\pi$-ample 
$\mathbb R$-divisor on $X$. 
Then there are only finitely many $R_j$'s included in $(K_X+\Delta
+A)_{<0}$. In particular, the $R_j$'s are discrete in the half 
space $(K_X+\Delta)_{<0}$. 
\item[(2)] Let $R$ be a $(K_X+\Delta)$-negative 
extremal ray. Then, 
after shrinking $Y$ around $W$ suitably, 
there exists a contraction morphism $\varphi_R\colon 
X\to Z$ over $Y$ satisfying: 
\begin{itemize}
\item[(i)] Let $C$ be a projective 
integral curve on $X$ such that 
$\pi(C)$ is a point in $W$. 
Then $\varphi_R(C)$ is a point if and only if 
$[C]\in R$. 
\item[(ii)] $\mathcal O_Z
\overset{\sim}{\longrightarrow} (\varphi_R)
_*\mathcal O_X$. 
\item[(iii)] Let $\mathcal L$ be a line bundle on $X$ such that 
$\mathcal L\cdot C=0$ for every curve $C$ 
with $[C]\in R$. 
Then there is a line bundle $\mathcal M$ on $Z$ 
such that $\mathcal L\simeq \varphi^*_R\mathcal M$. 
\end{itemize}
\end{itemize}
\end{thm}

\begin{proof}[Sketch of Proof of 
Theorem \ref{a-thm7.2}] 
When $K_X+\Delta$ is $\mathbb Q$-Cartier 
in a neighborhood of $\pi^{-1}(W)$, 
we have the desired properties as a consequence of 
the rationality theorem (see 
Theorem \ref{a-thm7.1}) 
and the basepoint-free theorem (see 
Theorem \ref{a-thm6.5}). 
This part is well known 
(for the details, see, for example, \cite[Theorem 4-2-1]{kmm}). 
From now on, we assume that 
$K_X+\Delta$ is $\mathbb R$-Cartier but is not 
$\mathbb Q$-Cartier. 
As usual, we will freely shrink $Y$ around $W$ without 
mentioning it explicitly. 
By the standard argument, we can find effective 
$\mathbb Q$-divisors $\Delta_1, \ldots, 
\Delta_k$ on $X$ and positive real numbers $r_1, \ldots, 
r_k$ with $\sum _{i=1}^kr_i=1$ such that 
$K_X+\Delta_i$ is $\mathbb Q$-Cartier and 
$(X, \Delta_i)$ is divisorial log terminal for every $i$ and 
that $\sum _{i=1}^kr_i \Delta_i=\Delta$ holds. 
Let $R$ be a $(K_X+\Delta)$-negative 
extremal ray of $\NE(W/Y; W)$. 
Then there exists some $i_0$ such that 
$R$ is a $(K_X+\Delta_{i_0})$-negative 
extremal ray of $\NE(X/Y; W)$. 
We have already known that 
extremal rays are discrete in the half space 
$(K_X+\Delta_i)_{<0}$ for every $i$ since $K_X+\Delta_i$ 
is $\mathbb Q$-Cartier. 
Hence it is not difficult to see that all the desired 
properties 
hold true even when $K_X+\Delta$ is only 
$\mathbb R$-Cartier.  
\end{proof}

In this paper, we sometimes treat log canonical pairs. 
Thus, we need: 

\begin{thm}\label{a-thm7.3}
Let $(X, \Delta)$ be a log canonical pair. 
Let $\pi\colon X\to Y$ be a projective 
morphism of complex analytic spaces and let $W$ be a 
compact subset of $Y$ such that $\pi\colon X\to Y$ and $W$ 
satisfies {\em{(P)}}. 
We assume that there exists $\Delta_0$ on $X$ 
such that $(X, \Delta_0)$ is kawamata log terminal. 
Let $A$ be a $\pi$-ample $\mathbb R$-divisor 
on $X$. 
Then there are only finitely many $(K_X+\Delta+A)$-negative 
extremal rays of $\NE(X/Y; W)$. 

Let $R$ be a $(K_X+\Delta)$-negative extremal 
ray of $\NE(X/Y; W)$. 
Then the contraction morphism $\varphi_R\colon X\to Z$ 
associated to $R$ as in Theorem \ref{a-thm7.2} (2) exists. 
\end{thm}

\begin{proof}
We take a sufficiently 
small positive rational number $\varepsilon$ and 
consider 
\begin{equation*}
K_X+\Delta+A=K_X+(1-\varepsilon) \Delta+\varepsilon 
\Delta_0 +\left(A-\varepsilon (\Delta_0-\Delta)\right). 
\end{equation*}
Since $\varepsilon$ is sufficiently small, 
$A-\varepsilon (\Delta_0-\Delta)$ is still $\pi$-ample. 
On the other hand, the pair 
$\left(X, (1-\varepsilon)\Delta+\varepsilon \Delta_0\right)$ 
is kawamata log terminal. 
Hence, by Theorem \ref{a-thm7.2} (1), there are 
only finitely many $(K_X+\Delta+A)$-negative extremal rays. 
Let $R$ be a $(K_X+\Delta)$-negative 
extremal ray of $\NE(X/Y; W)$. 
Then we can see it as a 
$\left( K_X+(1-\varepsilon) \Delta+\varepsilon 
\Delta_0\right)$-negative extremal ray for some 
small positive rational number $\varepsilon$. 
Therefore, by Theorem \ref{a-thm7.2} (2), 
we have the desired contraction morphism 
$\varphi_R\colon X\to Z$.  
\end{proof}

\section{Basepoint-free theorems, II}\label{a-sec8}

In this section, we will treat basepoint-free theorems 
for $\mathbb R$-divisors. 
We note that the use of $\mathbb R$-divisors 
is indispensable in the theory of minimal models. 

\begin{thm}[{Basepoint-free theorem for 
$\mathbb R$-divisors, see \cite[Theorem 3.9.1]{bchm}}]
\label{a-thm8.1}
Let $\pi\colon X\to Y$ be a projective 
morphism of normal complex variety $X$ to 
a complex analytic space $Y$ and let $W$ be a Stein 
compact subset of $Y$ such that 
$\pi\colon X\to Y$ and $W$ satisfies {\em{(P)}}. 
Let $D$ be a $\pi$-nef $\mathbb R$-divisor 
on $X$ such that $aD-(K_X+\Delta)$ 
is $\pi$-nef and $\pi$-big for some positive 
real number $a$, where $(X, \Delta)$ is 
kawamata log terminal.  

Then there exists an open neighborhood $U$ of $W$ such that 
$D|_{\pi^{-1}(U)}$ is semiample over $U$. 
\end{thm}

\begin{proof}[Sketch of Proof of Theorem \ref{a-thm8.1}]
By replacing $D$ with $aD$, we may assume that 
$a=1$. 
We take a small Stein open neighborhood $U'$ of $W$ and 
a Stein compact subset $W'$ of $Y$ such that 
$\Gamma (W', \mathcal O_Y)$ is noetherian and 
$U'\subset W'$. 
Throughout this proof, we will freely shrink $Y$ 
around $W'$ without mentioning it explicitly. 
By the standard argument (see, for example, 
the proof of \cite[Theorem 3.9.1]{bchm}), 
we may further assume that $D-(K_X+\Delta)$ 
is $\pi$-ample and $K_X+\Delta$ is $\mathbb Q$-Cartier. 
We take a small $\pi$-ample $\mathbb Q$-divisor 
$A$ on $X$ such that $D-(K_X+\Delta+A)$ is still 
$\pi$-ample. 
By the cone theorem (see Theorem \ref{a-thm7.2} (1)), 
there are only finitely many $(K_X+\Delta+A)$-negative 
extremal rays of $\NE(X/Y; W')$. 
Hence, we can write 
$D=\sum _{i=1}^kr_iD_i$, where 
$r_i$ is a positive real number, $D_i$ is a 
$\mathbb Q$-Cartier $\mathbb Q$-divisor on $X$ which is nef over 
$W'$, 
and $D_i-(K_X+\Delta)$ is $\pi$-ample for every 
$i$ (for the details, see the proof of 
\cite[Theorem 3.9.1]{bchm}). 
We replace $Y$ with $U'$.  
Then, by the usual basepoint-free theorem (see 
Theorem \ref{a-thm6.2}), 
$D_i|_{\pi^{-1}(U)}$ is semiample over $U$ for 
some $U$ and every $i$. 
This implies that $D|_{\pi^{-1}(U)}$ is 
semiample over $U$. 
\end{proof}

We used the cone theorem (see Theorem \ref{a-thm7.2}) 
in the above proof of 
Theorem \ref{a-thm8.1}. So it is much 
deeper than the usual basepoint-free theorem 
for Cartier divisors (see Theorem \ref{a-thm6.2}). 

By combining Lemma \ref{a-thm6.5} with 
the argument in the proof of Theorem \ref{a-thm8.1}, 
we have: 

\begin{thm}\label{a-thm8.2}
Let $\pi\colon X\to Y$ be a projective 
bimeromorphic contraction morphism of complex analytic 
spaces and let $y\in Y$ be a point. 
Let $D$ be an $\mathbb R$-Cartier $\mathbb R$-divisor 
on $X$ such that 
$D$ is numerically trivial over $y$, that is, 
$D\cdot C=0$ for every projective curve 
$C$ on $X$ such that 
$\pi(C)=y$. 
Assume that $aD-(K_X+\Delta)$ is $\pi$-nef for 
some positive real number $a$, where $(X, \Delta)$ is 
a kawamata log terminal pair. 
Then $\pi_*D$ is $\mathbb R$-Cartier at $y$. 
\end{thm}

\begin{proof}
We put $W=\{y\}$. 
Then $W$ is a Stein compact subset of $Y$ such that 
$\Gamma (W, \mathcal O_Y)$ is noetherian. 
We will freely shrink $Y$ around $W$ without mentioning it 
explicitly. As in the proof of 
Theorem \ref{a-thm8.1}, 
we may assume that 
$a=1$ and $D-(K_X+\Delta)$ is $\pi$-ample. 
By using the cone theorem as in the proof of 
Theorem \ref{a-thm8.1}, 
we can write $D=\sum _{i=1}^k r_iD_i$, where 
$r_i$ is a positive real number, 
$D_i$ is numerically trivial over $y$, and $D_i-(K_X+\Delta)$ 
is $\pi$-ample 
for every $i$. 
By replacing $D_i$ with $m_iD_i$ for some 
positive integer $m_i$, 
we may further assume that 
$D_i$ is a Cartier divisor on $X$ for every $i$. 
Then, by Theorem \ref{a-thm6.5}, 
$\pi_*D_i$ is $\mathbb Q$-Cartier for every $i$. 
Hence $\pi_*D$ is $\mathbb R$-Cartier. 
This is what we wanted. 
\end{proof}

The final theorem in this section is suitable 
for our framework of the minimal model program. 
We only assume that $K_X+\Delta$ is $\mathbb R$-Cartier and 
is nef over $W$ in Theorem \ref{a-thm8.3}. 
The conclusion says that it is semiample over some open 
neighborhood of $W$. 

\begin{thm}\label{a-thm8.3}
Let $\pi\colon X\to Y$ be a projective morphism 
of complex analytic spaces and let $W$ be a compact subset 
of $Y$ such that $\pi\colon X\to Y$ and $W$ satisfies 
{\em{(P)}}. Let $(X, \Delta)$ be a log canonical pair. 
Assume that there exists $\Delta_0$ such that $(X, \Delta_0)$ is kawamata 
log terminal. We further assume that $\Delta=A+B$, where 
$A$ is $\pi$-ample, 
$A\geq 0$, and $B\geq 0$. 
If $K_X+\Delta$ is nef over $W$, 
then there exists an open neighborhood $U$ of 
$W$ such that 
$K_X+\Delta$ is semiample over $U$. 
\end{thm}

\begin{proof}
Throughout this proof, we will shrink $Y$ around $W$ suitably 
without mentioning it explicitly. 
By assumption, we can take $\Delta'$ such that 
$(X, \Delta')$ is kawamata log terminal and $K_X+\Delta'\sim _{\mathbb 
R}K_X+\Delta$. 
Hence, by replacing $\Delta$ with $\Delta'$, 
we may assume that $(X, \Delta)$ is kawamata log 
terminal. 
Then $(X, B)$ is kawamata log terminal and 
$2(K_X+\Delta)-(K_X+B)$ is ample over 
$Y$. 
We take a general $\pi$-ample $\mathbb Q$-divisor 
$H$ on $X$ such that $(X, B+H)$ is kawamata log terminal 
and that $2(K_X+\Delta)-(K_X+B+H)$ is still ample 
over $Y$. As in the proof of \cite[Theorem 3.9.1]{bchm}, 
by using Theorem \ref{a-thm7.3}, 
we take positive real numbers $r_1, \ldots, r_k$ and 
$\mathbb Q$-divisors $\Delta_1,\ldots, \Delta_k$ such that 
$K_X+\Delta_i$ is nef over $W$ and $2(K_X+\Delta_i)-(K_X+B+H)$ 
is ample over $Y$ for every $i$ and 
that $\sum _{i=1}^k\Delta_i =\Delta$. 
By Theorem 
\ref{a-thm6.5}, we obtain that $K_X+\Delta_i$ is $\pi$-semiample 
for every $i$. 
This means that $K_X+\Delta$ is semiample over $Y$. 
This is what we wanted. 
\end{proof}

We close this section with conjectures related to Theorem \ref{a-thm8.3}. 

\begin{conj}\label{a-conj8.4}
Let $(X, \Delta)$ be a log canonical pair and 
let $\pi\colon X\to Y$ be a projective 
morphism of complex analytic spaces. 
We put 
\begin{equation*}
\mathfrak N:=\{y\in Y\, |\, {\text{$(K_X+\Delta)|_{\pi^{-1}(y)}$ 
is nef}}\}. 
\end{equation*} 
Then $\mathfrak N$ is open. 
\end{conj}

If we can establish the minimal model 
program for projective morphisms between complex 
analytic spaces in full generality, 
then we see that Conjecture \ref{a-conj8.4} holds true. 

\begin{rem}\label{a-rem8.5} 
In Conjecture \ref{a-conj8.4}, 
it is sufficient to prove that $\mathfrak N$ contains an 
open neighborhood 
of $y_0$ under the assumption that 
$(K_X+\Delta)|_{\pi^{-1}(y_0)}$ is nef. 
We take an open neighborhood $U$ of $y_0$ and a Stein 
compact subset $W$ of $Y$ such that 
$y_0\in U\subset W$ and 
that $\Gamma (W, \mathcal O_Y)$ is noetherian. 
We will freely shrink $Y$ around $W$. 
By Theorem \ref{a-thm1.27}, we can reduce the problem 
to the case where $X$ is $\mathbb Q$-factorial 
over $W$ and $(X, \Delta)$ is divisorial log terminal. 
Then we run a $(K_X+\Delta)$-minimal model 
program over $Y$ around $W$. 
Note that $K_X+\Delta$ is $\pi$-pseudo-effective 
since $(K_X+\Delta)|_{\pi^{-1}(y_0)}$ is nef. 
If the above minimal model program terminates after 
finitely many steps, 
then it is easy to see that $K_X+\Delta$ 
is nef over some open neighborhood of $y_0$. 
Hence Conjecture \ref{a-conj8.4} would be 
absolutely correct. 
\end{rem}

Conjecture \ref{a-conj8.4} 
is also related to the following abundance conjecture 
for projective morphisms of complex analytic spaces. 

\begin{conj}[Abundance conjecture]\label{a-conj8.6}
Let $(X, \Delta)$ be a log canonical pair and let 
$\pi\colon X\to Y$ be a projective 
morphism of complex analytic spaces. 
Assume that $K_X+\Delta$ is $\pi$-nef, that is, 
$(K_X+\Delta)\cdot C\geq 0$ for every 
projective integral curve $C$ on $X$ such that $\pi(C)$ is a point. 
Then $K_X+\Delta$ is $\pi$-semiample. 
\end{conj}

As is well known, the abundance conjecture (see Conjecture 
\ref{a-conj8.6}) is widely open even when 
$Y$ is a point. We will treat the abundance conjecture 
in Section \ref{a-sec23}. 

\section{Lengths of extremal rational curves}\label{a-sec9}

In this paper, we will repeatedly use the 
fact that every extremal ray is spanned by 
a rational curve of low degree, which is essentially due to 
Kawamata (see \cite[Theorem 1]{kawamata-length}). 
Note that Kawamata's result comes from 
the result obtained by Mori's bend and break technique, 
which relies on methods in 
positive characteristic. 

\begin{thm}[{see \cite[Theorem 1]{kawamata-length} and 
\cite[Theorem 1.12]{fujino-cone}}]\label{a-thm9.1}
Let $(X, \Delta)$ be a kawamata log terminal pair and let $\varphi 
\colon X\to Z$ be a projective morphism of complex analytic spaces such that 
$-(K_X+\Delta)$ is $\varphi$-ample. 
Let $P$ be an arbitrary point of $Z$. 
Let $E$ be any positive-dimensional irreducible component 
of $\varphi^{-1}(P)$. 
Then $E$ is covered by possibly singular rational curves $\ell$ 
with 
\begin{equation*}
0<-(K_X+\Delta)\cdot \ell\leq 2 \dim E. 
\end{equation*} 
In particular, $E$ is uniruled. 
\end{thm}

Here, we quickly reduce Theorem \ref{a-thm9.1} to 
\cite[Theorem 1.12]{fujino-cone}. 
So we use the framework of quasi-log schemes. 

\begin{proof}[Proof of Theorem \ref{a-thm9.1}] 
If $\varphi(X)=P$, then $E=X$ holds. 
In this case, the statement follows from 
\cite[Theorem 1.12]{fujino-cone} since $[X, K_X+\Delta]$ 
is a quasi-log scheme. 
From now on, we may assume that $\varphi(X)\ne P$. 
We shrink $Z$ around $P$. Then 
we can take an effective $\mathbb R$-Cartier divisor 
$B$ on $Z$ such that $(X, \Delta+\varphi^*B)$ is 
kawamata log terminal 
outside $\varphi^{-1}(P)$ 
and that $E$ is a log 
canonical center of $(X, \Delta+\varphi^*B)$. 
We consider the non-kawamata log terminal locus 
$V:=\Nklt(X, \Delta+\varphi^*B)$. 
Note that $\varphi(V)=P$. 
Let $f\colon Y\to X$ be a projective bimeromorphic 
morphism from a smooth variety $Y$ such that 
$K_Y+\Delta_Y=f^*(K_X+\Delta+\varphi^*B)$ and 
that $\Supp \Delta_Y$ is a simple normal crossing divisor on $Y$. 
We put $U:=\Delta^{=1}_Y$ and 
$(K_Y+\Delta_Y)|_U=K_U+\Delta_U$ by adjunction. 
Note that $U$ is projective since $\varphi\circ f(U)=P$. 
We consider the following short exact sequence: 
\begin{equation*}
0\to \mathcal O_Y(-\lfloor \Delta_Y\rfloor)\to 
\mathcal O_Y(-\lfloor \Delta_Y\rfloor +U)\to \mathcal O_U(\lceil 
-\Delta^{<1}_U\rceil -\lfloor \Delta^{>1}_U\rfloor)\to 0. 
\end{equation*}
By the Kawamata--Viehweg vanishing 
theorem (see Theorem \ref{a-thm5.1}), 
$R^1f_*\mathcal O_Y(-\lfloor \Delta_Y\rfloor)=0$. 
Then we have the following commutative diagram: 
\begin{equation*}
\xymatrix{
0 \ar[r]
& \ar@{=}[d]\mathcal J(X, \Delta+\varphi^*B) \ar[r]
& \mathcal J_{\NLC}
(X, \Delta+\varphi^*B) \ar[r]\ar@{^{(}->}[d]& f_*\mathcal O_U(\lceil 
-\Delta^{<1}_U\rceil -\lfloor \Delta^{>1}_U\rfloor) \ar[r]\ar@{^{(}->}[d]
& 0 \\ 
0\ar[r]&\mathcal J(X, \Delta+\varphi^*B) \ar[r]&
\mathcal O_X\ar[r]&\mathcal O_V\ar[r]&0, 
}
\end{equation*} 
where $\mathcal J(X, \Delta+\varphi^*B)=f_*\mathcal O_Y(-\lfloor 
\Delta_Y\rfloor)$ is the {\em{multiplier ideal 
sheaf}} of $(X, \Delta+\varphi^*B)$  and 
$\mathcal J_{\NLC}
(X, \Delta+\varphi^*B)=f_*\mathcal O_Y(-\lfloor \Delta_Y\rfloor 
+U)$, which is called the {\em{non-lc ideal sheaf}} 
of $(X, \Delta+\varphi^*B)$, is an ideal sheaf that defines 
the non-log canonical locus of $(X, \Delta+\varphi^*B)$. 
Hence $\mathcal J:=f_*\mathcal O_U(\lceil 
-\Delta^{<1}_U\rceil -\lfloor \Delta^{>1}_U\rfloor)$ is an ideal sheaf 
on $V$ such that 
\begin{equation*}
\mathcal O_X/\mathcal J_{\NLC}
(X, \Delta+\varphi^*B)=\mathcal O_V/\mathcal J. 
\end{equation*} 
Therefore, since $U$ is projective, 
\begin{equation*}
\bigl(V, (K_X+\Delta+\varphi^*B)|_V, f\colon (U, \Delta_U)\to V \bigr)
\end{equation*}
gives 
a quasi-log scheme structure on 
$[V, (K_X+\Delta+\varphi^*B)|_V]$ 
(see \cite[Theorem 4.9]{fujino-pull-back}) 
such that 
$E$ is a qlc stratum of 
$[V, (K_X+\Delta+\varphi^*B)|_V]$ by construction. 
Then, by \cite[Theorem 1.12]{fujino-cone}, 
$E$ is covered by rational curves $\ell$ with 
\begin{equation*}
0<-(K_X+\Delta+\varphi^*B)\cdot \ell \leq 2\dim E
\end{equation*} 
and is uniruled. 
Since $\varphi^*B\cdot \ell =0$, 
we obtain the desired statement. 
\end{proof}

By combining Theorem \ref{a-thm9.1} with 
a standard argument, we have the following theorem, which is 
well known when $f\colon X\to Y$ is a projective morphism 
between quasi-projective varieties. 

\begin{thm}[{see \cite[Theorem 3.8.1]{bchm}}]\label{a-thm9.2}
Let $\pi\colon X\to Y$ be a projective morphism of 
complex analytic spaces and let $W$ be a compact subset of $Y$ 
such that $\pi\colon X\to Y$ and $W$ satisfies {\em{(P)}}. 
Suppose that $(X, \Delta)$ is a log canonical pair. 
Suppose that there is an $\mathbb R$-divisor $\Delta_0$ such that 
$(X, \Delta_0)$ is kawamata log terminal. 
If $R$ is a $(K_X+\Delta)$-negative extremal ray of $\NE(X/Y; W)$, 
then there exists a rational curve $\ell$ spanning $R$ such that 
\begin{equation*}
0<-(K_X+\Delta)\cdot \ell\leq 2\dim X. 
\end{equation*}
\end{thm}

\begin{proof}
By assumption, we can find $\mathbb R$-divisors 
$\Delta_i$ with $\lim_{i\to \infty} \Delta_i =\Delta$ such that 
$(X, \Delta_i)$ is kawamata log terminal. By replacing 
$\pi$ by 
the contraction defined by the extremal ray $R$, we may further 
assume that $-(K_X+\Delta)$ is $\pi$-ample. 
By Theorem \ref{a-thm9.1}, for some $P\in Y$, 
we can find a rational curve $\ell_i$ in $\pi^{-1}(P)$ 
such that 
\begin{equation*}
0<-(K_X+\Delta_i)\cdot \ell _i \leq 2\dim X
\end{equation*} 
for every $i\gg 0$. 
We note that $\pi^{-1}(P)$ is projective. 
We take a $\pi$-ample 
$\mathbb Q$-divisor $A$ on $X$ such that 
$-(K_X+\Delta+A)$ is also $\pi$-ample. 
In particular, $-(K_X+\Delta_i+A)$ is $\pi$-ample 
for every $i\gg 0$. 
Hence 
\begin{equation*}
0<A\cdot \ell _i =(K_X+\Delta_i +A)\cdot \ell _i 
-(K_X+\Delta_i)\cdot \ell _i <2 \dim X. 
\end{equation*}
This means that the curves $\ell_i$ belong to a bounded family. 
Thus, possibly passing to a subsequence, 
we may assume that $\ell =\ell _i$ is constant. 
Therefore, we have 
\begin{equation*}
-(K_X+\Delta)\cdot \ell =\lim _{i\to \infty} -(K_X+\Delta_i)
\cdot \ell \leq 2\dim X. 
\end{equation*} 
This is what we wanted. 
\end{proof}

The following easy observation is very useful. 

\begin{thm}\label{a-thm9.3}
Let $(X, \Delta)$ be a log canonical pair. 
Let $\pi\colon X\to Y$ be a projective 
morphism of complex analytic spaces and let $W$ be a 
compact subset of $Y$ such that $\pi\colon X\to Y$ and $W$ 
satisfies {\em{(P)}}. 
We assume that there exists $\Delta_0$ on $X$ 
such that $(X, \Delta_0)$ is kawamata log terminal. 
Suppose that 
\begin{equation*}
\xymatrix{
\pi\colon X\ar[r]^-g& Y^\flat\ar[r]^-h & Y
}
\end{equation*}
such that $Y^\flat$ is projective over $Y$. 
Let $H$ be a general 
$h$-ample $\mathbb Q$-divisor 
on $Y^\flat$ with $H\cdot C>2\dim X$ for 
every projective curve $C$ such that 
$h(C)$ is a point. 
Let $R$ be a $(K_X+\Delta+g^*H)$-negative 
extremal ray of $\NE(X/Y; W)$ and 
let $\varphi_R\colon X\to Z$ be the contraction morphism 
over $Y$ associated to $R$. 
Then $\varphi_R\colon X\to Z$ is a 
contraction morphism over $Y^\flat$, that is, 
$Z\to Y$ factors through $Y^\flat$. 
\end{thm}
\begin{proof}
We note that we can see $R$ as a $(K_X+\Delta)$-negative 
extremal ray of $\NE(X/Y; W)$ since $g^*H$ is nef over $Y$. 
Therefore, by Theorem \ref{a-thm9.2}, 
$R$ is spanned by a rational curve $\ell$ on $X$ such 
that $0<-(K_X+\Delta)\cdot \ell \leq 2\dim X$. 
If $g(\ell)$ is a curve, then $(K_X+\Delta+g^*H)\cdot 
\ell >0$ since $\ell\cdot g^*H>2\dim X$. 
Therefore, this means that $g(\ell)$ is a point. 
Hence, the contraction morphism $\varphi_R\colon X\to Z$ exists 
over $Y^\flat$. 
\end{proof}

The next lemma is an easy consequence of 
Theorem \ref{a-thm9.3}. We will repeatedly use 
it in the subsequent sections. 

\begin{lem}\label{a-lem9.4}
Let $\pi\colon X\to Y$ be a projective 
morphism of complex analytic spaces and let $W$ be a 
compact subset of $Y$ such that $\pi\colon X\to Y$ and $W$ 
satisfies {\em{(P)}}. 
Assume that $(X, \Delta)$ is divisorial log terminal 
and that $X$ is $\mathbb Q$-factorial over $W$. 
Suppose that 
\begin{equation*}
\xymatrix{
\pi\colon X\ar[r]^-g& Y^\flat\ar[r]^-h & Y
}
\end{equation*}
such that $Y^\flat$ is projective over $Y$. 
Let $H$ be a general 
$h$-ample $\mathbb Q$-divisor 
on $Y^\flat$ with $H\cdot C>2\dim X$ for 
every projective curve $C$ such that 
$h(C)$ is a point. 
Let 
\begin{equation*}
(X_0, \Delta_0)\overset{\phi_0}{\dashrightarrow} (X_1, \Delta_1)
\overset{\phi_1}{\dashrightarrow} \cdots 
\overset{\phi_{i-1}}{\dashrightarrow} (X_i, \Delta_i)
\overset{\phi_i}{\dashrightarrow} 
\cdots
\end{equation*}
be a $(K_X+\Delta+g^*H)$-minimal 
model program over $Y$ starting from $(X_0, \Delta_0):=(X, 
\Delta)$. 
Then it is a $(K_X+\Delta)$-minimal model 
program over $Y^\flat$. 
\end{lem}
\begin{proof}
We apply Theorem \ref{a-thm9.3} to 
each extremal contraction. 
Then we can see that it is a minimal model 
program over $Y^\flat$. 
\end{proof}

By combining Theorem \ref{a-thm9.2} 
with Theorem \ref{a-thm8.1}, 
we have: 

\begin{thm}\label{a-thm9.5}
Let $\pi\colon X\to Y$ be a projective 
morphism of complex analytic spaces and let $W$ be a 
compact subset of $Y$ such that $\pi\colon X\to Y$ and $W$ 
satisfies {\em{(P)}}. 
Suppose that 
\begin{equation*}
\xymatrix{
\pi\colon X\ar[r]^-g& Y^\flat\ar[r]^-h & Y
}
\end{equation*}
such that $Y^\flat$ is projective over $Y$. 
Assume that $(X, \Delta)$ is kawamata log terminal 
such that $\Delta$ is $\pi$-big. 
We further assume that $K_X+\Delta$ is $g$-nef. 
Then there exists an open neighborhood 
$U$ of $W$ such that 
$(K_X+\Delta)|_{\pi^{-1}(U)}$ is semiample 
over $h^{-1}(U)$. 
\end{thm}
\begin{proof}
We take a relatively compact Stein open subset $U'$ of 
$Y$ and a Stein compact subset $W'$ of 
$Y$ such that $W\subset U'\subset W'$ and 
$\Gamma (W', \mathcal O_Y)$ is noetherian. 
From now on, we will freely shrink $Y$ around $W'$ suitably without 
mentioning it explicitly. 
Since $\Delta$ is $\pi$-big, 
there exists $\Delta'$ such that 
$\Delta'\sim _{\mathbb R} \Delta$, $\Delta'=A+B$, $A\geq 0$, 
$A$ is $\pi$-maple, $B\geq 0$, and 
$(X, \Delta')$ is kawamata log terminal. 
Let $H$ be a general 
$h$-ample $\mathbb Q$-divisor 
on $Y^\flat$ with $H\cdot C>2\dim X$ for 
every projective curve $C$ such that 
$h(C)$ is a point.  
Then $K_X+\Delta+g^*H$ is nef over $W'$. 
Hence $2(K_X+\Delta+g^*H)-(K_X+B)$ is 
$\pi$-ample. 
We apply Theorem \ref{a-thm8.1} to $K_X+\Delta+g^*H$ 
over $U'$. 
Then there exists an open neighborhood 
$U$ of $W$ such that 
$K_X+\Delta+g^*H$ is semiample over $U$. 
This implies that $K_X+\Delta$ is semiample 
over $h^{-1}(U)$.  
\end{proof}

\section{Real linear systems, stable base loci, and 
augmented base loci}\label{a-sec10}

In the theory of minimal models, we have to 
treat $\mathbb R$-divisors. 
Throughout this section, we always assume that 
$\pi\colon X\to Y$ is a projective morphism of normal varieties and let $W$ be 
a Stein compact subset of $Y$. We further assume that $Y$ is Stein for 
simplicity. An $\mathbb R$-divisor on $X$ may have infinitely 
many irreducible components. 
Hence we frequently have to restrict it to a relatively compact open 
subset of $X$ in order to make the number of the 
irreducible components finite. 

Let us start with the definition of {\em{real linear systems}} 
and {\em{stable base loci}}. 

\begin{defn}[Real linear systems and stable base loci]
\label{a-def10.1}
Let $D$ be an $\mathbb R$-divisor on $X$. 
Then we put 
\begin{equation*} 
|D/Y|_{\mathbb R} =\{C\geq 0\, |\, C\sim _{\mathbb R} D\}
\end{equation*}  
and call it the {\em{real linear system}} associated 
to $D$ over $Y$. 
We sometimes simply write $|D|_{\mathbb R}$ to 
denote $|D/Y|_{\mathbb R}$ if there is no danger of confusion. 
The {\em{stable base locus}} 
of $D$ over $Y$ is the Zariski closed 
subset $\mathbf B(D/Y)$ given 
by the 
intersection of the support of the 
elements of $|D/Y|_{\mathbb R}$. 
If $|D/Y|_{\mathbb R}=\emptyset$, then we put 
$\mathbf B(D/Y)=X$. 
Similarly, we consider 
\begin{equation*} 
|D/Y|_{\mathbb Q} =\{C\geq 0\, |\, C\sim _{\mathbb Q} D\}. 
\end{equation*} 
and define 
the Zariski closed subset $\mathbf B(D/Y)_{\mathbb Q}$ as 
the intersection of the support of the elements of $|D/Y|_{\mathbb Q}$. 
If $|D/Y|_{\mathbb Q}=\emptyset$, 
then we put $\mathbf B(D/Y)_{\mathbb Q}=X$. 
We note that the inclusion $\mathbf B(D/Y)\subset \mathbf 
B(D/Y)_{\mathbb Q}$ holds since $|D/Y|_{\mathbb Q}\subset 
|D/Y|_{\mathbb R}$. 
\end{defn}

We make an important remark on the definition of $\mathbf B(D/Y)$ and 
$\mathbf B(D/Y)_{\mathbb Q}$. 

\begin{rem}[{see \cite[Remark 3.5.2]{bchm}}]\label{a-rem10.2}
In Definition \ref{a-def10.1}, $\mathbf B(D/Y)$ and 
$\mathbf B(D/Y)_{\mathbb Q}$ are only defined as closed 
analytic subsets. 
\end{rem}

We will repeatedly use the following basic property of 
$\mathbf B(D/Y)$ implicitly. 

\begin{lem}\label{a-lem10.3}
Let $U$ be any Stein open subset of $Y$. 
If $D\sim _{\mathbb R} C\geq 0$, then $D|_{\pi^{-1}(U)}\sim 
_{\mathbb R} C|_{\pi^{-1}(U)}\geq 0$ obviously 
holds. Hence the inclusion 
\begin{equation*}
\mathbf B(D|_{\pi^{-1}(U)}/U)\subset 
\mathbf B(D/Y)|_{\pi^{-1}(U)}
\end{equation*} 
always holds true. 
\end{lem}

\begin{proof} This is obvious. 
\end{proof}

When we treat $\mathbb Q$-divisors, 
we need: 

\begin{lem}[{see \cite[Lemma 3.5.3]{bchm}}]\label{a-lem10.4}
Let $D$ be an integral Weil divisor. 
Then we have the following inclusions 
\begin{equation*}
\mathbf B(D/Y)_{\mathbb Q} \supset \mathbf B(D/Y), 
\end{equation*}
and 
\begin{equation*}
\mathbf B(D|_{\pi^{-1}(U)}/U)_{\mathbb Q}
\subset 
\mathbf B(D/Y)|_{\pi^{-1}(U)}
\end{equation*} 
for any relatively compact Stein open subset $U$ of $Y$. 
\end{lem}
\begin{proof}
Since $|D/Y|_{\mathbb Q} \subset 
 |D/Y|_{\mathbb R}$, 
 the first inclusion $\mathbf B(D/Y)_{\mathbb Q} \supset 
 \mathbf B(D/Y)$ is obvious. We take 
 $x\in \pi^{-1}(U)$ such that 
$x\not\in \mathbf B(D/Y)|_{\pi^{-1}(U)}$. 
Then, by the proof of \cite[Lemma 3.5.3]{bchm}, 
we can check that 
$x\not \in \mathbf B(D|_{\pi^{-1}(U)}/U)_{\mathbb Q}$. 
Hence the desired second inclusion 
$\mathbf B(D|_{\pi^{-1}(U)}/U)_{\mathbb Q}\subset 
\mathbf B(D/Y)|_{\pi^{-1}(U)}$ holds. 
\end{proof}

Although it may be dispensable, as in the algebraic case, we have: 

\begin{lem}\label{a-lem10.5}
Let $A$ be any $\pi$-ample divisor on $X$ and 
let $U$ be any relatively compact open subset of $Y$. 
Then $\mathbf B((D-\varepsilon A)/Y)|_{\pi^{-1}(U)}$ is independent 
of $\varepsilon$ if $0<\varepsilon \ll 1$. 
\end{lem}

\begin{proof}
It is sufficient to note that 
\begin{equation*}
\mathbf B((D-\varepsilon _1A)/Y)\subset 
\mathbf B((D-\varepsilon _2 A)/Y)
\end{equation*} 
holds for $0<\varepsilon _1<\varepsilon _2$ by definition. 
On a relatively compact open subset $\pi^{-1}(U)$, 
the loci $\mathbf B((D-\varepsilon A)/Y)$ stabilize 
for sufficiently small $\varepsilon >0$ (see Lemma \ref{a-lem2.17}). 
\end{proof}

Let us define {\em{augmented base loci}}. 

\begin{defn}[Augmented base loci]\label{a-def10.6}
The {\em{augmented base locus}} of $D$ over $Y$ is the Zariski closed 
subset 
\begin{equation*}
\mathbf B_+(D/Y):=\bigcap _{\varepsilon >0}\mathbf B((D-\varepsilon A)/Y), 
\end{equation*} where $A$ is some $\pi$-ample divisor on $X$. 
It is not difficult to see that $\mathbf B_+(D/Y)$ is independent of 
the choice of $A$. 
We note that $D$ is $\pi$-big if and only if 
$\mathbf B_+(D/Y)\subsetneq X$ holds. 
\end{defn}

We recall the definition of {\em{fixed divisors}}. 
We need it in Theorem \ref{thm-f} (3). 

\begin{defn}\label{a-def10.7}
Let $D$ be an integral Weil divisor on $X$. 
We put 
\begin{equation*}
|D|:=\{C\geq 0\, |\, C\sim D\}. 
\end{equation*}
Then $\Fix(D)$ denotes the {\em{fixed divisor}} of $D$ so 
that 
\begin{equation*}
|D|=|D-\Fix(D)|+\Fix(D), 
\end{equation*} 
where the base locus of $|D-\Fix(D)|$ contains 
no divisors. 
If $\Fix(D)=0$, then $D$ is said to be {\em{mobile}}. 
\end{defn}

Since we are mainly interested in the minimal model 
program over some open neighborhood of $W$, 
the following definition is useful. 

\begin{defn}[Stable base divisors]\label{a-def10.8}
A divisor $E$ defined on $\pi^{-1}(U)$, where 
$U$ is an open neighborhood of $W$, 
is called a {\em{stable base divisor of $D$ near 
$W$}} if 
$E|_{\pi^{-1}(U')}\subset \mathbf B (D|_{\pi^{-1}(U')}/U')$ holds 
for any Stein open neighborhood $U'$ of $W$ 
with $U'\subset U$. 
\end{defn}

For our purposes, we have to reformulate 
\cite[Proposition 3.5.4]{bchm} as follows. 
We will use Lemma \ref{a-lem10.9} in the proof of Theorem \ref{thm-g}. 

\begin{lem}[{see \cite[Proposition 3.5.4]{bchm}}]\label{a-lem10.9}
Let $\pi\colon X\to Y$ be a projective morphism 
of normal complex varieties and let $W$ be a Stein compact 
subset of $Y$. 
Let $D\geq 0$ be an $\mathbb R$-divisor on $X$. 
Then, after replacing $Y$ with a Stein open neighborhood 
of $W$ suitably, 
we can find $\mathbb R$-divisors $M$ and $F$ on $X$ such 
that 
\begin{itemize}
\item[(1)] $M\geq 0$ and $F\geq 0$, 
\item[(2)] $D\sim _{\mathbb R} M+F$, 
\item[(3)] every component of 
$\Supp F$ is a stable base divisor of $D$ near $W$, 
\item[(4)] if $B$ is a component of $\Supp M$, then some 
multiple is mobile. 
\end{itemize}
\end{lem}
\begin{proof}
The proof of \cite[Proposition 3.5.4, Lemma 3.5.5, and 
Lemma 3.5.6]{bchm} works in our setting with some minor 
modifications. 
As we mentioned above, an $\mathbb R$-divisor 
on $X$ may have infinitely many irreducible components. 
Therefore, we have to replace $Y$ with a relatively compact 
open neighborhood of $W$ in the proof of this lemma. 
For the details, see the proof of \cite[Proposition 3.5.4]{bchm}. 
\end{proof}

\section{Some basic definitions and properties, I}\label{a-sec11}

In this section, we will explain some basic definitions 
which are indispensable for the main results and their proof. 

Let us start with the definition of {\em{$D$-nonpositivity}} 
and {\em{$D$-negativity}}. 

\begin{defn}[{\cite[Definition 3.6.1]{bchm}}]\label{a-def11.1}
Let $\phi\colon X\dashrightarrow Z$ be a bimeromorphic 
contraction of normal complex varieties and let 
$D$ be an $\mathbb R$-Cartier $\mathbb R$-divisor 
on $X$ such that $D':=\phi_*D$ is also $\mathbb R$-Cartier. 
We say that $\phi$ is {\em{$D$-nonpositive}} 
(resp.~{\em{$D$-negative}}) if for some 
common resolution $p\colon V\to X$ and $q\colon V\to Y$, 
we may write 
\begin{equation*}
p^*D=q^*D'+E
\end{equation*}
where $E$ is effective and $q$-exceptional 
(resp.~$E$ is effective, $q$-exceptional, and 
the support of $E$ contains the strict 
transform of the 
$\phi$-exceptional divisors). 
\end{defn}

The so-called {\em{negativity lemma}} (see, for example, 
\cite[Lemma 3.6.2]{bchm}) holds 
true in our complex analytic setting. 
This is because everything follows from 
the negative definiteness of intersection 
form of contractible curves on surfaces (see, for example, 
\cite[Theorem 4-6-1]{matsuki}). 
Therefore, from now on, we will freely 
use the negativity lemma for projective 
morphisms of normal complex varieties. 
Note that the results obtained in \cite[Lemmas 3.6.2, 3.6.3, and 
3.6.4]{bchm} hold true in our complex analytic setting with 
some obvious modifications. 

Let us define {\em{semiample models}} and 
{\em{ample models}} following \cite{bchm}. 

\begin{defn}[{\cite[Definition 3.6.5]{bchm}}]\label{a-def11.2} 
Let $\pi\colon X\to Y$ be a projective morphism 
of complex analytic spaces and let $W$ be a 
compact subset 
of $Y$ such that $\pi\colon X\to Y$ and $W$ satisfies (P). 
Let $D$ be an $\mathbb R$-Cartier $\mathbb R$-divisor 
on $X$. 

Let $f\colon X\dashrightarrow Z$ be a bimeromorphic 
contraction over $Y$ after shrinking $Y$ around $W$ suitably. 
\begin{itemize}
\item We say that $f\colon X\dashrightarrow Z$ is 
a {\em{semiample model of $D$ over some open 
neighborhood of $W$}} if, after shrinking $Y$ around 
$W$ suitably, 
$Z$ is a normal variety and is projective over $Y$, 
$f$ is $D$-nonpositive, and $H:=f_*D$ is semiample 
over $Y$. 
\end{itemize}
Let $g\colon X\dashrightarrow Z$ be a meromorphic 
map over $Y$ after shrinking $Y$ around $W$ suitably. 
\begin{itemize}
\item We say that $g\colon X\dashrightarrow Z$ 
is the {\em{ample model of $D$ over some open 
neighborhood of $W$}} if, 
after shrinking $Y$ around $W$ suitably,  
$Z$ is a normal variety and is projective over $Y$, 
and there exists an ample $\mathbb R$-divisor 
$H$ over $Y$ on $Z$ such that 
if $p\colon V\to X$ and $q\colon V\to Z$ resolve the 
indeterminacy of $g$ then $q$ is a contraction morphism 
and we can write $p^*D\sim _{\mathbb R} q^*H+E$, 
where $E\geq 0$ and $B\geq E$ holds 
for every $B\in |p^*D/Y|_{\mathbb R}$. 
\end{itemize}
\end{defn}

The basic properties of semiample models and ample models 
are summarized as follows. 

\begin{lem}[{see \cite[Lemma 3.6.6]{bchm}}]\label{a-lem11.3} 
Let $\pi\colon X\to Y$ be a projective morphism 
of complex analytic spaces and 
let $W$ be a compact subset of $Y$ such that $\pi\colon X\to Y$ and 
$W$ satisfies {\em{(P)}}. 
Let $D$ be an $\mathbb R$-Cartier $\mathbb R$-divisor on $X$. 
\begin{itemize}
\item[(1)] If $g_i\colon X\dashrightarrow X_i$, $i=1, 2$, are two 
ample models of $D$ over some open neighborhood of $W$, 
then there exists an isomorphism 
$\chi\colon X_1\to X_2$ over some open neighborhood 
of $W$ such that $g_2=\chi\circ g_1$. 
\item[(2)] If $f\colon X\dashrightarrow Z$ is a semiample 
model of $D$ over some open neighborhood 
of $W$, 
then, after shrinking $Y$ around $W$ suitably, 
the ample model $g\colon X\dashrightarrow Z'$ of $D$ over some open 
neighborhood 
of $W$ exists and $g=h\circ f$, 
where $h\colon Z\to Z'$ is a contraction morphism and 
$f_*D\sim _{\mathbb R} 
h^*H$ holds such that $H$ is an $\mathbb R$-divisor 
on $Z'$ which is ample over some open neighborhood 
of $W$. 
\item[(3)] If $f\colon X\dashrightarrow Z$ is a bimeromorphic 
map over some open neighborhood of $W$, 
then $f$ is the ample model of $D$ over some open 
neighborhood of $W$ if and only if 
$f$ is a semiample model of $D$ over some open neighborhood 
of $W$ and $f_*D$ is ample over some open neighborhood of $W$. 
\end{itemize}
\end{lem}
\begin{proof}
For the details, see the proof of (1), (3), and (4) in 
\cite[Lemma 3.6.6]{bchm}. 
\end{proof}

The definition of {\em{weak log canonical models}} and 
{\em{log terminal models}} becomes subtle in our 
complex analytic setting. 

\begin{defn}[{\cite[Definition 3.6.7]{bchm}}]\label{a-def11.4} 
Let $\pi\colon X\to Y$ be a projective 
morphism of complex analytic spaces and let 
$W$ be a compact subset of $Y$. 
Suppose that $K_X+\Delta$ is log canonical 
and let $\phi\colon X\dashrightarrow Z$ 
be a bimeromorphic contraction of normal complex 
varieties over $Y$ after shrinking $Y$ around 
$W$ suitably, 
where $Z$ is projective 
over $Y$. 
We set $\Gamma =\phi_*\Delta$. 
\begin{itemize}
\item[(i)] $Z$ is a {\em{weak log canonical model for 
$K_X+\Delta$ over $W$}} 
if $\phi$ is $(K_X+\Delta)$-nonpositive over some 
open neighborhood of $W$ 
and $K_Z+\Gamma$ is nef over $W$. 
\item[(ii)] $Z$ is a {\em{weak log canonical model for 
$K_X+\Delta$ over some open neighborhood of 
$W$}} 
if, after shrinking $Y$ around 
$W$ suitably, $\phi$ is $(K_X+\Delta)$-nonpositive 
and $K_Z+\Gamma$ is nef over $Y$. 
\item[(iii)] $Z$ is a {\em{log terminal model 
for $K_X+\Delta$ over $W$}} 
if $\phi$ is $(K_X+\Delta)$-negative over some 
open neighborhood of $W$, 
$(Z, \Gamma)$ is divisorial log terminal, 
$K_Z+\Gamma$ is nef over $W$, and $Z$ is 
$\mathbb Q$-factorial 
over $W$. 
\item[(iv)] $Z$ is a {\em{log terminal model 
for $K_X+\Delta$ over some open neighborhood of 
$W$}} 
if, after shrinking $Y$ around 
$W$ suitably, $\phi$ is $(K_X+\Delta)$-negative, 
$(Z, \Gamma)$ is divisorial log terminal, 
$K_Z+\Gamma$ is nef over $Y$, and $Z$ is 
$\mathbb Q$-factorial over $W$. 
\item[(v)] $Z$ is a {\em{good log terminal model for $K_X+\Delta$ 
over some open neighborhood of $W$}} 
if, after shrinking $Y$ around 
$W$ suitably, $\phi$ is $(K_X+\Delta)$-negative, 
$(Z, \Gamma)$ is divisorial log terminal, 
$K_Z+\Gamma$ is semiample over $Y$, and $Z$ is 
$\mathbb Q$-factorial over $W$. 
\end{itemize}
We further assume that $\pi\colon X\to Y$ and $W$ satisfies (P). 
\begin{itemize}
\item[(vi)] $Z$ is the {\em{log canonical model 
for $K_X+\Delta$ over some open neighborhood of 
$W$}} if $\phi$ is 
the ample model of $K_X+\Delta$ over some open 
neighborhood of $W$. 
\end{itemize}
\end{defn}

We give some remarks on Definitions \ref{a-def11.2} and 
\ref{a-def11.4}. 

\begin{rem}[{see \cite[Remark 3.6.8]{bchm}}]\label{a-rem11.5}
A log terminal model is sometimes simply called a 
{\em{log minimal model}} or a {\em{minimal model}}. 
\end{rem}

\begin{rem}\label{a-rem11.6}
In Definitions \ref{a-def11.2}  and \ref{a-def11.4} , 
we only require that $f\colon X\dashrightarrow 
Z$, $g\colon X \dashrightarrow Z$, and 
$\phi\colon X\dashrightarrow Z$ exist after replacing 
$Y$ with a small open 
neighborhood of $W$ suitably. 
If there is no danger of confusion, then 
we simply say that 
$\phi\colon X\dashrightarrow Z$ is a log terminal 
model (weak log canonical model, log canonical model, and so 
on) for $K_X+\Delta$ 
{\em{over $Y$}} 
when it is a log terminal model (weak log canonical model, 
log canonical model, and so on) 
for $K_X+\Delta$ over 
some open neighborhood of $W$. 
\end{rem}

In our complex analytic setting, the definition of 
{\em{Mori fiber spaces}} becomes as follows. 
 
\begin{defn}[Mori fiber spaces]\label{a-def11.7}
Let $(X, \Delta)$ be a divisorial log 
terminal pair. Let $\pi\colon X\to Y$ be a projective 
morphism of complex analytic spaces and let 
$W$ be a compact 
subset of $Y$ such that $\pi\colon X\to Y$ and $W$ satisfies 
(P). 
Let $f\colon X\to Z$ be a projective morphism 
of normal complex varieties over $Y$. 
Then $f\colon (X, \Delta)\to Z$ is a {\em{Mori 
fiber space over $Y$}} if 
\begin{itemize}
\item[(i)] $X$ is $\mathbb Q$-factorial over $W$, 
\item[(ii)] $f$ is a contraction morphism associated 
to a $(K_X+\Delta)$-negative extremal ray of $\NE(X/Y; W)$, and 
\item[(iii)] $\dim Z<\dim X$. 
\end{itemize}
\end{defn}

The following definition is essentially the same as 
\cite[Definition 1.1.4]{bchm}. 
However, we need some modifications 
since we treat only curves mapped to points in $W$ by $\pi$. 

\begin{defn}[{\cite[Definition 1.1.4]{bchm}}]\label{a-def11.8} 
Let $\pi\colon X\to Y$ be a projective morphism 
of complex analytic spaces such that 
$X$ is a normal variety and let $W$ be a compact 
subset of $Y$. 
Let $V$ be a finite-dimensional affine subspace of the 
real vector space $\WDiv_\mathbb R (X)$ spanned by the 
prime divisors on $X$. 
We fix an $\mathbb R$-divisor $A\geq 0$ on $X$ 
such that $\Supp A$ has only finitely many irreducible components 
and 
define 
\begin{align*}
V_A&= \{\Delta \, |\, \Delta=A+B, B\in V\}, \\
\mathcal L_A(V; \pi^{-1}(W))&=\{\Delta=A+B\in V_A\, |\, 
\text{$K_X+\Delta$ is log 
canonical at $\pi^{-1}(W)$ and $B\geq 0$}\}, \\ 
\mathcal E_{A, \pi}(V; W) &
=
\left\{\Delta\in \mathcal L_A(V; \pi^{-1}(W))\, \left|\, 
\begin{array}{l} {\text{$K_X+\Delta$ is pseudo-effective over  
}}\\
{\text{some open 
neighborhood of $W$}}
\end{array}
\right.\right\}, 
\\
\mathcal N^\sharp_{A, \pi}(V; W) &
=\{\Delta\in \mathcal L_A(V; \pi^{-1}(W))\, |\, 
\text{$K_X+\Delta$ is nef over $W$}\}, \  \text{and} 
\\ 
\mathcal N_{A, \pi}(V; W) &=
\left\{\Delta\in \mathcal L_A(V; \pi^{-1}(W))\, \left|\, 
\begin{array}{l} {\text{$K_X+\Delta$ is nef over 
some open 
}}\\
{\text{neighborhood of $W$}}
\end{array}
\right.\right\}. 
\end{align*}
Given a bimeromorphic contraction $\phi\colon X\dashrightarrow 
Z$ after shrinking $Y$ around $W$ suitably, 
define 
\begin{equation*}
\mathcal W^{\sharp}_{\phi, A, \pi}(V; W)=\left\{\Delta
\in \mathcal E_{A, \pi} (V; W)\, \left|\, 
\begin{array}{l} {\text{$\phi$ is a weak log canonical model 
for $(X, \Delta)$}}\\
{\text{over $W$}}
\end{array}
\right.\right\},  
\end{equation*}
and 
\begin{equation*}
\mathcal W_{\phi, A, \pi}(V; W)=\left\{\Delta
\in \mathcal E_{A, \pi} (V; W)\, \left|\, 
\begin{array}{l} {\text{$\phi$ is a weak log canonical model 
for $(X, \Delta)$}}\\
{\text{over some open 
neighborhood of $W$}} 
\end{array}
\right.\right\}.  
\end{equation*} 
Given a meromorphic map 
$\psi\colon X\dashrightarrow Z$ after shrinking $Y$ around 
$W$ suitably, define 
\begin{equation*}
\mathcal A_{\psi, A, \pi}(V; W)=\left\{\Delta\in 
\mathcal E_{A, \pi} (V; W)\, \left|\, 
\begin{array}{l}
{\text{$\psi$ is the ample model for 
$(X, \Delta)$}} \\ 
{\text{over some open 
neighborhood of $W$}}
\end{array}\right.\right\}. 
\end{equation*}
\end{defn}

We make some elementary remarks. 

\begin{rem}\label{a-rem11.9}
By the same argument as in the 
proof of Lemma \ref{a-lem3.5}, we can check that  
$\mathcal L_A(V; \pi^{-1}(W))$ in Definition \ref{a-def11.8} 
is a polytope. 
We further assume that $A$ is a $\mathbb Q$-divisor and that 
$V$ is defined over the rationals. 
Then $\mathcal L_A(V; \pi^{-1}(W))$ 
is a rational polytope. 
\end{rem}

\begin{rem}\label{a-rem11.10}
By definition, it is easy to see that $\mathcal N^\sharp_{A, \pi}(V; W)$ and 
$\mathcal W^\sharp_{\phi, A, \pi}(V; W)$ are closed subsets 
of $\mathcal L_A(V; \pi^{-1}(W))$.  
\end{rem}

We note the following elementary fact. 

\begin{rem}\label{a-rem11.11}
In Definition \ref{a-def11.8}, 
let $S$ be an effective $\mathbb R$-divisor 
on $X$ such that $\Supp S$ has only finitely 
many irreducible components. 
If $\Supp A$ and $\Supp B$ have no common irreducible 
components for every $B\in V$, then 
\begin{equation*}
\mathcal L_{S+A}(V; \pi^{-1}(W))=\mathcal L_S(V_A; \pi^{-1}(W))
\end{equation*} 
holds. Of course, if $\Supp S$ 
and $\Supp B$ have no common irreducible 
components for every $B\in V$, then 
\begin{equation*}
\mathcal L_{S+A}(V; \pi^{-1}(W))=\mathcal L_A(V_S; \pi^{-1}(W))
\end{equation*} 
holds. 

From now on, we assume that $S$ is reduced. 
We put 
\begin{equation*}
V':=\{B\in V\, |\, {\text{$\Supp B$ and $\Supp S$ have no common 
irreducible components}}\}. 
\end{equation*} 
Then $V'$ is an affine subspace of $V$ such that 
\begin{equation*}
\mathcal L_{S+A}(V; \pi^{-1}(W))=
\mathcal L_{S+A}(V'; \pi^{-1}(W))=
\mathcal L_A(V'_S; \pi^{-1}(W))
\end{equation*}
\end{rem}

There are no difficulties to adapt 
\cite[Lemmas 3.6.9, 3.6.10, and 3.6.11]{bchm} 
to our complex analytic setting. 
Roughly speaking, they are easy consequences of 
the negativity lemma. 
Hence we omit the details here. 
On the other hand, \cite[Lemma 3.6.12]{bchm} 
is subtle and needs some reformulation 
for our purposes. We will discuss it in Section \ref{a-sec12}. 

\begin{say}[{see \cite[Lemmas 3.7.3, 3.7.4, and 3.7.5]{bchm}}]\label{a-say11.12}
Note that \cite[Lemmas 3.7.3, 3.7.4, and 3.7.5]{bchm} 
are very important. We need them to reduce various 
problems for log canonical pairs to simpler ones for 
kawamata log terminal pairs. 
We state them here explicitly in our complex analytic setting for the 
sake of completeness. In the following three lemmas 
\ref{a-lem11.13}, \ref{a-lem11.14}, and \ref{a-lem11.15}, 
we assume that $\pi\colon X\to Y$ is a projective 
morphism of complex analytic spaces such that 
$X$ is a normal variety and $Y$ is Stein and 
that $W$ is a Stein compact subset of $Y$. 

\begin{lem}[{see \cite[Lemma 3.7.3]{bchm}}]\label{a-lem11.13}
Let $V$ be a finite-dimensional affine 
subspace of $\WDiv_{\mathbb R}(X)$ and let 
$A\geq 0$ be a $\pi$-big $\mathbb R$-divisor on $X$. 
Let $\mathcal C\subset \mathcal L_A(V; \pi^{-1}(W))$ be a polytope. 

If $\mathbf B_+(A/Y)$ does not contain any non-kawamata 
log terminal centers of $(X, \Delta)$ 
for every $\Delta\in \mathcal C$, then, 
after shrinking $Y$ around $W$ suitably, 
we can find a general $\pi$-ample $\mathbb Q$-divisor 
$A'$ on $X$, a finite-dimensional affine subspace 
$V'$ of $\WDiv_{\mathbb R}(X)$, 
and a translation 
\begin{equation*}
L\colon \WDiv_{\mathbb R}(X)\to \WDiv_{\mathbb R} (X), 
\end{equation*}
by an $\mathbb R$-divisor $T$ with $T\sim 
_{\mathbb R}0$ such that 
$L(\mathcal C)\subset \mathcal L_{A'}(V'; \pi^{-1}(W))$ and 
$(X, \Delta-A)$ and $(X, L(\Delta))$ have the 
same non-kawamata log terminal centers. 
Furthermore, if $A$ is a $\mathbb Q$-divisor, 
then we may assume that $T\sim _{\mathbb Q}0$ holds. 
\end{lem}
\begin{lem}[{see \cite[Lemma 3.7.4]{bchm}}]\label{a-lem11.14}
Let $V$ be a finite-dimensional 
affine subspace of $\WDiv_{\mathbb R}(X)$, which 
is defined over the rationals, and let $A$ be a general $\pi$-ample 
$\mathbb Q$-divisor on $X$. 
Let $S$ be a finite sum of prime divisors on $X$ such that 
each irreducible component of $S$ intersects with $\pi^{-1}(W)$. 
Suppose that there exists a divisorial log terminal pair 
$(X, \Delta_0)$ with $S=\lfloor \Delta_0\rfloor$ and let $G\geq 0$ 
be any divisor whose support does not contain any non-kawamata 
log terminal centers of $(X, \Delta_0)$. 

Then, after shrinking $Y$ around $W$ suitably, 
we can find a general $\pi$-ample 
$\mathbb Q$-divisor $A'$ on $X$, and affine subspace 
$V'$ of $\WDiv_{\mathbb R}(X)$, which is defined over the 
rationals, and a rational affine linear isomorphism 
\begin{equation*}
L\colon V_{S+A}\to V'_{S+A'} 
\end{equation*} 
such that 
\begin{itemize}
\item $L$ preserves $\mathbb Q$-linear equivalence, 
\item $L\left(\mathcal L_{S+A}(V; \pi^{-1}(W))\right)$ is 
contained in the interior of $\mathcal L_{S+A'}(V'; \pi^{-1}(W))$, 
\item for any $\Delta\in L\left(\mathcal L_{S+A}(V; \pi^{-1}(W))\right)$, 
$K_X+\Delta$ is divisorial log terminal and 
$\lfloor \Delta\rfloor =S$, and 
\item for any $\Delta\in L\left(\mathcal L_{S+A}(V; \pi^{-1}(W))\right)$, 
the support of $\Delta$ contains the support of $G$. 
\end{itemize}
\end{lem}
\begin{lem}[{see \cite[Lemma 3.7.5]{bchm}}]\label{a-lem11.15}
Let $(X, \Delta=A+B)$ be a log canonical 
pair, where $A\geq 0$ and $B\geq 0$. 

If $A$ is $\pi$-big and $\mathbf B_+(A/Y)$ does not 
contain any non-kawamata log terminal centers of $(X, \Delta)$ 
and 
there exists a kawamata log terminal pair $(X, \Delta_0)$, 
then we can find a kawamata log terminal pair $(X, \Delta'=A'+B')$, 
where $A'\geq 0$ is a general $\pi$-ample $\mathbb Q$-divisor 
on $X$, $B'\geq 0$, and 
$K_X+\Delta'\sim _{\mathbb R} K_X+\Delta$. If in addition 
$A$ is a $\mathbb Q$-divisor, then $K_X+\Delta'\sim 
_{\mathbb Q} K_X+\Delta$. 
\end{lem}
Here we omit the proof of Lemmas \ref{a-lem11.13}, 
\ref{a-lem11.14}, and \ref{a-lem11.15}. This is because there are 
no difficulties to translate the proof of 
\cite[Lemmas 3.7.3, 3.7.4, and 3.7.5]{bchm} into 
our complex analytic setting. 
\end{say}

In this paper, we are mainly interested in kawamata log 
terminal pairs $(X, \Delta)$ such that 
$\Delta$ is big over $Y$. 
For such pairs, we have some good properties. 

\begin{lem}[{\cite[Lemma 3.9.3]{bchm}}]\label{a-lem11.16} 
Let $\pi\colon X\to Y$ be a projective morphism 
of complex analytic spaces and let $W$ be a compact 
subset of $Y$ such that $\pi\colon X\to Y$ and $W$ satisfies 
{\em{(P)}}. 
Suppose that $(X, \Delta)$ is a kawamata log terminal 
pair, where $\Delta$ is $\pi$-big. 
If $\phi\colon X\dashrightarrow Z$ is a weak 
log canonical model of $K_X+\Delta$ over $W$, 
then 
\begin{itemize}
\item[(1)] $\phi$ is a weak log canonical model 
of $K_X+\Delta$ over some 
open neighborhood of $W$, 
\item[(2)] $\phi$ is a semiample model over some 
open neighborhood of $W$, 
\item[(3)] after shrinking $Y$ around $W$ suitably, 
there exists a contraction morphism 
$h\colon Z\to Z'$ such that $K_Z+\Gamma \sim 
_{\mathbb R} h^*H$, for some $\mathbb R$-divisor 
$H$ on $Z'$, which is ample over $Y$, where $\Gamma =\phi_*\Delta$, and 
\item[(4)] the ample model $\psi\colon X\dashrightarrow 
Z'$ of $K_X+\Delta$ over some open neighborhood 
of $W$ exists. 
\end{itemize}
\end{lem}

\begin{proof}
Throughout this proof, we will freely shrink $Y$ around 
$W$ without mentioning it explicitly. 
We put $\Gamma =\phi_*\Delta$. 
Then $(Z, \Gamma)$ is kawamata log terminal 
by the negativity lemma. 
Since $\Delta$ is big, 
we can write $\Gamma\sim _{\mathbb R} A+B$ such that 
$A$ is ample over $Y$, $A\geq 0$, 
$B\geq 0$, 
and $(Z, A+B)$ is kawamata log terminal. 
Then, by Theorem \ref{a-thm8.3}, 
we can check that 
$K_Z+\Gamma$ is semiample over $Y$. 
This means that $\phi$ is a weak log canonical model 
of $K_X+\Delta$ over $Y$ and that $K_Z+\Gamma$ is semiample 
over $Y$. 
Hence we obtain (1) and (2). 
Since $K_Z+\Gamma$ is semiample over $Y$, 
we get a contraction morphism $h\colon Z\to Z'$ such that 
$\psi:=h\circ \phi\colon X\dashrightarrow Z'$ is 
the ample model of $(X, \Delta)$ (see Lemma \ref{a-lem11.3} (2)). 
Therefore, we have (3) and (4). 
\end{proof}

The following theorem is very important. 

\begin{thm}[{\cite[Theorem 3.11.1]{bchm}}]\label{a-thm11.17}
Let $\pi\colon X\to Y$ be a projective morphism 
of complex analytic spaces and let $W$ be a 
compact subset of $Y$ such that 
$\pi\colon X\to Y$ and $W$ satisfies {\em{(P)}}. 
Let $V$ be a finite-dimensional affine subspace of 
$\WDiv_{\mathbb R}(X)$, which is defined over the rationals. 
Fix a $\pi$-ample $\mathbb Q$-divisor $A$ on $X$. 
Suppose that there exists a kawamata log terminal pair $(X, 
\Delta_0)$. 
Then $\mathcal N_{A, \pi}(V; W)=\mathcal N^\sharp_{A, \pi} 
(V; W)$ holds and the set of hyperplanes 
$R^{\perp}$ is finite in $\mathcal L_A(V; \pi^{-1}(W))$, 
as $R$ ranges over the set of extremal rays of $\NE 
(X/Y; W)$. In particular, $\mathcal N_{A, \pi}(V; W)=
\mathcal N^\sharp_{A, \pi} 
(V; W)$ is a rational polytope. 
\end{thm}

\begin{proof}[Sketch of Proof of Theorem \ref{a-thm11.17}]
By Theorem \ref{a-thm8.3}, $K_X+\Delta$ is semiample 
over some open neighborhood of $W$ for every 
$\Delta\in \mathcal N^\sharp_{A, \pi}(V; W)$. 
In particular, $K_X+\Delta$ is nef over some open neighborhood 
of $W$. This implies that $\mathcal N_{A, \pi}(V; W)=
\mathcal N^\sharp_{A, \pi} 
(V; W)$ holds. On the other hand, 
the proof of \cite[Theorem 3.11.1]{bchm} works 
by Theorem \ref{a-thm7.3}. 
Hence we see that $R^{\perp}$ is finite 
in $\mathcal L_A(V; \pi^{-1}(W))$ and 
$\mathcal N^\sharp_{A, \pi}(V; W)$ is a rational polytope. 
\end{proof}

We prepare an easy lemma. 

\begin{lem}\label{a-lem11.18}
In Theorem \ref{a-thm11.17}, we consider $\Delta_1, 
\Delta_2\in \mathcal L_A(V; \pi^{-1}(W))$. 
Let $f_i\colon X\to Z_i$ be a contraction morphism between 
normal varieties over $Y$ such that 
$K_X+\Delta_i\sim _{\mathbb R} f^*_iD_i$ for 
some $g_i$-ample $\mathbb R$-divisor 
$D_i$ on $Z_i$, 
where $g_i\colon Z_i\to Y$ is the structure morphism, 
for $i=1, 2$. 
Then the following conditions are equivalent. 
\begin{itemize}
\item[(i)] $\Delta_1$ and $\Delta_2$ belong to 
the same interior of a unique face of $\mathcal N^\sharp _{A, \pi} 
(V; W)$. 
\item[(ii)] $Z_1$ and $Z_2$ are isomorphic over some 
open neighborhood of $W$.  
\end{itemize}
\end{lem}

\begin{proof}
We note that $\Delta_1, \Delta_2
\in \mathcal N^\sharp _{A, \pi}(V; W)$. 
If (ii) holds, then (i) obviously holds true. 
From now on, we will prove (ii) under the assumption that 
(i) holds. 
Let $\overline Z$ be the image of the map $(f_1, f_2)\colon 
X\to Z_1\times _Y Z_2$ given by $x\mapsto (f_1(x), f_2(x))$ Let 
$p_i\colon \overline 
Z\to Z_i$ be the projection for $i=1, 2$. 
We take any point $z_i\in g^{-1}_i(W)$. 
Then we can easily see that $p^{-1}_i(z_i)$ is a point 
by (i). By using the Stein factorization 
(see, for example, \cite[Chapter III, Corollary 2.13]{banica}), 
$p_i\colon \overline Z\to Z_i$ is an isomorphism over some 
open neighborhood of $g^{-1}_i(W)$. 
Hence $Z_1$ and $Z_2$ are isomorphic over some 
open neighborhood of $W$. 
This is what we wanted. 
\end{proof}

As an easy consequence of Theorem \ref{a-thm11.17}, 
we have: 

\begin{cor}[{\cite[Corollary 3.11.2]{bchm}}]\label{a-cor11.19}
Let $\pi\colon X\to Y$ be a projective morphism 
of complex analytic spaces and let 
$W$ be a compact subset of $Y$ such that 
$\pi\colon X\to Y$ and $W$ satisfies {\em{(P)}}. 
Let $V$ be a finite-dimensional 
affine subspace of $\WDiv_{\mathbb R}(X)$, 
which is defined over the rationals. 
Fix a general $\pi$-ample $\mathbb Q$-divisor $A$ 
on $X$. 
Suppose that there exists a kawamata log terminal 
pair $(X, \Delta_0)$. 
Let $\phi\colon X\dashrightarrow Z$ be any 
bimeromorphic contraction over $Y$. Then we obtain: 
\begin{itemize}
\item[(1)] $\mathcal W_{\phi, A, \pi}(V; W)=
\mathcal W^\sharp_{\phi, A, \pi}(V; W)$ holds and 
$\mathcal W_{\phi, A, \pi}(V; W)$ is a rational polytope. 
\end{itemize}
Moreover, we have: 
\begin{itemize}
\item[(2)] There are finitely many 
contraction 
morphisms $f_i\colon Z\to Z_i$ over $Y$, $1\leq i\leq k$, 
such that if $f\colon Z\to Z'$ is any contraction morphism 
over $Y$ and there is an $\mathbb R$-divisor 
$D$ on $Z'$, which is ample over $Y$, such that 
$K_Z+\Gamma:=\phi_*(K_X+\Delta)\sim _{\mathbb R} f^*D$ for 
some $\Delta\in \mathcal W_{\phi, A, \pi}(V; W)$, then there 
is an index $1\leq i\leq k$ and an isomorphism 
$\eta\colon Z_i\to Z'$ such that 
$f=\eta\circ f_i$. 
\end{itemize}
Note that in (2) we require that $f_i$, $f$, 
$D$, and $\eta$ exist only after shrinking 
$Y$ around $W$ suitably. 
\end{cor}

The proof of \cite[Corollary 3.11.2]{bchm} works 
with some minor modifications. 
 
\begin{proof}[Sketch of Proof of Corollary \ref{a-cor11.19}]
Note that $\mathcal L_A(V; \pi^{-1}(W))$ is a rational polytope. 
Therefore, its span is an affine subspace of $V_A$, which is 
defined over the rationals. 
By replacing $V$, we may assume that $\mathcal L_A(V; \pi^{-1}(W))$ 
spans $V_A$. To prove that $\mathcal W^\sharp_{\phi, A, \pi}(V; W)$ is 
a rational polytope, we may work locally about 
a divisor $\Delta \in \mathcal W^\sharp_{\phi, A, \pi}(V; W)$. 
By Lemma \ref{a-lem11.14}, 
we may assume that $(X, \Delta)$ is kawamata log terminal. 
In this case, $(Z, \Gamma)$ is automatically kawamata log 
terminal. We put $C:=\phi_*A$. Then $C$ is big over $Y$. Let 
$V^\dag \subset \WDiv_{\mathbb R}(Z)$ be the image 
of $V$. 
By Lemmas \ref{a-lem11.13} and \ref{a-lem11.14}, 
we can reduce the problem to the case 
where $C$ is a $\psi$-ample $\mathbb Q$-divisor 
and $\Gamma$ belongs to the interior 
of $\mathcal L_C(V^\dag; \psi^{-1}(W))$, 
where $\psi\colon Z\to Y$ is the structure morphism. 
By Theorem \ref{a-thm11.17}, 
$\mathcal N^\sharp_{C, \psi} (V^\dag, W)
=\mathcal N_{C, \psi} (V^\dag, W)$ is a rational polytope. 
Hence we can easily check that 
$\mathcal W^\sharp_{\phi, A, \pi} (V; W)=
\mathcal W_{\phi, A, \pi} (V; W)$ holds and 
$\mathcal W^\sharp _{\phi, A, \pi} (V; W)$ is a rational 
polytope. 
Thus we obtain (1). 
Let $f\colon Z\to Z'$ be a contraction morphism 
over some open neighborhood 
of $W$ such that $\phi_*(K_X+\Delta)=K_Z+\Gamma \sim 
_{\mathbb R} f^*D$ for some $\psi'$-ample 
$\mathbb R$-divisor $D$ on $Z'$, where $\psi'\colon Z'\to Y$ 
is the structure morphism. 
Then $\Gamma$ belongs to the interior of a unique 
face $G$ of $\mathcal N^\sharp_{C, \psi} (V^\dag; W)=
\mathcal N_{C, \psi} (V^\dag; W)$. 
Note that $\Delta$ belongs to the interior of a unique face $F$ 
of $\mathcal W^\sharp_{\phi, A, \pi} (V; W)=
\mathcal W_{\phi, A, \pi} (V; W)$ 
and $G$ is determined by 
$F$. 
Thus we can check that (2) holds true by Lemma \ref{a-lem11.18}.  
\end{proof}

\begin{say}[{see \cite[Lemma 3.10.11]{bchm}}]\label{a-say11.20}
When we run a minimal model program, we have to 
check that several properties are preserved by flips and 
divisorial contractions. 

Let $\pi\colon X\to Y$ be a projective morphism between 
complex analytic spaces and let $W$ be a compact subset of $Y$ 
such that $\pi\colon X\to Y$ and $W$ satisfies (P). 
Assume that $(X, \Delta)$ is divisorial log terminal and that 
$X$ is $\mathbb Q$-factorial over $W$. 
Let $\varphi\colon X\to Z$ be a bimeromorphic 
contraction morphism over $Y$ associated to a $(K_X+\Delta)$-negative 
extremal ray $R$ of $\NE(X/Y; W)$. 
Let $A$ be a $\pi$-big $\mathbb R$-divisor 
on $X$ such that $\mathbf B_+(A/Y)$ does not 
contain any non-kawamata log terminal centers 
of $(X, \Delta)$. 

\begin{lem}[Divisorial contractions]\label{a-lem11.21}
In the above setting, we further assume that 
$\varphi$ is divisorial. 
Then, after shrinking $Y$ around $W$ suitably, 
we have the following properties. 
\begin{itemize}
\item[(1)] $Z$ is $\mathbb Q$-factorial over $W$. 
\item[(2)] $(Z, \Gamma)$ is divisorial log terminal, 
where $\Gamma:=\varphi_*\Delta$. 
\item[(3)] $\Exc(\varphi)$ is a prime divisor 
on $X$. 
\item[(4)] $\rho(Z/Y; W)=\rho(X/Y; W)-1$. 
\item[(5)] $\mathbf B_+(\varphi_*A/Y)$ does not 
contain any non-kawamata log terminal centers of $(Z, \Gamma)$.  
\end{itemize}
\end{lem}
\begin{lem}[Flips]\label{a-lem11.22} 
In the above setting, 
we further assume that 
$\varphi$ is a flipping contraction and that 
the flip $\varphi^+\colon X^+\to Z$ of $\varphi$ exists. 
\begin{equation*}
\xymatrix{
X\ar[dr]_-\varphi\ar@{-->}[rr]^-\phi& & \ar[dl]^-{\varphi^+}X^+\\ 
& Z &
}
\end{equation*}
Then, after shrinking $Y$ and $W$ suitably, 
we have the following properties. 
\begin{itemize}
\item[(1)] $X^+$ is $\mathbb Q$-factorial over $W$. 
\item[(2)] $(X^+, \Delta^+)$ is divisorial log terminal, 
where $\Delta^+:=\phi_*\Delta$. 
\item[(3)] $\rho(X^+/Y; W)=\rho (Z/Y; W)+1=\rho(X/Y; W)$. 
\item[(4)] $X^+$ is projective over $Y$. 
\item[(5)] $\mathbf B_+(\phi_*A/Y)$ does not contain any 
non-kawamata log terminal centers of $(X^+, \Delta^+)$. 
\end{itemize}
\end{lem}
\begin{proof}[Proof of Lemmas \ref{a-lem11.21} and \ref{a-lem11.22}]
The proof for algebraic varieties 
works with only some obvious modifications 
even in the complex analytic setting. 
Here, we will only prove (5). There are no difficulties 
to prove the other properties. 
Let $f\colon X\dashrightarrow X'$ denote the 
divisorial contraction $\varphi\colon X\to Z$ in 
Lemma \ref{a-lem11.21} or the flip 
$\phi\colon X\dashrightarrow X^+$ in Lemma 
\ref{a-lem11.22}. 
We will freely shrink $Y$ around $W$ suitably without 
mentioning it explicitly. 
We take a general $\pi'$-ample 
$\mathbb Q$-divisor $C$ on $X'$, 
where $\pi'\colon X'\to Y$ is the structure morphism. 
We may assume that 
$\mathbf B((A-\varepsilon f^{-1}_*C)/Y)$ does not 
contain any non-kawamata log terminal centers 
of $(X, \Delta)$ for some $0<\varepsilon <1$. 
Therefore, we have an effective $\mathbb R$-divisor 
$D$ on $X$ such that 
$D\sim _{\mathbb R} A-\varepsilon f^{-1}_*C$ and 
that $(X, \Delta':=\Delta+\varepsilon 'D)$ is a divisorial 
log terminal pair for $0<\varepsilon '\ll 1$. 
Note that if $0<\varepsilon '\ll 1$ then 
$R$ is still a $(K_X+\Delta')$-negative extremal ray 
of $\NE(X/Y; W)$. 
Therefore, $(X', f_*\Delta'=f_*\Delta+\varepsilon 'f_*D)$ 
is still a divisorial log terminal pair. 
Hence the support of $f_*D$ contains no non-kawamata 
log terminal centers of $(X', f_*\Delta)$. 
Since $f_*A\sim _{\mathbb R}f_*D+\varepsilon C$, 
$\mathbf B_+(f_*A/Y)$ contains no non-kawamata log 
terminal centers of $(X', f_*\Delta')$. 
This is what we wanted. 
\end{proof}
\end{say}

\section{Some basic definitions and properties, II}\label{a-sec12}

In this section, we will treat \cite[Lemma 3.6.12]{bchm} in the complex 
analytic setting. 
We change the formulation suitable for our complex analytic setting. 
The main result of this section is Lemma \ref{a-lem12.3}. 
For the proof of Lemma \ref{a-lem12.3}, we prepare two lemmas. 

Let us start with {\em{small projective $\mathbb Q$-factorializations}} 
(see Theorem \ref{a-thm1.24}). 

\begin{lem}[Small projective $\mathbb Q$-factorializations]\label{a-lem12.1}
 Assume that Theorem \ref{thm-g}$_n$ holds true. 

Let $\pi\colon X\to Y$ be a projective morphism 
between complex analytic spaces 
with $\dim X=n$ and let $W$ 
be a compact subset of $Y$ such that 
$\pi\colon X\to Y$ and $W$ satisfies {\em{(P)}}. 
Assume that $(X, \Delta)$ is kawamata log terminal. 
Then, after shrinking $Y$ around $W$ suitably, 
there exists a small projective bimeromorphic contraction 
morphism $f\colon X'\to X$ such that 
$X'$ is projective over $Y$ and that 
$X'$ is $\mathbb Q$-factorial 
over $W$. 
\end{lem}

\begin{proof}
Throughout this proof, we will freely shrink $Y$ around 
$W$ 
suitably without mentioning it explicitly. 
By taking a resolution, 
we have a bimeromorphic contraction morphism 
$g\colon V\to X$ such that 
$V$ is smooth, $V$ is projective over $Y$, and $\Exc(g)$ and 
$\Exc (g)\cup \Supp g^{-1}_*\Delta$ are 
simple normal crossing divisors on $V$. 
Then we can take an $\mathbb R$-divisor $\Delta_V$ 
on $V$ such that $(V, \Delta_V)$ is kawamata log 
terminal and that $K_V+\Delta_V=g^*(K_X+\Delta)+E$, 
where $E\geq 0$ and $\Supp E=\Exc(g)$. 
We take a general $\pi$-ample $\mathbb Q$-divisor 
$H$ on $X$ such that 
$K_X+\Delta+H\sim _{\mathbb R} D\geq 0$. 
We apply Theorem \ref{thm-g}$_n$ to 
$K_V+\Delta_V+g^*H\sim _{\mathbb R} g^*D+E\geq 0$. 
Then we get a bimeromorphic contraction 
$\phi\colon V\dashrightarrow X'$ over $X$ such that 
$X'$ is $\mathbb Q$-factorial over $W$, $X'$ is projective 
over $Y$, and $K_{X'}+\Gamma$ is nef over $Y$, 
where $\Gamma:=\phi_*\Delta_V$. 
By the negativity lemma, we see that $f\colon X'\to X$ is small. 
This is what we wanted. 
\end{proof}

By combining Lemma \ref{a-lem12.1} with Theorem \ref{a-thm8.2}, we have: 

\begin{lem}\label{a-lem12.2} 
Assume that Theorem \ref{thm-g}$_n$ holds true. 

Let $\pi\colon X\to Y$ be a projective morphism 
between complex analytic spaces 
with $\dim X=n$ and let $W$ 
be a compact subset of $Y$ such that 
$\pi\colon X\to Y$ and $W$ satisfies {\em{(P)}}. 
Let $\phi\colon X\dashrightarrow Z$ be a bimeromorphic 
contraction of normal complex varieties over $Y$ such that 
$(Z, \Delta_Z)$ is kawamata log terminal for some 
$\Delta_Z$. 
We consider the following commutative diagram:  
\begin{equation*} 
\xymatrix{
&\ar[ld]_-pV\ar[rd]^-q&\\ 
X\ar[dr]_-\pi\ar@{-->}[rr]^-\phi&&\ar[dl]^-{\pi'}  Z \\
& Y &
}
\end{equation*} 
where $p$ and $q$ are projective bimeromorphic morphisms 
and $V$ is a normal complex variety. 
Let $H$ be an $\mathbb R$-Cartier $\mathbb R$-divisor 
on $Z$ such that $H':=p_*q^*H$ is also 
$\mathbb R$-Cartier. 
Let $B$ be an $\mathbb R$-Cartier $\mathbb R$-divisor 
on $X$ such that $B$ is numerically equivalent to $H'$ over 
$W$. 
Then $\phi_*B:=q_*p^*B$ is $\mathbb R$-Cartier over 
some open neighborhood of $W$ and 
is numerically equivalent to $H$ over 
$W$. 
\end{lem}

\begin{proof}
It is sufficient to prove that $\phi_*B$ is 
$\mathbb R$-Cartier over some open neighborhood 
of $W$, equivalently, $\phi_*B$ is $\mathbb R$-Cartier 
at any point $z$ of $(\pi')^{-1}(W)$. 
We will freely shrink $Y$ around $W$ without mentioning it 
explicitly. 
By applying Lemma \ref{a-lem12.1} to $\pi'\colon Z\to Y$, 
we can construct a small projective bimeromorphic 
contraction morphism $f\colon Z'\to Z$ such that 
$Z'$ is projective over $Y$ and that $Z'$ is $\mathbb Q$-factorial 
over $W$. 
By replacing $V$, we may assume that 
$q\colon V\to Z$ factors through $Z'$. 
Hence we have the following commutative diagram. 
\begin{equation*}
\xymatrix{& \ar[dl]_-pV\ar[dr]& \\
X' \ar@{-->}[drr]_-\phi\ar@{-->}[rr]_-{\phi'}&& Z'\ar[d]^-f \\ 
&& Z
}
\end{equation*}
Since $Z'$ is $\mathbb Q$-factorial over $W$, we may assume 
that $\phi'_*B$ is $\mathbb R$-Cartier $\mathbb R$-divisor 
on $Z'$. 
We put $K_{Z'}+\Delta_{Z'}=f^*(K_Z+\Delta_Z)$. 
Then $(Z', \Delta_{Z'})$ is kawamata log terminal 
since $f$ is small. 
Note that $\phi'_*B$ is numerically equivalent 
to $f^*H$ over $W$ by construction. 
Therefore, $\phi'_*B$ is numerically trivial over $z$ for any 
$z\in (\pi')^{-1}(W)$. 
Since $f\colon Z'\to Z$ is bimeromorphic, by replacing 
$Z$ with a small Stein open neighborhood of some $z\in (\pi')^{-1}(W)$, we can 
take $\Theta$ on $Z'$ such that 
$(Z', \Theta)$ is kawamata log terminal and 
that $-(K_{Z'}+\Theta)$ is ample over $Z$. 
Hence, by Theorem \ref{a-thm8.2}, 
$\phi_*B$ is $\mathbb R$-Cartier at $z$. 
Note that $W$ is compact. 
Therefore, this means that $\phi_*B$ is 
$\mathbb R$-Cartier over some open neighborhood of $W$. 
This is what we wanted. 
\end{proof}

The following lemma is essentially the same as 
\cite[Lemma 3.6.12]{bchm}. 
We note that we do not assume that $A$ is $\pi$-ample 
in Lemma \ref{a-lem12.3} (see Remark \ref{a-rem12.4} below). 

\begin{lem}[{\cite[Lemma 3.6.12]{bchm}}]\label{a-lem12.3} 
Let $\pi\colon X\to Y$ be a projective morphism 
of complex analytic spaces 
and let $W$ be a compact subset of $Y$ 
such that $\pi\colon X\to Y$ and $W$ satisfies {\em{(P)}} and 
that 
$X$ is $\mathbb Q$-factorial over $W$ and has only 
kawamata log terminal singularities. 
Let $\phi\colon X\dashrightarrow Z$ 
be a bimeromorphic contraction over $Y$ and let 
$A$ be an effective $\mathbb R$-divisor on $X$ such that 
$\Supp A$ has only finitely many irreducible components. 
We assume one of the following conditions: 
\begin{itemize}
\item[(i)] $Z$ is $\mathbb Q$-factorial over $W$, 
or 
\item[(ii)] Theorem \ref{thm-g}$_n$ holds, where 
$n=\dim X$. 
\end{itemize} 
If $V$ is any finite-dimensional 
affine subspace of $\WDiv_{\mathbb R} (X)$ such that 
$\mathcal L_A(V; \pi^{-1}(W))$ spans 
$\WDiv_{\mathbb R}(X)$ modulo numerical equivalence 
over $W$ and 
$\mathcal W^\sharp_{\phi, A, \pi}(V; W)$ intersects the 
interior of $\mathcal L_A(V; \pi^{-1}(W))$, 
then 
\begin{equation*}
\mathcal W^\sharp_{\phi, A, \pi}(V; W)=
\overline{\mathcal A} _{\phi, A, \pi}(V; W)
\end{equation*} 
holds, where $\overline {\mathcal A}_{\phi, A, \pi}(V; W)$ is 
the closure 
of $\mathcal A _{\phi, A, \pi}(V; W)$. 
\end{lem}

Let us prove Lemma \ref{a-lem12.3}. 

\begin{proof}[Proof of Lemma \ref{a-lem12.3}]
It is easy to see that 
\begin{equation*}
\mathcal W^\sharp_{\phi, A, \pi}(V; W)\supset \mathcal A_{\phi, A, \pi}(V; W)
\end{equation*} 
holds. Since $\mathcal W^\sharp_{\phi, A, \pi}(V; W)$ is closed, it follows 
that 
\begin{equation*}
\mathcal W^\sharp_{\phi, A, \pi}(V; W)\supset 
\overline{\mathcal A}_{\phi, A, \pi}(V; W). 
\end{equation*} 
In order to prove the opposite inclusion, 
it is sufficient to prove that 
a dense subset of $\mathcal W^\sharp_{\phi, A, \pi}(V; W)$ 
is contained in $\mathcal A_{\phi, A, \pi}(V; W)$. 

From now on, we will freely shrink $Y$ around $W$ suitably without 
mentioning it explicitly. 
We take $\Delta$ belonging to the interior of 
$\mathcal W^\sharp_{\phi, A, \pi}(V; W)$. 
We put $\Gamma:=\phi_*\Delta$. 
Then $(Z, \Gamma)$ is a weak log canonical 
model of $(X, \Delta)$ over $W$ by definition. 
Since $\mathcal L_A (V; \pi^{-1}(W))$ spans 
$\WDiv_{\mathbb R}(X)$ modulo 
numerical equivalence over $W$, 
we can find $\Delta_0\in \mathcal L_A(V; \pi^{-1}(W))$ 
such that $\Delta_0-\Delta$ is numerically equivalent over 
$W$ to $\mu \Delta$ for some 
$\mu >0$. 
We consider 
\begin{equation*}
\Delta':=\Delta+\varepsilon 
\left((\Delta_0-\Delta)-\mu \Delta\right) =(1-\varepsilon) \nu \Delta
+\varepsilon \Delta_0, 
\end{equation*} 
where 
\begin{equation*}
\nu=\frac{1-\varepsilon -\varepsilon \mu}{1-\varepsilon}<1. 
\end{equation*} 
Hence, $\Delta'$ is numerically equivalent to 
$\Delta$ over $W$ and if $\varepsilon >0$ is sufficiently 
small then $\Delta'$ is effective. 
Since $(X, \nu \Delta)$ is kawamata log terminal, it follows 
that $(X, \Delta')$ is also kawamata log terminal. 
We put $\Gamma':=\phi_*\Delta'$. 
If (i) holds, then $K_Z+\Gamma'$ is obviously $\mathbb R$-Cartier. 
If (ii) holds, then we can check 
that $K_Z+\Gamma'$ is $\mathbb R$-Cartier 
by Lemma \ref{a-lem12.2}. 

Let $H$ be a general $\pi'$-ample 
$\mathbb Q$-divisor on $Z$, where $\pi'\colon Z\to Y$ is the structure 
morphism. 
Let $p\colon U\to X$ and $q\colon U\to Z$ resolve the indeterminacy 
locus of $\phi$. 
We put $H':=p_*q^*H$. 
It is obvious that $\phi$ is $H'$-nonpositive. 
We take $\Delta_1\in \mathcal L_A(V; \pi^{-1}(W))$ such that 
$B:=\Delta_1-\Delta$ is numerically equivalent to 
$\eta H'$ over $W$ 
for some $\eta >0$. 
By replacing $H$ with $\eta H$, we may assume that 
$\eta=1$. If (i) holds, then $\phi_*B$ is obviously 
$\mathbb R$-Cartier. 
If (ii) holds, then we can check that $\phi_*B$ 
is $\mathbb R$-Cartier by Lemma \ref{a-lem12.2}. 
Therefore, we obtain that $\phi$ is $(K_X+\Delta+\lambda B)$-nonpositive 
and $\phi_*(K_X+\Delta+\lambda B)$ is ample over 
$Y$ for every $\lambda >0$. 
On the other hand, we have 
\begin{equation*}
\Delta+\lambda B=\Delta+\lambda(\Delta_1-\Delta)=(1-\lambda) \Delta
+\lambda \Delta_1\in \mathcal L_A(V; \pi^{-1}(W))
\end{equation*} 
for every $\lambda \in [0, 1]$. This implies that 
$\phi$ is the ample model of $K_X+\Delta+\lambda B$ 
over $Y$ for every $\lambda \in (0, 1]$. 
\end{proof}

We close this section with a useful remark. 

\begin{rem}\label{a-rem12.4}
In Lemma \ref{a-lem12.3}, 
we further assume that $A=S+A'$, where $S$ is reduced and 
$A'\geq 0$ is a general $\pi$-ample $\mathbb Q$-divisor 
on $X$, and that $V$ is defined over the rationals. 
We put 
\begin{equation*}
V':=\{B\in V\, |\, {\text{$\Supp B$ and $\Supp S$ have no common 
irreducible components}}\}
\end{equation*} 
as in Remark \ref{a-rem11.11}. Hence $V'$ is also defined over the 
rationals. 
Then we have 
\begin{equation*}
\begin{split}
\mathcal W^\sharp_{\phi, S+A', \pi}(V; W)
&=\mathcal W^\sharp_{\phi, S+A', \pi}(V'; W)
=\mathcal W^\sharp_{\phi, A', \pi}(V'_S; W)
\\ &=\mathcal W_{\phi, A', \pi}(V'_S; W)
=\mathcal W_{\phi, S+A', \pi}(V'; W)
=\mathcal W_{\phi, S+A', \pi}(V; W)
\end{split}
\end{equation*} 
by Remark \ref{a-rem11.11} 
and Corollary \ref{a-cor11.19} (1). 
In particular, 
\begin{equation*}
\mathcal W^\sharp_{\phi, S+A', \pi}(V; W)
=\mathcal W_{\phi, S+A', \pi}(V; W)
\end{equation*} 
is 
a rational polytope 
since $\mathcal W^\sharp_{\phi, A', \pi}(V'_S; W)=
\mathcal W_{\phi, A', \pi}(V'_S; W)$ is a rational polytope 
by Corollary \ref{a-cor11.19} (1). 
\end{rem}

\section{Minimal model program with scaling}\label{a-sec13}

In this section, we will explain the minimal model program 
with scaling. It is very important for various geometric 
applications. 
 
\begin{say}[Minimal model program with scaling]\label{a-say13.1}
Let $\pi\colon X\to Y$ be a projective morphism 
of complex analytic spaces and let $W$ be a compact 
subset of $Y$ such that 
$\pi\colon X\to Y$ and 
$W$ satisfies (P). 
Precisely speaking, $X$ is a normal complex variety, 
$Y$ is a Stein space, and $W$ is a Stein compact subset of $Y$ 
such that $\Gamma (W, \mathcal O_Y)$ is noetherian. 
Let $(X, \Delta)$ be a divisorial log terminal pair such that 
$X$ is $\mathbb Q$-factorial over $W$ and 
let $C$ be an effective $\mathbb R$-Cartier $\mathbb R$-divisor 
on $X$ such that 
$(X, \Delta+C)$ is log canonical and that 
$K_X+\Delta+C$ is nef over $W$. 
We assume that one of the following conditions hold. 
\begin{itemize}
\item[(i)] $\Delta=S+A+B$, $S=\lfloor \Delta\rfloor$, 
$A\geq 0$ is $\pi$-big, $\mathbf B_+(A/Y)$ does not 
contain any non-kawamata log terminal 
centers of $(X, \Delta)$, and $B\geq 0$. 
\item[(ii)] $C$ is $\pi$-big and $\mathbf B_+(C/Y)$ does 
not contain any non-kawamata log terminal centers 
of $(X, \Delta)$. 
\end{itemize}

We recall the following elementary fact 
for the reader's convenience. 

\begin{rem}\label{a-rem13.2}
Assume that $(X, \Delta)$ and 
$(X, \Delta+C)$ are both log canonical and 
that $C$ is effective. 
Then $V$ is a non-kawamata log terminal center 
of $(X, \Delta)$ if and only if 
$V$ is a non-kawamata log terminal 
center of $(X, \Delta+\varepsilon C)$ for every $0<\varepsilon <1$. 
\end{rem}

Although we have already treated more general 
lemmas, we explicitly state an easy 
lemma for the sake of completeness. 

\begin{lem}\label{a-lem13.3}
Suppose that (i) holds true. Then, 
after shrinking $Y$ around $W$ suitably, 
we can find $\Delta'$ such that $K_X+\Delta\sim _{\mathbb R} 
K_X+\Delta'$, 
$\Delta'=A'+B'$, $A'\geq 0$ is a $\pi$-ample $\mathbb Q$-divisor, 
$B'\geq 0$, and $(X, \Delta')$ is kawamata log terminal. 
Suppose that (ii) holds true. Then, 
after shrinking $Y$ around $W$ suitably, 
for any $0<\varepsilon < 1$, there exists 
$\Delta'$ such that $K_X+\Delta+\varepsilon C\sim _{\mathbb R} 
K_X+\Delta'$, 
$\Delta'=A'+B'$, $A'\geq 0$ is a $\pi$-ample $\mathbb Q$-divisor, 
$B'\geq 0$, and $(X, \Delta')$ is kawamata log terminal. 
\end{lem}
\begin{proof}
Throughout this proof, we will freely shrink $Y$ around $W$ without 
mentioning it explicitly. 
We assume that (i) holds. 
By the assumption on $\mathbf B_+(A/Y)$, 
we can write $A\sim _{\mathbb R}A_1+A_2$ such 
that $A_1$ is a $\pi$-ample $\mathbb Q$-divisor 
on $X$ and $A_2$ does not contain any non-kawamata 
log terminal centers of $(X, \Delta)$. 
Then $K_X+S+A+B\sim_{\mathbb R} 
K_X+S+B+(1-\alpha)A+\alpha A_2+\alpha A_1$ 
such that $(X, S+B+(1-\alpha)A+\alpha A_2)$ is 
divisorial log terminal for some 
positive rational number $\alpha$ with $0<\alpha \ll 1$. 
By replacing $A$ and $B$ with $\alpha A_1$ and 
$B+(1-\alpha)A+\alpha A_2$, respectively, 
we may assume that 
$A$ itself is a $\pi$-ample 
$\mathbb Q$-divisor. 
Since $\beta S+\frac{1}{2}A$ is $\pi$-ample 
for some rational number $\beta$ with $0<\beta \ll 1$, 
we can take $A_3\sim _{\mathbb Q} \beta S+\frac{1}{2}A$ 
such that $K_X+(1-\beta )S+\frac{1}{2}A+B+A_3$ is kawamata log 
terminal. 
If we put $A'=\frac{1}{2}A$ and 
$B'=(1-\beta) S+\frac{1}{2}A+B+A_3$, 
then $\Delta'=A'+B'$ satisfies 
the desired properties. 
From now on, we assume that (ii) holds. 
We note that 
$V$ is a non-kawamata log terminal 
center of $(X, \Delta)$ if and only if $V$ is a non-kawamata 
log terminal center of $(X, \Delta+\varepsilon C)$ for 
$0<\varepsilon <1$. 
Therefore, we can apply the above argument to 
$\Delta+\varepsilon C$. 
Thus we have a desired divisor $\Delta'$ on $X$. 
\end{proof}

We put 
\begin{equation*}
\lambda:=\inf \{ \mu \in \mathbb R_{\geq 0} \, 
|\, \text{$K_X+\Delta+\mu C$ is nef over $W$}\}. 
\end{equation*}
If $\lambda=0$, equivalently, $K_X+\Delta$ is nef over $W$, 
then we stop. In this case, $(X, \Delta)$ itself is a log terminal model 
of $(X, \Delta)$ over $W$. 
If further (i) holds true, then $(X, \Delta)$ is a log 
terminal model 
of $(X, \Delta)$ over some open neighborhood 
of $W$ by Theorem \ref{a-thm8.3} (see 
also Lemma \ref{a-lem11.16}). 
Moreover, it is a good log terminal model of $(X, \Delta)$ over some 
open neighborhood of $W$. 

\begin{lem}\label{a-lem13.4}
If $\lambda>0$ holds, then there exists a $(K_X+\Delta)$-negative 
extremal ray $R$ of $\NE(X/Y; W)$ such that 
$(K_X+\Delta+\lambda C)\cdot R=0$. 
\end{lem}

\begin{proof}
If (i) holds, then there are only finitely many $(K_X+\Delta)$-negative 
extremal rays by Theorems \ref{a-thm7.2} and \ref{a-thm7.3}. 
Hence it is not difficult to find a desired extremal ray $R$. 
If (ii) holds, then we consider $K_X+\Delta+\varepsilon C$ for 
$0<\varepsilon \ll 1$. 
By Lemma \ref{a-lem13.3}, 
after shrinking $Y$ around $W$ suitably, 
$K_X+\Delta+\varepsilon C\sim _{\mathbb R} 
K_X+\Delta'$ such that 
$(X, \Delta')$ is kawamata log terminal, $\Delta'=A'+B'$, $A'\geq 0$ is 
a $\pi$-ample $\mathbb Q$-divisor, and $B'\geq 0$. 
Hence, by Theorem \ref{a-thm7.3}, 
there are only finitely many $(K_X+\Delta+\varepsilon C)$-negative 
extremal rays. 
Thus, we can easily take a desired extremal ray $R$. 
\end{proof}

From now on, we assume that $\lambda>0$ holds. 
Let $\varphi\colon X\to Z$ be the extremal 
contraction over $Y$ defined 
by $R$ (see 
Theorems \ref{a-thm7.2} and \ref{a-thm7.3}). 
We note that in general the contraction morphism 
$\varphi\colon X\to Z$ over $Y$ exists only after 
shrinking $Y$ around $W$ suitably. 
If $\varphi$ is not birational, then we have a Mori fiber space 
over $Y$ (see Definition \ref{a-def11.7}) 
and we stop. 

\begin{lem}\label{a-lem13.5}
Assume that Theorem \ref{thm-g}$_n$ holds true. 

Let $\varphi\colon X\to Z$ be a flipping contraction 
associated to a $(K_X+\Delta)$-negative 
extremal ray $R$ of $\NE(X/Y; W)$ with $\dim X=n$.  
Then the flip $\varphi^+\colon X^+\to Z$ exists. 
\end{lem}

\begin{proof}
For the details, see Theorem \ref{a-thm17.9} and its proof. 
\end{proof} 
If $\varphi$ is birational, 
then either $\varphi$ is divisorial and we replace 
$X$ by $Z$ or $\varphi$ is small, that is, flipping, and 
we replace $X$ by the flip $X^+$ (see Lemmas \ref{a-lem11.21} and 
\ref{a-lem11.22}). 
In either case, $K_X+\Delta+\lambda C$ is nef over 
$W$ and $K_X+\Delta$ is divisorial log terminal. 
Hence we may repeat the process under the assumption that 
Theorem \ref{thm-g}$_n$ holds true 
for $n=\dim X$. 
In this way, we obtain a sequence of flips and 
divisorial contractions starting from $X_0:=X$: 
\begin{equation*}
X_0\overset{\phi_0}{\dashrightarrow} X_1
\overset{\phi_1}{\dashrightarrow} \cdots 
\overset{\phi_{i-1}}{\dashrightarrow} 
X_i \overset{\phi_i}{\dashrightarrow} X_{i+1}
\overset{\phi_{i+1}}{\dashrightarrow} \cdots, 
\end{equation*}
and a real numbers 
\begin{equation*}
1\geq \lambda=:\lambda_0\geq \lambda _1\geq \cdots 
\end{equation*} 
such that $K_{X_i}+\Delta_i+\lambda_iC_i$ is nef over 
$W$, where $\Delta_i:=(\phi_{i-1})_*\Delta_{i-1}$ and 
$C_i:=(\phi_{i-1})_*C_{i-1}$ for 
every $i\geq 1$. We note that 
each step $\phi_i$ exists only after shrinking $Y$ around 
$W$ suitably. 
We can easily check that 
each step of this minimal model program preserves 
the conditions (i) and (ii) by the negativity lemma 
(see, for example, Lemmas \ref{a-lem11.21}, 
\ref{a-lem11.22} and \cite[Lemma 3.10.11]{bchm}).  
The above minimal model program is usually called the 
{\em{minimal model program with scaling over $Y$ around $W$}}. 
We sometimes simply say that it is a 
{\em{$(K_X+\Delta)$-minimal model program 
with scaling}}. 
If (i) holds true 
and $\mathbf B_+(A/Y)$ does not 
contain any non-kawamata log terminal 
centers of $(X, \Delta+C)$, 
then this minimal model program always terminates 
after finitely many steps under the assumption that 
Theorem \ref{thm-e} holds true. 

\begin{thm}\label{a-thm13.6}
Assume that Theorem \ref{thm-g}$_n$ and 
Theorem \ref{thm-e}$_n$ hold true, where 
$n=\dim X$. 

Suppose that (i) holds. 
We further assume that $\mathbf B_+(A/Y)$ does not 
contain any non-kawamata log terminal 
centers of $(X, \Delta+C)$. 
Then the minimal model program with scaling 
explained above 
always terminates after finitely many steps. 
\end{thm}

\begin{proof}
By the proof of Lemma \ref{a-lem13.3}, 
we may assume that $\Delta=A+B$, $A\geq 0$ is a $\pi$-ample 
$\mathbb Q$-divisor, $B\geq 0$, $(X, \Delta)$ is 
kawamata log terminal, and 
$(X, \Delta+C)$ is still log canonical. 
By construction, 
after shrinking $Y$ around $W$ suitably, 
$(X_i, \Delta_i +\lambda_i C_i)$ is a weak 
log canonical model of $(X, \Delta+\lambda_i C)$ over $W$ for 
every $i$. 
By Theorem \ref{thm-e}$_n$ and the negativity lemma 
(see \cite[Lemma 3.10.12]{bchm}), 
we know that there are no infinite sequences of flips and 
divisorial contractions. 
This is what we wanted. 
\end{proof}

If $C$ is $\pi$-ample, then we can run a minimal model 
program with scaling of $C$ by (ii). 
We conjecture that the minimal model 
program with scaling always terminates 
after finitely many steps. 
Unfortunately, however, this conjecture 
is still widely open. 
The following easy lemma is useful for some geometric 
applications (see \cite{fujino-semistable}). 

\begin{lem}\label{a-lem13.7} 
Assume that Theorem \ref{thm-g}$_n$ and 
Theorem \ref{thm-e}$_n$ hold true, where 
$n=\dim X$. 

Let $\pi\colon X\to Y$ be a projective morphism of complex analytic 
spaces and let $W$ be a compact subset of $Y$ such that 
$\pi\colon X\to Y$ and $W$ satisfies {\em{(P)}}. 
Let $(X, \Delta)$ be a divisorial log terminal pair such that 
$X$ is $\mathbb Q$-factorial over $W$ and let $C\geq 0$ be a 
$\pi$-ample $\mathbb R$-divisor on $X$ such 
that $(X, \Delta+C)$ is log canonical and that $K_X+\Delta+C$ is nef over 
$W$. We consider a $(K_X+\Delta)$-minimal model program 
with scaling of $C$ over $Y$ around $W$ starting from 
$(X_0, \Delta_0):=(X, \Delta)${\em{:}} 
\begin{equation*}
X_0\overset{\phi_0}{\dashrightarrow} X_1
\overset{\phi_1}{\dashrightarrow} \cdots 
\overset{\phi_{i-1}}{\dashrightarrow} 
X_i \overset{\phi_i}{\dashrightarrow} X_{i+1}
\overset{\phi_{i+1}}{\dashrightarrow} \cdots, 
\end{equation*} 
with 
\begin{equation*}
1\geq \lambda=:\lambda_0\geq \lambda _1\geq \cdots 
\end{equation*} 
such that $K_{X_i}+\Delta_i+\lambda_iC_i$ is nef over 
$W$, $\Delta_i:=(\phi_{i-1})_*\Delta_{i-1}$, and 
$C_i:=(\phi_{i-1})_*C_{i-1}$ for 
every $i\geq 1$. 
We further assume that 
$K_X+\Delta$ is $\pi$-pseudo-effective. 
Then there exists $i_0$ such that 
$K_{X_{i_0}}+\Delta_{i_0}\in \Mov(X_{i_0}/Y; W)$. 
\end{lem}

\begin{proof}
If the minimal model program terminates after finitely many steps, 
then $K_{i_0}+\Delta_{i_0}$ is nef over $W$ 
for some $i_0$ since $K_X+\Delta$ is $\pi$-pseudo-effective. 
This means that $K_{X_{i_0}}+\Delta_{i_0}\in \Mov(X_{i_0}/Y; W)$. 

From now on, we assume that the minimal model program 
does not terminate. We put $\lambda_\infty:=\lim 
_{i\to \infty}\lambda_i\geq 0$. 
If $\lambda_{\infty}>0$, then the given minimal model 
program can be seen as 
a $(K_X+\Delta+\frac{\lambda_{\infty}}{2}C)$-minimal 
model program with scaling of $C$. 
Without loss of generality, we may assume that 
$C$ does not contain any non-kawamata log terminal 
centers of $(X, \Delta+C)$ since $C$ is $\pi$-ample. 
Hence, by Theorem \ref{a-thm13.6}, 
it must terminate. 
This is a contradiction. 
Therefore, we may assume that $\lambda_\infty=0$. 
By replacing $(X, \Delta)$ with $(X_{i_0}, \Delta_{i_0})$ for 
some $i_0$, we may further assume that 
every step of the $(K_X+\Delta)$-minimal model program 
is a flip. Let $G_i$ be a $\mathbb Q$-divisor 
on $X_i$ such that $G_i$ is ample over $Y$. 
We assume that $G_{iX}\to 0$ in $N^1(X/Y; W)$ for 
$i\to \infty$, where $G_{iX}$ is the strict transform of $G_i$ on $X$. 
We note that $K_{X_i}+\Delta_i +\lambda_i C_i 
+G_i$ is ample over some open neighborhood 
of $W$ for every $i$. 
Since $X\dashrightarrow X_i$ is an isomorphism 
in codimension one, 
the strict transform 
$K_X+\Delta+\lambda_i C
+G_{iX}$ is in $\Mov (X/Y; W)$ for every $i$. 
By taking $i\to \infty$, we obtain $K_X+\Delta
\in \Mov (X/Y; W)$. 
This is what we wanted. 
\end{proof}
\end{say}

Anyway, if Theorem 
\ref{thm-g}$_n$ and 
Theorem \ref{thm-e}$_n$ hold true, 
then we can run the minimal model program with scaling 
explained in this section in dimension $n$, although we do not know 
whether it terminates or not. 

\section{Nonvanishing theorem; 
\ref{thm-d}$_n$}
\label{a-sec14}

One of the most difficult results in \cite{bchm} 
is the nonvanishing 
theorem (see \cite[Theorem D]{bchm}). Fortunately, 
we can generalize it for projective morphisms 
of complex varieties without any difficulties. 
For a completely different approach to the nonvanishing 
theorem (see \cite[Theorem D]{bchm}), 
see \cite[Section 3]{birkar-paun} and 
\cite[Theorem 0.1 and Corollary 3.3]{ckp}. 

\begin{thm}[{Nonvanishing theorem, \cite[Theorem D]{bchm}}]\label{a-thm14.1}
Let $(X, \Delta)$ be a kawamata log terminal pair and let 
$\pi\colon X\to Y$ be a projective morphism 
of complex varieties such that $Y$ is Stein. 
Assume that $\Delta$ is big over $Y$ and that 
$K_X+\Delta$ 
is pseudo-effective over $Y$. 
Let $U$ be any relatively compact Stein open subset of $Y$. 
Then there exists a globally $\mathbb R$-Cartier 
$\mathbb R$-divisor $D$ on 
$\pi^{-1}(U)$ such that $(K_X+\Delta)|_{\pi^{-1}(U)}
\sim _{\mathbb R}D\geq 0$. 
\end{thm}
\begin{proof}
By Lemma \ref{a-lem2.37} , 
$(K_X+\Delta)|_{\pi^{-1}(U)}$ is a globally 
$\mathbb R$-Cartier $\mathbb R$-divisor 
on $\pi^{-1}(U)$. 
We take an analytically sufficiently general fiber $F$ of $\pi\colon X\to Y$. 
Then $(F, \Delta|_F)$ is kawamata log terminal, 
$(K_X+\Delta)|_F=K_F+\Delta|_F$, 
$\Delta|_F$ is big, and $K_F+\Delta|_F$ is pseudo-effective. 
Hence, by the nonvanishing theorem for 
projective varieties (see \cite[Theorem D]{bchm}), 
there exists an effective $\mathbb R$-divisor 
$D'$ on $F$ such that $K_F+\Delta|_F\sim _{\mathbb R} D'\geq 0$. 
By Lemma \ref{a-lem2.53}, 
we can find a globally $\mathbb R$-Cartier $\mathbb R$-divisor 
$D$ on $\pi^{-1}(U)$ with $(K_X+\Delta)|_{\pi^{-1}(U)}
\sim _{\mathbb R} D\geq 0$. 
This is what we wanted. 
\end{proof}

\begin{rem}\label{a-rem14.2}
In \cite[Section 6]{bchm}, 
Birkar, Cascini, Hacon, and M\textsuperscript{c}Kernan proved 
\cite[Theorem D$_n$]{bchm} by using 
\cite[Theorems B$_n$, C$_n$, and D$_{n-1}$]{bchm}. 
\end{rem}

We close this section with a very important conjecture. 

\begin{conj}[Nonvanishing conjecture]\label{a-conj14.3} 
Let $X$ be a smooth projective variety such that 
$K_X$ is pseudo-effective. Then there exists a positive 
integer $m$ such that \begin{equation*}
H^0(X, \mathcal O_X(mK_X))\ne 0. 
\end{equation*}
\end{conj}

For the details of Conjecture \ref{a-conj14.3}, 
see \cite{hashizume}. 
Note that if Conjecture \ref{a-conj14.3} holds true 
then the existence problem of minimal models 
for projective log canonical pairs 
is completely solved (see \cite{hashizume}). 

\section{Existence of analytic pl-flips; \ref{thm-f}$_{n-1}$
$\Rightarrow$ \ref{thm-a}$_n$}\label{a-sec15}

In this section, we will see that 
\cite{hacon-mckernan} works with some minor 
modifications 
for projective morphisms between complex analytic 
spaces. 

Let us start with 
the definition of {\em{analytic pl-flipping contractions}}. 

\begin{defn}[Analytic pl-flipping contractions]\label{a-def15.1}
Let $(X, \Delta)$ be a purely log terminal 
pair and let $\varphi\colon X\to Z$ be 
a projective morphism of complex varieties. 
Then $\varphi$ is called an {\em{analytic 
pl-flipping contraction}} 
if $\Delta$ is a $\mathbb Q$-divisor 
and 
\begin{itemize}
\item[(i)] $\varphi$ is small, 
\item[(ii)] $-(K_X+\Delta)$ is $\varphi$-ample, 
\item[(iii)] $S=\lfloor \Delta\rfloor$ is 
irreducible and $-S$ is $\varphi$-ample, 
and 
\item[(iv)] $a(K_X+\Delta)\sim bS$ holds for 
some positive integers $a$ and $b$. 
\end{itemize}
\end{defn}

\begin{rem}\label{a-rem15.2}
Here, we replaced the condition that the 
relative Picard number is one in the usual 
definition of pl-flipping contractions for algebraic 
varieties with (iv) in Definition \ref{a-def15.1}. 
This is because the definition of relative Picard numbers 
is not so clear in the setting of Definition \ref{a-def15.1}. 
Moreover, (iv) is sufficient for the 
proof of the existence of pl-flips. 
\end{rem}

We can define {\em{analytic pl-flips}}. 

\begin{defn}[Analytic pl-flip]\label{a-def15.3}
Let $\varphi\colon (X, \Delta)\to Z$ be 
an analytic pl-flipping contraction as in Definition \ref{a-def15.1}. 
The {\em{flip}} of $\varphi\colon (X, \Delta)\to Z$ 
is a small projective morphism 
$\varphi^+\colon X^+\to Z$ from a normal 
complex variety $X^+$ such that $K_{X^+}+\Delta^+$ 
is $\varphi^+$-ample, where $\Delta^+$ is the strict transform 
of $\Delta$. 
This flip is sometimes called the {\em{$($analytic$)$ 
pl-flip}} of $\varphi\colon (X, \Delta)\to Z$. 
It is not difficult to see 
that 
the existence of $\varphi^+\colon (X^+, \Delta^+)\to Z$ 
is equivalent to 
the condition that 
\begin{equation*}
\bigoplus_{m \in \mathbb N} \varphi_*\mathcal O_X(\lfloor 
m(K_X+\Delta)\rfloor)
\end{equation*} 
is a locally finitely generated graded $\mathcal O_Z$-algebra. 
The flip $\varphi^+$ of $\varphi$ is nothing but 
\begin{equation*}
X^+=\Projan_Z 
\bigoplus_{m \in \mathbb N} \varphi_*\mathcal O_X(\lfloor 
m(K_X+\Delta)\rfloor)\to Z. 
\end{equation*}
\end{defn}

We prepare an easy but important 
lemma. 

\begin{lem}\label{a-lem15.4}
Let $\varphi\colon (X, \Delta)\to Z$ be 
an analytic pl-flipping contraction as in 
Definition \ref{a-def15.1}. 
We put $T:=\varphi(S)$. 
Then $T$ is normal and $\varphi\colon S\to T$ 
has connected fibers, that is, 
$\mathcal O_T\overset{\sim}{\longrightarrow} \varphi_*
\mathcal O_S$ holds. Hence, 
for any open subset $U$ of $Z$ such that 
$T|_U$ is connected, 
$T|_U$ and $S|_{\varphi^{-1}(U)}$ are normal 
irreducible varieties. 
\end{lem}
\begin{proof}
In this proof, we do not need (iv) in 
Definition \ref{a-def15.1}. 
We will only use (i), (ii), and (iii). 
We note that $\varphi$ is bimeromorphic 
by (i). 
We consider the following short exact sequence 
\begin{equation*}
0\to \mathcal O_X(-S)\to \mathcal O_X\to \mathcal O_S
\to 0. 
\end{equation*} 
Since $-S-(K_X+\Delta)$ is $\varphi$-ample and 
$(X, \Delta)$ is purely log terminal, 
we obtain $R^1\varphi_*\mathcal O_X(-S)=0$. 
This implies that 
\begin{equation*}
0\to \mathcal O_Y(-T)=\varphi_*\mathcal O_X(-S)
\to \mathcal O_Y\to \varphi_*\mathcal O_S\to 0
\end{equation*} 
is exact. Hence we get 
$\mathcal O_T\overset{\sim}{\longrightarrow}
\varphi_*\mathcal O_S$. 
Therefore, $T$ is normal and $\varphi\colon S\to T$ 
has connected fibers. 
Note that every normal complex variety is locally 
irreducible. 
Thus, $T|_U$ is a normal irreducible 
complex variety. 
So, $S|_{\varphi^{-1}(U)}$ is also 
a normal irreducible variety. 
\end{proof}

\begin{say}[Theorem \ref{thm-f}$_{n-1}$ $\Rightarrow$ 
Theorem \ref{thm-a}$_n$]\label{a-say15.5}
From now on, let us see how to modify some 
arguments in \cite{hacon-mckernan}. 
\setcounter{step}{0}
\begin{step}
[{see 
\cite[Section 3]{hacon-mckernan}}]\label{a-say15.5-step1}
Let $\varphi\colon (X, \Delta)\to Z$ be an analytic 
pl-flipping contraction with $\dim X=n$. 
In order to prove the existence of the flip of $\varphi$, 
it is sufficient to check that 
\begin{equation*}
\bigoplus_{m \in \mathbb N} \varphi_*\mathcal O_X(\lfloor 
m(K_X+\Delta)\rfloor)
\end{equation*} 
is a locally finitely generated graded $\mathcal O_Z$-algebra. 
Therefore, we take an arbitrary 
point $z\in Z$ and assume that $Z$ is a Stein open neighborhood of $z$ by shrinking 
$Z$ (see Lemma \ref{a-lem15.4}). 
We can always take a Stein compact subset $W$ of 
$Z$ containing $z$ such that 
$\Gamma(W, \mathcal O_Z)$ is noetherian. 

The preliminary results in \cite[Section 3]{hacon-mckernan} 
hold with some minor modifications with the aid of 
Lemma \ref{a-lem2.26} . 
Hence, the existence problem of the flip $\varphi^+$ 
can be reduced to the condition that 
the restricted algebra is a finitely generated graded 
$\mathcal O_{\varphi(S)}$-algebra. 
\end{step}

\begin{step}[{see 
\cite[Section 4]{hacon-mckernan}}]\label{a-say15.5-step2}
As in Step \ref{a-say15.5-step1}, we consider 
a projective morphism $\pi\colon X\to Z$ of normal 
complex varieties such that 
$Z$ is Stein and that there exists a Stein compact 
subset $W$ of $Z$ such that $\Gamma(W, \mathcal O_Z)$ is 
noetherian. 
We take a relatively compact Stein open 
neighborhood $U$ of $W$ in $Z$. 
Then every argument in \cite[Section 4]{hacon-mckernan} 
works in a neighborhood of $\pi\colon 
\pi^{-1}(\overline U)\to \overline U$. 
This means that we can define multiplier ideal sheaves 
(see \cite[Definition-Lemma 4.2]{hacon-mckernan}) 
and check the basic properties. 
Then we can establish \cite[Theorem 4.1]{hacon-mckernan} 
for $\pi\colon \pi^{-1}(U)\to U$. 
All we need here are a relative Kawamata--Viehweg 
vanishing theorem 
and a suitable resolution theorem for complex analytic 
spaces. 
\end{step}

\begin{step}[{see 
\cite[Section 5]{hacon-mckernan}}]\label{a-say15.5-step3} 
Let us see \cite[Section 5]{hacon-mckernan}. 
As in Step \ref{a-say15.5-step2}, 
we work over a neighborhood of 
$\pi\colon \pi^{-1}(\overline U)\to \overline U$. 
Then we can define {\em{asymptotic 
multiplier ideal sheaves}} (see 
\cite[Definition-Lemma 5.2]{hacon-mckernan}) 
with the aid of Lemma \ref{a-lem2.17}. 
Thus, we can establish \cite[Theorem 5.3]{hacon-mckernan}, 
which is the main result of \cite[Section 5]{hacon-mckernan}, 
for $\pi\colon \pi^{-1}(U)\to U$. 
We note that the topics in \cite[Sections 4 and 5]{hacon-mckernan} 
are independent of the theory of minimal models. 
\end{step}

\begin{step}[{see 
\cite[Section 6]{hacon-mckernan}}]\label{a-say15.5-step4} 
The main result of \cite[Section 6]{hacon-mckernan}, 
which is \cite[Theorem 6.3]{hacon-mckernan}, 
is a consequence of \cite[Theorems 4.1 and 5.3]{hacon-mckernan}. 
Therefore, we can formulate and 
prove it for $\pi^{-1}(U)\to U$ without any 
difficulties, where $U$ is a sufficiently small relatively 
compact Stein open neighborhood of a given Stein compact 
subset $W$ of $Z$ with $z\in W$. 
We note that we do need the assumption that 
$\Gamma (W, \mathcal O_Z)$ is noetherian in 
Steps \ref{a-say15.5-step2}, \ref{a-say15.5-step3}, 
and \ref{a-say15.5-step4}. 
\end{step}

\begin{step}[{see 
\cite[Section 7]{hacon-mckernan}}]\label{a-say15.5-step5}
We can formulate \cite[Theorem 7.1]{hacon-mckernan} 
in a neighborhood of $\pi^{-1}(W)\to W$. 
By taking a Stein open neighborhood 
$U$ of $W$ suitably, 
we can use Theorem \ref{thm-f}$_{n-1}$ and the 
results in the previous sections for 
$\pi^{-1}(U)\to U$. In this step, we need the assumption that 
$\Gamma (W, \mathcal O_Z)$ is noetherian in order to 
apply Theorem \ref{thm-f}$_n$. 
\end{step}

\begin{step}[{see 
\cite[Section 8]{hacon-mckernan}}]\label{a-say15.5-step6}
Note that \cite[Section 8]{hacon-mckernan} 
is an easy consequence of \cite[Section 7]{hacon-mckernan}. 
Therefore, we need no new ideas. 
Hence we obtain that 
\begin{equation*}
\bigoplus _{m \in \mathbb N}\varphi_*\mathcal O_X(\lfloor 
m(K_X+\Delta)\rfloor)
\end{equation*}  
is a locally finitely generated graded $\mathcal O_Z$-algebra 
for every $n$-dimensional analytic pl-flipping 
contraction $\varphi\colon (X, \Delta)\to Z$ under 
the assumption that 
Theorem \ref{thm-f}$_{n-1}$ holds. 
\end{step}
Anyway, we have understood that 
Theorem \ref{thm-f}$_{n-1}$ implies Theorem 
\ref{thm-a}$_n$, that is, 
the existence of analytic pl-flips 
in dimension $n$ follows from Theorem \ref{thm-f}$_{n-1}$. 
This is a very important step of the whole proof of the main theorem 
(see Theorems \ref{a-thm1.6} and \ref{a-thm1.13}). 
\end{say}

\section{Special finiteness; \ref{thm-e}$_{n-1}$ 
$\Rightarrow$ \ref{thm-b}$_{n}$}\label{a-sec16}

This section corresponds to \cite[Section 4]{bchm}. 
We will check that the arguments in \cite[Section 4]{bchm} 
can work with some minor obvious modifications. 
We do not need no new ideas here. 

\begin{say}[Theorem \ref{thm-e}$_{n-1}$ $\Rightarrow$ 
Theorem \ref{thm-b}$_n$]\label{a-say16.1}
Let us see \cite[Section 4]{bchm} and make some comments. 
\setcounter{step}{0}
\begin{step}\label{a-say16.1-step1}
In \cite[Lemmas 4.1, 4.2, and 4.3]{bchm}, 
some elementary results are prepared. 
Although they are formulated for quasi-projective varieties, 
there are no difficulties to translate 
them into our complex analytic setting. 
Of course, we are interested in the situation where 
$\pi\colon X\to Y$ is a projective morphism 
of complex analytic spaces and 
$W$ is a compact subset of $Y$ such that 
$\pi\colon X\to Y$ and $W$ satisfies (P) 
and consider everything over some Stein open 
neighborhood of $W$. 

We make an important remark for the reader's convenience. 

\begin{rem}\label{a-rem16.2}
When we formulate \cite[Lemma 4.1]{bchm} for 
our complex analytic setting, 
there are no differences between the notion 
of weak log canonical models {\em{over $W$}} 
and that of weak log canonical models {\em{over 
some open neighborhood of $W$}} by Theorem \ref{a-thm8.3}. 
\end{rem}

\end{step}
\begin{step}\label{a-say16.1-step2}
The main result in \cite[Section 4]{bchm} 
is \cite[Lemma 4.4]{bchm}, 
where we prove Theorem B$_n$ under 
the assumption that Theorem E$_{n-1}$ holds. 
We note that we can use Lemma \ref{a-lem11.14} instead of 
\cite[Lemma 3.7.4]{bchm}. 
In the proof of \cite[Lemma 4.4]{bchm}, 
$Y_i$ is $\mathbb Q$-factorial for every $i$ 
by assumption. 
In our complex analytic setting, 
the corresponding condition becomes the one 
that $Z_i$ is 
$\mathbb Q$-factorial over $W$ for every $i$. 
Therefore, (i) in Lemma \ref{a-lem12.3} 
is satisfied. Thus, we can use Lemma \ref{a-lem12.3} 
instead of \cite[Lemma 3.6.12]{bchm} and 
check that the arguments in the proof of 
\cite[Lemma 4.4]{bchm} can be adapted for our 
complex analytic setting. 
\end{step}
Hence we see that Theorem \ref{thm-e}$_{n-1}$ implies 
Theorem \ref{thm-b}$_n$. 
\end{say}

We close this section with a remark. 

\begin{rem}\label{a-rem16.3}
Theorem B is not in the 
first version of \cite{bchm} circulated in 2006, 
where the special termination, which is 
a more traditional approach originally due to 
Shokurov, is used. 
In \cite{hacon-kovacs}, 
Hacon adopts the special termination instead of 
the special finiteness. 
For the details, see \cite[8.A Special termination]{hacon-kovacs} 
(see also \cite{fujino-special}). 
\end{rem}

\section{Existence of log terminal models; 
\ref{thm-a}$_n$ and \ref{thm-b}$_n$ 
$\Rightarrow$ \ref{thm-g}$_{n}$}\label{a-sec17}

This section corresponds to \cite[Section 5]{bchm}. 
This part is not difficult once we know the existence of 
analytic pl-flips (see Theorem \ref{thm-a}) and 
the special finiteness (see Theorem \ref{thm-b}). 
Precisely speaking, we 
prove Theorem \ref{thm-g}$_n$, which is a slight 
generalization of Theorem \ref{thm-c}$_n$, under the 
assumption that Theorem \ref{thm-a}$_n$ and Theorem 
\ref{thm-b}$_n$ hold true. 

\begin{say}[Theorem \ref{thm-a}$_n$ and 
Theorem \ref{thm-b}$_n$ 
$\Rightarrow$ Theorem \ref{thm-g}$_{n}$]\label{a-say17.1} 
Note that \cite[Lemmas 5.1, 5.2, 5.4, 5.5, and 5.6]{bchm} 
hold true for our complex analytic setting with only 
minor suitable modifications. 
Since \cite[Section 5]{bchm} is easily accessible 
for everyone who studies the minimal model 
program, there are no difficulties to translate 
it into our complex analytic setting. 

The first lemma is an easy consequence of Theorem \ref{thm-b}$_n$. 

\begin{lem}[{see \cite[Lemma 5.1]{bchm}}]\label{a-lem17.2}
Assume that Theorem \ref{thm-b}$_n$ holds true. 

Let $\pi\colon X\to Y$ be a projective morphism 
between complex analytic spaces with 
\begin{equation*}
\xymatrix{
\pi\colon X\ar[r]^-g& Y^\flat\ar[r]^-h & Y
}
\end{equation*} 
such that $Y^\flat$ is projective over $Y$ and let $W$ 
be a compact subset of $Y$ such that 
$\pi\colon X\to Y$ and $W$ satisfies {\em{(P)}}. 
Let $H$ be a general $h$-ample $\mathbb Q$-divisor 
on $Y^\flat$ satisfying $H\cdot \ell>2\dim X$ 
for every projective curve $\ell$ such that 
$h(\ell)$ is a point. Suppose that 
$X$ is $\mathbb Q$-factorial over $W$ with $\dim X=n$, 
\begin{equation*}
K_X+\Delta+C=K_X+S+A+B+C
\end{equation*}
is nef over $W$ and is divisorial log terminal 
with $A\geq 0$, $B\geq 0$, 
and $C\geq 0$, 
where $S$ is a finite sum of prime divisors, and 
$\mathbf B_+(A/Y)$ does not contain any non-kawamata 
log terminal centers of $(X, \Delta+C)$. 
Then any sequence of flips and divisorial 
contractions for the $(K_X+\Delta+g^*H)$-minimal 
model program with scaling over $Y$ around $W$ 
which does not contract $S$, is 
eventually disjoint from $S$. 
\end{lem}
\begin{proof}
Although we made the formulation suitable for our complex analytic setting, 
the proof of \cite[Lemma 5.1]{bchm} works. 
The desired statement is an almost direct consequence of 
Theorem \ref{thm-b}$_n$. 
For the details, see the proof of \cite[Lemma 5.1]{bchm}. 
\end{proof}

We note that the $(K_X+\Delta+g^*H)$-minimal model 
program over $Y$ in Lemma \ref{a-lem17.2} can be seen 
as a $(K_X+\Delta)$-minimal model program 
over $Y^\flat$ by Lemma \ref{a-lem9.4}. 

\begin{lem}[{see \cite[Lemma 5.2]{bchm}}]\label{a-lem17.3}
Assume that Theorem \ref{thm-a}$_n$ and Theorem 
\ref{thm-b}$_n$ hold true. 

Let $\pi\colon X\to Y$ be a projective morphism 
between complex analytic spaces with 
\begin{equation*}
\xymatrix{
\pi\colon X\ar[r]^-g& Y^\flat\ar[r]^-h & Y
}
\end{equation*} 
such that $Y^\flat$ is projective over $Y$ and let $W$ 
be a compact subset of $Y$ such that 
$\pi\colon X\to Y$ and $W$ satisfies {\em{(P)}}. 
Let $H$ be a general $h$-ample $\mathbb Q$-divisor 
on $Y^\flat$ satisfying $H\cdot \ell>2\dim X$ 
for every projective curve $\ell$ such that 
$h(\ell)$ is a point. Suppose that 
$X$ is $\mathbb Q$-factorial over $W$ with $\dim X=n$, 
$(X, \Delta+C=S+A+B+C)$ is a divisorial log terminal 
pair such that $\lfloor \Delta\rfloor =S$, 
$A\geq 0$ is big over $Y$, $\mathbf B_+(A/Y)$ does not 
contain any non-kawamata log terminal centers of 
$(X, \Delta+C)$ with $B\geq 0$ and $C\geq 0$. 
Suppose that there is an $\mathbb R$-divisor $D\geq 0$ whose 
support is contained in $S$ and a real number $\alpha\geq 0$ such that 
\begin{equation*}
K_X+\Delta+g^*H\sim _{\mathbb R} D+\alpha C. 
\end{equation*}
If $K_X+\Delta+C$ is nef over $W$, then, after 
shrinking $Y$ around $W$ suitably, there 
is a log terminal model $\phi\colon X\dashrightarrow 
Z$ for $K_X+\Delta+g^*H$ over $W$, 
where $\phi$ is a 
bimeromorphic contraction over $Y^\flat$, 
such that $\mathbf B_+(\phi_*A/Y)$ does not 
contain any non-kawamata log terminal centers of 
$(Z, \Gamma:=\phi_*\Delta)$. 
\end{lem}
\begin{proof} 
We can run the $(K_X+\Delta+g^*H)$-minimal model 
program over $Y$ around $W$ explained in Section \ref{a-sec13}. 
As usual, we put 
\begin{equation*}
\lambda:=\inf \{t\in \mathbb R_{\geq 0}\, |\, {\text{$K_X+\Delta+g^*H+tC$ is 
nef over $W$}}\}. 
\end{equation*}
If $\lambda=0$, there is nothing to do. 
Otherwise, we can find a $(K_X+\Delta+g^*H)$-negative extremal 
ray $R$ of $\NE(X/Y; W)$ such that 
$(K_X+\Delta+g^*H+\lambda C)\cdot R=0$. 
Let $\varphi_R\colon X\to W$ be the contraction morphism 
over $Y$ associated 
to $R$. By Theorem \ref{a-thm9.3}, $\varphi_R$ is a 
contraction morphism over $Y^\flat$. 
Since $\lambda>0$, $C\cdot R>0$. 
Hence we have $D\cdot R<0$. 
In particular, $\varphi_R$ is always birational. 
When $\varphi_R$ is divisorial, we can replace 
everything with its image. 
When $\varphi_R$ is small, we can see it as an analytic 
pl-flipping contraction 
because $D\cdot R<0$ and $\Supp D\subset S=\lfloor 
\Delta\rfloor$. 
Therefore, by Theorem \ref{thm-a}$_n$, 
we know that the flip $\varphi^+_R\colon X^+\to Z$ exists. 
In this case, we replace $X$ with $X^+$. 
Note that we have to replace $Y$ with a relatively compact Stein open 
neighborhood of $W$ in each step. 
Then the condition $\mathbf B_+(A/Y)$ does not 
contain any non-kawamata log terminal centers of $(X, \Delta)$ 
is preserved by Lemmas \ref{a-lem11.21} and \ref{a-lem11.22}. 
By construction, this minimal model program is not 
an isomorphism in a neighborhood of $S$. 
Hence it terminates by Lemma \ref{a-lem17.2} and 
Theorem \ref{thm-b}$_n$. 
Thus, we finally get a log terminal model $\phi\colon X\dashrightarrow 
Z$. By Lemma \ref{a-lem9.4}, 
the above minimal model program can be seen as a 
$(K_X+\Delta)$-minimal model program 
over $Y^\flat$. 
Therefore, the bimeromorphic contraction $\phi\colon X\dashrightarrow 
Z$ is a bimeromorphic contraction over $Y^\flat$. 
\end{proof}

We need the notion of {\em{neutral models}} in our complex 
analytic setting. 
 
\begin{defn}[{see \cite[Definition 5.3]{bchm}}]\label{a-def17.4}
Let $\pi\colon X\to Y$ be a projective morphism 
between complex analytic spaces with 
\begin{equation*}
\xymatrix{
\pi\colon X\ar[r]^-g& Y^\flat\ar[r]^-h & Y
}
\end{equation*} 
such that $Y^\flat$ is projective over $Y$ and let $W$ 
be a compact subset of $Y$ such that 
$\pi\colon X\to Y$ and $W$ satisfies (P). 
Let $H$ be a general $h$-ample $\mathbb Q$-divisor 
on $Y^\flat$ satisfying $H\cdot \ell>2\dim X$ 
for every projective curve $\ell$ such that 
$h(\ell)$ is a point. 
Let $(X, \Delta=A+B)$ be a divisorial log terminal 
pair with $A\geq 0$ and $B\geq 0$ such that 
$X$ is $\mathbb Q$-factorial over $W$ and let $D$ 
be an effective $\mathbb R$-divisor on $X$. 
A {\em{neutral model}} over $Y^\flat$ for $(X, \Delta+g^*H)$ 
with respect to $A$ and $D$ is any 
bimeromorphic map $f\colon X\dashrightarrow Z$ 
over $Y^\flat$ such that 
\begin{itemize}
\item $f$ is a bimeromorphic contraction, 
\item the only divisors contracted by $f$ are components of $D$, 
\item $Z$ is $\mathbb Q$-factorial over $W$ and is 
projective over $Y$, 
\item $\mathbf B_+(f_*A/Y)$ does not contain 
any non-kawamata log terminal 
centers 
of $(Z, \Gamma:=f_*\Delta)$, and 
\item $K_Z+\Gamma +g_Z^*H$ is divisorial log terminal 
and is nef over $W$, where $g_Z\colon Z\to Y^\flat$ 
is the structure morphism. 
\end{itemize} 
\end{defn}

\begin{lem}[{see \cite[Lemma 5.4]{bchm}}]\label{a-lem17.5}
Assume that Theorem \ref{thm-a}$_n$ and Theorem 
\ref{thm-b}$_n$ hold true. 

Let $\pi\colon X\to Y$ be a projective morphism 
between complex analytic spaces with 
\begin{equation*}
\xymatrix{
\pi\colon X\ar[r]^-g& Y^\flat\ar[r]^-h & Y
}
\end{equation*} 
such that $Y^\flat$ is projective over $Y$ and let $W$ 
be a compact subset of $Y$ such that 
$\pi\colon X\to Y$ and $W$ satisfies {\em{(P)}}. 
Let $H$ be a general $h$-ample $\mathbb Q$-divisor 
on $Y^\flat$ satisfying $H\cdot \ell>2\dim X$ 
for every projective curve $\ell$ such that 
$h(\ell)$ is a point. 
Let $(X, \Delta=A+B)$ be a divisorial log terminal 
pair and let $D$ be an $\mathbb R$-divisor, where 
$A\geq 0$ is big over $Y$, $B\geq 0$, $D\geq 0$, and 
$D$ and $A$ have no common components. 
If 
\begin{itemize}
\item[(1)] $K_X+\Delta+g^*H\sim D$, 
\item[(2)] $X$ is smooth and $G$ is a simple normal 
crossing divisor on $X$ such that $\Supp (\Delta+D)=G$, and 
\item[(3)] $\mathbf B_+(A/Y)$ does 
not contain any non-kawamata log terminal centers of $(X, G)$, 
\end{itemize}
then, after shrinking $Y$ around $W$ suitably, 
$(X, \Delta+g^*H)$ has a neutral model over $Y^\flat$ with respect 
to $A$ and $D$. 
\end{lem}
\begin{proof}
Although our formulation is slightly different from 
\cite[Lemma 5.4]{bchm}, 
the proof of \cite[Lemma 5.4]{bchm} works by using 
Lemma \ref{a-lem17.3} instead of \cite[Lemma 5.2]{bchm}. 
We note that we have to shrink $Y$ around $W$ in each 
step of the proof. 
For the details, see the proof of \cite[Lemma 5.4]{bchm}. 
\end{proof}

\begin{lem}[{see \cite[Lemma 5.5]{bchm}}]\label{a-lem17.6}
Let $\pi\colon X\to Y$ be a projective morphism 
between complex analytic spaces with 
\begin{equation*}
\xymatrix{
\pi\colon X\ar[r]^-g& Y^\flat\ar[r]^-h & Y
}
\end{equation*} 
such that $Y^\flat$ is projective over $Y$ and let $W$ 
be a compact subset of $Y$ such that 
$\pi\colon X\to Y$ and $W$ satisfies {\em{(P)}}. 
Let $H$ be a general $h$-ample $\mathbb Q$-divisor 
on $Y^\flat$ satisfying $H\cdot \ell>2\dim X$ 
for every projective curve $\ell$ such that 
$h(\ell)$ is a point. 
Let $(X, \Delta=A+B)$ be a divisorial log terminal 
pair such that $X$ is $\mathbb Q$-factorial 
over $W$ and let $D$ be an $\mathbb R$-divisor, where 
$A\geq 0$ is big over $Y$, $B\geq 0$, and $D\geq 0$. 
If every component of $D$ is either semiample over $Y$ 
or a stable base divisor of $K_X+\Delta+g^*H$ near $W$ and 
$f\colon X\dashrightarrow Z$ is a neutral model 
over $Y^\flat$ for $(X, \Delta+g^*H)$ with 
respect to $A$ and $D$, then $f$ is a log 
terminal model over $U$ for some 
open neighborhood $U$ of $W$. 
Moreover, $K_Z+\Gamma +g_Z^*H$ is semiample over $U$, 
where $\Gamma:=f_*\Delta$ and $g_Z\colon Z\to Y^\flat$ is the structure 
morphism. 
\end{lem}

\begin{proof}
The proof of \cite[Lemma 5.5]{bchm} works with only some minor modifications. 
In the proof of this lemma, we have to shrink $Y$ around 
$W$ repeatedly. 
We note that $K_Z+\Gamma$ is nef over $h^{-1}(W)$ if and only if 
$K_Z+\Gamma +g^*_Z H$ is nef over $W$. 
We also note that $K_Z+\Gamma +g^*_ZH$ is semiample 
over some open neighborhood of $W$ when $K_Z+\Gamma+g^*_ZH$ 
is nef over $W$. 
For the details, see the proof of \cite[Lemma 5.5]{bchm}. 
\end{proof}

By using the above lemmas, there are no difficulties to 
prove Theorem \ref{thm-g}$_n$ under the assumption that 
Theorem \ref{thm-a}$_n$ and Theorem \ref{thm-b}$_n$ hold true. 

\begin{lem}[{see \cite[Lemma 5.6]{bchm}}]\label{a-lem17.7}
Theorem \ref{thm-a}$_n$ and Theorem \ref{thm-b}$_n$ imply 
Theorem \ref{thm-g}$_n$. 
\end{lem}
\begin{proof}
The proof of \cite[Lemma 5.6]{bchm} works in our setting 
by using Lemmas \ref{a-lem17.5} and \ref{a-lem17.6} instead of 
\cite[Lemmas 5.4 and 5.5]{bchm}. 
We note that we can use Lemma \ref{a-lem10.9} instead of 
\cite[Proposition 3.5.4]{bchm}. 
As usual, we have to replace $Y$ with a relatively 
compact Stein open neighborhood of $W$ finitely many times 
in the proof of this lemma. 
For the details, see the proof of \cite[Lemma 5.6]{bchm}. 
\end{proof}

Finally, we explicitly state the following obvious result 
for the sake of completeness. 

\begin{lem}\label{a-lem17.8}Theorem \ref{thm-g}$_n$ implies 
Theorem \ref{thm-c}$_n$ for every $n$. 
\end{lem}

\begin{proof}
It is sufficient to put $Y^\flat=Y$ and apply Theorem \ref{thm-g}$_n$. 
\end{proof}

Anyway, we see that Theorem \ref{thm-g}$_n$ and 
Theorem \ref{thm-c}$_n$ hold under the assumption 
that Theorem \ref{thm-a}$_n$ and Theorem \ref{thm-b}$_n$ are 
true. 
\end{say}

We need the existence theorem of $\mathbb Q$-factorial 
divisorial log terminal flips, which is an easy 
consequence of Theorem \ref{thm-g}, in order to run minimal model 
programs. 

\begin{thm}[Existence of $\mathbb Q$-factorial divisorial log 
terminal flips]\label{a-thm17.9}
Assume that Theorem \ref{thm-g}$_n$ holds true. 

Let $\pi\colon X\to Y$ be a projective morphism 
of complex analytic spaces with $\dim X=n$ and 
let $W$ be a compact subset of $Y$ such that 
$\pi\colon X\to Y$ and $W$ satisfies {\em{(P)}}. 
We further assume that 
$(X, \Delta)$ is divisorial log terminal and that 
$X$ is $\mathbb Q$-factorial over $W$. 
Let $\varphi\colon X\to Z$ be a small projective bimeromorphic 
contraction morphism 
associated to a 
$(K_X+\Delta)$-negative extremal ray $R$ of $\NE(X/Y; W)$, 
that is, $\varphi\colon (X, \Delta)\to Z$ is a {\em{flipping contraction}} 
associated to $R$. 
Then, after shrinking $Y$ around $W$ suitably, 
the flip $\varphi^+\colon X^+\to Z$ always exists. 
\begin{equation*}
\xymatrix{
X\ar[dr]_-{\varphi} \ar@{-->}^-\phi[rr]&& \ar[dl]^-{\varphi^+}X^+\\ 
& Z&
}
\end{equation*}
This means that 
\begin{itemize}
\item[(1)] $\varphi^+\colon X^+\to Z$ is a small projective 
bimeromorphic contraction morphism, and 
\item[(2)] $K_{X^+}+\Delta^+$ is $\varphi^+$-ample, 
where $\Delta^+:=\phi_*\Delta$. 
\end{itemize}
Moreover, $(X^+, \Delta^+)$ is divisorial 
log terminal, $X^+$ is $\mathbb Q$-factorial 
over $W$, and the equality $\rho(X^+/Y; W)=\rho(X/Y; W)$ holds. 
\end{thm}

\begin{proof}
We will freely shrink $Y$ around $W$. 
We note that $(X, (1-\varepsilon)\Delta)$ is kawamata log 
terminal and $-(K_X+(1-\varepsilon)\Delta)$ is $\varphi$-ample for 
some $0<\varepsilon \ll 1$. 
We take a general $\pi_Z$-ample $\mathbb Q$-divisor 
$A$ on $Z$, where $\pi_Z\colon Z\to Y$ is the 
structure morphism, 
such that $(X, (1-\varepsilon)\Delta+\varphi^*A)$ is kawamata 
log terminal and $K_X+(1-\varepsilon)\Delta+\varphi^*A\sim 
_{\mathbb R} D\geq 0$. 
By Theorem \ref{thm-g}$_n$, $(X, (1-\varepsilon)\Delta)$ 
has a good log terminal model over $Z$. 
Hence $\varphi\colon (X, (1-\varepsilon)\Delta)\to Z$ has a 
flip $\phi\colon X\dashrightarrow X^+$. 
We can easily see that 
$\varphi^+\colon (X^+, \Delta^+)\to Z$ is the flip of $\varphi\colon 
(X, \Delta)\to Z$ and satisfies all the desired properties. 
\end{proof}

We close this section with an almost obvious remark, 
which may be useful for some applications. 

\begin{rem}\label{a-rem17.10}
In Theorem \ref{a-thm17.9}, $(X, \Delta)$ is 
assumed to be a divisorial log terminal pair. 
There are no difficulties to see that 
the existence of flips (see Theorem \ref{a-thm17.9}) 
also holds true 
under a slightly weaker assumption that 
$(X, \Delta)$ is log canonical and 
that there exists $\Delta_0$ such that 
$(X, \Delta_0)$ is kawamata log terminal. 
\end{rem}

\section{Finiteness of models; \ref{thm-g}$_n$ 
$\Rightarrow$ \ref{thm-e}$_n$}\label{a-sec18}

This section corresponds to \cite[Section 7]{bchm}, where 
Theorem E$_n$ is proved under the assumption that 
Theorem C$_n$ and Theorem D$_n$ hold true. 
In our complex analytic setting, in Section \ref{a-sec14}, 
we have already checked that the nonvanishing 
theorem (see Theorem \ref{thm-d}) holds true 
in any dimension by reducing it to the 
original nonvanishing theorem formulated 
for algebraic varieties (see \cite[Theorem D]{bchm}). 
Therefore, we can freely use Theorem \ref{thm-d} here. 
In this section, we will use Theorem \ref{thm-g}$_n$, which is 
slightly stronger than Theorem \ref{thm-c}$_n$, for 
the proof of Theorem \ref{thm-e}$_n$. 

\begin{say}[Theorem \ref{thm-g}$_n$ 
$\Rightarrow$ Theorem \ref{thm-e}$_n$]
\label{a-say18.1}
\setcounter{step}{0}
Let us see \cite[Section 7]{bchm} in detail. 
\begin{step}\label{a-say18.1-step1}
The proof of \cite[Lemma 7.1]{bchm} is well known 
and can work for our complex analytic setting. 
Note that it is sufficient to assume that 
Theorem \ref{thm-g}$_n$ holds true since 
we can freely use Theorem \ref{thm-d} in arbitrary 
dimension as mentioned above. The correct formulation 
of \cite[Lemma 7.1]{bchm} for our complex analytic setting is 
as follows. 

\begin{lem}[{see \cite[Lemma 7.1]{bchm}}]\label{a-lem18.2}
Assume that Theorem \ref{thm-g}$_n$ holds. 

Let $\pi\colon X\to Y$ be a projective morphism 
of complex analytic spaces with $\dim X=n$ and let 
$W$ be a compact subset of $Y$ such that 
$\pi\colon X\to Y$ and $W$ satisfies {\em{(P)}}. 
Let $V$ be a finite-dimensional affine 
subspace of $\WDiv_{\mathbb R}(X)$, 
which is defined over the rationals. 
Fix a general $\pi$-ample $\mathbb Q$-divisor 
$A$ on $X$. Let $\mathcal C\subset \mathcal L_A(V; \pi^{-1}(W))$ 
be a rational polytope such that 
if $\Delta\in \mathcal C$ then $(X, \Delta)$ is 
kawamata log terminal at $\pi^{-1}(W)$. 

Then, after shrinking $Y$ around $W$ suitably, 
there are finitely many bimeromorphic maps $\phi_i\colon 
X\dashrightarrow Z_i$ over $Y$, $1\leq i\leq k$, with the property 
that if $\Delta\in \mathcal C\cap \mathcal E_{A, \pi}(V; W)$, then 
there exists an index $1\leq i\leq k$ such that 
$\phi_i$ is a log terminal model of $K_X+\Delta$ over some open 
neighborhood of $W$. 
\end{lem}
\begin{proof}
We have already proved Theorem \ref{thm-d} in any dimension. 
We can use Lemma \ref{a-lem11.14} instead of \cite[Lemma 3.7.4]{bchm}. 
Therefore, by using Theorem \ref{thm-g}$_n$ instead of 
\cite[Theorem C$_n$]{bchm}, we see that the proof of 
\cite[Lemma 7.1]{bchm} works in our complex analytic setting. 
Precisely speaking, we formulate Theorem \ref{thm-g}$_n$ in order 
to make the proof of \cite[Lemma 7.1]{bchm} work in 
the complex analytic setting. For the details, see the proof of 
\cite[Lemma 7.1]{bchm}. 
\end{proof}
\end{step}
\begin{step}\label{a-say18.1-step2}
The proof of \cite[Lemma 7.2]{bchm} uses 
\cite[Lemma 3.6.12]{bchm}. In our 
case, we can use Lemma \ref{a-lem12.3} instead 
of \cite[Lemma 3.6.12]{bchm}. 
We note that (ii) in Lemma \ref{a-lem12.3}, 
which is nothing but Theorem \ref{thm-g}$_n$, 
is satisfied by assumption. As in Step \ref{a-say18.1-step1}, 
it is sufficient to assume Theorem \ref{thm-g}$_n$ 
since we can freely use Theorem \ref{thm-d} in any 
dimension. 

\begin{lem}[{\cite[Lemma 7.2]{bchm}}]\label{a-lem18.3}
Assume that Theorem \ref{thm-g}$_n$ holds true. 

Let $\pi\colon X\to Y$ be a projective morphism of 
complex analytic spaces with $\dim X=n$ 
and let $W$ be a compact subset 
of $Y$ such that $\pi\colon X\to Y$ and $W$ satisfies 
{\em{(P)}}. Suppose that 
there is a kawamata log terminal pair $(X, \Delta_0)$. 
We fix a general $\pi$-ample 
$\mathbb Q$-divisor 
$A$ on $X$. 
Let $V$ be a finite-dimensional affine subspace of $\WDiv_{\mathbb R}(X)$ 
which is defined over the rationals. 
Let $\mathcal C\subset \mathcal L_A(V; \pi^{-1}(W))$ be a rational 
polytope. 

Then, after shrinking $Y$ around $W$ suitably, 
there are finitely many bimeromorphic contractions 
$\psi_j\colon X\dashrightarrow Z_j$ over $Y$, $1\leq j\leq l$, such that 
if $\psi\colon X\dashrightarrow Z$ is a weak log canonical 
model of $K_X+\Delta$ over $W$ for some $\Delta\in \mathcal C$ then 
there exist an index $1\leq j\leq l$ and an isomorphism 
$\xi\colon Z_j\to Z$ over some open neighborhood of $W$ such that 
$\psi=\xi\circ \psi_j$ holds. 
\end{lem}
\begin{proof}
If we use Lemma \ref{a-lem18.2}, Corollary \ref{a-cor11.19}, and 
Lemma \ref{a-lem12.3}  
instead of \cite[Lemma 7.1]{bchm}, \cite[Corollary 3.11.2]{bchm}, 
and \cite[Lemma 3.6.12]{bchm}, 
then the proof of \cite[Lemma 7.2]{bchm} works in our complex 
analytic setting. The idea is as follows. 
By using Lemma \ref{a-lem11.14} and so on, we 
can reduce the problem to the case where we can use 
Lemma \ref{a-lem12.3}. 
Let $\psi\colon X\dashrightarrow Z$ be a weak log canonical model 
over $W$. Then we can take $\Delta'$ such that 
$\psi\colon X\dashrightarrow Z$ is an ample model of $(X, \Delta')$. 
Then, by Lemma \ref{a-lem18.2} and Corollary \ref{a-cor11.19}, 
we obtain all the desired properties. 
For the details, see the proof of \cite[Lemma 7.2]{bchm}. 
\end{proof}
\end{step}
\begin{step}\label{a-say18.1-step3}
The final step is obvious. 
\begin{lem}[{\cite[Lemma 7.3]{bchm}}]\label{a-lem18.4}
Theorem \ref{thm-g}$_n$ implies Theorem \ref{thm-e}$_n$. 
\end{lem}
\begin{proof} We note that $\mathcal L_A(V; \pi^{-1}(W))$ is a rational 
polytope. Therefore, it 
is sufficient to put $\mathcal C=\mathcal L_A(V; \pi^{-1}(W))$ 
in Lemma \ref{a-lem18.3}. 
\end{proof} 
\end{step}
Hence we see that Theorem \ref{thm-e}$_n$ holds 
under the assumption that Theorem \ref{thm-g}$_n$ holds 
true. This is what we wanted. 
\end{say}

By the above arguments, we think that the reader can understand 
the reason why we prepared Theorem \ref{thm-g}, 
Corollary \ref{a-cor11.19}, and 
Lemma \ref{a-lem12.3}. 

\section{Finite generation; 
\ref{thm-g}$_n$ 
$\Rightarrow$ \ref{thm-f}$_n$}\label{a-sec19}

This section corresponds to \cite[Section 8]{bchm}.
Note that \cite[Section 8]{bchm} is a very short section, 
which consists of only one lemma (see \cite[Lemma 8.1]{bchm}). 
The proof of Theorem \ref{thm-f}$_n$ given below 
is slightly more complicated than the original 
algebraic version in \cite[Section 
8]{bchm}. This is because we formulated 
everything only over some open neighborhood of 
a given compact subset of the base space. 

\begin{say}[Theorem \ref{thm-g}$_n$ 
$\Rightarrow$ Theorem \ref{thm-f}$_n$]\label{a-say19.1} 
Here, we will prove Theorem \ref{thm-f}$_n$ 
under the assumption that 
Theorem \ref{thm-g}$_n$ holds true.  

First, we will prove (1). In the proof of (1), we will 
freely shrink $Y$ around $W$ suitably 
without mentioning it explicitly. 
If $K_X+\Delta$ is $\pi$-pseudo-effective, 
then 
$K_X+\Delta\sim _{\mathbb R} D\geq 0$ by Theorem 
\ref{thm-d}. 
Hence, by Theorem \ref{thm-g}$_n$, $(X, \Delta)$ has 
a good log terminal model $(Z, \Gamma)$ over $Y$. 
Hence $K_Z+\Gamma$ is semiample over $Y$. 
We take any point $y\in Y$. By applying the above 
result to $\pi\colon X\to Y$ with $W:=\{y\}$. 
Then we see that 
there exits an open neighborhood $U_y$ of $y$ such that 
$(X, \Delta)$ has a good log terminal model over $U_y$. 
By this observation, we obtain that $R(X/Y, K_X+\Delta)$ 
is a locally finitely generated graded $\mathcal O_Y$-algebra. 
Thus, we get (1) in Theorem \ref{thm-f}$_n$. 

From now on, we will prove (2). 
Let $\mu\colon X\dashrightarrow Z$ be a good 
log terminal model for $K_X+\Delta$ over $Y$ after 
shrinking $Y$ around $W$ (see Theorem \ref{thm-g}$_n$). 
Since $G$ is a prime divisor contained 
in the stable base locus of $K_X+\Delta$ over $Y$, 
$G$ is $\mu$-exceptional. 
We take a small positive real number $\delta$ such that 
if $|\!| \Xi-\Delta|\!|<\delta$ then 
$(Z, \mu_*\Xi)$ is kawamata log terminal and 
$a(G, X, \Xi)<a(G, Y, \mu_*\Xi)$. 
If $K_X+\Xi$ is not $\pi$-pseudo-effective, 
then $\mathbf B((K_X+\Xi)/Y)=X$. 
Therefore, $G\subset \mathbf B((K_X+\Xi)/Y)$ is 
obvious. 
Hence we may assume that $K_X+\Xi$ is 
$\pi$-pseudo-effective. We take any point $y$ of $Y$. 
Then there exists an open neighborhood 
$U_y$ of $y$ such that 
$(Z, \mu_*\Xi)$ has a good log terminal model 
over $U_y$ by Theorem 
\ref{thm-g}$_n$. 
This means that $(X, \Xi)$ has a good log terminal model 
over $U_y$. 
Hence we can easily check that 
$G|_{\pi^{-1}(U_y)}\subset 
\mathbf B((K_X+\Xi)|_{\pi^{-1}(U_y)}/U_y)$. 
Therefore, 
$G|_{\pi^{-1}(U_y)}\subset \mathbf B((K_X+\Xi)/Y)$. 
Thus, we obtain that $G\subset \mathbf B((K_X+\Xi)/Y)$ 
holds since $y$ is any point of $Y$. 
This is (2). 

Finally, we will prove (3). We take a good log terminal model 
$\mu\colon X\dashrightarrow Z$ of $K_X+\Delta$ over $Y$ 
after shrinking $Y$ around $W$ (see Theorem \ref{thm-g}$_n$).  
By Corollary \ref{a-cor11.19} (1), 
$\mathcal W_{\mu, A, \pi}(V'; W)=
\mathcal W^\sharp_{\mu, A, \pi}(V'; W)$ 
is a rational polytope and $\Delta\in \mathcal W_{\mu, A, \pi}(V'; W)$. 
Therefore, we may assume that 
$K_Z+\mu_*\Xi$ is nef over $Y$ for every $\Xi \in 
\mathcal W_{\mu, A, \pi}(V'; W)$ after shrinking $Y$ around $W$ again. 
Hence, after shrinking $Y$ around $W$ suitably,  
there exists a positive constant $\eta$ such that 
if $\Xi \in V'$ and $|\!| \Xi -\Delta|\!|<\eta$ then 
$(Z, \mu_*\Xi)$ is kawamata log 
terminal and $K_Z+\mu_*\Xi$ is semiample 
over $Y$. 
By Theorem \ref{a-thm3.12}, 
$Z$ has only rational singularities. 
Since $Z$ is $\mathbb Q$-factorial over $W$, 
there is a positive integer $l$ such that if 
$m(K_Z+\mu_*\Xi)$ is an integral Weil divisor then 
$lm(K_Z+\mu_*\Xi)$ is Cartier over some open neighborhood 
of $W$ (see Lemma \ref{a-lem2.42}). 
By replacing $Y$ with a small open neighborhood of $W$, 
we may assume that $lm(K_Z+\mu_*\Xi)$ is Cartier on $Z$. 
Therefore, by Theorem \ref{a-thm6.4}, 
there exists $r>0$ such that $rm(K_Z+\mu_*\Xi)$ is free over $Y$ when 
$m(K_Z+\mu_*\Xi)$ is an integral Weil divisor. 
It follows that 
if $k(K_X+\Xi)/r$ is Cartier then every component of 
$\Fix (k(K_X+\Xi))$ is contracted by $\mu$ and so every component 
is in $\mathbf B((K_X+\Delta)/Y)$. 
We finish the proof of (3) in Theorem \ref{thm-f}$_n$. 
\end{say}

In Section \ref{a-sec21}, 
we will prove Theorems \ref{a-thm1.18} and \ref{a-thm1.22}, 
which are much more general than 
the finite generation in 
Theorem \ref{thm-f} (1). 

\section{Proof of theorems}\label{a-sec20}

In this section, we will prove theorems. 
Note that we postpone the proof of Theorems \ref{a-thm1.18} and 
\ref{a-thm1.22} until Section \ref{a-sec21} because 
it needs some deep results from the theory of variations of 
Hodge structure. Theorem \ref{a-thm1.28} will be proved 
in Section \ref{a-sec22} after we explain some supplementary 
results on the minimal model program with scaling. 
Theorem \ref{a-thm1.30} (see 
Theorem \ref{a-thm23.2}) will be treated in Section \ref{a-sec23}. 
Note that the proof of Theorem \ref{a-thm1.30} 
uses Theorem \ref{a-thm1.18}. 

Let us start with 
the proof of Theorems \ref{a-thm1.6} and 
\ref{a-thm1.13}, which is now almost obvious. 

\begin{proof}[Proof of Theorems \ref{a-thm1.6} and 
\ref{a-thm1.13}]
As we explained in Subsection \ref{a-subsec1.2}, 
we will prove Theorems \ref{thm-a}, 
\ref{thm-b}, \ref{thm-c}, \ref{thm-d}, \ref{thm-e}, 
\ref{thm-f}, and \ref{thm-g} by induction on $\dim X$. 
In Section \ref{a-sec14}, 
we established Theorem \ref{thm-d} in arbitrary dimension. 
We use induction on $n=\dim X$. 
When $n=0$, all the statements are trivially true. 
From now on, we assume that 
Theorems \ref{thm-a}, \ref{thm-b}, \ref{thm-c}, 
\ref{thm-e}, \ref{thm-f}, \ref{thm-g} hold true when 
$\dim X\leq n-1$. 
In Section \ref{a-sec15}, 
we proved Theorem \ref{thm-a} in 
$\dim X=n$. 
In Section \ref{a-sec16}, 
we obtained Theorem \ref{thm-b} in $\dim X=n$. 
Hence we can prove Theorem \ref{thm-g} 
in $\dim X=n$ by Section \ref{a-sec17}. 
Note that Theorem \ref{thm-c}$_n$ is a very 
special case of Theorem \ref{thm-g}$_n$ (see Lemma 
\ref{a-lem17.8}). 
By Sections \ref{a-sec18} and \ref{a-sec19}, 
we have Theorems \ref{thm-e} and \ref{thm-f} in 
$\dim X=n$, respectively. 
This means that we have established Theorems 
\ref{thm-a}, \ref{thm-b}, \ref{thm-c}, 
\ref{thm-d}, \ref{thm-e}, \ref{thm-f}, and \ref{thm-g} in arbitrary dimension. 
\end{proof}

From now on, we can freely use Theorems 
\ref{thm-a}, \ref{thm-b}, \ref{thm-c}, 
\ref{thm-d}, \ref{thm-e}, \ref{thm-f}, and \ref{thm-g} in 
arbitrary dimension. 
Therefore, we can freely use the minimal model 
program with scaling explained in Section \ref{a-sec13}. 

\begin{proof}[Proof of Theorem \ref{a-thm1.7}]
In this theorem, we only treat kawamata log pairs. 
Therefore, we do not need any extra assumptions on 
non-kawamata log terminal centers. 
We can freely use the minimal model program with scaling 
explained in Section \ref{a-sec13}. 
Note that the pseudo-effectivity over $Y$ is preserved 
by the minimal model program. 
Therefore, this theorem is a special case of the minimal model 
program with scaling under the condition (i) explained 
in Section \ref{a-sec13}. 
We also note that the termination 
of the minimal model program follows from Theorem \ref{thm-e} 
(see the proof of Theorem \ref{a-thm13.6}). 
\end{proof}

Theorem \ref{a-thm1.8} is a direct generalization of 
\cite[Theorem 1.2]{bchm} in the complex analytic setting. 

\begin{proof}[Proof of Theorem \ref{a-thm1.8}]
Throughout this proof, we will freely shrink $Y$ around 
$W$ suitably without mentioning it explicitly. 
By Theorems \ref{thm-d} and \ref{thm-c}, 
$(X, \Delta)$ has a log terminal model over $Y$. 
This is (1). 
By Lemma \ref{a-lem11.16} (2), $\phi$ is a semiample 
model over $Y$. 
By Lemma \ref{a-lem11.16} (3) and (4), 
we know that $(X, \Delta)$ has a log 
canonical model 
over $Y$ when $K_X+\Delta$ is $\pi$-big. 
This is (2). 
We take an arbitrary point $y\in Y$. 
It is sufficient to prove (3) over some open neighborhood 
of $y\in Y$. 
Hence we may take a Stein compact subset 
$W$ of $Y$ such that $y\in W$ and 
that $\Gamma(W, \mathcal O_Y)$ is noetherian, and 
shrink $Y$ and enlarge $W$ suitably without 
mentioning it explicitly 
(see Lemma \ref{a-lem2.16}). 
We take a positive integer $a$ such that 
$a(K_X+\Delta)$ and $a(K_Z+\Gamma)$ are 
both Cartier. 
Since $a(K_Z+\Gamma)$ is semiample 
over $Y$, 
\begin{equation*}
\bigoplus _{m\in \mathbb N} 
\pi_*\mathcal O_X(ma(K_X+\Delta))
\simeq 
\bigoplus _{m\in \mathbb N} 
(\pi_Z)_*\mathcal O_Z(ma(K_Z+\Gamma)), 
\end{equation*}
where $\pi_Z\colon Z\to Y$ is the structure morphism, 
is a locally finitely generated graded $\mathcal O_Y$-algebra 
(see Lemma \ref{a-lem2.36}). 
Therefore, by Lemma \ref{a-lem2.26} , $R(X/Y, K_X+\Delta)$ is a 
locally finitely generated graded $\mathcal O_Y$-algebra. 
This is (3). 
\end{proof}

The existence of kawamata log terminal flips is a direct 
consequence of Theorem \ref{a-thm1.8} (2). 

\begin{proof}[Proof of Theorem \ref{a-thm1.14}]
We take a point $z\in Z$ and consider a small 
Stein open neighborhood $U$ of $z\in Z$. 
Then $(X^+, \Delta^+)|_{(\varphi^+)^{-1}(U)}$ is nothing 
but the log canonical model of $(X, \Delta)|_{\varphi^{-1}(U)}$ 
over $U$. 
Therefore, after shrinking $U$ around $z$ suitably, 
it exists by Theorem \ref{a-thm1.8} (2) and 
is unique. 
Hence the desired flip $\varphi^+\colon 
(X^+, \Delta^+)\to Z$ exists globally.  
\end{proof}

The existence of canonicalizations for complex variety 
is new. 

\begin{proof}[Proof of Theorem \ref{a-thm1.16}]
We take a point $x\in X$. Over some open neighborhood 
$U$ of $x\in X$, there exist a projective bimeromorphic 
morphism $\pi\colon V\to U$ and a log canonical model 
$\pi'\colon V'\to U$ of $\pi\colon V\to U$ by 
Theorem \ref{a-thm1.8} (2). 
Note that $\pi'$ is projective bimeromorphic, 
$K_{V'}$ is $\pi'$-ample, and $V'$ has only 
canonical singularities. 
We can easily check that $\pi'$ is an isomorphism 
over a nonempty Zariski open subset where 
$U$ has only canonical singularities. 
We also note that $\pi'\colon V'\to U$ is usually 
called a {\em{canonical model}} 
of $V$ over $U$ and is unique. 
Thus, the desired model $f\colon Z\to X$ exists globally. 
\end{proof}

When $K_X+\Delta$ is not pseudo-effective, we see that 
we can always run a minimal model program and finally get 
a Mori fiber space. 

\begin{proof}[Proof of Theorem \ref{a-thm1.17}] 
As usual, we will repeatedly shrink $Y$ around $W$ without mentioning 
it explicitly. We take a $\pi$-ample $\mathbb Q$-divisor 
$C$ on $X$ such that $K_X+\Delta+C$ is nef over 
$Y$ and that 
$(X, \Delta+(1+a)C)$ is a divisorial log terminal pair for some 
positive real number $a$. 
We run a $(K_X+\Delta)$-minimal model 
program with scaling of $C$. 
Since $K_X+\Delta$ is not pseudo-effective, 
$K_X+\Delta+\varepsilon C$ is still not pseudo-effective for 
some $0<\varepsilon \ll 1$. 
We can see the above minimal model program 
as a $(K_X+\Delta+\varepsilon C)$-minimal 
model program with scaling of $C$. 
By Theorem \ref{a-thm13.6}, 
this minimal model program always terminates 
and then we finally get a Mori fiber space structure over $Y$. 
This is what we wanted. 
\end{proof}

We note that Theorems \ref{a-thm1.19} and \ref{a-thm1.21} for 
quasi-projective varieties are not treated in \cite{bchm}. 
We also note that a 
key ingredient of the proof of Theorems \ref{a-thm1.19} 
and \ref{a-thm1.21} is Lemma \ref{a-lem13.7}. 

\begin{proof}[Proof of Theorem \ref{a-thm1.19}]
Throughout this proof, we will freely shrink $Y$ around 
$W$ suitably without mentioning it explicitly. 
By Lemma \ref{a-lem2.53}, there exists a globally 
$\mathbb R$-Cartier $\mathbb R$-divisor 
$B$ on $X$ such that 
$K_X+\Delta\sim _{\mathbb R} B\geq 0$. 
Since $(K_X+\Delta)|_F\sim _{\mathbb R}0$, 
we see that $\pi(B)\subsetneq Y$ holds. 
Hence we can write $K_X+\Delta\sim _{\mathbb R} \pi^*D+B'$, 
where $D$ is an $\mathbb R$-Cartier $\mathbb R$-divisor 
on $Y$, 
$B'$ is an effective $\mathbb R$-Cartier $\mathbb R$-divisor 
on $X$ such that $\pi(B')\subsetneq Y$ and that $\Supp B'$ does 
not contain any fibers of $\pi$. 
Without loss of generality, we may assume that $\pi(B')\subsetneq W$ 
by shrinking $Y$ around $W$ suitably. 
We take a general $\pi$-ample $\mathbb Q$-divisor 
$C\geq 0$ on $X$ such that $(X, \Delta+C)$ is divisorial 
log terminal and that $K_X+\Delta+C$ is nef over $Y$. 
Then we run a $(K_X+\Delta)$-minimal model 
program with scaling of $C$ over $Y$ around $W$ starting 
from $X_0:=X$: 
\begin{equation*}
X_0\dashrightarrow X_1\dashrightarrow 
\cdots \dashrightarrow X_i \dashrightarrow \cdots. 
\end{equation*}
In this case, any divisorial contraction 
contracts an irreducible component of $\Supp B'$. 
By Lemma \ref{a-lem13.7}, 
we finally obtain $(X_m, \Delta_m)$ such that 
$K_{X_m}+\Delta_m\in \Mov(X_m/Y; W)$. 
By Zariski's lemma, we can check that 
$K_{X_m}+\Delta_m\sim _{\mathbb R} (\pi_m)^*D$ holds. 
This is what we wanted. 
\end{proof}

In the proof of Theorem \ref{a-thm1.21}, Lemma \ref{a-lem4.6} 
plays an important role. 

\begin{proof}[Proof of Theorem \ref{a-thm1.21}]
We will freely shrink $X$ suitably without 
mentioning it explicitly. 
By taking a resolution of singularities, 
we have a projective bimeromorphic 
morphism $\pi\colon Y\to X$ from a 
complex variety $Y$ such that $\pi^{-1}(U)$ 
is smooth and $\Exc(\pi)$ and $\Exc(\pi)\cup \Supp \pi^{-1}_*\Delta$ 
are simple normal crossing divisors on $\pi^{-1}(U)$. 
Let $E$ be any $\pi$-exceptional divisor 
such that 
$\pi(E)\cap U\ne \emptyset$. 
Then, by enlarging $V$ suitably, we may assume that 
$\pi(E)\cap 
V\ne \emptyset$. 
By Lemma \ref{a-lem2.16}, 
we can take a Stein compact subset $W$ of $U$ such that 
$\Gamma (W, \mathcal O_X)$ is noetherian and 
that $V\subset W$. 
We write $K_Y+\Delta_Y=\pi^*(K_X+\Delta)$. 
Let $\Delta_Y=\sum _i a_i \Delta_i$ be the irreducible decomposition. 
We put 
\begin{equation*}
\Theta =\sum _{0<a_i <1} a_i \Delta_i +\sum 
_{a_i\geq 1}\Delta_i. 
\end{equation*} 
Then we have $K_Y+\Theta =\pi^*(K_X+\Delta)+F$ 
such that $-\pi_*F$ is effective. 
Let $C$ be a general $\pi$-ample 
$\mathbb Q$-divisor on $Y$ such that 
$(Y, \Theta+C)$ is divisorial log terminal and 
$K_Y+\Theta +C$ is nef over $W$. 
We run a $(K_Y+\Theta)$-minimal model program 
with scaling of $C$ over $X$ around $W$. 
We put $(Y_0, \Theta_0):=(Y, \Theta)$, 
$F_0:=F$, and $C_0:=C$. 
Then we obtain a sequence of flips and divisorial 
contractions starting from $(Y_0, \Theta_0)$: 
\begin{equation*}
(Y_0, \Theta_0)\overset{\phi_0}{\dashrightarrow} 
(Y_1, \Theta_1)\overset{\phi_1}{\dashrightarrow} 
\cdots 
\overset{\phi_{i-1}}{\dashrightarrow} 
(Y_i, \Theta_i)\overset{\phi_i}{\dashrightarrow}, 
\end{equation*}
where $\Theta_{i+1}:=(\phi_i)_*\Theta_i$, $C_{i+1}:=(\phi_i)_*C_i$, 
and $F_{i+1}:=(\phi_i)_*F_i$, for every $i$, and 
a sequence of real numbers 
\begin{equation*}
1\geq \lambda_0\geq \lambda _1\geq \cdots 
\geq \lambda_i \geq \cdots \geq 0
\end{equation*}
such that $K_{Y_i}+\Theta_i +\lambda_i C_i$ is 
nef over $W$. 
By Lemma \ref{a-lem13.7}, 
we can prove that $K_{Y_m}+\Theta_m$ is in 
$\Mov(Y_m/X; W)$ for some 
$m$. 
By the negativity lemma (see Lemma \ref{a-lem4.6}), 
we see that $-F_m\geq 0$ over $V$. 
Hence, $-F_m$ is effective over some open neighborhood 
of $W$. 
We put $Z:=Y_m$, $f\colon Z\to X$, 
and $K_Z+\Delta_Z=f^*(K_X+\Delta)$. 
Then, $(Z, \Delta_Z)$ has all the desired properties. 
\end{proof}

We have already proved Theorem \ref{a-thm1.24} 
in Section \ref{a-sec12}. 

\begin{proof}[Proof of Theorem \ref{a-thm1.24}] 
We have already known that Theorem \ref{thm-g}$_n$ holds true for 
every $n$. Therefore, 
Theorem \ref{a-thm1.24} is nothing but Lemma \ref{a-lem12.1}. 
\end{proof}

By Theorem \ref{a-thm1.24}, 
Corollary \ref{a-cor1.25} is almost obvious. 

\begin{proof}[Proof of Corollary \ref{a-cor1.25}]
We note that $Z$ has only kawamata log terminal singularities 
over some open neighborhood of $W$. 
We apply Theorem \ref{a-thm1.24} to 
$Z\to Y$ and $W'$. Then, after shrinking $Y$ around $W'$ suitably, 
there exists a small projective bimeromorphic contraction morphism 
$Z'\to Z$ such that $Z'$ is projective over $Y$ and is $\mathbb Q$-factorial 
over $W'$. 
Hence the induced bimeromorphic contraction $\phi'\colon X\dashrightarrow 
Z'$ satisfies 
the desired properties.  
\end{proof}

The argument in the proof of Theorem \ref{a-thm1.26} is more 
or less well known. 

\begin{proof}[Proof of Theorem \ref{a-thm1.26}]
We take an arbitrary point $x\in X$. 
It is sufficient to prove that $\bigoplus _{m\in \mathbb N} \mathcal 
O_X(mD)$ is a finitely generated graded $\mathcal O_X$-algebra 
on some open neighborhood of $x$. 
By shrinking $X$ around $x$ and replacing $D$ with a linearly 
equivalent integral Weil divisor, we may assume that 
$D$ is effective. 
We take a relatively compact Stein open 
neighborhood $U$ of $x$ and a Stein compact 
subset $W$ of $X$ 
such that $U\subset W$ and that 
$\Gamma (W, \mathcal O_X)$ is noetherian. 
By Theorem \ref{a-thm1.24}, after shrinking $X$ around $W$, 
there exists a small projective bimeromorphic 
morphism $f\colon Z\to X$ from a normal complex variety $Z$ 
such that $Z$ is $\mathbb Q$-factorial over $W$. 
We put $K_Z+\Delta_Z=f^*(K_X+\Delta)$. 
Then $(Z, \Delta_Z)$ is kawamata log terminal. 
Let $D_Z$ be the strict transform of $D$ on $Z$. 
By shrinking $X$ around $W$, 
we may assume that $D_Z$ is $\mathbb Q$-Cartier. 
We take a small rational number $\varepsilon$ 
such that $(Z, \Delta_Z+\varepsilon D_Z)$ is still 
kawamata log terminal. From now on, 
we will freely shrink $X$ around $W$ without mentioning 
it explicitly. We take a general $f$-ample 
$\mathbb Q$-divisor 
$H$ on $Z$ such that 
$K_Z+\Delta_Z+\varepsilon D_Z+H$ is nef over $W$ and 
$(Z, \Delta_Z+
\varepsilon D_Z+H)$ is kawamata log terminal. 
We run a $(K_Z+\Delta_Z+\varepsilon D_Z)$-minimal 
model program with scaling of $H$ over $X$ around $W$. 
Then we get a finite sequence of flips 
starting from $(Z_0, \Delta_{Z_0}+\varepsilon D_{Z_0}):=
(Z, \Delta_Z+\varepsilon D_Z)$: 
\begin{equation*}
Z_0\overset{\phi_0}{\dashrightarrow} 
Z_1\overset{\phi_1}{\dashrightarrow} 
\cdots 
\overset{\phi_{i-1}}{\dashrightarrow} 
Z_i\overset{\phi_i}{\dashrightarrow}\cdots 
\overset{\phi_{m-1}}{\dashrightarrow} Z_m, 
\end{equation*} 
such that $\Delta_{Z_i}:=(\phi_{i-1})_*\Delta_{Z_{i-1}}$ and 
$D_{Z_i}:=(\phi_{i-1})_*D_{Z_{i-1}}$ for every $i\geq 1$ 
and that $K_{Z_m}+\Delta_{Z_m}+\varepsilon D_{Z_m}$ is nef over 
$W$. 
Note that $K_{Z_m}+\Delta_{Z_m}=f^*_m(K_X+\Delta)$ holds 
by construction, where $f_m\colon Z_m\to X$ is the 
structure morphism. 
Since $f_m$ is bimeromorphic, 
we can take an effective $\mathbb Q$-divisor $B$ on $Z_m$ such that 
$-B$ is $f_m$-ample and 
that $(Z_m, \Delta_{Z_m}+B)$ is kawamata log terminal. 
Hence, by Theorem \ref{a-thm6.5}, 
$D_{Z_m}$ is semiample over $X$. 
By considering a contraction morphism $Z_m\to Z'$ over 
$X$ associated to $D_{Z_m}$, 
we obtain a small projective bimeromorphic 
contraction morphism $f'\colon Z'\to X$ from a normal variety $Z'$ and an 
integral Weil divisor $D'$ on $Z'$ such that 
$D'$ is ample over $X$ and that $f'_*D'=D$ holds. 
Since $f'$ is small, 
we obtain $f'_*\mathcal O_{Z'}(mD')=\mathcal O_X(mD)$ 
for every $m\in \mathbb N$. 
Since $D'$ is ample over $X$, 
$\bigoplus _{m\in \mathbb N} f'_*\mathcal O_{Z'}(mD')$ is a locally 
finitely generated graded $\mathcal O_X$-algebra by Lemma 
\ref{a-lem2.36}. 
This means that $\bigoplus _{m\in \mathbb N} \mathcal 
O_X(mD)$ is 
a finitely generated graded $\mathcal O_X$-algebra 
on some open neighborhood of $x$. 
This is what we wanted. 
\end{proof}

Now there are no difficulties to prove Theorem \ref{a-thm1.27}. 

\begin{proof}[Proof of Theorem \ref{a-thm1.27}]
We take an open neighborhood $U$ of $W$ and a Stein compact subset $W'$ 
of $Y$ such that $U\subset W'$ and that $\Gamma (W', \mathcal O_Y)$ 
is noetherian. Throughout this proof, 
we will freely shrink $Y$ suitably without mentioning it explicitly. 
Let $A$ be a general $\pi$-ample 
$\mathbb Q$-divisor on $X$ satisfying that 
$A\cdot C>2\dim X$ for every projective curve $C$ on $X$ such that 
$\pi(C)$ is a point. 
We take a resolution $g\colon X'\to X$ such that 
$\Supp g^{-1}_*\Delta\cup \Exc(g)$ and 
$\Exc(g)$ are simple normal crossing divisors on $X'$ and that 
$\pi'\colon X'\to Y$ is projective. 
We write $K_{X'}+\Delta_{X'}=g^*(K_X+\Delta)$. 
Let $\Delta_{X'}=\sum_i a_i \Delta'_i$ be the irreducible 
decomposition. We put 
\begin{equation*}
\Theta =\sum _{0<a_i<1} a_i \Delta'_i +\sum _{a_i\geq 1} \Delta'_i. 
\end{equation*} 
Then we can write $K_{X'}+\Theta =g^*(K_X+\Delta)+F$ such that 
$-g_*F\geq 0$. 
We take a general $\pi'$-ample $\mathbb Q$-divisor 
$H$ on $X'$ such that $K_{X'}+\Theta+g^*A+H$ 
is nef over $Y$. 
We run a $(K_{X'}+\Theta+g^*A)$-minimal model 
program over $Y$ around $W'$ with scaling of $H$. 
Then we obtain a sequence of flips and divisorial 
contractions starting from $(X', \Theta_0):=(X', \Theta)$: 
\begin{equation*}
(X'_0, \Theta_0)\overset{\phi_0}{\dashrightarrow} 
(X'_1, \Theta_1)\overset{\phi_1}{\dashrightarrow} 
\cdots 
\overset{\phi_{i-1}}{\dashrightarrow} 
(X'_i, \Theta_i)\overset{\phi_i}{\dashrightarrow}, 
\end{equation*}
where $\Theta_{i+1}:=(\phi_i)_*\Theta_i$, $H_{i+1}:=(\phi_i)_*H_i$, 
and $F_{i+1}:=(\phi_i)_*F_i$, for every $i$, and 
a sequence of real numbers 
\begin{equation*}
1\geq \lambda_0\geq \lambda _1\geq \cdots 
\geq \lambda_i \geq \cdots \geq 0
\end{equation*}
such that $K_{X'_i}+\Theta_i +g^*_iA+\lambda_i H_i$ is 
nef over $W'$, where $g_i\colon X'_i\to X$. 
We note that by Lemma \ref{a-lem9.4} 
the above minimal model program can be seen as a 
$(K_{X'}+\Theta)$-minimal model 
program over $X$. 
By Lemma \ref{a-lem13.7} and its proof, 
we can check that 
$K_{X'_m}+\Theta_m$ is in 
$\Mov (X'_m/X; \pi^{-1}(W'))$ for some $m$. 
By applying the negativity lemma 
(see Lemma \ref{a-lem4.6}) to 
$g_m\colon X'_m \to X$, we can check 
that $-F_m$ is effective on $(\pi\circ g_m)^{-1}(U)$. 
If $Y_m$ is not $\mathbb Q$-factorial over $W$, 
then we replace $Y_m$ with its small projective 
$\mathbb Q$-factorialization 
by Theorem \ref{a-thm1.24}. 
Hence we obtain a desired $f\colon Z\to X$. 
\end{proof}

\section{A canonical bundle formula in the complex 
analytic setting}\label{a-sec21}

In this section, we will quickly discuss a canonical bundle formula 
in the complex analytic setting and prove 
Theorems \ref{a-thm1.18} and \ref{a-thm1.22}. 
We need some deep results from the theory of variations of 
Hodge structure. 

Let us start with a generalization of Koll\'ar's famous torsion-freeness. 

\begin{thm}[{Torsion-freeness, see \cite{takegoshi}}]\label{a-thm21.1}
Let $\pi\colon X\to Y$ be a proper morphism 
from a K\"ahler manifold $X$ onto 
a complex analytic space $Y$. 
Then $R^i\pi_*\omega_X$ is torsion-free 
for every $i$. 
\end{thm}

If $Y$ is projective in Theorem \ref{a-thm21.1}, then 
$X$ is a compact K\"ahler manifold.  
In this case, there are no 
difficulties to prove Theorem \ref{a-thm21.1}. 
Unfortunately, however, we have to treat the case where $Y$ is 
a general noncompact complex analytic space. 
Hence the proof of Theorem \ref{a-thm21.1} is much harder than 
that of Koll\'ar's original torsion-freeness. 
For the details, see \cite{takegoshi} 
(see also \cite{fujino-transcendental}, 
and \cite{matsumura}). 
By combining Steenbrink's geometric description of 
canonical extensions of Hodge filtrations 
(see \cite{steenbrink1} and 
\cite{steenbrink2}) with Theorem \ref{a-thm21.1}, 
we have: 

\begin{thm}[{Hodge filtrations, 
see \cite[Chapter V.~3.7.~Theorem (4)]{nakayama3}}]\label{a-thm21.2} 
Let $\pi\colon X\to Y$ be a proper morphism from a K\"ahler manifold 
$X$ onto a smooth variety $Y$. 
Assume that there exists a simple normal crossing divisor 
$\Sigma_Y$ on $Y$ such that $\pi$ is smooth over $Y\setminus \Sigma_Y$. 
Then $R^i\pi_*\omega_{X/Y}$ is characterized as the 
upper canonical extension of the bottom Hodge filtration for every $i$. 
\end{thm}

The proof of \cite[Theorem 1]{nakayama-hodge} 
works in the above complex analytic setting once we 
know the torsion-freeness (see Theorem \ref{a-thm21.1}). 
For the details of Nakayama's argument, 
we recommend the interested reader to see \cite[Subsection 3.1]
{fujino-higher} and \cite[\S 7]{fujino-fujisawa} although 
they treat much more general settings than Nakayama's. 

In order to discuss a canonical bundle formula, we recall the definition of 
{\em{discriminant $\mathbb Q$-divisors}}. 

\begin{defn}\label{a-def21.3}
Let $f\colon X\to Y$ be a proper 
surjective morphism from a normal variety $X$ 
onto a smooth variety $Z$ with $f_*\mathcal O_X\simeq 
\mathcal O_Z$. Let $\Theta$ be a $\mathbb Q$-divisor 
on $X$ such that $K_X+\Theta$ is $\mathbb Q$-Cartier and 
that $(X, \Theta)$ is sub kawamata log terminal over a nonempty 
Zariski open subset of $Z$. 
Let $P$ be a prime divisor on $Z$. 
We put 
\begin{equation*}
b_P:=\max \left\{t\in \mathbb Q\, |\, 
{\text{$(X, \Theta +tf^*P)$ is sub log canonical 
over general points of $P$}}
\right\}. 
\end{equation*} 
Then we set $B_Z:=\sum _P (1-b_P)P$, where 
$P$ runs over prime divisors on $Z$, and 
call $B_Z$ the {\em{discriminant $\mathbb Q$-divisor}} of 
$f\colon (X, \Theta)\to Z$. 
We can easily check that $B_Z$ is a well-defined 
$\mathbb Q$-divisor on $Z$ satisfying $\lfloor B_Z\rfloor\leq 0$. 
Let $\mu\colon Z'\to Z$ be a projective 
bimeromorphic morphism from 
a smooth variety $Z'$. 
We consider the following commutative 
diagram: 
\begin{equation*}
\xymatrix{
X \ar[d]_-f& X' \ar[d]^-{f'}\ar[l]_-\sigma\\
Z & \ar[l]^-\mu Z', 
}
\end{equation*}
where $X'$ is the normalization 
of the main component of $X \times _Z Z'$. 
We define $\Theta'$ by $K_{X'}+\Theta'=\sigma^*(K_X+\Theta)$. 
We call $f'\colon (X', \Theta')\to Z'$ the {\em{induced fibration}} of 
$f\colon (X, \Theta)\to Z$ by $\mu\colon Z'\to Z$. 
We can define the discriminant 
$\mathbb Q$-divisor $B_{Z'}$ of $f'\colon (X', \Theta')\to Z'$. 
By construction, we see that 
$\sigma_*B_{Z'}=B_Z$. 
\end{defn}

The following theorem is a generalization of 
Ambro's result in the complex analytic setting 
(see \cite[Theorem 0.2]{ambro}). 

\begin{thm}\label{a-thm21.4}
Let $f\colon X\to Z$ be a proper morphism 
from a K\"ahler manifold $X$ 
onto a smooth complex variety $Z$ with 
$f_*\mathcal O_X\simeq \mathcal O_Z$. 
Let $g\colon Z\to Y$ be a projective morphism 
to a Stein space $Y$. 
Let $\Theta$ be a $\mathbb Q$-divisor on $X$ 
such that $K_X+\Theta\sim _{\mathbb Q} f^*D$ for 
some $\mathbb Q$-Cartier $\mathbb Q$-divisor $D$ on $Z$, 
$\Supp \Theta$ is a simple normal crossing divisor 
on $X$, $\Theta=\Theta^{<1}$ holds over 
general points of $Z$, and $\rank f_*\mathcal O_X(\lceil 
-\Theta^{<1}\rceil)=1$. 
Let $y$ be any point of $Y$. 
By replacing 
$Y$ with a relatively compact Stein open 
neighborhood 
of $y$ suitably, we have a commutative diagram: 
\begin{equation*}
\xymatrix{
X \ar[d]_-f& X'\ar[d]^-{f'}\ar[l]_-\sigma \\ 
Z & Z'\ar[l]^-\mu
}
\end{equation*}
with the following properties. 
\begin{itemize}
\item[(1)] $\mu \colon Z'\to Z$ is a projective 
bimeromorphic morphism from a smooth variety 
$Z'$. 
\item[(2)] $X'$ is a desingularization of $X\times _Z Z'$ 
such that $X'$ is K\"ahler with $K_{X'}+\Theta '=\sigma^*(K_X+\Delta)$. 
\item[(3)] Let $B_{Z'}$ be the discriminant $\mathbb Q$-divisor 
of $f'\colon (X', \Theta')\to Z'$. 
We write $\sigma^*D=K_{Z'}+B_{Z'}+M_{Z'}$. 
Then $M_{Z'}$ is nef over $Y$. Note that 
$M_{Z'}$ is usually called the {\em{moduli $\mathbb Q$-divisor}} 
of $f'\colon (X', \Theta')\to Z'$. 
\item[(4)] Let $\nu\colon Z''\to Z'$ be any 
projective bimeromorphic morphism 
from a smooth variety $Z''$. 
Then we can define $f''\colon (X'', \Theta'')\to Z''$, 
$B_{Z''}$, and $M_{Z''}$ as in (3) with 
$\nu^*\mu^*D=K_{Z''}+B_{Z''}+M_{Z''}$. 
In this situation, after shrinking $Y$ with a slightly smaller relatively 
compact Stein open neighborhood of $y$ again, 
$\nu^*M_{Z'}=M_{Z''}$ holds with 
$\nu_*K_{Z''}=K_{Z'}$ and $\nu_*B_{Z''}=B_{Z'}$. 
\end{itemize} 
\end{thm}

\begin{proof}
For the details, see \cite[Section 5]{ambro}. 
Although they treat much more general setting than 
Ambro's, \cite{fujino-slc-trivial} and 
\cite{fujino-hashizume} may help the reader understand 
the proof of this theorem. 
We note that Ambro's 
argument in \cite{ambro} is different from 
Kawamata's in \cite{kawamata-subadjunction} and 
is closer to Mori's (see \cite[Section 5, Part II]{mori} and 
\cite[Section 4]{fujino-certain}). 
\end{proof}

\begin{say}[A canonical bundle formula 
in the complex analytic setting, see \cite{fujino-mori}]
\label{a-say21.5}
Let $f\colon X\to Z$ be a proper morphism 
from a K\"ahler manifold $X$ onto a smooth 
variety $Z$ with $f_*\mathcal O_X\simeq \mathcal O_Z$ and 
let $g\colon Z\to Y$ be a projective 
morphism such that $Y$ is Stein. 
Let $\Delta$ be an effective $\mathbb Q$-divisor 
on $X$ such that 
$\Supp \Delta$ is a simple normal crossing divisor on $X$ and 
that $\lfloor \Delta\rfloor=0$. 
Suppose that 
$\kappa (F, (K_X+\Delta)|_F)=0$ holds 
for an analytically sufficiently general fiber $F$ 
of $f\colon X\to Z$. 
We fix an arbitrary point $y\in Y$. 
From now on, we sometimes replace $Y$ with 
a smaller relatively compact Stein open neighborhood 
of $y$ without mentioning it explicitly. 
Since $\kappa (F, (K_X+\Delta)|_F)=0$, 
we obtain $g_*\left(f_*\mathcal O_X(m(K_X+\Delta))\otimes 
\mathcal A\right)\ne 0$ for some 
divisible positive integer $m$ and some 
$g$-ample line bundle $\mathcal A$ on $Z$. 
Hence we can 
write 
$K_X+\Delta\sim_{\mathbb Q} f^*D+B$ such that 
\begin{itemize}
\item[(a)] $D$ is a $\mathbb Q$-Cartier 
$\mathbb Q$-divisor 
on $Z$, 
\item[(b)] $f_*\mathcal O_X(\lfloor iB_+\rfloor) 
\simeq \mathcal O_Z$ holds for every $i\geq 0$, 
and 
\item[(c)] $\codim _Zf(\Supp B_-)\geq 2$. 
\end{itemize} 
We take a projective bimeromorphic 
morphism 
$\psi\colon X^\dag\to X$ from a smooth 
variety $X^\dag$ such that 
$\Exc(\psi)\cup \Supp \psi^{-1}_*\Delta
\cup \Supp \psi^{-1}_*B$ is contained 
in a simple normal crossing divisor on $X^\dag$. 
We put 
$K_{X^\dag} +\Delta^\dag=\psi^*(K_X+\Delta)$ and 
consider $K_{X^\dag}+\Delta^\dag_+
\sim _{\mathbb Q} \psi^*f^*D+\psi^*B+\Delta^\dag_-$. 
By replacing $f\colon X\to Z$ with $f\circ \psi \colon X^\dag 
\to Z$, we may further assume that 
the support of $\Theta:=\Delta-B$ is a simple normal 
crossing divisor on $Z$. 
We apply Theorem \ref{a-thm21.4} to 
$f\colon (X, \Theta)\to Z$. 
Then we have a projective bimeromorphic morphism 
$\mu\colon Z'\to Z$ satisfying the properties 
in Theorem \ref{a-thm21.4}. 
By Hironaka's flattening theorem (see \cite{hironaka}), 
we can take a projective bimeromorphic 
morphism 
$p\colon Z_1\to Z$ from a smooth variety 
such that the main component of $X\times _Z Z_1$ is flat 
over $Z_1$. 
We may further assume that $p\colon Z_1\to Z$ factors 
through $Z'$. 
Then we consider the following 
big commutative diagram: 
\begin{equation*}
\xymatrix{
& X\ar[d]_-f & \ar[l]_-\alpha X_1\ar[d]_-{f_1}& 
X_2\ar[d]_-{f_2}\ar[l]_-\beta & X''\ar[l]_-\gamma 
\ar[d]^-{f''}\\ 
&Z\ar[dl]_-g &Z_1\ar[d]\ar[l]^-p\ar@{=}[r]& Z_2
\ar[dl]\ar@{=}[r]& Z''\ar[dll]^-\nu\\ 
Y && Z'\ar[lu]^-\mu &&
}
\end{equation*}
where $X_1$ is the main component of 
$X\times _Z Z_1$, 
$X_2$ is the normalization of $X_1$, and $\gamma\colon X''\to X_2$ is a 
proper bimeromorphic morphism from a smooth 
variety $X''$. 
We put $h:=\alpha\circ \beta\circ \gamma\colon X''\to X$ 
and $K_{X''}+\Delta''=h^*(K_X+\Delta)$. 
We may assume that 
there exist simple normal crossing divisors 
$\Sigma_{X''}$ and $\Sigma_{Z''}$ on $X''$ and $Z''$, respectively, 
such that $f''\colon X''\to Z''$ is smooth over $Z''\setminus \Sigma_{Z''}$, 
$\Sigma_{X''}$ is relatively simple normal crossing over 
$Z''\setminus \Sigma_{Z''}$, $(f'')^{-1}\Sigma_{Z''}\subset 
\Sigma_{X''}$, and $\Supp \Delta''\cup \Supp h^*B$ is contained 
in $\Sigma _{X''}$. 
Since $K_X+\Delta\sim _{\mathbb Q} f^*D+B$, 
we obtain $K_{X''}+\Delta''\sim _{\mathbb Q} 
(f'')^*p^*D+h^*B$. 
We can write 
\begin{equation*}
\Delta''_-+(f'')^*p^*D+h^*B=
(f'')^*D''+B''
\end{equation*} 
such that 
$\codim _{Z''}f''(\Supp B''_-)\geq 2$ and that 
$f''_*\mathcal O_{X''}(\lfloor iB''_+\rfloor)\simeq 
\mathcal O_{Z''}$ for every $i\geq 0$, 
where $B''=B''_+-B''_-$ as usual. 
Hence, we can write 
\begin{equation*}
K_{X''}+\Delta''_+\sim _{\mathbb Q} (f'')^*(K_{Z''}+B_{Z''}+M_{Z''})
+B''. 
\end{equation*}
By construction, we can check that 
\begin{itemize}
\item[(d)] $M_{Z''}=\mu^*M_{Z'}$ is nef over 
$Y$, and 
\item[(e)] $\Supp B_{Z''} \subset \Sigma_{Z''}$, 
$B_{Z''}$ is effective, and $\lfloor B_{Z''}\rfloor =0$. 
\end{itemize} 
We put $\pi:=g\circ f\colon X\to Y$ and $\pi''\colon 
X''\to Y$. Let $k$ be a divisible positive integer. 
Then \begin{equation*}
\begin{split}
\pi_*\mathcal O_X(k(K_X+\Delta))&
\simeq \pi''_*\mathcal O_{X''}(k(K_{X''}+\Delta''))\\ 
&\simeq \pi''_*\mathcal O_{X''}(k(K_{X''}+\Delta''_++B''_-))
\\ 
&\simeq \pi''_*\mathcal O_{X''}((f'')^*(k(K_{Z''}+B_{Z''}+M_{Z''}))+kB''_+)\\
&\simeq g''_*\mathcal O_{Z''}(k(K_{Z''}+B_{Z''}+M_{Z''})), 
\end{split}
\end{equation*}
where $g''\colon Z''\to Y$. 
Here, we used the fact that 
$\Delta''_-+B''_-$ is effective 
and $h$-exceptional. 
\end{say}

\begin{rem}\label{a-rem21.6}
In \cite{fujino-mori}, we used Kawamata's positivity 
theorem (see \cite[Theorem 2]{kawamata-subadjunction}) 
to prove the nefness of the moduli part $M_{Z'}$. 
In Theorem \ref{a-thm21.4}, we adopted 
Ambro's formulation of klt-trivial fibrations 
(see Theorem \ref{a-thm21.4} and \cite{ambro}) instead of 
\cite[Theorem 2]{kawamata-subadjunction}. 
\end{rem}

Let us go to the proof of Theorems \ref{a-thm1.18} and 
\ref{a-thm1.22}. 

\begin{proof}[Proof of Theorems \ref{a-thm1.18} and \ref{a-thm1.22}]
Let $y$ be any point of $Y$. Throughout this proof, 
we will feely replace $Y$ with a relatively compact 
Stein open neighborhood of $y$. 
In Theorem \ref{a-thm1.18}, 
by taking a resolution of singularities, we may assume that 
$X$ is smooth and $\Supp \Delta$ is a simple 
normal crossing divisor on $X$. 
Let $f\colon X\dashrightarrow Z$ be the 
Iitaka fibration with respect to $K_X+\Delta$ over $Y$. 
By replacing $X$ and $Z$, 
we may assume that $Z$ is a smooth variety and is projective over $Y$ and 
that $f$ is a morphism with $f_*\mathcal O_X\simeq \mathcal O_Z$. 
We use the canonical bundle formula discussed in \ref{a-say21.5}. 
Then, by Lemma \ref{a-lem2.26}, it is sufficient to 
prove that 
\begin{equation*}
\bigoplus _{m\in \mathbb N}g''_*
\mathcal O_{Z''}(mk(K_{Z''}+B_{Z''}+M_{Z''}))
\end{equation*} 
is a locally finitely generated graded $\mathcal O_Y$-algebra. 
By construction, $K_{Z''}+B_{Z''}+M_{Z''}$ is big over $Y$. 
We can find $\Delta_{Z''}$ such that 
$(Z'', \Delta_{Z''})$ is kawamata log terminal and 
that 
\begin{equation*}
a(K_{Z''}+B_{Z''}+M_{Z''})\sim b(K_{Z''}+\Delta_{Z''})
\end{equation*}
for some positive integers $a$ and $b$. 
Hence, by Lemma \ref{a-lem2.26} again, 
it is sufficient fo prove that 
\begin{equation*}
\bigoplus _{m\in \mathbb N} g''_*\mathcal O_{Z''}
(\lfloor m(K_{Z''}+\Delta_{Z''})\rfloor)
\end{equation*}
is a locally finitely generated graded $\mathcal O_Y$-algebra. 
Since $K_{Z''}+\Delta_{Z''}$ is big over $Y$, 
it follows from Theorem \ref{a-thm1.8} (3). 
Therefore, we get the desired result. 
\end{proof}

\section{Minimal model program with scaling revisited}\label{a-sec22}

In this section, we will discuss the minimal 
model program with scaling again for future usage. 
The original results for 
algebraic varieties are not covered by \cite{bchm}. 
Here, we will closely follow the presentation of 
\cite{birkar} and \cite{birkar2}. 

Let us recall the definition of {\em{extremal curves}}. 
\begin{defn}[Extremal curves]\label{a-def22.1} 
Let $\pi\colon X\to Y$ 
be a projective morphism of complex analytic spaces and let $W$ be a compact 
subset of $Y$ such that 
$\pi\colon X\to Y$ and $W$ satisfies (P). 
A curve $\Gamma$ on $X$ is called {\em{extremal over $W$}} 
if the following properties hold. 
\begin{itemize}
\item[(i)] $\Gamma$ generates an extremal ray $R$ of $\NE(X/Y; W)$. 
\item[(ii)] There exists a $\pi$-ample 
Cartier divisor $H$ on $X$ such that 
\begin{equation*}
H\cdot \Gamma=\min \{H\cdot \ell\}, 
\end{equation*} 
where $\ell$ ranges over curves generating $R$. 
\end{itemize}
\end{defn}

The following theorem is very useful 
when we run the minimal model program with scaling. 

\begin{thm}[{see \cite[Theorem 4.7.2]{fujino-foundations}}]
\label{a-thm22.2} 
Let $\pi\colon X\to Y$ 
be a projective morphism of complex analytic spaces and let $W$ be a compact 
subset of $Y$ such that 
$\pi\colon X\to Y$ and $W$ satisfies {\em{(P)}}. 
Let $V$ be a finite-dimensional 
affine subspace of $\WDiv_{\mathbb R}(X)$, which 
is defined over the rationals. Assume that 
there is an $\mathbb R$-divisor $\Delta_0$ on $X$ such that 
$(X, \Delta_0)$ is kawamata log terminal. 
We fix an $\mathbb R$-divisor $\Delta\in \mathcal L(V; \pi^{-1}(W))$. 
Then we can find positive real numbers $\alpha$ and $\delta$, which 
depend on $(X, \Delta)$ and $V$, with the following properties. 
\begin{itemize}
\item[(1)] If $\Gamma$ is any extremal curve over $W$ and 
$(K_X+\Delta)\cdot \Gamma>0$, 
then $(K_X+\Delta)\cdot \Gamma >\alpha$. 
\item[(2)] If $D\in \mathcal L(V; \pi^{-1}(W))$, 
$|\!| D-\Delta|\!|<\delta$, and $(K_X+D)\cdot \Gamma \leq 0$ for 
an extremal curve $\Gamma$ over $W$, 
then $(K_X+\Delta)\cdot \Gamma\leq 0$. 
\item[(3)] Let $\{R_t\}_{t\in T}$ be any set of extremal rays of $\NE(X/Y; W)$. 
Then 
\begin{equation*}
\mathcal N_T:=\{D\in \mathcal L(V; \pi^{-1}(W))\, |\, 
{\text{$(K_X+D)\cdot R_t\geq 0$ for every $t\in T$}}\} 
\end{equation*} 
is a rational polytope in $V$. 
In particular, 
\begin{equation*}
\mathcal N^\sharp _{\pi} (V; W)=\{\Delta\in 
\mathcal L(V; \pi^{-1}(W)) \, |\, 
{\text{$K_X+\Delta$ is nef over $W$}}\}
\end{equation*}
is a rational polytope. 
\end{itemize}
\end{thm}

\begin{proof}
This theorem is a formal consequence of Theorem \ref{a-thm9.2} 
and Theorem \ref{a-thm7.3}. More precisely, 
(1) easily follows from Theorem \ref{a-thm9.2}. 
We can check that (2) holds true by using (1). 
By (2) and Theorem \ref{a-thm7.3}, we can prove (3). 
For the details, see, for example, 
the proof of \cite[Theorem 4.7.2]{fujino-foundations}. 
\end{proof}

By Theorem \ref{a-thm22.2} (3) and Theorem \ref{a-thm9.2}, 
we can prove: 

\begin{thm}[{see \cite[Theorem 4.7.3]{fujino-foundations}}]\label{a-thm22.3} 
Let $\pi\colon X\to Y$ 
be a projective morphism of complex analytic spaces and let $W$ be a compact 
subset of $Y$ such that 
$\pi\colon X\to Y$ and $W$ satisfies {\em{(P)}}. 
Let $(X, \Delta)$ be a log canonical pair and let $H$ be an effective 
$\mathbb R$-Cartier $\mathbb R$-divisor on $X$ such that 
$(X, \Delta+H)$ is log canonical and that $K_X+\Delta+H$ is nef over 
$W$. 
Assume that there exists $\Delta_0$ such that 
$(X, \Delta_0)$ is kawamata log terminal. 
Then, either $K_X+\Delta$ is nef over $W$ or there 
is a $(K_X+\Delta)$-negative extremal ray $R$ of $\NE(X/Y; W)$ such that 
$(K_X+\Delta+\lambda H)\cdot R=0$, 
where 
\begin{equation*}
\lambda:=\inf \{t\in \mathbb R_{\geq 0}
\, |\, {\text{$K_X+\Delta+tH$ is nef over $W$}}\}. 
\end{equation*} 
Of course, $K_X+\Delta+\lambda H$ is nef over $W$. 
\end{thm}
\begin{proof}
The proof of \cite[Theorem 4.7.3]{fujino-foundations} works without 
any modifications. 
\end{proof}

By Theorems \ref{a-thm22.2} and 
\ref{a-thm22.3}, the minimal model 
program with scaling explained in Section \ref{a-sec13} becomes 
much more useful. 

\begin{say}[Minimal model program with scaling]\label{a-say22.4}
Let $\pi\colon X\to Y$ 
be a projective morphism of complex analytic spaces and let 
$W$ be a compact 
subset of $Y$ such that 
$\pi\colon X\to Y$ and $W$ satisfies (P). 
Let $(X, \Delta)$ be a log canonical pair such that 
$X$ is $\mathbb Q$-factorial over $W$. 
Assume that there exists $\Delta_0$ such that 
$(X, \Delta_0)$ is kawamata log terminal. 
Let $H$ be an effective $\mathbb R$-Cartier 
$\mathbb R$-divisor on $X$ such that 
$(X, \Delta+H)$ is log canonical and that 
$K_X+\Delta+H$ is nef over $W$. 
By Theorem \ref{a-thm22.3}, we can take a 
$(K_X+\Delta)$-negative extremal ray $R$ of 
$\NE(X/Y; W)$ such that 
$(K_X+\Delta+\lambda H)\cdot R=0$ if 
$(K_X+\Delta)$ is not nef over $W$. 
We can consider the contraction morphism 
$\varphi_R\colon X\to Z$ associated to $R$ over some open neighborhood 
of $W$ by Theorem \ref{a-thm7.3}. 
By Remark \ref{a-rem17.10}, we know that 
the desired flip always exists. 
We note that we can always find $\Delta'_0$ such that 
$(X, \Delta'_0)$ is kawamata log terminal and 
that $R$ is a $(K_X+\Delta'_0)$-negative 
extremal ray of $\NE(X/Y; W)$. 
Therefore, we 
can run a minimal model program similar to 
the one explained in 
Section \ref{a-sec13}. 
We call it 
the {\em{$(K_X+\Delta)$-minimal model program with 
scaling of $H$ over $Y$ around $W$}}. 
We sometimes simply say that 
it is the {\em{minimal model program with scaling}} 
if there is no danger of confusion. 
 \end{say}

It is well known that Theorem \ref{a-thm1.28} is 
an easy consequence of the minimal model 
program with scaling. 
The main ingredient of the following 
proof of Theorem \ref{a-thm1.28} 
is Theorem \ref{a-thm22.2} (2). 

\begin{proof}[Proof of Theorem \ref{a-thm1.28}]
Throughout this proof, we will freely shrink $Y$ around 
$W$ suitably without mentioning it explicitly. 
Let $H_2$ be a general $\pi_2$-ample 
$\mathbb Q$-divisor on $X_2$ and 
let $H_1$ be its strict transform on $X_1$. 
Then there is a small positive real number $\delta$ such that 
$(X_1, \Delta_1+\delta H_1)$ is kawamata log 
terminal. We take a general $\pi_1$-ample 
$\mathbb Q$-divisor $H'_1$ on $X_1$ such that 
$(X_2, \Delta_2+\delta H_2+\delta' H'_2)$ is kawamata 
log terminal for some positive real number 
$\delta'$, where $H'_2$ is the strict transform of $H'_1$. 
If $\delta$ is sufficiently small, 
$K_{X_1}+\Delta_1+\delta H_1 +\delta' H'_1$ is nef 
over $W$. 
We can run the $(K_{X_1}+\Delta_1+\delta H_1)$-minimal 
model program with scaling over $Y$ around $W$ 
(see \ref{a-say22.4}). 
After finitely many flips, we finally end up with $X_2$. 
On the other hand, by Theorem \ref{a-thm22.2} (2), 
we see that each step is a flop with respect to 
$K_{X_1}+\Delta_1$ if $\delta$ is sufficiently small. 
Therefore, we obtain the desired statement. 
\end{proof}

\section{On abundance conjecture}\label{a-sec23}

In this final section, we will treat the abundance conjecture 
for kawamata log terminal pairs in the complex analytic setting. 

Let us recall the following famous conjecture, which is 
one of the most difficult conjectures in the 
theory of minimal models. 

\begin{conj}[Abundance conjecture for projective 
kawamata log terminal pairs]\label{a-conj23.1}
Let $(X, \Delta)$ be a projective 
kawamata log terminal pair such that $K_X+\Delta$ is nef. 
Then $K_X+\Delta$ is semiample. 
\end{conj}

The main result of this section is as follows. 

\begin{thm}[see Theorem \ref{a-thm1.30}]\label{a-thm23.2} 
Assume that Conjecture \ref{a-conj23.1} holds in dimension 
$n$. 

Let $\pi\colon X\to Y$ be a projective surjective 
morphism of normal complex varieties with $\dim 
X-\dim Y=n$ and let 
$(X, \Delta)$ be a kawamata log terminal pair. 
Assume that $K_X+\Delta$ is $\pi$-nef. 
Let $W$ be a Stein compact subset of $Y$ such that 
$\Gamma (W, \mathcal O_Y)$ is noetherian. 
Then $K_X+\Delta$ is $\pi$-semiample 
over some open neighborhood of $W$. 
\end{thm}

Theorem \ref{a-thm23.2} says that we can reduce 
the abundance conjecture for projective morphisms 
of complex analytic spaces to the original abundance 
conjecture for projective varieties. 
Before we prove Theorem \ref{a-thm23.2}, we prepare some lemmas. 
The following lemma is Wilson's 
theorem (see \cite[Theorem 2.3.9]{lazarsfeld1}) 
for projective morphisms of complex varieties. 

\begin{lem}\label{a-lem23.3}
Let $f\colon Z\to Y$ be a projective morphism 
from a smooth complex variety $Z$ onto a normal Stein 
variety $Y$ and let $D$ be a Cartier divisor 
on $X$ such that $D$ is nef and big over $Y$. 
Let $y$ be any point of $Y$. 
Then, by replacing $Y$ with any relatively compact Stein 
open neighborhood of $y$, 
there exist a positive integer $m_0$ and an effective 
Cartier divisor $B$ on $Z$ such that 
$\mathcal O_Z(mD-B)$ is $f$-free. 
\end{lem}
\begin{proof}
We can take an $f$-very ample Cartier divisor $H$ on $Z$ after 
replacing $Y$ with any relatively compact Stein open neighborhood of $y$. 
Since $D$ is big over $Y$, there exists a positive integer $m_0$ such that 
$m_0D\sim A+B$, 
$A$ is $f$-ample, $B\geq 0$, and $A-(K_X+(n+1)H)$ is 
$f$-ample with $n=\dim Z$. 
Then, $R^if_*\mathcal O_Z(mD-B-iH)=0$ holds for every $i>0$ and 
every $m\geq m_0$ since 
$mD-B-iH-K_X\sim A-(K_X+(n+1)H)+(n+1-i)H$ is $f$-ample 
for $0<i\leq n+1$
(see, for example, Theorem \ref{a-thm5.1}). 
Therefore, by Castelnuovo--Mumford 
regularity (see, for example, \cite[Example 1.8.24]{lazarsfeld}), 
we obtain that $\mathcal O_Z(mD-B)$ is $f$-free for 
every $m\geq m_0$. 
\end{proof}

As an easy consequence, we obtain: 

\begin{lem}\label{a-lem23.4}
Let $f\colon Z\to Y$ be a projective morphism 
from a smooth complex variety $Z$ onto a normal Stein 
variety $Y$ and let $D$ be a Cartier divisor 
on $X$ such that $D$ is nef and big over $Y$. 
Assume that 
\begin{equation*}
R(Z, D):=\bigoplus _{m\in \mathbb N} f_*\mathcal O_Z(mD)
\end{equation*} 
is a locally finitely generated graded $\mathcal O_Y$-algebra. 
Then $D$ is $f$-semiample. 
\end{lem}

Lemma \ref{a-lem23.4} is well known for normal projective varieties 
(see, for example, \cite[Theorem 2.3.15]{lazarsfeld1}). 

\begin{proof}[Proof of Lemma \ref{a-lem23.4}]
By taking the Stein factorization, we may assume that 
$f_*\mathcal O_Z\simeq \mathcal O_Y$. 
Suppose that, for every positive integer $m$, $f^*f_*\mathcal O_Z
(mD)\to \mathcal O_Z(mD)$ is not surjective at $z\in Z$. 
We take an open neighborhood $U$ of $f(z)$ and a 
Stein compact subset $W$ of $Y$ such that 
$f(z)\in U\subset W$ and that $\mathcal O_Y(W)=\Gamma 
(W, \mathcal O_Y)$ is noetherian. 
If we make $U$ and $W$ sufficiently small, 
then $\Gamma \left (W, \bigoplus _{m\in \mathbb N} 
f_*\mathcal O_Z(mD)\right)\simeq 
\bigoplus_{m\in \mathbb N} f_*\mathcal O_Z(mD)(W)$ is 
a finitely generated $\mathcal O_Y(W)$-algebra. 
Therefore, there exists a positive integer $l$ such that 
$\bigoplus_{m\in \mathbb N} f_*\mathcal O_Z(mlD)(W)$ 
is generated by $f_*\mathcal O_Z(lD)(W)$. 
Let $V$ be a relatively compact Stein open neighborhood of $W$. Then, 
by Lemma \ref{a-lem23.3}, 
there exist $k>0$ and $g\in 
\Gamma (f^{-1}(V), \mathcal O_Z(klD))
=\Gamma (V, f_*\mathcal O_Z(klD))$ such that 
$C=(g=0)$ is an effective divisor on $f^{-1}(V)$ 
with $\mult _z C<k$. On the other hand, since 
$f_*\mathcal O_Z(klD)(W)$ is generated by $f_*\mathcal O_Z(lD)(W)$, 
$\mult _z C\geq k$ holds. 
It is a contradiction. This means that 
$D$ is $f$-semiample. 
\end{proof}

Let us prove Theorem \ref{a-thm23.2}. 

\begin{proof}[Proof of Theorem \ref{a-thm23.2}]
\setcounter{step}{0} In Step \ref{a-thm23.2-step1}, we will 
reduce the problem to the case where $K_X+\Delta$ is 
$\mathbb Q$-Cartier. Then, in Step \ref{a-thm23.2-step2}, 
we will prove that it is semiample by using the finite generation of 
log canonical rings. 
\begin{step}\label{a-thm23.2-step1}
We take a Stein open neighborhood $U$ of $W$ and 
a Stein compact subset $W'$ such that 
$U\subset W'$ and that $\Gamma (W', \mathcal O_Y)$ is 
noetherian. 
By Theorem \ref{a-thm22.2}, 
after shrinking $Y$ around $W'$, 
we can find $\mathbb Q$-divisors $\Delta_1, 
\ldots, \Delta_l$ on $X$ such that 
$K_X+\Delta=\sum _i r_i (K_X+\Delta_i)$, 
$(X, \Delta_i)$ is kawamata log terminal, 
$K_X+\Delta_i$ is nef over $W'$, and $r_i\in \mathbb R_{>0}$ with 
$\sum _i r_i =1$. 
Therefore, it is sufficient to prove that 
$K_X+\Delta_i$ is semiample over some open neighborhood 
of $W$. 
Hence, from now on, we may further assume that 
$K_X+\Delta$ is $\mathbb Q$-Cartier. 
Moreover, we may assume that there exists a positive 
integer $k$ such that 
$k(K_X+\Delta)$ is Cartier. 
Without loss of generality, we may assume that $\pi_*\mathcal O_X\simeq 
\mathcal O_Y$ by taking the Stein factorization. 
\end{step}
\begin{step}\label{a-thm23.2-step2}
Let $F$ be an analytically sufficiently general fiber of $\pi\colon X\to Y$. 
Then $(F, \Delta|_F)$ is kawamata log terminal with 
$K_F+\Delta|_F=(K_X+\Delta)|_F$. 
Hence, by assumption, $K_F+\Delta|_F$ is semiample. 
We put $L=k(K_X+\Delta)$. 
From now on, we will freely replace $Y$ with a smaller 
Stein open neighborhood 
of $W$ without mentioning it explicitly. 
We consider a meromorphic map 
$g\colon X\dashrightarrow Z_0$ over $Y$ associated to 
$\pi^*\pi_*\mathcal O_X(mL)\to \mathcal O_X(mL)$ for some 
sufficiently large and divisible positive integer $m$ such that 
$\dim Z_0=\dim Y+\kappa (F, K_F+\Delta|_F)$. 
As in the proof of \cite[Proposition 2.1]{kawamata-pluricanonical}, 
by using Hironaka's flattening theorem (see \cite{hironaka}), and so on, 
we can construct the following commutative diagram: 
\begin{equation*}
\xymatrix{
X\ar@{-->}[d]_-g & \ar[l]_-{\mu_0}
\ar[d]^-{g_1}X_1 &\ar[l]_-{\mu_1}\ar[d]^-{g_2} X_2 &
\ar[l]_-{\mu_2} \ar[d]^-{g_3}X_3 & \ar[l]_-{\mu_3}
\ar[d]^-\phi X' \\
Z_0 &\ar[l]^-{\pi_0} Z_1 &\ar[l]^-{\pi_1} Z& Z\ar@{=}[l]& Z \ar@{=}[l]
}
\end{equation*} 
which satisfies the following conditions. 
\begin{itemize}
\item[(i)] All the varieties in the diagram are projective over $Y$. 
\item[(ii)] $X_1$, $X'$, $Z_1$, and 
$Z$ are smooth, and $X_3$ is normal. 
\item[(iii)] $\mu_0$, $\mu_1$, 
$\mu_2$, $\mu_3$, and $\pi_1$ are projective 
bimeromorphic 
morphisms, $g_1$, $g_2$, $g_3$ and $\phi$ are 
surjective morphisms with connected fibers, and 
$\pi_0$ is a generically finite surjective morphism. 
\item[(iv)] $g_2$ is flat, $\mu_2$ is finite, and $g_3$ is equidimensional. 
\end{itemize}
We put $\mu:=\mu_0\circ \mu_1\circ \mu_2\circ \mu_3\colon 
X'\to X$. 
Then we finally get the following commutative diagram: 
\begin{equation*}
\xymatrix{
&X' \ar[dr]^-\phi\ar[dl]_-\mu& \\ 
X\ar[dr]_-\pi && Z \ar[dl]^-f\\ 
& Y&
}
\end{equation*}
such that 
\begin{itemize}
\item[(a)] $X'$ and $Z$ are projective over $Y$, 
\item[(b)] $X'$ and $Z$ are smooth, 
\item[(c)] $\mu$ is bimeromorphic and $\phi$ is a surjective morphism 
with connected fibers, and 
\item[(d)] there exists a Cartier divisor $D$ on $Z$ such that 
$D$ is nef and big over $Y$ with $a\mu^*L\sim b\phi^*D$ for some 
positive integers $a$ and $b$. 
\end{itemize} 
For the details, see the proof of \cite[Proposition 2.1]{kawamata-pluricanonical}. 
Since $R(X/Y, K_X+\Delta)$ is a locally finitely generated 
graded $\mathcal O_Y$-algebra by Theorem \ref{a-thm1.18}, 
$R(Z, D)$ is also a locally finitely generated graded $\mathcal O_Y$-algebra 
by Lemma \ref{a-lem2.26}. 
Hence, by Lemma \ref{a-lem23.4}, 
$D$ is $f$-semiample. 
This means that $L$ is $\pi$-semiample. 
\end{step}
Anyway, $K_X+\Delta$ is a finite $\mathbb R_{>0}$-linear 
combination of semiample 
Cartier divisors over some open 
neighborhood of $W$. 
This is what we wanted. 
\end{proof}

The abundance conjecture for log canonical pairs 
in the complex analytic setting seems to be much 
more difficult than the one for kawamata log terminal pairs. 


\end{document}